\documentclass{memo-l}

\usepackage{graphicx}
\usepackage{ifthen}
\usepackage{xparse}
\usepackage{longtable}
\usepackage{boldline}
\usepackage[table]{xcolor}
\usepackage{enumitem}
\usepackage[noadjust,nosort]{cite}
\usepackage{upref}
\usepackage[font={small},labelfont=md]{subfig}
\usepackage{amssymb}
\usepackage{mathtools}
\usepackage{mathrsfs}
\usepackage{dsfont}
	\newcommand{\one}{\mathds{1}}
\usepackage{stackengine}
	\newcommand\barbelow[1]{\stackunder[1pt]{$#1$}{\rule{1.96ex}{.075ex}}}
	\newcommand\barabove[1]{\stackon[1pt]{$#1$}{\rule{1.96ex}{.075ex}}}
\usepackage{hyperref}
	\hypersetup{colorlinks=true,linkcolor=blue,citecolor=blue,urlcolor=black}	
	
\newtheorem{thm}{Theorem}[chapter]
\newtheorem{lemma}[thm]{Lemma}
\newtheorem{prop}[thm]{Proposition}
\newtheorem{question}[thm]{Question}
\newtheorem{cor}[thm]{Corollary}
\newtheorem{claim}[thm]{Claim}
\newtheorem{theirthm}{Theorem} 

\theoremstyle{definition}
\newtheorem{defn}[thm]{Definition}
\newtheorem{eg}[thm]{Example}

\theoremstyle{remark}
\newtheorem{remark}[thm]{Remark}

\numberwithin{section}{chapter}
\numberwithin{equation}{chapter}

\newcommand{\stackref}[2]{
\readlist*\mylist{#1}
\stackrel{\mbox{\footnotesize\foreachitem\x\in\mylist[]{\ifnum\xcnt=1\else,\fi\eqref{\x}}}}{#2}
}
\newcommand{\stackrefp}[2]{
\readlist*\mylist{#1}
\stackrel{\hphantom{\mbox{\footnotesize\foreachitem\x\in\mylist[]{\ifnum\xcnt=1\else,\fi\eqref{\x}}}}}{#2}
}
\newcommand{\stackrefpp}[3]{
\readlist*\mylist{#1}
\readlist*\mylistt{#2}
\stackrel{\parbox{\widthof{\footnotesize\foreachitem\x\in\mylistt[]{\ifnum\xcnt=1\else,\fi\eqref{\x}}}}{\centering\footnotesize\foreachitem\x\in\mylist[]{{\ifnum\xcnt=1\else,\fi\eqref{\x}}}}}{#3}
}

\newcommand{\eq}[1]{\begin{align*} #1 \end{align*}}
\newcommand{\eeq}[1]{\begin{align} \begin{split} #1 \end{split} \end{align}}
\newcommand{\eeqs}[2]{\begin{subequations}\label{#1}\begin{align} #2 \end{align}\end{subequations}}

\newcommand{\eps}{\varepsilon}
\newcommand{\vphi}{\varphi}

\newcommand{\E}{\mathbb{E}}

\newcommand{\N}{\mathbb{N}}
\renewcommand{\P}{\mathbb{P}}
\newcommand{\Q}{\mathbb{Q}}
\newcommand{\R}{\mathbb{R}}
\renewcommand{\S}{\mathbb{S}}
\newcommand{\T}{\mathbb{T}}
\newcommand{\Z}{\mathbb{Z}}

\renewcommand{\AA}{\mathcal{A}}
\newcommand{\BB}{\mathcal{B}}

\newcommand{\FF}{\mathcal{F}}

\newcommand{\HH}{\mathcal{H}}
\newcommand{\II}{\mathcal{I}}

\newcommand{\KK}{\mathcal{K}}
\newcommand{\LL}{\mathcal{L}}

\newcommand{\OO}{\mathcal{O}}
\newcommand{\PP}{\mathcal{P}}

\newcommand{\RR}{\mathcal{R}}
\renewcommand{\SS}{\mathcal{S}}
\newcommand{\TT}{\mathcal{T}}

\newcommand{\WW}{\mathcal{W}}

\newcommand{\YY}{\mathcal{Y}}

\newcommand{\BBB}{\mathscr{B}}
\newcommand{\CCC}{\mathscr{C}}
\newcommand{\DDD}{\mathscr{D}}

\newcommand{\FFF}{\mathscr{F}}

\newcommand{\OOO}{\mathscr{O}}
\newcommand{\PPP}{\mathscr{P}}

\newcommand{\TTT}{\mathscr{T}}

\newcommand{\Asf}{\mathsf{A}}
\newcommand{\Bsf}{\mathsf{B}}
\newcommand{\Csf}{\mathsf{C}}
\newcommand{\Dsf}{\mathsf{D}}
\newcommand{\Esf}{\mathsf{E}}

\newcommand{\Usf}{\mathsf{U}}
\newcommand{\Vsf}{\mathsf{V}}

\newcommand{\vc}[1]{{\boldsymbol #1}}

\newcommand{\wt}[1]{\widetilde{#1}}
\newcommand{\wh}[1]{\widehat{#1}}

\newcommand{\KL}[3][]{D_{\mathrm{KL}}#1(#2\: #1|#1|\: #3#1)}

\DeclareMathOperator*{\esssup}{ess\,sup}

\DeclareMathOperator{\e}{e} 
\newcommand{\cc}{\mathrm{c}} 
\newcommand{\dd}{\mathrm{d}} 
\newcommand{\oo}{\mathrm{o}} 
\newcommand{\pp}{\mathfrak{R}}
\newcommand{\ww}{\mathrm{w}} 
\DeclareMathOperator{\TV}{TV} 
\newcommand{\bv}{\mathrm{bv}} 

\DeclareMathOperator{\Geo}{Geo}

\newcommand{\pert}{\mathrm{pert}}
\renewcommand{\v}{\mathbf{v}}	
\newcommand{\bxi}{{\boldsymbol\xi}}
\newcommand{\zetab}{{\boldsymbol\zeta}}
\newcommand{\Xib}{{\boldsymbol\Xi}}
\newcommand{\af}{\mathfrak{h}}
\newcommand{\tf}{\mathfrak{t}}
\newcommand{\mf}{\mathfrak{m}}
\newcommand{\qf}{\mathfrak{q}}
\newcommand{\Avc}{\vc A}
\newcommand{\Rvc}{\vc R}
\newcommand{\Lvc}{\vc L}
\newcommand{\Qvc}{\vc Q}
\newcommand{\Bbf}{\mathbf{B}}

\DeclareRobustCommand{\SkipTocEntry}[5]{} 

\makeindex

\begin{document}

\frontmatter

\title[Variational Formula and Applications in FPP]{Empirical Measures, Geodesic Lengths, and a Variational Formula in First-Passage Percolation}

\author{Erik Bates}
\address{Department of Mathematics \newline\indent University of Wisconsin--Madison \newline\indent Van Vleck Hall \newline\indent 480 Lincoln Drive \newline\indent Madison, Wisconsin 53706-1324 \newline}
\email{ewbates@wisc.edu}
\urladdr{https://www.ewbates.com/}
\thanks{This research was partially supported by NSF grant DMS-1902734}

\date{}




\begin{abstract}
This monograph resolves---in a dense class of cases---several open problems concerning geodesics in i.i.d.~first-passage percolation on $\mathbb{Z}^d$.
Our primary interest is in the empirical measures of edge-weights observed along geodesics from $0$ to $n{\boldsymbol\xi}$, where ${\boldsymbol\xi}$ is a fixed unit vector.
For various dense families of edge-weight distributions, we prove that these empirical measures converge weakly to a deterministic limit as $n\to\infty$, answering a question of \mbox{Hoffman}.
These families include arbitrarily small $L^\infty$-perturbations of any given distribution, almost every finitely supported distribution, uncountable collections of continuous distributions, and certain discrete distributions whose atoms can have any prescribed sequence of probabilities.
Moreover, the constructions are explicit enough to guarantee examples possessing certain features, for instance: both continuous and discrete distributions whose support is all of $[0,\infty)$, and distributions given by a density function that is $k$-times differentiable.
All results also hold for ${\boldsymbol\xi}$-directed infinite geodesics.
In comparison, we show that if $\mathbb{Z}^d$ is replaced by the infinite $d$-ary tree, then any distribution for the weights admits a unique limiting empirical measure along geodesics.
In both the lattice and tree cases, our methodology is driven by a new variational formula for the time constant, which requires no assumptions on the edge-weight distribution.
Incidentally, this variational approach also allows us to obtain new convergence results for geodesic lengths, which have been unimproved in the subcritical regime since the seminal 1965 manuscript of \mbox{Hammersley} and \mbox{Welsh}.
\end{abstract}

\maketitle

\tableofcontents

\counterwithout{footnote}{chapter}
\counterwithout{figure}{chapter}

\chapter*{Outline of Manuscript}

In a broad view, this monograph proposes a method for studying asymptotic properties of geodesics in first-passage percolation (FPP). 
We are motivated by several well-known open problems.
Chief among these is the following question: as the endpoint of a geodesic is brought to infinity in a fixed direction, does the empirical measure of edge-weights appearing along the path converge weakly to some limit?
In brief, we show in Section \ref{main_results} that the answer is \textit{yes} 
for edge-weight distributions belonging to any one of a variety of dense collections.
Rather than rush the reader to these statements, we provide here a roadmap which highlights our methodology and its implications for other long-standing problems in FPP.
On the other hand, for the fastest possible entry into our main results, see Theorem \ref{dense_thm} after reading \eqref{hat_nu_plus_def} and Definition \ref{good_def}.

For the purposes of this outline, we use the following notation despite having not yet defined the model:
$\tau_e\geq0$ is an edge-weight, $\LL$ is the law of $\tau_e$, and $F$ its distribution function.
For a unit vector $\bxi$, the time constant in the $\bxi$-direction is $\mu_{\bxi}$.
The geodesics under consideration are either finite and between $0$ and $[n\bxi]$ (with $n\to\infty$), or infinite and ${\bxi}$-directed. 
Finally, $p_\cc(\Z^d)$ is the critical probability for Bernoulli bond percolation on $\Z^d$.
A complete list of symbols is provided on page \pageref{list_of_symbols}.

\addtocontents{toc}{\SkipTocEntry}
\section*{Chapter \ref{intro} (\textit{Introduction: Definitions and Main Questions})}
The FPP model is formally defined in Section \ref{model_def}, the time constant $\mu_\bxi$ in Section \ref{time_constant_sec}, and geodesics in Section \ref{geo_intro}.
After these preliminaries, Section \ref{empirical_intro} introduces the notion of empirical measures and the main motivation for this monograph: Question \ref{main_q}, which was paraphrased above.
We also offer a first example:
\setlength{\leftmargini}{2.5em}
\begin{itemize}
\item Theorems \ref{zero_thm} and \ref{zero_thm_infty}: If $F(0)\geq p_\cc(\Z^d)$, then empirical measures along geodesics converge weakly to the Dirac delta mass at $0$.
\end{itemize}
On the other hand, if $0<F(0)<p_\cc(\Z^d)$, then variability in the number of zero-weight edges prevents convergence of empirical measures; we thus turn our focus to empirical measures of only the \textit{positive}-weight edges, rephrasing Question \ref{main_q} as Question \ref{main_q_2}.
Before developing our approach to answering these questions, we discuss in Section \ref{geo_length_sec} the related matter of geodesic lengths.
The following statements are proved in Section \ref{one_d_subsec} en route to our main results concerning empirical measures:
\begin{itemize}
\item Theorem \ref{krs_thm}: Replace every edge-weight $\tau_e$ by $\tau_e+h$.
Then the scaled geodesic length converges for all but countably many $h>0$.
\item Theorem \ref{new_krs_thm}: Replace every $\tau_e$ by $\tau_e +h\one_{\{\tau_e>0\}}$.
The scaled number of positive-weight edges converges for all but countably many $h>0$.
\end{itemize}
The first theorem was originally obtained in restricted cases by \mbox{Hammersley} and \mbox{Welsh} \cite{hammersley-welsh65}.
Partial removals of the restrictions have only been realized by improved generality for the time constant;
our version has removed all restrictions.

Meanwhile, the second theorem is a new observation.
By allowing for an atom at zero, it works to make progress on the following open problem discussed in Example \ref{supercritical_length_eg}. 
If $0<F(0)<p_\cc(\Z^d)$, does the scaled length of the shortest geodesic from $0$ to $[n\bxi]$ converge?
In both theorems, however, the perturbation of $\tau_e$ introduces a ``mass gap" above $0$.
We will give the first proof that this gap is not a necessary feature, a fact that is unachievable from the approach of \mbox{Hammersley} and \mbox{Welsh}.
\begin{itemize}
\item Theorem \ref{no_gap_thm}: There exist edge-weight distributions satisfying $F(h)>F(0)$ for all $h>0$, such that the scaled geodesic length converges.
\end{itemize}

\addtocontents{toc}{\SkipTocEntry}
\section*{Chapter \ref{var_form_sec} (\textit{Variational Formula for the Time Constant})} 

The manuscript's theoretical engine is a new variational formula for $\mu_\bxi$, stated in \eqref{var_form_eq} as part of Theorem \ref{var_form_thm}.
This formula is advantaged in three key ways: it requires no assumptions on $\LL$, it minimizes a \textit{linear} functional, and the underlying constraint set has no dependence on $\LL$.
These features combine with the following fact (stated imprecisely) to make for an evidently useful tool in studying asymptotic properties of geodesics:
\begin{itemize}
\item Theorem \ref{minimizers_thm}: If $F(0)<p_\cc(\Z^d)$, then every sequence of geodesics admits a subsequence whose empirical measures converge to a minimizer of \eqref{var_form_eq}.
An analogous statement holds for infinite ${\bxi}$-directed geodesics.
\end{itemize}

\addtocontents{toc}{\SkipTocEntry}
\section*{Chapter \ref{applications} (\textit{Applications of Variational Formula})}

As will be described in Section \ref{sketch_and_statement}, our perspective is to view $\tau_e$ as some deterministic function $\tau(U_e)$ of a uniform random variable $U_e$ on $[0,1]$.
The simple but crucial observation---stated as Lemma \ref{concavity_lemma}---is that the variational formula \eqref{var_form_eq} is concave in the (nonnegative) function $\tau$.
It is then not difficult to prove that empirical measures along geodesics have a unique limit when $\tau$ belongs to any of the following families: 
\begin{itemize}
\item Theorem \ref{finite_emp}: almost every function taking finitely many values.
\item Theorem \ref{countable_emp}: various functions taking countably many values.
\item Theorem \ref{continuous_emp}: a dense $G_\delta$ subset of $C^k([0,1],[0,\infty))$ for any $k\geq0$.
\item Theorem \ref{bounded_emp}:  a dense $G_\delta$ set of $L^\infty$-perturbations of any given $\tau$.
\end{itemize}
See Remarks \ref{optimality_remark}, \ref{good_examples_1}, \ref{good_examples_2}, \ref{mix_remark} and Examples \ref{dense_discrete}, \ref{prescribed_discrete}, \ref{continuous_eg}, \ref{better_choice_eg} for further interpretation.

\addtocontents{toc}{\SkipTocEntry}
\section*[Variational formula and applications in FPP]{Chapter \ref{tree_sec} (\textit{First-Passage Percolation on $d$-ary Tree})}

If the lattice $\Z^d$ is replaced by the infinite complete $d$-ary tree $\T_d$,
then the constraint set for our variational formula is a sublevel set of relative entropy.
In particular, this set is strictly convex and so admits a unique minimizer.
This allows us to completely answer Question \ref{main_q} in the tree case:
\begin{itemize}
\item Theorem \ref{tree_thm}: In FPP on $\T_d$, \textit{any} weight distribution admits a unique limit for the empirical measures along geodesics. 
\end{itemize}

\addtocontents{toc}{\SkipTocEntry}
\section*[Variational formula and applications in FPP]{Chapter \ref{modification_sec} (\textit{Negative Weights and Passage Times along Geodesics})}

Of possible independent interest are three utilitarian results. 
The first of these allows us to handle edge-weight distributions supported at $0$, the second permits us to forgo any moment assumption, and the third enables us to parlay results about finite geodesics into ones about infinite geodesics.
\begin{itemize}
\item Proposition \ref{negative_thm}: If $F(0)<p_\cc(\Z^d)$, then a shape theorem holds even if edge-weights are perturbed slightly in the negative direction.
\item Proposition \ref{replacement_thm}: Passage times along geodesics between $0$ and $[n\bxi]$ scale to the time constant $\mu_\bxi$ once a small number of edges are removed.
\item Proposition \ref{addition_thm}: When appropriately scaled, the passage time along any $\bxi$-directed infinite geodesic converges to $\mu_\bxi$.
\end{itemize}
These statements are not novel in spirit, but were unavailable from the literature in their present form.
The proofs are largely independent from the rest of the manuscript and use mostly standard tools. 

\addtocontents{toc}{\SkipTocEntry}
\section*[Variational formula and applications in FPP]{Chapters \ref{constraint_sec}, \ref{var_form_proof}, and \ref{tree_proof} (\textit{remaining proofs})} 

The constraint set for the variational formula is constructed in Chapter \ref{constraint_sec}, and then Chapter \ref{var_form_proof} proves the actual formula in tandem with Theorem \ref{minimizers_thm}.
Finally, the proof of Theorem \ref{tree_thm} is given in Chapter \ref{tree_proof} and incidentally leads to the following deterministic statement unrelated to FPP:
\begin{itemize}
\item Theorem \ref{weak_to_strong_thm}:
If a sequence of probability measures on $[0,1]$ converges weakly, and their relative entropies with respect to some fixed Radon probability measure are uniformly bounded, then the sequence converges strongly.
\end{itemize}
I have not been able to locate any previous observation of this fact.

\chapter*{Acknowledgments}
I am very grateful to \mbox{Louigi} \mbox{Addario-Berry}, \mbox{Sourav} \mbox{Chatterjee}, \mbox{Alexander} \mbox{Dunlap}, \mbox{Shirshendu} \mbox{Ganguly}, \mbox{Nicholas} \mbox{Miller}, \mbox{Timo} \mbox{Sepp{\"a}l{\"a}inen}, \mbox{Xiao} \mbox{Shen}, and \mbox{Bernd} \mbox{Sturmfels} for invaluable discussions.
I also kindly thank \mbox{Michael} \mbox{Damron}, \mbox{Christopher} \mbox{Hoffman}, \mbox{Arjun} \mbox{Krishnan},  \mbox{Firas} \mbox{Rassoul-Agha}, \mbox{Allan} \mbox{Sly}, and \mbox{Zhe} \mbox{Wang} for their feedback on a preliminary draft.
I thank a referee whose insightful comments and questions lead to numerous improvements. 
Finally, special thanks are due to \mbox{Wai-Kit} \mbox{Lam}, whose careful reading, thoughtful discussion, and pointers to several useful references significantly improved the manuscript.

\mainmatter
\chapter{Introduction: Definitions and Main Questions} \label{intro}

\section{First-passage percolation model} \label{model_def}
Let $E(\Z^d)$ \label{edges_def}
denote the undirected edges of the integer lattice $\Z^d$, $d\geq 2$.
Consider a family of i.i.d.~nonnegative random variables $\{\tau_e :\, e\in E(\Z^d)\}$, \label{edge_weights_def}
called the \textit{edge-weights}, defined on some complete probability space $(\Omega,\FF,\P)$. \label{cps_def}
The shared law will be denoted by $\LL$, \label{weight_law_def_1}
and the associated distribution function will be written as 
\eeq{ \label{distribution_fnc_def}
F(t) \coloneqq \P(\tau_e \leq t) = \LL((-\infty,t]).
}
For each pair $x,y\in\Z^d$, let $\PP(x,y)$ \label{path_xy_def}
denote the collection of all self-avoiding nearest-neighbor paths starting at $x$ and ending at $y$.
(A path is viewed as a set of edges.)
The \textit{passage time} between $x$ and $y$ is the random quantity
\eeq{ \label{fpp_def}
T(x,y) \coloneqq \inf_{\gamma \in \PP(x,y)} T(\gamma), \quad \text{where} \quad T(\gamma) \coloneqq \sum_{e\in\gamma} \tau_e.
}
We allow the empty path in $\PP(x,x)$ so that $T(x,x) = 0$.
So that we are not restricted to integer coordinates, for general $x\in\R^d$ we will write $[x]$ \label{integer_approx_def}
to denote the unique element of $\Z^d$ such that $x\in[x]+[0,1)^d$.
With this notation, we take $\PP(x,y) = \PP([x],[y])$ for $x,y\in\R^d$, so that $T(x,y)=T([x],[y])$.

\section{Time constant and limit shape} \label{time_constant_sec}
A classical result is the existence of a \textit{time constant}, which records the law of large numbers for passage times between $0$ and $[n\v]$.
In the theorem quoted below, $p_\cc(\Z^d)$ denotes the critical probability for bond percolation on $\Z^d$. \label{pczd_def}

\begin{theirthm} \label{time_constant_thm}
\textup{\cite[Thm.~3.1, Thm.~6.1]{kesten86}}
For any $\v\in\R^d$, there is a constant $\mu_{\v}=\mu_{\v}(\LL)\in[0,\infty)$ such that
\eeq{ \label{time_constant_def}
\frac{T(0,n \v)}{n} \to \mu_\v \quad \text{in probability as $n\to\infty$}.
}
Furthermore, $\mu_{\v}>0$ if and only if $F(0) < p_\cc(\Z^d)$.
\end{theirthm}

If $F(0)<p_\cc(\Z^d)$, then $\v\mapsto\mu_{\v}$ is a norm on $\R^d$, and
the unit ball $B_0\coloneqq \{\v\in\R^d:\mu_{\v} \leq 1\}$ \label{limit_shape_def}
under this norm is called the \textit{limit shape}.
If instead $F(0)\geq p_\cc(\Z^d)$, then $\mu_{\v} = 0$ for all ${\v}$, in which case $B_0=\R^d$.
In either case, we have $\mu_{\alpha\v}=\alpha\mu_\v$ for any $\alpha\geq0$, and so $\mu_{(\cdot)}$ is completely determined by its values on the unit sphere $\S^{d-1}$. \label{unit_sphere_def}
Therefore, we will henceforth restrict our attention to any fixed unit vector $\bxi\in\S^{d-1}$.

\begin{remark}
It is common in FPP to restrict attention to the case when $\bxi$ is equal to $\mathbf{e}_1$, the first standard basis vector. \label{sbv_def}
In order to make our results available for general directions, we will not afford ourselves this convenience.
Consequently, one annoying detail is that the location $[n\bxi]$ can change at non-integer values of $n$, and in no periodic fashion.
To avoid rogue subsequences, we specify now that \underline{all limits as $n\to\infty$ hold even if $n$ varies \textit{continuously}}, unless $n$ is explicitly declared an integer.
We continue to use the symbol $n$ only to match standard notation, and the reader is encouraged to imagine that $n$ is a positive integer.
Indeed, nothing is lost conceptually in doing so, as most quantities we consider will change at only countably many values of $n$.
In any case, the symbols $i$, $j$, $k$, $\ell$, will always denote integers.
\end{remark}

\section{Geodesics, finite and infinite} \label{geo_intro}
We say that $\gamma\in\PP(x,y)$ is a \textit{geodesic} if $T(\gamma) = T(x,y)$; that is, $\gamma$ achieves the infimum in \eqref{fpp_def}.
Let $\Geo(x,y)$ denote the set of all \label{geo_def}
geodesics between $[x]$ and $[y]$.
To give our results context, it is important to know that geodesics actually exist.

\begin{theirthm}[{\cite[Cor.~1.3]{wierman-reh78}, \cite[Sec.~9.23]{kesten86}, \cite[Thm.~2]{zhang95}}] 
\label{geo_exist_thm}
Almost surely, $\Geo(x,y)$ is nonempty for all $x,y\in\Z^d$, provided one of the following holds:
\begin{enumerate}[label=\textup{(\alph*)}]
\item \label{geo_exist_thm_a} $d=2$;
\item \label{geo_exist_thm_b} $F(0)<p_\cc(\Z^d)$;
\item \label{geo_exist_thm_c} $F(0)>p_\cc(\Z^d)$; 
\item \label{geo_exist_thm_d} $F(0)=F(h)$ for some $h>0$.
\end{enumerate}
\end{theirthm}

\begin{remark}
For $d\geq3$, it remains a long-standing open problem to prove the existence of geodesics if $F(0)=p_\cc(\Z^d)$ without further assumptions.
There does exist a sufficient condition for existence (see \cite[Thm.~8.1.8]{zhang99} and \cite[pg.~74]{auffinger-damron-hanson17}), although all the scenarios currently known to satisfy this condition are already included above in \ref{geo_exist_thm_a} or \ref{geo_exist_thm_d}.
Nevertheless, in the very special cases for which Theorem \ref{geo_exist_thm} is not known, the results of this manuscript concerning geodesics can be read as conditional on their existence.
The same comment applies to infinite geodesics, which are discussed below.
\end{remark}

\begin{remark}
I am not aware of any written proof for case \ref{geo_exist_thm_d}, but it is essentially trivial: Take any sequence of paths $\gamma_k\in\PP(x,y)$ such that $T(\gamma_k)\searrow T(x,y)$ as $k\to\infty$.
If $|\gamma_k|$ does not diverge to $\infty$, then there is some $\gamma\in\PP(x,y)$ equal to $\gamma_k$ for  infinitely many $k$, and so $T(\gamma) = T(x,y)$.
Otherwise, take some pair of edges $e_1^x$ and $e_1^y$, containing $x$ and $y$ respectively, that appear as the first and last edges of $\gamma_k$ for infinitely many $k$.
After passing to a subsequence of such $k$, we can identify $e_2^x$ and $e_2^y$, adjacent to $e_1^x$ and $e_1^y$ respectively, that appear simultaneously in infinitely many $\gamma_k$.
Repeating this indefinitely, we will identify infinite paths $\gamma^{x}$ and $\gamma^{y}$ starting at $x$ and $y$.
Since $|\gamma_k|\to\infty$, these paths will never intersect, and so $T(\gamma^{x})+T(\gamma^{y}) = T(x,y)$.
Because $\tau_e\geq h$ whenever $\tau_e>0$, each of $\gamma^{x}$ and $\gamma^{y}$ can only use finitely many nonzero-weight edges; in particular, both paths eventually enter an infinite connected cluster of zero-weight edges.
As this cluster is almost surely unique (e.g.~\cite[Thm.~8.1]{grimmett99}), we can link the two points of entry (via zero-weight edges) to form a finite path $\gamma$ of passage time $T(\gamma) = T(x,y)$.
\end{remark}

An infinite path of the form $\Gamma = \{e_1,e_2,\dots\}$, where $e_1$ is incident with $e_2$, $e_2$ with $e_3$, and so on, is an \textit{infinite geodesic} if for each $\ell$, the subpath 
\eeq{ \label{sub_of_infinite}
\Gamma^{(\ell)} \coloneqq \{e_1,\dots,e_\ell\}
} 
is a finite geodesic. 
Let us write $x_\ell$ for the vertex shared by $e_\ell$ and $e_{\ell+1}$.
Then $\Gamma$ is said to be $\bxi$-\textit{directed} 
if $x_\ell/\|x_\ell\|_2\to \bxi$ as $\ell\to\infty$.
More generally, for $\Xib\subset\S^{d-1}$, we say $\Gamma$ is $\Xib$-\textit{directed} if the limit points of the sequence $(x_\ell/\|x_\ell\|_2)_{\ell\geq1}$ are contained in $\Xib$.

It is known in $d=2$ that for each linear face of the limit shape $B_0$, corresponding say to the directions $\Xib\subset\S^{1}$, there is at least one $\Xib$-directed geodesic starting at the origin \cite[Thm.~1.1]{damron-hanson14}.
For a conjectural picture and the most recent results, 
the reader is referred to \cite{ahlberg-hoffman?,alexander?,brito-damron-hanson?}.

\begin{remark} \label{LPP_remark}
The methods of this monograph could be adapted to last-passage percolation (LPP) and potentially other stochastic optimization models.
In one respect, the LPP setting is simpler than FPP, as the paths under consideration all have the same length: see \cite{martin06} for a definition.
Consequently, more is known regarding the existence and coalescence of infinite geodesics than in FPP; for the state of the art, see \cite{janjigian-rassoul-seppalainen?} and references therein.
On the other hand, ($1+1$)-dimensional LPP admits integrable models whose rich algebraic structure or stationarity properties allow for much finer asymptotic analysis.
For one such case, similar objectives as those of this manuscript are pursued by Martin, Sly, and Zhang \cite{martin-sly-zhang?}; the methods, however, are very different.
Whereas we develop results for abstract dense collections of weight distributions, \cite{martin-sly-zhang?} focuses on a fixed, exactly solvable model.
\end{remark}

\section{Empirical measures associated to paths} \label{empirical_intro}
For a topological space $\T$ with its \mbox{Borel} sigma-algebra $\BB$, \label{borel_1}
a measurable $f \colon \T\to\R$, and a measure $\nu$ on $(\T,\BB)$, let us write 
\eeq{ \label{f_nu_def} 
\langle f,\nu\rangle \coloneqq \int_\T f(t)\ \nu(\dd t).
}
Recall that a net of finite measures $(\nu_\alpha)$ on $(\T,\BB)$ is said to \textit{converge weakly} to $\nu$, and we write $\nu_\alpha\Rightarrow\nu$, if
\eq{ 
\lim_{\alpha} \langle \vphi,\nu_\alpha\rangle = \langle \vphi,\nu\rangle \quad \text{for every bounded, continuous $\vphi \colon \T\to\R$}.
}
Let $\PPP$ denote the space of \mbox{Borel} probability measures on $[0,\infty)$, \label{PPP_def}
equipped with the topology of weak convergence.
Given a realization of the edge-weights $\tau_e$, each finite, nonempty path $\gamma$ can be associated to an element of $\PPP$, namely the following \textit{empirical measure}:
\eeq{ \label{hat_nu_def}
\hat\nu_\gamma \coloneqq \frac{1}{|\gamma|}\sum_{e\in\gamma}\delta_{\tau_e},
}
where $|\gamma|$ is the \textit{length} of $\gamma$ (the number of edges it contains), \label{gamma_length_def}
and $\delta_t$ is the Dirac delta measure at $t\in[0,\infty)$. \label{dirac_def}
In this notation, the passage time along $\gamma$ can be expressed as
\eeq{ \label{T_gamma}
T(\gamma) = |\gamma|\int_0^\infty t\ \hat\nu_\gamma(\dd t) = |\gamma|\langle t,\hat\nu_\gamma\rangle.
}
The main purpose of this monograph is to address the following question, a version of which was raised by C. Hoffman during a 2015 workshop at the American Institute of Mathematics \cite{aim_15}.
Recall that $\bxi$ denotes a unit vector. 

\begin{question} \label{main_q}
If $\gamma_n \in \Geo(0,n\bxi)$, does $\hat\nu_{\gamma_n}$ converge weakly as $n\to\infty$?
\end{question}

Our answer is generically \textit{yes}, where ``generically" is given a variety of meanings in Section \ref{main_results}.
Nevertheless, we should be clear from the outset that exceptional scenarios do exist (see Remark \ref{optimality_remark}), a fact which makes Question \ref{main_q} all the more intriguing.
For a straightforward first example, though, we look to the so-called critical and supercritical cases.

\begin{thm}
\label{zero_thm}
If $F(0) \geq p_\cc(\Z^d)$, then almost surely the following holds for any sequence $(x_k)_{k\geq1}$ in $\Z^d$ such that $\|x_k\|_2\to\infty$:
\eq{
\hat\nu_{\gamma_k} \Rightarrow \delta_0 \quad \text{as $k\to\infty$}, \quad \text{for any choice of $\gamma_k\in\Geo(0,x_k)$}.
}
\end{thm}

Indeed, Theorem \ref{zero_thm} is fairly clear if \eqref{time_constant_def} is known to hold almost surely, for then it is not difficult to deduce $\hat\nu_{\gamma_n}\Rightarrow\delta_{0}$ from \eqref{T_gamma}.\footnote{The convergence in \eqref{time_constant_def} is almost sure if and only if $\E\min(\tau_e^{(1)},\dots,\tau_e^{(2d)})<\infty$, where $\tau_e^{(1)},\dots,\tau_e^{(2d)}$ are i.i.d.~copies of $\tau_e$, at least when $\mathbf{v}$ has rational coordinates \cite[Lem.~2.3]{auffinger-damron-hanson17}.}
But absent any moment assumption on $\tau_e$, the ratio $T(0,n\bxi)/n$ may actually be unbounded in $n$. 
Fortunately, Proposition \ref{replacement_thm} says we can recover almost sure convergence to the time constant by deleting $O(1)$ many edges from the beginnings of geodesics, and $o(n^\eps)$ many edges from the ends.
The simple proof of Theorem \ref{zero_thm} can then be carried out in all cases, as we do in Theorem \hyperref[var_formula_thm_super_c1]{\ref*{var_formula_thm_super}\ref*{var_formula_thm_super_c1}}.

\begin{eg}[\textit{Supercritical nonzero atom}]  \label{nonzero_atom_eg}
A situation similar to Theorem \ref{zero_thm} can occur in certain directions when there is $\mathfrak{b}>0$ such that
\eeq{ \label{inf_supp_nonzero}
F(\mathfrak{b}) > 0 \quad \text{and} \quad F(\mathfrak{b}-) = 0.
}
Namely, if $F(\mathfrak{b})$ exceeds the critical probability for oriented bond percolation on $\Z^2$, then $\mu_{\bxi} = \mathfrak{b}\|\bxi\|_1$ for all $\bxi$ within the so-called percolation cone (see \cite{durrett84}).
This corresponds to a flat edge in the limit shape, as discovered by \mbox{Durrett} and \mbox{Liggett} \cite[Thm.~9]{durrett-liggett81} and identified precisely by \mbox{Marchand} \cite[Thm.~1.3]{marchand02}.
Considering that \eqref{inf_supp_nonzero} implies every $\gamma\in\PP(0,x)$ satisfies $T(\gamma)\geq \mathfrak{b}\|x\|_1$, the equality $\mu_{\bxi} = \mathfrak{b}\|\bxi\|_1$ forces $\hat\nu_{\gamma_n}\to\delta_{\mathfrak{b}}$ for any choice of $\gamma_n\in\Geo(0,n\bxi)$.
\end{eg}

With the critical and supercritical cases resolved in Theorem \ref{zero_thm}, we can focus on answering Question \ref{main_q} in the subcritical regime.
But first we need to account for the complication created by an atom at zero.
Indeed, when $0 < F(0) < p_\cc(\Z^d)$, Question \ref{main_q} has a negative answer for the reason that a geodesic will occasionally pass through a box consisting entirely of zero-weight edges.
By following either longer or shorter routes through such boxes, the geodesic can raise or lower the proportion of its edges which have weight $0$.
Recent work of \mbox{Krishnan}, \mbox{Rassoul-Agha}, and \mbox{Sepp\"al\"ainen} \cite{krishnan-rassoul-seppalainen?} makes this idea precise, to the effect of extending results from \cite{steele-zhang03}.

This discussion leads us to ask a refined version of Question \ref{main_q}.
Let us define the following empirical measure, modified from \eqref{hat_nu_def} to now include only nonzero-weight edges:
\eeq{ \label{hat_nu_plus_def}
\hat\nu_\gamma^+ \coloneqq \frac{1}{|\{e\in\gamma:\, \tau_e\neq 0\}|}\sum_{e\in\gamma:\, \tau_e\neq 0}\delta_{\tau_e}.
}
We take the convention that $\hat\nu_\gamma^+$ is the zero measure if $\tau_e=0$ for all $e\in\gamma$.

\begin{question} \label{main_q_2}
If $F(0)<p_\cc(\Z^d)$ and $\gamma_n\in\Geo(0,n\bxi)$, does $\hat\nu_{\gamma_n}^+$ converge weakly as $n\to\infty$?
\end{question}

Notice that $\hat\nu_\gamma^+ = \hat\nu_\gamma$ when $F(0) = 0$, in which case Questions \ref{main_q} and \ref{main_q_2} are equivalent.
So let us make the following definition,
recalling that $\LL\in\PPP$ is the law of $\tau_e$.

\begin{defn} \label{good_def}
For $\bxi\in\S^{d-1}$, we will write $\LL\in\PPP_\mathrm{emp}(\bxi)$ if either 
\begin{itemize}
\item[(a)] $F(0)\geq p_\cc(\Z^d)$; or
\item[(b)] $F(0)<p_\cc(\Z^d)$ and there is a deterministic measure $\hat\nu^+ = \hat\nu^+(\LL,\bxi)\in\PPP$ such that, almost surely we have the following weak convergence:
\eq{
\hat\nu_{\gamma_n}^+ \Rightarrow \hat\nu^+ \quad \text{as $n\to\infty$}, \quad \text{for any choice of $\gamma_n\in\Geo(0,n\bxi)$}.
}
\end{itemize}
\end{defn}

\begin{eg} \label{trivial_emp_eg}
Any distribution of the form $\LL=p\delta_0+(1-p)\delta_t$, with $p\in(0,1)$ and $t>0$, trivially belongs to $\PPP_\mathrm{emp}(\bxi)$.
Indeed, either $p\geq p_\cc(\Z^d)$ or $\hat\nu_{\gamma_n}^+=\delta_t$ for all $n$ large enough that $T(0,n\bxi)>0$.
\end{eg}

After introducing in Chapter \ref{var_form_sec} a variational formula for the time constant, we will state and prove several results in Chapter \ref{applications} to the effect that $\PPP_\mathrm{emp}(\bxi)$ is dense in $\PPP$.
Here is one result that already shows this fact, where ``dense" is with respect to the topology of weak convergence.

\begin{thm} \label{finite_emp_0}
Fix any positive integer $N$, $\bxi\in\S^{d-1}$, and $p_0,p_1,\dots,p_N\in[0,1]$ satisfying 
$\sum_{i=0}^N p_i = 1$.
For Lebesgue-almost every $(t_1,\dots,t_N)\in(0,\infty)^N$, we have
$p_0\delta_0 + \sum_{i=1}^N p_i\delta_{t_i} \in \PPP_\mathrm{emp}(\bxi)$.
\end{thm}

This result is later stated and proved as Theorem \ref{finite_emp}.
Our other main results, namely Theorems \ref{countable_emp}, \ref{continuous_emp}, \ref{bounded_emp}, manifest further dense families belonging to $\PPP_\mathrm{emp}(\bxi)$. 
Their exact statements, however, are best read after we have introduced certain notations related to the variational formula, hence our postponing these results until Section \ref{main_results}.
Nevertheless, they allow us to make the following, stronger denseness statement.
Said in the most concise way, it establishes that $\PPP_\mathrm{emp}(\bxi)$ is dense in $\PPP$ with respect to the $\infty$-Wasserstein distance.

\begin{thm} \label{dense_thm}
Fix any $\bxi\in\S^{d-1}$.
For every edge-weight distribution $\LL\in\PPP$ and every $\eps>0$, there is $\LL'\in\PPP_\mathrm{emp}(\bxi)$ that can be coupled with $\LL$ using random variables $X\sim\LL$ and $Y\sim\LL'$ satisfying
\eq{
|X-Y|\leq \eps \qquad \text{and} \qquad X = 0 \iff Y = 0.
}
\end{thm}

Finally, one can extend Questions  \ref{main_q} and \ref{main_q_2} to infinite geodesics, and we address the case when these geodesics have the correct asymptotic direction.
Recall the notation $\Gamma^{(\ell)} = \{e_1,\dots,e_\ell\}$ from \eqref{sub_of_infinite}, and let $\Geo_\infty$ \label{geo_infty_def}
be the set of all infinite geodesics.
First, we have the following analogue of Theorem \ref{zero_thm}.

\begin{thm}
\label{zero_thm_infty}
If $F(0) \geq p_\cc(\Z^d)$, then almost surely
\eq{
\hat\nu_{\Gamma^{(\ell)}} \Rightarrow \delta_0 \quad \text{as $\ell\to\infty$}, \quad \text{for every $\Gamma\in\Geo_\infty$}.
}
\end{thm}

\begin{remark}
If $d=2$ and $F(0)>p_\cc(\Z^2) = 1/2$, then it is not difficult to reason that all infinite geodesics eventually consist of \textit{entirely} zero-weight edges (because the complement of the infinite open cluster has no infinite component).
This statement holds also for $d\geq3$ so long as $F(0)$ is sufficiently close to $1$ \cite{grimmett-holroyd-kozma14}, although it is not clear if the eventually-always-zero property of geodesics is in effect whenever $F(0)>p_\cc(\Z^d)$.
When $F(0)=p_\cc(\Z^d)$, the existence of infinite geodesics is itself unclear.
\end{remark}

Given Theorem \ref{zero_thm_infty}, it makes sense to extend Definition \ref{good_def} as follows.
For $\bxi\in\S^{d-1}$, let $\Geo_\infty(\bxi)$ denote the set of all $\bxi$-directed infinite geodesics.
\label{geo_infty_bxi_def}

\begin{defn} \label{good_def_infty}
For $\bxi\in\S^{d-1}$, we will write $\LL\in\PPP_\mathrm{emp}^\infty(\bxi)$ if either 
\begin{itemize}
\item[(a)] $F(0)\geq p_\cc(\Z^d)$; or
\item[(b)] $F(0)<p_\cc(\Z^d)$ and there is a deterministic $\hat\nu^+_\infty = \hat\nu^+_\infty(\LL,\bxi)\in\PPP$ such that, almost surely we have
\eq{
\hat\nu_{\Gamma^{(\ell)}}^+ \Rightarrow \hat\nu^+_\infty \quad \text{as $\ell\to\infty$}, \quad \text{for every $\Gamma\in\Geo_\infty(\bxi)$}.
}
\end{itemize}
\end{defn}

It seems reasonable to believe that $\PPP_\mathrm{emp}(\bxi) = \PPP_\mathrm{emp}^\infty(\bxi)$ with $\hat\nu^+ = \hat\nu^+_\infty$.
While we are unable to show this in general, the aforementioned dense families identified in Section \ref{main_results} also belong to $\PPP_\mathrm{emp}^\infty(\bxi)$.
In particular, Theorem \ref{dense_thm} still holds if we demand $\LL'\in\PPP_\mathrm{emp}(\bxi)\cap\PPP_\mathrm{emp}^\infty(\bxi)$.
Finally, for comparison:

\begin{eg}
In \cite[pg.~673]{chaika-krishnan19}, \mbox{Chaika} and \mbox{Krishnan} exhibit a stationary ergodic FPP model admitting infinite geodesics with neither an asymptotic direction nor a limiting empirical measure.
\end{eg}

\section{Lengths of geodesics} \label{geo_length_sec} 

Here we discuss a problem very much related to Question \ref{main_q}.
The following question was posed by M. \mbox{Damron}, also at the 2015 AIM workshop \cite{aim_15}.
Its origins can be traced back to the landmark paper of \mbox{Hammersley} and \mbox{Welsh} \cite[Sec.~8.2]{hammersley-welsh65} credited with initiating the study of FPP.
The problem was also promoted in the early work of \mbox{Smythe} and \mbox{Wierman} \cite[Chap.~VIII]{smythe-wierman78}, and more recently in \cite[Question 4.1.4]{auffinger-damron-hanson17}.

\begin{question} \label{cor_q}
If $\gamma_n \in \Geo(0,n\bxi)$, does $|{\gamma_n}|/n$ converge as $n\to\infty$?
\end{question}

Just as we did with Question \ref{main_q}, we will refine our inquiry after first noting some counterexamples.

\begin{eg}[\textit{Non-critical atom at zero}] \label{supercritical_length_eg}
If $F(0)>0$, then the answer to Question \ref{cor_q} is \textit{no}. 
This is by the same discussion as preceded \eqref{hat_nu_plus_def}.
Nevertheless, one can instead consider the quantities
\eeq{ \label{Ln_def} 
\barbelow{N}_n^\bxi\coloneqq \inf\{|\gamma|:\, \gamma\in\Geo(0,n\bxi)\}, \ \
\barabove{N}_n^\bxi \coloneqq \sup\{|\gamma| :\,\gamma\in\Geo(0,n\bxi)\}.
}
In the case $F(0)>p_\cc(\Z^d)$, \mbox{Zhang} \cite[Thm.~4]{zhang95} showed that $n^{-1}\barbelow{N}_n^{\mathbf{e}_1}$ converges to a deterministic constant, 
improving previous results from \cite{zhang93,zhang-zhang84}.
The case $0 < F(0) < p_\cc(\Z^d)$ remains unsettled.
\end{eg}

\begin{eg}[\textit{Critical atom at zero}] \label{critical_length}
For $d=2$ and $F(0) = p_\cc(\Z^2)=1/2$, it was shown by \mbox{Damron} and \mbox{Tang} \cite{damron-tang19} that geodesics have superlinear length, verifying a conjecture of \mbox{Kesten} \cite[Sec.~9.24]{kesten86}.
\end{eg}

To motivate our next definition, consider the unsettled case mentioned in Example \ref{supercritical_length_eg}. 
While this problem remains challenging, one possible approach is to study the decomposition
$|\gamma| = |\gamma|_0 + |\gamma|_+$, 
where 
\eeq{ \label{gamma_length_0_def} 
|\gamma|_0 \coloneqq |\{e\in\gamma :\,\tau_e = 0\}|, \qquad
|\gamma|_+ \coloneqq |\{e\in\gamma :\,\tau_e > 0\}|.
}
One can hope (possibly falsely) that the convergence of $\barbelow{N}_n^\bxi/n$ holds at least when $\LL$ belongs to the following class.

\begin{defn} \label{good_length_def}
For $\bxi\in\S^{d-1}$, we will write $\LL\in\PPP_\mathrm{length}(\bxi)$ if there is a deterministic constant $\lambda = \lambda(\LL,\bxi)$ such that, almost surely,
\eq{
\lim_{n\to\infty}\frac{|\gamma_n|_+}{n} = \lambda \quad \text{for any choice of $\gamma_n\in\Geo(0,n\bxi)$}.
}
\end{defn}

For completeness, we include the following (partial) companion result to Theorem \ref{zero_thm}.
A short proof can be found in Section \ref{length_lemma_proof}.
The supercritical case is a straightforward consequence of \cite[Thm.~1]{zhang95}, and the second scenario is easy once we have Proposition \ref{replacement_thm}.
The general critical case remains open.

\begin{thm} \label{zero_length_thm}
If $F(0) > p_\cc(\Z^d)$ or $F(0)=F(h)=p_\cc(\Z^d)$ for some $h>0$, then almost surely the following holds for any sequence $(x_k)_{k\geq1}$ in $\Z^d$ such that $\|x_k\|_2\to\infty$: 
\eeq{ \label{zero_length_thm_eq}
\lim_{k\to\infty} \frac{|\gamma_k|_+}{\|x_k\|_2}=0 \quad \text{for any choice of $\gamma_k\in\Geo(0,x_k)$}.
}
In particular, $\LL\in\PPP_\mathrm{length}(\bxi)$ for every $\bxi\in\S^{d-1}$, with $\lambda=0$.
\end{thm}

\begin{remark}
If $\tau_e$ is bounded and $F(0)=F(h) = p_\cc(\Z^d)$ for some $h>0$, then it is immediate from \cite[Thm.~B]{chayes91} that $|\gamma_n|_+ = O(n^\eps)$ for all $\eps>0$. 
See also \cite[Rmk.~3]{kesten93}.
If we further assume $d=2$, then $|\gamma_n|_+$ is approximately $\log n$ \cite{chayes-chayes-durrett86,kesten-zhang97,damron-lam-wang17}, and there is a related body of work on the triangular lattice.
For Bernoulli($1/2$) site weights, \mbox{Yao} \cite{yao14,yao18} proved an explicit limit: $|\gamma_n|_+/\log n$ converges to $1/(\sqrt{3}\pi)$ in probability but not almost surely.
To obtain the latter, one needs to instead consider the point-to-line passage time (rather than point-to-point), and for this quantity a universality result was obtained by \mbox{Damron}, \mbox{Hanson}, and \mbox{Lam} \cite{damron-hanson-lam?}.
\end{remark}

When $F(0) < p_\cc(\Z^d)$, Question \ref{cor_q} is simply a relaxation of Question \ref{main_q_2}.
This is intuitively clear from \eqref{T_gamma}: if a sequence of geodesics admits a limiting empirical measure, then $T(\gamma_n)/n\to\mu_\bxi$ implies $|\gamma_n|/n$ must converge as well.
We make this precise in the following lemma, the proof of which can be found in Section \ref{length_lemma_proof}.
In the interest of generality, we state the result for any sequence $(n_k)_{k\geq1}$ of real numbers such that $n_k\to\infty$ as $k\to\infty$, rather than insisting on a geodesic for every $n$.
See also Remark \ref{good_to_length_generality_remark}.

\begin{lemma} \label{good_to_length}
If $F(0)<p_\cc(\Z^d)$, then with probability one the following implications hold for every $\bxi\in\S^{d-1}$, any sequence of real numbers $(n_k)_{k\geq1}$ tending to $+\infty$, and any choice of $\gamma_{n_k}\in\Geo(0,n_k\bxi)$:
\begin{align}
\label{good_to_length_a}
\hat\nu_{\gamma_{n_k}} \Rightarrow \hat\nu \quad \text{as $k\to\infty$} \quad &\implies \quad 
\lim_{k\to\infty} \frac{|\gamma_{n_k}|}{n_k} = \frac{\mu_\bxi}{\int_0^\infty t\, \hat\nu(\dd t)},
\intertext{and}
\label{good_to_length_b}
\hat\nu_{\gamma_{n_k}}^+ \Rightarrow \hat\nu^+ \quad \text{as $k\to\infty$} \quad &\implies \quad
\lim_{k\to\infty} \frac{|\gamma_{n_k}|_+}{n_k} = \frac{\mu_\bxi}{\int_0^\infty t\, \hat\nu^+(\dd t)}.
\end{align}
\end{lemma}

\begin{eg} \label{trivial_length_eg}
As in Example \ref{trivial_emp_eg}, consider $\LL = p\delta_0+(1-p)\delta_t$ with $t>0$ and $p<p_\cc(\Z^d)$.
In this case, for every path $\gamma$ we have $T(\gamma) = t|\gamma|_+$, and for $\gamma_n\in\Geo(0,n\bxi)$ we have $T(\gamma_n)/n\to\mu_\bxi$ as $n\to\infty$.
It follows that $|\gamma_n|_+/n \to \mu_\bxi/t$, in agreement with \eqref{good_to_length_b}.
Furthermore, if we identify a sequence $(n_k)_{k\geq1}$ such that $\hat\nu_{\gamma_{n_k}}\Rightarrow\hat\nu$, where $\hat\nu = (1-p')\delta_0 + p'\delta_t$, then $|\gamma_{n_k}|_+/|\gamma_{n_k}| \to p'$.
Hence $|\gamma_{n_k}|/n_k\to \mu_\bxi/(tp')$, in agreement with \eqref{good_to_length_a}.
\end{eg}

The crucial consequence of Lemma \ref{good_to_length} is the following implication:
\eeq{ \label{good_to_length_implication}
F(0)<p_\cc(\Z^d),\ \LL\in\PPP_\mathrm{emp}(\bxi) \quad \implies \quad \LL\in\PPP_\mathrm{length}(\bxi).
}
In particular, the upcoming results of Section \ref{main_results} will provide many examples of $\LL\in\PPP_\mathrm{length}(\bxi)$.
Having relaxed to Question \ref{cor_q}, though, we are able to obtain stronger results: Theorems \ref{krs_thm} and \ref{new_krs_thm} stated below.

If $F(0)=0$, then $|\gamma_n|_+=|\gamma_n|$, and so determining whether or not $\LL$ belongs to $\PPP_\mathrm{length}(\bxi)$ 
is equivalent to answering Question \ref{cor_q}.
In this case, there is the following theorem due essentially to \mbox{Hammersley} and \mbox{Welsh} \cite[Thm.~8.2.3]{hammersley-welsh65}, with refinements given by \mbox{Smythe} and \mbox{Wierman} \cite[Thm.~7.9 and 8.2]{smythe-wierman78} and more recently by \mbox{Krishnan}, \mbox{Rassoul-Agha}, and \mbox{Sepp\"al\"ainen} \cite[Thm.~2.2 and 2.3]{krishnan-rassoul-seppalainen?}.
The result as stated here offers a further improvement, namely the elimination of all moment assumptions.
For $h\in\R$, denote by $\LL\oplus h$ the law of $\tau_e+h$. \label{variable_law_1}

\begin{thm} \label{krs_thm}
Fix any $\bxi\in\S^{d-1}$.
For every $\LL\in\PPP$, there are at most countably many values of $h>0$ for which $\LL\oplus h\notin\PPP_\mathrm{length}(\bxi)$.
\end{thm}

In general, one cannot improve on this result.
Indeed, for distributions $\LL$ with at least two atoms, \cite[Thm.~2.7]{krishnan-rassoul-seppalainen?} constructs (in a very simple way) a countable, dense set of $h$ for which $\LL\oplus h\notin\PPP_\mathrm{length}(\bxi)$.
For continuous distributions, however, the nature of the exceptional set remains unknown; it may even be empty.

Unfortunately, Theorem \ref{krs_thm} says nothing about the case $h=0$.
In particular, the open problem from Example \ref{supercritical_length_eg} remains unaddressed, since $\LL\oplus h$ clearly has no atom at zero when $h>0$.
Nevertheless, by considering $|\lambda_n|_+$ instead of $|\lambda_n|$, we are able to obtain examples of $\LL\in\PPP_\mathrm{length}(\bxi)$ with $F(0)>0$. 
This offers a simple but apparently new shift in paradigm with respect to Question \ref{cor_q}.
Let $\LL\oplus h\one_{\{t>0\}}$ denote the law of $\tau_e + h\one_{\{\tau_e>0\}}$. \label{variable_law_2}

\begin{thm} \label{new_krs_thm}
Fix any $\bxi\in\S^{d-1}$.
For every $\LL\in\PPP$, there are at most countably many values of $h>0$ for which $\LL\oplus h\one_{\{t>0\}}\notin\PPP_\mathrm{length}(\bxi)$.
\end{thm}

\begin{remark}
Note that geodesics do exist in the environments generated by both $\LL\oplus h$ and $\LL\oplus h\one_{\{t>0\}}$.
In the first case, we have $F(0)=0$; in the second, $F(0)=F(h)$.
Hence parts \ref{geo_exist_thm_b} and \ref{geo_exist_thm_d} of Theorem \ref{geo_exist_thm} apply.
Furthermore, if $F(0)<p_\cc(\Z^d)$, then Proposition \ref{slightly_negative_lemma} will allow us to extend Theorems \ref{krs_thm} and \ref{new_krs_thm} to $h>-\af$, where $\af$ is a small positive number depending on $\LL$ and $d$.
\end{remark}

In Theorems \ref{krs_thm} and \ref{new_krs_thm}, the price we pay to guarantee membership in $\PPP_\mathrm{length}(\bxi)$ is the introduction of a ``mass gap" above zero, i.e.~$F(h)=F(0)$ for some $h>0$.
Given that the only affirmative cases to date for Question \ref{cor_q} have come from the argument of Hammersley and Welsh, there have been no examples of distributions belonging to $\PPP_\mathrm{length}(\bxi)$ without this gap.
The following result fills this void.

\begin{thm} \label{no_gap_thm}
For any $\bxi\in\S^{d-1}$ and any $p_0\in[0,1]\setminus\{p_\cc(\Z^d)\}$, there exists $\LL\in\PPP_\mathrm{length}(\bxi)$ such that $F(0)=p_0$ and $F(h)>p_0$ for all $h>0$.
\end{thm}

The edge-weight distributions claimed here include both discrete and continuous distributions, or even a combination of the two: see Examples \ref{dense_discrete}, \ref{better_choice_eg} and Remark \ref{mix_remark}, respectively.
For $p_0<p_\cc(\Z^d)$, the proof of Theorem \ref{no_gap_thm} is to simply pair these examples---which actually show that $\LL$ belongs to $\PPP_\mathrm{emp}(\bxi)$---with \eqref{good_to_length_implication}. 
If $F(0)>p_\cc(\Z^d)$, then Theorem \ref{zero_length_thm} automatically gives $\LL\in\PPP_\mathrm{length}(\bxi)$.

Finally, it is worth mentioning that regardless of the answer to Question \ref{cor_q}, one can ask for upper and lower bounds on the lengths of geodesics. 
If $F(0)<p_\cc(\Z^d)$, then a result of Kesten \cite[Prop.~5.8]{kesten86} gives $\barabove{N}_n^\bxi \leq Cn$ for all large $n$ and some deterministic $C$.
On the other hand, there is the trivial bound $\barbelow{N}_n^\bxi \geq \|[n\bxi]\|_1$.
Under an additional moment assumption on $\tau_e$, \mbox{Krishnan}, \mbox{Rassoul-Agha}, and \mbox{Sepp\"al\"ainen} \cite[Thm.~2.5]{krishnan-rassoul-seppalainen?} recently improved this to $\barbelow{N}_n^\bxi \geq (1+\delta)\|n\bxi\|_1$ for all large $n$, where $\delta>0$ is deterministic and uniform in $\bxi$.

\chapter{Variational Formula for the Time Constant} \label{var_form_sec}

\section{Coupling the environment to uniform variables} \label{coupling_env_sub}
Recall that $\LL$ is the law of $\tau_e$, and $F$ is the associated distribution function. 
Our variational formula arises from viewing the environment $(\tau_e)_{e\in E(\Z^d)}$ as given by
\eeq{ \label{function_coupling}
\tau_e = \tau(U_e),
}
where $\tau \colon  [0,1]\to[0,\infty)$ is measurable, \label{tau_measurable_def_1}
and $(U_e)_{e\in E(\Z^d)}$ \label{U_e_def_1}
is a collection of i.i.d.~uniform $[0,1]$-valued random variables supported on $(\Omega,\FF,\P)$.
Our assuming such a coupling poses no loss of generality, for if $\tau$ is equal to\footnote{If $\tau_e$ is unbounded, then $F^{-1}(1) = \infty$, but one can choose an arbitrary value for $\tau(1)$ without changing the distribution of $\tau(U_e)$.}
\eeq{ \label{inverse_cdf}
F^{-1}(u) \coloneqq \inf\{t \in \R :\,F(t) \geq u\},
}
then $\tau(U_e)$ does indeed have $F$ as its distribution function.
In other words, if $\Lambda$ denotes Lebesgue measure on $[0,1]$, \label{lebesgue_meas_def}
then $\LL$ has been realized as the pushforward measure $\tau_*(\Lambda)$, defined by
\eq{
[\tau_*(\Lambda)](B) \coloneqq \Lambda(\tau^{-1}B), \quad \text{measurable $B\subset[0,\infty)$}.
}
The representation \eqref{function_coupling} with $\tau = F^{-1}$ has occasionally been used to couple passage times for different edge-weight distributions (e.g.~\cite[Sec.~2]{cox80}, \cite[pg.~811]{cox-kesten81}, and \cite[pg.~226]{kesten86}).
But here---and this is crucial---we allow $F^{-1}$ to be replaced by any other $\tau$ such that $\tau_*(\Lambda) = \LL$.
Moreover, the underlying $U_e$'s will be more than just a technical device for us; we will reduce questions about empirical measures to questions about the family $(U_e)_{e\in E(\Z^d)}$.

\section{Statement of theorem} \label{sketch_and_statement}
Ignoring technical details, the idea for our variational formula is as follows.
For simplicity, we assume $\bxi=\mathbf{e}_1$.
Associate to each $\gamma\in\PP(0,n\mathbf{e}_1)$ an empirical measure with respect to the $U_e$'s:
\eeq{ \label{sigma_gamma_0}
\sigma_\gamma \coloneqq \frac{1}{n}\sum_{e\in\gamma}\delta_{U_e}.
}
Note that we have scaled by $n$ instead of $|\gamma|$.
The reason for doing so is that under the coupling \eqref{function_coupling}, we now have
\eq{
\frac{T(\gamma)}{n} = \int_0^1\tau(u)\ \sigma_\gamma(\dd u) = \langle\tau,\sigma_\gamma\rangle.
}
In this notation, the time constant can be expressed as
\eq{
\mu_{\mathbf{e}_1} = \lim_{n\to\infty} \frac{T(0,n\mathbf{e}_1)}{n} = \lim_{n\to\infty} \inf_{\gamma\in\PP(0,n\mathbf{e}_1)} \langle\tau,\sigma_\gamma\rangle.
}
Now we have the luxury of working with measures on the compact set $[0,1]$.
In particular, given any choice of $\gamma_n\in\Geo(0,n\mathbf{e}_1)$, there exists some sequence $(\gamma_{n_k})_{k\geq1}$ such that $\sigma_{\gamma_{n_k}}$ converges weakly to some measure $\sigma_\infty$.
From this we obtain
\eeq{ \label{mu_is_limit}
\langle\tau,\sigma_\infty\rangle = \lim_{k\to\infty} \langle\tau,\sigma_{\gamma_{n_k}}\rangle = \mu_{\mathbf{e}_1}.
}
But we could have used the same compactness trick with any sequence of paths, not just geodesics.
So if we define $\RR_\infty^{\mathbf{e}_1}$ as the set of all (finite) positive measures $\sigma_\infty$ on $[0,1]$ that can be obtained as a limit $\sigma_{\gamma_{n_k}}\Rightarrow \sigma_\infty$ for \textit{some} sequence of paths $\gamma_{n_k}\in\PP(0,n_k\mathbf{e}_1)$, then in general we have
\eeq{ \label{mu_is_lower_bd}
\langle \tau,\sigma_\infty\rangle = \lim_{k\to\infty}\langle\tau,\sigma_{\gamma_{n_k}}\rangle =
\lim_{k\to\infty} \frac{T(\gamma_{n_k})}{n_k} \geq \mu_{\mathbf{e}_1}.
}
Viewing \eqref{mu_is_limit} and \eqref{mu_is_lower_bd} together, we deduce
\eq{
\mu_{\mathbf{e}_1} = \inf_{\sigma_\infty\in\RR_\infty^{\mathbf{e}_1}}\langle \tau,\sigma_\infty\rangle.
}
In order for this variational formula to be useful, however, we need to know that $\RR_\infty^{\mathbf{e}_1}$ is a \textit{deterministic} set. This fact is established in Theorem \ref{deterministic_limits_thm}.
While the technical aspects of the proof are somewhat delicate, little intuition is lost in regarding this fact as nothing more than a consequence of Kolmogorov's zero-one law.

To now make this discussion formal, let $\Sigma$ \label{Sigma_def}
be the set of finite, positive \mbox{Borel} measures $\sigma$ on $[0,1]$  with total mass at least $1$. 
When we wish to normalize a measure $\nu$ 
to be a probability measure, we will write $\hat \nu = \nu/\langle 1,\nu\rangle$ if $\nu\neq0$; otherwise $\hat\nu=0$. \label{normalization_def}
By $\mu_\bxi(\tau)$ we mean the time constant \eqref{time_constant_def} associated to the edge-weight law $\LL = \tau_*(\Lambda)$, where $\Lambda$ is Lebesgue measure on $[0,1]$.

\begin{thm} \label{var_form_thm}
For each $\bxi\in\S^{d-1}$, there is a deterministic subset $\RR^\bxi\subset\Sigma$ such that:

\begin{enumerate}[label=\textup{(\alph*)}]

\item \label{var_form_thm_a}
For any measurable $\tau \colon  [0,1]\to[0,\infty)$, the $\bxi$-direction time constant under the coupling \eqref{function_coupling} is given by
\eeq{ \label{var_form_eq}
\mu_\bxi(\tau) = \inf_{\sigma\in\RR^\bxi} \langle \tau,\sigma\rangle. 
}

\item \label{var_form_thm_b}
If $F(0)\neq p_\cc(\Z^d)$, then the set of minimizers is nonempty:
\eeq{ \label{minimizers_def}
\RR^\bxi_\tau \coloneqq \{\sigma\in\RR^\bxi :\,\langle\tau,\sigma\rangle = \mu_\bxi(\tau)\} \neq \varnothing.
}

\item \label{var_form_thm_c}
Finally, there is a constant $C = C(d)$ such that
\eeq{ \label{abs_cont_quantitative}
\hat\sigma(B) \leq C(\log \Lambda(B)^{-1})^{-1} \quad \text{for all $\sigma\in\RR^\bxi$, measurable $B\subset[0,1]$.}
}

\end{enumerate}
\end{thm}

\begin{remark}
Part \ref{var_form_thm_c}, which is proved in Theorem \ref{deterministic_limits_thm} separately from \ref{var_form_thm_a} and \ref{var_form_thm_b}, implies that every $\sigma\in\RR^\bxi$ is absolutely continuous with respect to Lebesgue measure $\Lambda$, and \eqref{abs_cont_quantitative} gives some control on the Radon--Nikodym derivative $\dd\hat\sigma/\dd\Lambda$.
Using the proof method of Proposition \ref{tree_prop_1}, one can also show that for every $\sigma\in\RR^\bxi$, the relative entropy of $\hat\sigma$ with respect to $\Lambda$ is at most $\log \kappa_d$, where $\kappa_d$ is the connective constant for self-avoiding walks on $\Z^d$ (see \cite[Sec.~6.2]{lawler13}).
It is also worth mentioning that the method by which we obtain \eqref{abs_cont_quantitative} could be used to directly establish tightness (and thus subsequential limits) for the family $(\hat\nu_{\gamma_n})_{n\in[0,\infty)}$ from Question \ref{main_q}.
Therefore, the fact that \eqref{function_coupling} allows us to consider a compact space is less important---although extremely convenient technically---than the fact that the constraint set $\RR^\bxi$ in \eqref{var_form_eq} is non-random and has no dependence on $\tau$.
\end{remark}

An earlier variational formula for the time constant was proved for bounded, stationary-ergodic edge-weights by \mbox{Krishnan} \cite{krishnan14,krishnan16}, whose view of FPP as a homogenization problem led to $\bxi\mapsto\mu_\bxi$ being understood as solving a certain Hamilton--Jacobi--Bellman equation.
In the setting of  LPP and directed polymers (the positive-temperature version of LPP), two different methodologies were pursued in \cite{rassoul-seppalainen-yilmaz13,rassoul-seppalainen14,rassoul-seppalainen-yilmaz17,georgiou-rassoul-seppalainen16}, yielding ``cocycle" and ``entropy" variational formulas.
In yet another direction, directed polymers also enjoy a Markovian structure---not present in FPP---which has led to ``endpoint" variational formulas \cite{bates18,bates-chatterjee20,broker-mukherjee19II,bakhtin-seo20}.
Unlike the cocycle and entropy formulas, these have not yet met analogous expressions in zero-temperature models.
 
The ``cocycle" branch of \cite{rassoul-seppalainen-yilmaz13,rassoul-seppalainen14,rassoul-seppalainen-yilmaz17,georgiou-rassoul-seppalainen16} has also been developed for FPP \cite{janjigian-nurbavliyev-rassoul?,krishnan-rassoul-seppalainen??} and bears connections to the formula of Krishnan.
Meanwhile, the ``entropy" branch led to an LPP formula \cite[Thm.~7.3]{georgiou-rassoul-seppalainen16} very similar to \eqref{var_form_eq}; see also \cite[Thm.~7.2, disp.~(1.4)]{georgiou-rassoul-seppalainen16} and the follow-up paper \cite{rassoul-seppalainen-yilmaz17II}.
Despite this similarity, the approach of \cite{georgiou-rassoul-seppalainen16} is entirely different from ours: the authors derive their LPP variational formula from one for the positive-temperature model, which in turn is obtained using large deviations principles from \cite{rassoul-seppalainen-yilmaz13,rassoul-seppalainen14} for empirical measures.
In this monograph, however, we work directly in a zero-temperature setting; this allows us to not only generate a variational formula, but also prove that empirical measures associated to geodesics converge to its minimizers.
This is discussed in Section \ref{empirical_connection}.

There is currently no nontrivial edge-weight distribution for which the exact numerical value of $\mu_{\mathbf{e}_1}$ is known; see \cite[Ques.~2.1.1]{auffinger-damron-hanson17}.
If the set $\RR^{\bxi}$ were given a sufficiently explicit description---certainly a difficult problem---in principle \eqref{var_form_eq} could enable the computation of $\mu_\bxi$ for any distribution.
More important, it remains a compelling but unrealized possibility that one could use variational formulas to derive geometric properties of the limit shape $B_0$ from Section \ref{time_constant_sec}.
In turn, these properties are fundamental to the study of fluctuations, coalescence of geodesics, and existence of infinite geodesics (e.g.~see \cite[Chap.~3--5]{auffinger-damron-hanson17}).
These connections underscore the importance and likely difficulty of understanding how the sets $\RR^\bxi$ and $\RR^\bxi_\tau$ vary with the direction $\bxi$.
This question is left unexplored in the present work.

\section{Connection to empirical measures of geodesics} \label{empirical_connection}
As the proof sketch preceding Theorem \ref{var_form_thm} suggests, the minimizing set $\RR^\bxi_\tau$ is related to geodesics through \eqref{mu_is_limit}.
But the relation \eqref{mu_is_limit} involves just a single test function, namely $\tau$ from the coupling \eqref{function_coupling}.\footnote{The reader would be correct to raise the following concern regarding \eqref{mu_is_limit}.  We do not assume anything more than measurability of the function $\tau$, but  \textit{a priori} the weak convergence $\sigma_{\gamma_{n_k}}\Rightarrow\sigma_\infty$ implies $\langle f,\sigma_{\gamma_{n_k}}\rangle\to\langle f,\sigma_\infty\rangle$ only for bounded, continuous $f$.  This is where Lemma \ref{test_fnc_lemma} enters: because we can specify $\tau$ before realizing the $U_e$'s, the weak convergence will \textit{almost surely} apply to $\tau$, or to any other specified function $f$.
For the deterministic implications of these statements, see Corollary \ref{weak_to_strong_cor}.}
That is, the mean of the empirical measure converges to the mean of some limiting measure.
It stands to reason, though, that by using the full strength of the weak convergence $\sigma_{\gamma_{n_k}}\Rightarrow\sigma_\infty$, we can conclude that the empirical measures themselves converge.
Indeed, this is the claim of Theorem \ref{minimizers_thm} stated below.

Generalizing \eqref{sigma_gamma_0}, we define
\eeq{ \label{sigma_gamma}
\sigma_\gamma \coloneqq \frac{1}{\|x-y\|_2}\sum_{e\in\gamma}\delta_{U_e}, 
\quad \gamma\in\PP(x,y).
}
Recall the empirical measures $\hat\nu_\gamma$ and $\hat\nu^+_\gamma$ from \eqref{hat_nu_def} and \eqref{hat_nu_plus_def}.
In order to match the scaling for $\sigma_\gamma$, let us also define
\eeq{ \label{nu_def}
\nu_\gamma \coloneqq \frac{1}{\|x-y\|_2}\sum_{e\in\gamma}\delta_{\tau_e}, \quad
\nu_\gamma^+ \coloneqq \frac{1}{\|x-y\|_2}\sum_{e\in\gamma:\, \tau_e\neq 0}\delta_{\tau_e}, \quad \gamma\in\PP(x,y).
}
Under the coupling \eqref{function_coupling}, we have $\nu_\gamma = \tau_*(\sigma_\gamma)$ and $\hat\nu_\gamma = \hat\tau_*(\sigma_\gamma) = \tau_*(\hat\sigma_\gamma)$, where $\tau_*(\sigma)$ is the pushforward of the measure $\sigma$ by the map $\tau$:
\eeq{ \label{push_def}
\langle f,\tau_*(\sigma)\rangle = \langle f\circ \tau,\sigma\rangle \quad \text{for measurable $f \colon \R\to\R$}.
}
We will write $\tau_*^+(\sigma)$ for the measure defined by
\eeq{ \label{pos_push_def}
\langle f,\tau_*^+(\sigma)\rangle = \langle (f\circ \tau)\one_{\{\tau(u)\neq0\}},\sigma\rangle \quad \text{for measurable $f \colon \R\to\R$},
}
so that $\nu_\gamma^+ = \tau_*^+(\sigma_\gamma)$ and $\hat\nu_\gamma^+ = \hat\tau_*^+(\sigma_\gamma)$.\footnote{Note that $\hat\tau_*^+(\sigma)\neq\tau_*^+(\hat\sigma)$ whenever $\langle\one_{\{\tau(u)=0\}},\sigma\rangle > 0$, but $\hat\tau_*^+(\sigma) = \hat\tau_*^+(\hat\sigma).$}
We can now state our convergence result for empirical measures associated to geodesics.
Recall the notation $\Gamma^{(\ell)}$ from \eqref{sub_of_infinite} for subsets of infinite paths, and note that the case $F(0)\geq p_\cc(\Z^d)$ is already addressed in Theorems \ref{zero_thm} and \ref{zero_thm_infty}.

\begin{thm} \label{minimizers_thm}
Assume $F(0)<p_\cc(\Z^d)$ and that the function $\tau \colon [0,1]\to[0,\infty)$ satisfies $\tau_*(\Lambda) = \LL$.\footnote{We do distinguish $\tau$ and $\tau'$ even if $\tau=\tau'$ Lebesgue-almost everywhere.
The reason for this distinction is very minor: while $\mu_\bxi(\tau) = \mu_\bxi(\tau')$ in Theorem \ref{var_form_thm}, the probability-one events $\Omega_{\tau}^\bxi$ and $\Omega_{\tau'}^\bxi$ in Theorem \ref{minimizers_thm} may differ.}
Under the coupling \eqref{function_coupling}, for each $\bxi\in\S^{d-1}$ there is a probability-one event $\Omega^\bxi_\tau$ on which the following statements hold:
\begin{enumerate}[label=\textup{(\alph*)}]

\item  \label{minimizers_thm_a}
For any sequence $(n_k)_{k\geq1}$ of real numbers with $n_k\to\infty$ as $k\to\infty$, and any choice of $\gamma_{n_k}\in\Geo(0,n_k\bxi)$, there is a subsequence $(\gamma_{n_{k_j}})_{j\geq1}$ such that
\eeq{ \label{optimizer_weak_limit}
\sigma_{\gamma_{n_{k_j}}}\Rightarrow\sigma \quad \text{as $j\to\infty$}, \quad \text{for some $\sigma\in\RR^\bxi_\tau$},
}
in which case
\eeq{ \label{weak_pushforwards}
\nu_{\gamma_{n_{k_j}}}\Rightarrow \tau_*(\sigma), \quad
\hat\nu_{\gamma_{n_{k_j}}}\Rightarrow \hat\tau_*(\sigma), \quad
\nu_{\gamma_{n_{k_j}}}^+\Rightarrow \tau_*^+(\sigma), \quad
\hat\nu_{\gamma_{n_{k_j}}}^+\Rightarrow \hat\tau_*^+(\sigma).
}

\item \label{minimizers_thm_b}
For any increasing sequence $(\ell_k)_{k\geq1}$ of positive integers and any $\Gamma\in\Geo_\infty(\bxi)$, there is a subsequence $(\ell_{k_j})_{j\geq1}$ such that
\eq{
\sigma_{\Gamma^{(\ell_{k_j})}}\Rightarrow\sigma \quad \text{as $j\to\infty$}, \quad \text{for some $\sigma\in\RR^\bxi_\tau$},
}
in which case
\eq{ 
\nu_{\Gamma^{(\ell_{k_j})}}\Rightarrow \tau_*(\sigma), \quad
\hat\nu_{\Gamma^{(\ell_{k_j})}}\Rightarrow \hat\tau_*(\sigma), \quad
\nu_{\Gamma^{(\ell_{k_j})}}^+\Rightarrow \tau_*^+(\sigma), \quad
\hat\nu_{\Gamma^{(\ell_{k_j})}}^+\Rightarrow \hat\tau_*^+(\sigma).
}
\end{enumerate}
\end{thm}

Theorems \ref{zero_thm}, \ref{zero_thm_infty}, \hyperref[var_form_thm]{\ref*{var_form_thm}(a,b)}, \ref{minimizers_thm} are proved in Section \ref{main_final_proof_sec} following the more general Theorems \ref{var_formula_thm_super} and \ref{var_formula_thm_sub}, which allow sequences that are not strictly geodesics.
In the next section, we will combine Theorem \ref{minimizers_thm} with elementary convex analysis to identify various families in $\PPP_\mathrm{emp}(\bxi)\cap\PPP_\mathrm{emp}^\infty(\bxi)$.
Namely, we will determine cases in which $\tau_*^+(\sigma)$ is constant over $\sigma\in\RR^\bxi_\tau$.
By Theorem \ref{minimizers_thm}, this is sufficient to show that empirical measures along geodesics have a unique, almost sure limit.

\begin{remark} \label{realization_remark}
This approach to proving convergence of empirical measures is certainly sufficient but may not be necessary.
More precisely, it is not clear whether every $\sigma\in\RR^\bxi_\tau$ can be realized as the weak limit $\sigma_{\gamma_{n_k}}\Rightarrow\sigma$ for some sequence of $\gamma_{n_k}\in\Geo(0,n_k\bxi)$.
Hence it remains a possibility that $\LL\in\PPP_\mathrm{emp}(\bxi)$ despite $\tau_*(\sigma)$ not being constant over $\RR^\bxi_\tau$.
If, however, every $\sigma\in\RR^\bxi_\tau$ can be realized as a weak limit with geodesics, then we would be able to make additional progress on the open problem from Example \ref{supercritical_length_eg}.
\end{remark}

\begin{remark} \label{abs_cont_remark}
By Theorem \hyperref[var_form_thm_c]{\ref*{var_form_thm}\ref*{var_form_thm_c}}, every $\sigma\in\RR^\bxi$ is absolutely continuous with respect to Lebesgue measure, which implies that any subsequential limit $\hat\nu$ of empirical measures is absolutely continuous with respect to $\LL$.
When these empirical measures are those of geodesics, it is generally expected that the reverse is also true: $\LL$ is absolutely continuous with respect to any such limit.
This was effectively proved by \mbox{van den Berg} and \mbox{Kesten} \cite{vandenberg-kesten93} when $\E(\tau_e)<\infty$; see \cite[Rmk.~2.15]{vandenberg-kesten93}.
In $d=2$, Marchand \cite{marchand02} showed that the assumption of a finite mean is not needed.
In addition, there has been recent work on tail estimates for empirical measures and their limits \cite{janjigian-lam-shen?}.
\end{remark}

\chapter{Applications of Variational Formula} \label{applications}

\section{Concavity of variational formula}

The utility of Theorem \ref{var_form_thm} comes from the following fact, which is a more general version of \cite[disp.~(6.5.2)]{hammersley-welsh65}.

\begin{lemma} \label{concavity_lemma}
For any $\tau_0,\tau_1:[0,1]\to[0,\infty)$ and $\alpha\in[0,1]$, we have
\eeq{ \label{concavity}
\mu_\bxi((1-\alpha)\tau_0 + \alpha\tau_1) \geq (1-\alpha)\mu_\bxi(\tau_0) + \alpha\mu_\bxi(\tau_1).
}
\end{lemma}

\begin{proof}
First observe that $\mu_\bxi(\cdot)$ is homogeneous: for every $\alpha\geq0$, we have
\eq{
\mu_\bxi(\alpha\tau) = \inf_{\sigma\in\RR^\bxi}\langle\alpha\tau,\sigma\rangle = \alpha\inf_{\sigma\in\RR^\bxi}\langle\tau,\sigma\rangle
= \alpha\mu_\bxi(\tau).
}
Also, $\mu_\bxi(\cdot)$ is superadditive:
\eq{
\mu_\bxi(\tau_0+\tau_1) &= \inf_{\sigma\in\RR^\bxi} \langle\tau_0+\tau_1,\sigma\rangle \\
&\geq
\inf_{\sigma\in\RR^\bxi}\langle\tau_0,\sigma\rangle+\inf_{\sigma\in\RR^\bxi}\langle\tau_1,\sigma\rangle
= \mu_\bxi(\tau_0)+\mu_\bxi(\tau_1).
}
The concavity statement \eqref{concavity} is immediate from these two properties.
\end{proof}

To leverage this concavity, we will restrict $\mu_\bxi$ to various function spaces.
As usual, concavity implies some level of differentiability depending on the dimension of the space.
We will always interpret this differentiability as some uniqueness property for the minimizing set $\RR^\bxi_\tau$ from \eqref{minimizers_def}, with a tradeoff between dimension and the scope of the differentiability.
To witness this tradeoff, compare Lemma \ref{one_d_lemma} (one-dimensional), Lemma \ref{finite_d_lemma} (finite-dimensional), and Theorems \ref{countable_emp}, \ref{continuous_emp}, \ref{bounded_emp} (infinite-dimensional).

The conceptual link between differentiability of the time constant and properties of geodesics is not new.
But to my knowledge, this relationship has previously been developed only in the context of Question \ref{cor_q}. 
That is, only the derivative in the ``direction" $\tau \equiv 1$ was considered (i.e.~the effect of adding a constant $h$ to every $\tau_e)$.
A key contribution of this monograph is to relate the ``full" derivative to the more complex Questions \ref{main_q} and \ref{main_q_2}. 
In the process, we can still return to the question of geodesic lengths to recover Theorem \ref{krs_thm} and also establish the new Theorem \ref{new_krs_thm}.
On this topic, let us give some additional background.

\mbox{Hammersley} and \mbox{Welsh} \cite[Sec.~8.2]{hammersley-welsh65} noticed that the map $h\mapsto \mu_\bxi(\tau+h)$ is concave and thus has right and left derivatives at all $h>0$.
This observation was refined by \mbox{Smythe} and \mbox{Wierman} \cite[Sec.~8.1]{smythe-wierman78} to include $h=0$ when $F(0)$ is small enough\footnote{The requirement was $F(0)<\kappa_d^{-1}$, where $\kappa_d$ is the connective constant for self-avoiding walks on $\Z^d$; see \cite[Sec.~6.2]{lawler13}.}, resulting in
\eeqs{derivative_ineq}{ 
\label{derivative_ineq_a}
D^+\mu_\bxi(\tau+h)\big|_{h=0}&\leq \liminf_{n\to\infty} \frac{1}{n}\barbelow{N}_n^\bxi, \\
\label{derivative_ineq_b}
D^-\mu_\bxi(\tau+h)\big|_{h=0} &\geq \limsup_{n\to\infty} \frac{1}{n}\barabove{N}_n^\bxi \quad \mathrm{a.s.},
}
where $\barbelow{N}_n^\bxi$ and $\barabove{N}_n^\bxi$ were defined in Example \ref{supercritical_length_eg}.
\mbox{Kesten} \cite[Cor.~to Thm.~3]{kesten80} showed that ``small enough" actually extends all the way to criticality; that is, $F(0)<p_\cc(\Z^d)$ is sufficient for \eqref{derivative_ineq} (proving the same for $F(0)\geq p_\cc(\Z^d)$ is not hard; see Section \ref{length_lemma_proof}).
The key input to Kesten's argument was to show that the likelihood of finding a self-avoiding path of length $n$ with passage time less than $h_0 n$, decays exponentially in $n$ if $h_0>0$ is chosen small enough \cite[Prop.~5.8]{kesten86}.
In fact, we will build on this fundamental estimate in Chapter \ref{modification_sec}.
The goal is to eliminate the hypothesis $\E(\tau_e)<\infty$, which was needed in all the works just mentioned in order to apply the subadditive ergodic theorem.
Indeed, because of Proposition \ref{replacement_thm}, we do not need any moment assumption at all.

\begin{thm} \label{derivative_ineq_thm}
If $F(0-)=0$, then \eqref{derivative_ineq} holds for all $\bxi\in\S^{d-1}$.
\end{thm}

Notice that to interpret \eqref{derivative_ineq_b}, one must allow slightly negative edge-weights.
The proposition stated below permits this consideration.
Given some measurable function $\tau^\pert \colon [0,1]\to\R$, \label{pert_fnc_1}
we will write $\wt T$ for the passage time when each $\tau_e=\tau(U_e)$ is replaced by $\tilde\tau_e = \tilde\tau(U_e)$, where $\tilde \tau = \tau + \tau^\pert$.

\begin{prop} \label{slightly_negative_lemma}
If $F(0)<p_\cc(\Z^d)$, then there exists $\af = \af(\LL,d)>0$ such that the following statements hold whenever 
$\|\tau^\pert\|_\infty \leq \af$.
\begin{enumerate}[label=\textup{(\alph*)}]

\item \label{slightly_negative_lemma_a}
For every $\bxi\in\S^{d-1}$, there is $\mu_\bxi = \mu_\bxi(\tilde\tau,\bxi)\in[\af,\infty)$ such that
\eq{ 
\frac{\wt T(0,n\bxi)}{n} \to \mu_\bxi \quad \text{in probability as $n\to\infty$}.
}

\item \label{slightly_negative_lemma_c}
The variational formula \eqref{var_form_eq} still applies,
\eq{
\mu_\bxi(\tilde \tau) = \inf_{\sigma\in\RR^\bxi}\langle\tilde \tau,\sigma\rangle,
}
and $\RR^\bxi_{\tilde \tau} \coloneqq \{\sigma\in\RR^\bxi :\,\langle \tilde\tau,\sigma\rangle = \mu_\bxi(\tilde\tau)\}$ is still nonempty.

\end{enumerate}
\end{prop}

Part \ref{slightly_negative_lemma_a} is a special case of Proposition \hyperref[negative_thm_b]{\ref*{negative_thm}(b,c)}, and part \ref{slightly_negative_lemma_c} is included in Theorem \ref{var_formula_thm_sub}.
Under certain moment assumptions, statements like \ref{slightly_negative_lemma_a} have been proved before when $\tau^\pert$ is a constant function, for instance \cite[Thm.~5.13]{smythe-wierman78} and \cite[disp.~(4.2)]{kesten80}.
The full generalization to \ref{slightly_negative_lemma_a} requires some additional analysis aided by modern inputs.
The argument uses anyway a setup already required for Propositions \ref{replacement_thm} and \ref{addition_thm}.
On a technical note, Proposition \ref{slightly_negative_lemma} is possible because we have stipulated in our definition \eqref{fpp_def} of passage time that all paths under consideration are self-avoiding (otherwise a single negative-weight edge could be traversed repeatedly to generate passage times of $-\infty$).
Outside the subcritical regime, however, the restriction to self-avoiding paths does not prevent this degeneracy; the following statement is also proved in Chapter \ref{modification_sec}.

\begin{prop} \label{slightly_negative_lemma_2}
If $F(0)\geq p_\cc(\Z^d)$ and $\tau^\pert \equiv -h$ with $h>0$, then almost surely $\wt T(x,y)=-\infty$ for all $x\neq y$.
\end{prop}

Propositions \ref{slightly_negative_lemma} and \ref{slightly_negative_lemma_2} will allow us to prove Theorem \ref{derivative_ineq_thm} in Section \ref{length_lemma_proof}.
As a final comment, one consequence of \eqref{derivative_ineq} is that if
\eeq{ \label{non_differentiable_implication}
\liminf_{n\to\infty}\frac{1}{n}\barbelow{N}_n^\bxi<\limsup_{n\to\infty}\frac{1}{n}\barabove{N}_n^\bxi \quad \text{with positive probability},
}
then $D^+\mu_\bxi(\tau+h)\big|_{h=0}<D^-\mu_\bxi(\tau+h)\big|_{h=0}$.
This strategy was used by Steele and Zhang to exhibit a point of \textit{non}-differentiability for $\LL = \mathrm{Bernoulli}(p)$ and $d=2$, with $p \in (\frac{1}{2}-\eps,\frac{1}{2})$.\footnote{\mbox{Steele} and \mbox{Zhang} worked with $\bxi=\mathbf{e}_1$, and strictly speaking, they did not obtain \eqref{non_differentiable_implication}.
Rather, they proved the analogous inequality for paths that are optimal among those constrained to the rectangle $(0,n)\times(-3n,3n)$. 
Because these paths still achieve the correct time constant, this is sufficient to conclude $D^+\mu_\bxi(\tau+h)\big|_{h=0}<D^-\mu_\bxi(\tau+h)\big|_{h=0}$.}
In \cite[Thm.~2.6]{krishnan-rassoul-seppalainen?}, this result is extended to all edge-weight distributions with an atom at the origin, subject to $F(0)<p_\cc(\Z^d)$ and a moment bound.
For LPP on the complete directed graph, a similar differentiability question is studied in \cite{foss-konstantopoulous-pyatkin?}.

\section{Differentiability and uniqueness of minimizers} \label{main_results}
In general, the proper setting for us will be an open, convex subset $\CCC$ of a Banach space $\BBB$.
Recall that a continuous, concave function $f \colon \CCC\to\R$ is G\^{a}teaux differentiable at $v\in\CCC$ if and only if there exists a unique continuous linear functional $D f[v]$ such that
\eq{
f(w)-f(v) \leq D f[v](w-v) \quad \text{for all $w\in\CCC$}.
}
For example, see \cite[Prop.~1.8]{phelps93}.
Of course, when $\BBB=\R^N$, this notion coincides with usual differentiability of concave functions.

\subsection{One-dimensional spaces} \label{one_d_subsec}
When $\CCC=(0,\infty)$, we obtain the strongest conclusion with respect to differentiability.
Here and elsewhere, ``countable" means finite or countably infinite.

\begin{lemma} \label{one_d_lemma}
Fix $\bxi\in\S^{d-1}$ and measurable functions $\tau,\psi\colon[0,1]\to[0,\infty)$.
\begin{enumerate}[label=\textup{(\alph*)}]
\item \label{one_d_lemma_a}
The map $[0,\infty)\to[0,\infty)$ given by $h\mapsto f(h)=\mu_\bxi(\tau+h\psi)$ is concave. 
In particular, $f$ is differentiable in $(0,\infty)$ off a countable set $\CCC_{\mathrm{bad}} = \CCC_{\mathrm{bad}}(\LL,\bxi,\tau,\psi)$.
Furthermore,
\eeqs{derivative_ineq_minimizers}{ \label{derivative_ineq_minimizers_a}
D^+f(h) \leq \inf_{\sigma\in\RR^\bxi_{\tau+h\psi}}\langle\psi,\sigma\rangle \quad \text{for all $h\geq0$}, \\
\label{derivative_ineq_minimizers_b}
D^-f(h) \geq \sup_{\sigma\in\RR^\bxi_{\tau+h\psi}}\langle\psi,\sigma\rangle \quad \text{for all $h>0$}.
}

\item \label{one_d_lemma_b}
If $D f^+(h)=D f^-(h)$, then $\langle \psi,\sigma\rangle=D f(h)$ for all $\sigma\in\RR^\bxi_{\tau+h\psi}$.

\item \label{one_d_lemma_c}
If $\Lambda(\{u\in[0,1]:\tau(u)=0\}) <p_\cc(\Z^d)$ and $\|\psi\|_\infty\leq 1$, then parts \ref{one_d_lemma_a} and \ref{one_d_lemma_b} hold even for $h\in(-\af,\infty)$, where $\af>0$ is the constant from Proposition \ref{slightly_negative_lemma}.

\end{enumerate}
\end{lemma}

\begin{proof}
Let us write $\tau^{(h)} = \tau+h\psi$ and $f(h)=\mu_\bxi(\tau^{(h)})$.
Lemma \ref{concavity_lemma} can be used to show $f$ is concave on $[0,\infty)$. 
Then observe that every $\sigma\in\RR^\bxi_{\tau^{(h)}}$ satisfies the following whenever $h,h+\eps\geq0$:
\eq{
f(h+\eps)-f(h)
&= \inf_{\sigma'\in\RR^\bxi}\langle\tau^{(h+\eps)},\sigma'\rangle - \langle\tau^{(h)},\sigma\rangle  \\
&\leq \langle\tau^{(h+\eps)}-\tau^{(h)},\sigma\rangle = \eps\langle\psi,\sigma\rangle.
}
Since $f$ is concave, we conclude from this inequality that \eqref{derivative_ineq_minimizers} holds, thus completing the proof of \ref{one_d_lemma_a}.
When $D f^+(h)=D f^-(h)$, the quantity $\langle\psi,\sigma\rangle$ must be equal to $D f(h)$ for every $\sigma\in\RR^\bxi_{\tau^{(h)}}$; this proves part \ref{one_d_lemma_b}.
Finally, we obtain \ref{one_d_lemma_c} using the same arguments supplemented by Proposition \ref{slightly_negative_lemma}.
\end{proof}

We can now obtain Theorems \ref{krs_thm} and \ref{new_krs_thm}.
Recall the various empirical measures $\sigma_\gamma$, $\nu_\gamma$, and $\nu_\gamma^+$ from \eqref{sigma_gamma} and \eqref{nu_def}, defined for $\gamma\in\PP(0,n\bxi)$.

\begin{remark} \label{more_direct_remark}
The inequalities in \eqref{derivative_ineq} can be easily generalized to $h>0$ by considering geodesics in the shifted environment $\{\tau_e+h:\, e\in E(\Z^d)\}$, and then Theorem \ref{krs_thm} would follow from the concavity of $h\mapsto\mu_\bxi(\tau+h)$.
While the proof offered below does not speak in these terms, it is conceptually the same.
The reason for our more abstract exposition is to unify Theorems \ref{krs_thm} and \ref{new_krs_thm}, and also to provide a first example of the functional analytic methodology used in the rest of Chapter \ref{applications}.
\end{remark}

\begin{proof}[Proof of Theorems \ref{krs_thm} and \ref{new_krs_thm}]
Let $\tau$ be such that $\tau_*(\Lambda) = \LL$.
For $h>0$, let us write $\tau^{(h)} = \tau + h$ and $\tau^{(h+)}=\tau+h\one_{\{\tau>0\}}$.
Then $\tau^{(h)}_*(\Lambda) = \LL\oplus h$ and $\tau^{(h+)}_*(\Lambda) = \LL\oplus h\one_{\{t>0\}}$.
By Lemma \ref{one_d_lemma} with $\psi\equiv 1$, if  $h>0$ avoids a certain countable set, then
\eq{
\lambda = \langle 1,\sigma\rangle 
\quad \text{is constant over $\sigma\in\RR^\bxi_{\tau^{(h)}}$}.
} 
Similarly, applying Lemma \ref{one_d_lemma} with $\psi\equiv \one_{\{\tau>0\}}$ shows
\eq{
\lambda_+=\langle \one_{\{\tau(u)>0\}},\sigma\rangle \quad \text{is constant over $\sigma\in\RR^\bxi_{\tau^{(h+)}}$},
}
again assuming $h>0$ avoids a certain countable set.
Now take any $\gamma_n\in\Geo(0,n\bxi)$ in the edge-weight environment $(\tau_e+h)_{e\in E(\Z^d)}$.
By Theorem \ref{minimizers_thm}, it is almost surely the case that within any sequence $(n_k)_{k\geq1}$ satisfying $n_k\to\infty$ as $k\to\infty$, there is a subsequence $({n_{k_\ell}})_{\ell\geq1}$ such that $\sigma_{\gamma_{n_{k_\ell}}}\Rightarrow\sigma$ as $\ell\to\infty$, for some $\sigma\in\RR^\bxi_{\tau^{(h)}}$.
Upon realizing that $|\gamma_n|/\|[n\bxi]\|_2 = \langle 1,\nu_{\gamma_n}\rangle$ and $\|[n\bxi]\|_2/n\to1$ as $n\to\infty$, we now have
\eq{
\lim_{\ell\to\infty} \frac{|\gamma_{n_{k_\ell}}|}{n_{k_\ell}} 
= \lim_{\ell\to\infty} \langle 1,\nu_{\gamma_{n_{k_\ell}}}\rangle
\stackref{weak_pushforwards}{=} \langle 1,\tau_*^{(h)}(\sigma)\rangle 
\stackref{push_def}{=} \langle1,\sigma\rangle = \lambda. 
}
As this holds for any $(n_k)_{k\geq1}$ such that $n_k\to\infty$, we conclude $|\gamma_n|/n\to\lambda$ as $n\to\infty$.

Similarly, consider any geodesics $\gamma_n\in\Geo(0,n\bxi)$ in the environment $(\tau_e+h\one_{\{\tau_e>0\}})_{e\in E(\Z^d)}$.
If $\P(\tau_e=0) \geq p_\cc(\Z^d)$, then Theorem \ref{zero_length_thm} tells us $\LL\oplus h\one_{\{t>0\}}\in\PPP_\mathrm{length}(\bxi)$ with $\lambda_+ = 0$.
Otherwise, Theorem \ref{minimizers_thm} permits us to take, along any sequence $n_k\to\infty$, a further subsequence such that $\sigma_{\gamma_{n_{k_\ell}}}\Rightarrow\sigma$ as $\ell\to\infty$, for some $\sigma\in\RR^\bxi_{\tau^{(h+)}}$.
We then have
\eq{
\lim_{\ell\to\infty} \frac{|\gamma_{n_{k_\ell}}|_+}{n_{k_\ell}}
= \lim_{\ell\to\infty} \langle 1,\nu_{\gamma_{n_{k_\ell}}}^+\rangle
\stackref{weak_pushforwards}{=} \langle 1,\tau_*^{(+h)+}(\sigma)\rangle
&\stackref{pos_push_def}{=} \langle \one_{\{\tau^{(+h)}(u)>0\}},\sigma\rangle \\
&\stackrefp{pos_push_def}{=}\langle \one_{\{\tau(u)>0\}},\sigma\rangle = \lambda_+.
}
As this holds for every $(n_k)_{k\geq1}$ with $n_k\to\infty$, we conclude $|\gamma_n|_+/n\to\lambda_+$.
\end{proof}

\begin{remark}
Echoing Remark \ref{more_direct_remark}, the key to the above proof is concavity of the time constant with respect to shifting the edge-weights.
Interestingly, this concavity is strict, a fact recently shown in \cite[Thm.~2.2]{krishnan-rassoul-seppalainen?}
using so-called ``modification arguments" originating in \cite{vandenberg-kesten93}.
The authors ask (see their Open Problem 3.5) whether strict concavity can be proved for other types of perturbations to the edge-weights, for instance those of Lemma \ref{one_d_lemma}.
Answering this question may lead to information about geodesic statistics other than length.
\end{remark}

\subsection{Finite-dimensional spaces}
When $\CCC=(0,\infty)^N$ for any positive integer $N$, we obtain the following by the same proof as in Lemma \ref{one_d_lemma}, now using the fact that concave functions on $\R^N$ are differentiable almost everywhere.

\begin{lemma} \label{finite_d_lemma}
Fix $\bxi\in\S^{d-1}$ and measurable functions $\tau_0,\psi_1,\dots,\psi_N\colon [0,1]\to[0,\infty)$.
Write $\vc \psi = (\psi_1,\dots,\psi_N)$ and $\vc t\cdot\vc \psi = t_1\psi_1+\cdots+t_N\psi_N$ for $\vc t\in \R^N$.
\begin{enumerate}[label=\textup{(\alph*)}]
\item \label{finite_d_lemma_a}
The map $[0,\infty)^N\to[0,\infty)$ given by $\vc t\mapsto f(\vc t) = \mu_\bxi(\tau_0+\vc t\cdot\vc \psi)$ is concave.
In particular, $f$ is differentiable in $(0,\infty)^N$ off a Lebesgue null set $\CCC_{\mathrm{bad}} = \CCC_{\mathrm{bad}}(\LL,\bxi,\tau_0,\vc\psi)$.

\item \label{finite_d_lemma_b}
For $\vc t\in(0,\infty)^N\setminus\CCC_\mathrm{bad}$ and every $\sigma\in\RR^\bxi_{\tau_0+\vc t\cdot\vc\psi}$, we have
\eq{
\langle \vc s\cdot\vc\psi,\sigma\rangle = D f[\vc t](\vc s) \quad \text{for all $\vc s\in\R^N$}.
}
In particular, for each $i=1,\dots,N$, we have
\eq{ 
\langle \psi_i,\sigma\rangle = \frac{\partial f}{\partial t_i}[\vc t]  \quad \text{for all $\sigma\in\RR^\bxi_{\tau_0+\vc t\cdot\vc\psi}$}.
}
\end{enumerate}
\end{lemma}

We can now prove our first main result regarding Question \ref{main_q_2}.
It says that almost every distribution supported on $N$ points, and possibly also $0$, 
admits a unique limit for empirical measures along geodesics.

\begin{thm} \label{finite_emp}
Fix $\bxi\in\S^{d-1}$ and $p_0,p_1,\dots,p_N\in[0,1]$ satisfying 
$\sum_{i=0}^N p_i = 1$.
For Lebesgue-almost every $\vc t = (t_1,\dots,t_N)\in(0,\infty)^N$, we have
\eq{
p_0\delta_0 + \sum_{i=1}^N p_i\delta_{t_i} \in \PPP_\mathrm{emp}(\bxi)\cap\PPP_\mathrm{emp}^\infty(\bxi).
}
\end{thm}

\begin{proof}
Consider the partition of $[0,1)$ into intervals $\biguplus_{i=0}^N\II_i$, where
\eeq{ \label{interval_def}
\II_i \coloneqq \bigg[\sum_{j=0}^{i-1} p_j,\sum_{j=0}^ip_j\bigg), \quad i \in \{0,1,\dots,N\}.
}
By design, we have $\Lambda(\II_i) = p_i$, where $\Lambda$ is Lebesgue measure.
Set $\psi_i = \one_{\II_i}$ for $i\geq1$ and write $\tau = \vc t\cdot\vc \psi$ as in Lemma \ref{finite_d_lemma} (here $\tau_0\equiv0$).
Under the coupling $\tau_e = \tau(U_e)$ from \eqref{function_coupling}, the edge-weight distribution is
\eq{
\LL=\tau_*(\Lambda) = p_0\delta_0+\sum_{i=1}^N p_i\delta_{t_i}.
} 
If $p_0\geq p_\cc(\Z^d)$, then we trivially have $\LL\in\PPP_\mathrm{emp}(\bxi)\cap\PPP_\mathrm{emp}^\infty(\bxi)$ by Definitions \ref{good_def} and \ref{good_def_infty}.
Otherwise, Theorem \hyperref[var_form_thm_b]{\ref*{var_form_thm}\ref*{var_form_thm_b}} says $\RR^\bxi_\tau$ is nonempty.
In this latter scenario, whenever $Df[\vc t]$ exists as in Lemma \hyperref[finite_d_lemma_b]{\ref*{finite_d_lemma}\ref*{finite_d_lemma_b}}, we have
\eq{
\sigma(\II_i) = \langle \psi_i,\sigma\rangle = \frac{\partial f}{\partial t_i}[\vc t] \quad \text{for all $\sigma\in\RR^\bxi_\tau$, $i\in\{1,\dots,N\}$}.
}
In particular, $\sigma(\II_i)$ takes the same value for every $\sigma\in\RR^\bxi_\tau$.
Since $\tau$ is constant on each $\II_i$ and equal to zero on $[0,1]\setminus\biguplus_{i=1}^N\II_i$, it follows that $\tau_*^+(\sigma)$ is the same measure for every $\sigma\in\RR^\bxi_\tau$.
By Theorem \ref{minimizers_thm}, we conclude $\tau_*(\Lambda) \in \PPP_\mathrm{emp}(\bxi)\cap\PPP_\mathrm{emp}^\infty(\bxi)$.
\end{proof}

\begin{remark} \label{optimality_remark}
For $p_0<p_\cc(\Z^d)$, Theorem \ref{finite_emp} is optimal in that the set of $\vc t\in(0,\infty)^N$ for which the conclusion is false has Hausdorff dimension at least $N-1$.
Indeed, whenever $\LL$ contains atoms satisfying a suitable linear relation with integer coefficients (for example,
$3t_1 = t_2$), the result \cite[Thm.~6.2]{krishnan-rassoul-seppalainen?} says that $\LL\notin\PPP_\mathrm{length}(\bxi)$, in which case  $\LL\notin\PPP_\mathrm{emp}(\bxi)$ by \eqref{good_to_length_implication}.
In $\R^N$, the union of varieties associated to these relations is a set of Hausdorff dimension $N-1$.
\end{remark}

\subsection{Infinite-dimensional spaces, I}
Turning to the infinite-dimensional case, we will offer two different methods for leveraging the variational formula \eqref{var_form_thm} with Lemma \ref{concavity_lemma}.
The first is to adapt the approach of Theorem \ref{finite_emp} to generate further examples of discrete distributions belonging to $\PPP_\mathrm{emp}(\bxi)\cap\PPP_\mathrm{emp}^\infty(\bxi)$, but now having infinite support.
Let $\BBB_\bv$ \label{bv_banach_space}
denote the Banach space of real-valued sequences $\vc t = (t_1,t_2,\dots)$ of bounded variation:
\eq{
\|\vc t\|_\bv \coloneqq |t_1| + \sum_{i=1}^\infty |t_{i+1}-t_i| < \infty.
}
Let $\BBB_\bv^\eps$ be the subset of $\BBB_\bv$ consisting of $\vc t$ such that $t_i\geq\eps$ for all $i$.

\begin{thm} \label{countable_emp}
Fix $\bxi\in\S^{d-1}$, $p_0\in[0,1)$, and any two sequences of positive numbers $\vc p = (p_1,p_2,\dots)$, $\vc \beta = (\beta_1,\beta_2,\dots)$ such that
\eeq{ \label{B0_assumption}
\sum_{i=0}^\infty p_i = 1 \qquad \text{and} \qquad \sum_{i=1}^\infty \frac{\beta_i}{\log p_i^{-1}}< \infty.
}
For every $\eps>0$, there is a $G_\delta$ subset $\DDD_\eps=\DDD_\eps(\LL,\bxi,\vc p,\vc \beta)\subset\BBB_\mathrm{bv}^0$ such that every $\vc t\in\BBB_\bv^\eps$ is a limit point of $\DDD_{\eps}$, and
\eq{
p_0\delta_0 + \sum_{i=1}^\infty p_i\delta_{\beta_i t_i} \in \PPP_\mathrm{emp}(\bxi)\cap\PPP_\mathrm{emp}^\infty(\bxi) \quad \text{for every $\vc t\in\DDD_{\eps}$}.
}
\end{thm}

Before proving Theorem \ref{countable_emp}, let us elaborate on two types of interesting examples it produces. 

\begin{eg}[\textit{Discrete distribution with full support}] \label{dense_discrete}
Choose any $\vc p$ and $\vc \beta$ such that $\beta_1,\beta_2,
\dots$ is an enumeration of the positive rationals, and \eqref{B0_assumption} holds (for example, $p_i \asymp \e^{-i^3}$ with $\beta_i \leq i$ for all $i$).
By Theorem \ref{countable_emp}, there is some $\vc t \in \DDD_{2}$ within distance $1$ of the constant sequence equal to $2$; hence $|t_i| \in [1,3]$ for all $i\geq1$.
Because $\vc t\in\BBB_\bv$, $t_i$ must converge to some $t_\infty\in[1,3]$ as $i\to\infty$.
It follows that $\{t_i\beta_i\}_{i=1}^\infty$ is dense in $[0,\infty)$.
That is, we have identified a discrete distribution in  $\PPP_\mathrm{emp}(\bxi)\cap\PPP_\mathrm{emp}^\infty(\bxi)$ whose support is all of $[0,\infty)$, including an atom at $0$ of mass $p_0$.
\end{eg}

\begin{eg}[\textit{Discrete distribution with prescribed probabilities}] \label{prescribed_discrete}
In the previous example, we imposed a decay condition on $\vc p$ so that $\vc \beta$ could enumerate an unbounded set.
Alternatively, we can let $\vc p$ be arbitrary while choosing any summable $\vc \beta$ in order to satisfy \eqref{B0_assumption}.
In this way, Theorem \ref{countable_emp} produces discrete distributions in $\PPP_\mathrm{emp}(\bxi)\cap\PPP_\mathrm{emp}^\infty(\bxi)$ whose atoms have any desired sequence of probabilities $\vc p$.
\end{eg}

\begin{proof}[Proof of Theorem \ref{countable_emp}]
Consider the partition $[0,1)= \biguplus_{i=0}^\infty \II_i$, where $\II_i$ is again given by \eqref{interval_def}.
Set $\psi_i = \beta_i\one_{\II_i}$ and write $\tau = \vc t\cdot\vc\psi = \sum_{i=1}^\infty t_i\psi_i$ for $\vc t\in\BBB_\bv$ so that under the coupling \eqref{function_coupling}, we have
\eq{
\LL = \tau_*(\Lambda) = p_0\delta_0 + \sum_{i=1}^\infty p_i\delta_{\beta_it_i}.
}
Note that $\tau$ is indeed a measurable function $[0,1]\to[0,\infty)$ because for any particular $u\in[0,1]$, there is at most  one $i$ such that $\psi_i(u)$ is nonzero.
Using Lemma \ref{concavity_lemma}, it is easy to check that the map $\BBB_\bv^0\to[0,\infty)$ given by $\vc t\mapsto f(\vc t) = \mu_\bxi(\vc t\cdot\vc\psi)$ is concave.
Furthermore, this map is continuous: if $(\vc t^{(j)})_{j\geq1}$ is a sequence of elements in $\BBB_\bv$ such that $\|\vc t^{(j)}-\vc t\|_\bv\to0$ as $j\to\infty$, then $\vc t^{(j)}\cdot\vc\psi$ converges to $\vc t\cdot\vc\psi$ pointwise everywhere, and thus $\mu_\bxi(\vc t^{(j)}\cdot\vc\psi)\to\mu_\bxi(\vc t\cdot\vc\psi)$ by \cite[Thm.~6.9]{kesten86}.
For the purposes of applying convex function theory, we will restrict $f$ to
\eeq{ \label{C_eps_def}
\CCC_\eps \coloneqq\bigcup_{j=1}^\infty \Big\{\vc t\in\BBB_\bv :\,t_i \geq \frac{\eps}{2}+j^{-1} \text{ for all $i\geq1$}\Big\},
}
which is a convex, open subset of $\BBB_\bv$ containing $\BBB_\bv^\eps$.
Since $\BBB_\bv$ is separable, Mazur's Theorem (see \cite[Thm.~1.20]{phelps93}) guarantees the existence of a dense $G_\delta$ subset $\DDD_{\eps}\subset\CCC_\eps$ such that $f$ is G\^{a}teaux differentiable at all $\vc t\in\DDD_{\eps}$.

Now consider any $\vc t\in\DDD_{\eps}$, and
let us continue to write $\tau = \vc t\cdot\vc\psi$.
Observe that $\Lambda(\{u\in[0,1]:\tau(u)=0\}) = \Lambda(\II_0)=p_0$. 
If $p_0\geq p_\cc(\Z^d)$, then $\LL\in\PPP_\mathrm{emp}(\bxi)\cap\PPP_\mathrm{emp}^\infty(\bxi)$ by Definitions \ref{good_def} and \ref{good_def_infty}.
So let us assume $p_0<p_\cc(\Z^d)$, in which case Theorem \hyperref[var_form_thm_b]{\ref*{var_form_thm}\ref*{var_form_thm_b}} ensures $\RR^\bxi_\tau$ is nonempty.
For every $\sigma\in\RR^\bxi_{\tau}$ and $\vc s\in\CCC_\eps$, we have
\eq{
\mu_\bxi(\vc s\cdot\vc\psi)-\mu_\bxi(\vc t\cdot\vc\psi)
&\stackrefp{abs_cont_quantitative}{=}\inf_{\sigma'\in\RR^\bxi}\langle\vc s\cdot\vc\psi,\sigma'\rangle - \langle \vc t\cdot\vc\psi,\sigma\rangle \\
&\stackrefp{abs_cont_quantitative}{\leq}  \langle (\vc s-\vc t)\cdot\vc\psi,\sigma\rangle \\
&\stackrefp{abs_cont_quantitative}{=} \sum_{i=1}^\infty \beta_i(s_i-t_i)\sigma(\II_i) \\
&\stackref{abs_cont_quantitative}{\leq} \|\vc s-\vc t\|_\bv \sum_{i=1}^\infty \frac{C\langle 1,\sigma\rangle\beta_i}{\log p_i^{-1}}.
}
Since the sum in the final expression is finite by \eqref{B0_assumption}, we have shown that the map $\vc s\mapsto\langle\vc s\cdot\vc\psi,\sigma\rangle$ is equal to $Df[\vc t]$ for every $\sigma\in\RR^\bxi_{\tau}$.
For each $i \geq 1$, we can apply this map to the element $\vc s^{(j)}\in\BBB_\bv$ given by $s_i^{(j)} = \beta_j^{-1}\one_{\{i=j\}}$, to determine
\eq{
\sigma(\II_j) = \langle \one_{\II_j},\sigma\rangle = \langle \vc s^{(j)}\cdot\vc\psi,\sigma\rangle = Df[\vc t](\vc s^{(j)}) \quad \text{for all $\sigma\in\RR^\bxi_{\tau}$.}
}
In particular, $\sigma(\II_j)$ is constant over $\sigma\in\RR^\bxi_{\tau}$.
Since $\tau$ is constant on each $\II_j$ and equal to zero on $[0,1]\setminus\biguplus_{j=1}^\infty\II_j$, it follows that $\tau_*^+(\sigma)$ is the same measure for each $\sigma\in\RR^\bxi_{\tau}$.
By Theorem \ref{minimizers_thm}, we conclude $\LL\in\PPP_\mathrm{emp}(\bxi)\cap\PPP_\mathrm{emp}^\infty(\bxi)$.
\end{proof}

\begin{remark} \label{baire_remark}
It is a standard exercise using Baire's theorem to show that if $\DDD$ is a $G_\delta$ subset of a Banach space $\BBB$, and $\DDD$ is also dense in some open subset of $\BBB$, then $\DDD$ is uncountable.
In the specific case of Theorem \ref{countable_emp}, we realized $\DDD_{\eps}$ as a dense $G_\delta$ subset of the open set $\CCC_\eps$ from \eqref{C_eps_def}; hence $\DDD_{\eps}$ is uncountable.
The same is true for the analogous sets in Theorems \ref{continuous_emp}, \ref{bounded_emp}.
\end{remark}

\begin{remark} \label{prescribed_probabilities_remark}
Continuing on Remark \ref{optimality_remark} and Example \ref{prescribed_discrete}, we point out that Theorems \ref{finite_emp} and \ref{countable_emp} both work for arbitrary probabilities $p_0$, $p_1$, and so on.
As a tentative heuristic, then, obstructions to a discrete distribution's belonging to $\PPP_\mathrm{emp}(\bxi)$ arise because of the locations of its atoms rather than these atoms' weights.
\end{remark}

\subsection{Infinite-dimensional spaces, II} \label{infinite_dim}
Our second approach is as follows:
\begin{itemize}
\item[1.] Start with a fixed, measurable $\tau \colon  [0,1]\to[0,\infty)$.
\item[2.] Consider all perturbations of $\tau$ by functions $\vphi:[0,1]\to[0,\infty)$ which belong to a Banach space $\BBB$ (or rather an open, convex subset $\CCC\subset\BBB$) that separates measures on $[0,1]$.
\item[3.] Appeal to convex analysis to guarantee that the map $\vphi\mapsto\mu_\bxi(\tau+\vphi)$ is G\^{a}teaux differentiable on a dense subset of $\CCC$.
\item[4.] Because $\CCC$ separates measures, conclude that differentiability at $\vphi$ implies $(\tau+\vphi)_*(\Lambda)\in\PPP_\mathrm{emp}(\bxi)\cap\PPP_\mathrm{emp}^\infty(\bxi)$.
\end{itemize}

\begin{remark} \label{good_examples_1}
The reader may find the upcoming Theorems \ref{continuous_emp} and \ref{bounded_emp} more transparent by simply setting $\tau\equiv0$.
Nevertheless, including a general $\tau$ in the statements does qualitatively widen the set of examples we have of $\LL\in\PPP_\mathrm{emp}(\bxi)\cap\PPP_\mathrm{emp}^\infty(\bxi)$.
For instance, for any bounded-moment condition one might wish to impose, there is $\LL\in\PPP_\mathrm{emp}(\bxi)\cap\PPP_\mathrm{emp}^\infty(\bxi)$ failing that assumption.
Moreover, both Theorem \ref{continuous_emp} and Theorem \ref{bounded_emp} imply Theorem \ref{dense_thm}.
\end{remark}

\begin{remark} \label{good_examples_2}
A finer point is that the procedure outlined above offers no guarantee of preserving the zero set of $\tau$.
That is, we may begin with a distribution such that $F(0)>0$ but be left with $F(0)=0$ after perturbation, which would defeat the purpose of our having modified Question \ref{main_q} to Question \ref{main_q_2}. 
To avoid this possibility, we will replace the interval $[0,1]$ by $[p_0,1]$ and assume that all relevant functions are identically zero on $[0,p_0)$.
This approach was already featured in Theorems \ref{finite_emp} and \ref{countable_emp}.
In this way, within every class of functions $\tau$ considered (piecewise constant, continuous, differentiable, \textit{etc}.) and for every $p_0\in[0,1)$, we obtain examples of $\LL=\tau_*(\Lambda)\in\PPP_\mathrm{emp}(\bxi)\cap\PPP_\mathrm{emp}^\infty(\bxi)$ with $F(0)=p_0$.
\end{remark}

For the choice of $\BBB$, we offer two flavors.
The first is the Banach space $C^k([p_0,1]) = C^k([p_0,1],\R)$ of functions $\vphi:[p_0,1]\to\R$ with $k$ continuous derivatives, equipped with the norm
\eq{
\|\vphi\|_{C^k} \coloneqq \sum_{j=0}^k \|\vphi^{(k)}\|_\infty.
}
Let us extend each $\vphi\in C^k([p_0,1])$ to all of $[0,1]$ by defining 
\eeq{ \label{extend_def}
\bar \vphi(u)=\begin{cases} 0 &0\leq u<p_0, \\
\vphi(u) &p_0\leq u\leq 1.
\end{cases}
}
We may not have $\bar\vphi\in C^k([0,1])$, but this will not be of concern. 

\begin{thm} \label{continuous_emp}
Fix $\bxi\in\S^{d-1}$.
Given $k\geq0$ and $p_0\in[0,1)$, let
\eq{
\CCC=\{\vphi \in C^k([p_0,1]) :\,\vphi(u)>0 \text{ for all $u\in[p_0,1]$}\}.
}
For any measurable function $\tau \colon  [0,1]\to[0,\infty)$ with $\tau|_{[0,p_0)} \equiv 0$,
there exists a dense $G_\delta$ subset $\DDD = \DDD(\LL,\bxi,k,p_0,\tau)\subset\CCC$ such that
\eq{
(\tau+\bar\vphi)_*(\Lambda)\in\PPP_\mathrm{emp}(\bxi)\cap\PPP_\mathrm{emp}^\infty(\bxi)
\quad \text{for every $\vphi\in\DDD$}.
}
\end{thm}

\begin{proof} 
Lemma \ref{concavity_lemma} shows that $\vphi\mapsto f(\vphi) = \mu_\bxi(\tau+\bar\vphi)$ is concave on $\CCC$, since
\eq{
\vphi = (1-\alpha)\vphi_0 + \alpha\vphi_1 \quad \implies \quad \tau+\bar\vphi = (1-\alpha)(\tau+\bar\vphi_0) + \alpha(\tau+\bar\vphi_1).
}
Furthermore, if $(\vphi_i)_{i\geq1}$ is a sequence in $\CCC$ such that $\|\vphi_i-\vphi\|_{C^k}\to0$ as $i\to\infty$, then clearly $\|(\tau+\bar\vphi_i)-(\tau-\bar\vphi)\|_\infty\to0$, which in turn implies $f(\vphi_i)\to f(\vphi)$ by \cite[Thm.~6.9]{kesten86}.
Since $C^k([p_0,1])$ is separable and $\CCC$ is open and convex, it now follows from Mazur's Theorem (again, see \cite[Thm.~1.20]{phelps93}) that there is a dense $G_\delta$ subset $\DDD\subset\CCC$ such that $f$ is G\^{a}teaux differentiable at all $\vphi\in\DDD$.
 
Consider any $\vphi\in\CCC$ that is a point of differentiability.
If $p_0\geq p_\cc(\Z^d)$, then we already have $(\tau+\bar\vphi)_*(\Lambda)\in\PPP_\mathrm{emp}(\bxi)\cap\PPP_\mathrm{emp}^\infty(\bxi)$ by Definitions \ref{good_def} and \ref{good_def_infty}.
So let us assume $p_0< p_\cc(\Z^d)$ and consider any $\sigma\in\RR^\bxi_{\tau+\bar\vphi}$.
For every $\phi\in\CCC$, we have
\eq{ 
f(\phi) - f(\vphi) &= \inf_{\sigma'\in\RR^\bxi}\langle\tau+\bar\phi,\sigma'\rangle - \langle \tau+\bar\vphi,\sigma\rangle 
\leq \langle\bar\phi-\bar\vphi,\sigma\rangle.
}
Considering that $\phi\mapsto\langle\bar\phi,\sigma\rangle$ is indeed a continuous linear functional on $C^k([p_0,1])$, 
we see that this map is precisely $Df[\vphi]$.
In particular, the quantity $\langle\bar\phi,\sigma\rangle$ assumes the same value for every $\sigma\in\RR^\bxi_{\tau+\bar\vphi}$.
Because $C^k([p_0,1])$ is dense in $C^0([p_0,1])$, this is enough to conclude that $\sigma$ restricted to $[p_0,1]$ is the same measure for every $\sigma\in\RR^\bxi_{\tau+\bar\vphi}$.
Since $(\tau+\bar\vphi)(u) = 0$ for all $u\in[0,p_0)$, it follows that 
$(\tau+\bar\vphi)_*^+(\sigma)$ is the same measure for every $\sigma\in\RR^\bxi_{\tau+\bar\vphi}$.
In light of Theorem \ref{minimizers_thm}, we conclude $(\tau+\bar\vphi)_*(\Lambda)\in\PPP_\mathrm{emp}(\bxi)\cap\PPP_\mathrm{emp}^\infty(\bxi)$.
\end{proof}

\begin{eg}[\textit{Continuous distribution with differentiable density}] \label{continuous_eg}
Let $\tau$ be any nonnegative, nondecreasing function that vanishes on $[0,p_0]$ and is $k$-times continuously differentiable on $(0,1)$.
For $k\geq1$, the proof of Theorem \ref{continuous_emp} works without modification if $\CCC$ is replaced by
\eq{
\CCC' \coloneqq \{\vphi \in C^k([p_0,1]) :\,\vphi(u),\vphi'(u)>0 \text{ for all $u\in[p_0,1]$}\}.
}
Whenever $\vphi\in\CCC'$, the sum $\tau+\bar\vphi$ is strictly increasing on $[p_0,1]$ and again $k$-times continuously differentiable on $(p_0,1)$.
Regarding $\tau+\bar\vphi$ as the inverse of some cumulative distribution function---in the sense of \eqref{inverse_cdf}---we conclude that
$(\tau+\bar\vphi)_*(\Lambda)$ is the law of $XY$, where $X\sim\mathrm{Bernoulli}(1-p_0)$ and $Y$ is an independent continuous random variable bearing a density (with respect to Lebesgue measure) possessing $k-1$ continuous derivatives.
If we further assume that $\tau$ is unbounded, then so too is $Y$.
\end{eg}

Our second flavor for $\BBB$ is realized as follows.
We continue to use the notation $\bar\vphi$ from {\eqref{extend_def} to extend functions on $[p_0,1]$ to all of $[0,1]$.
Recall that $\Sigma$ is the set of all finite measures on $[0,1]$ with total mass at least $1$.
Let $\Sigma_\Lambda\subset\Sigma$ denote the set of those $\sigma$ that are absolutely continuous with respect to Lebesgue measure $\Lambda$.
This subset is of interest to us because Theorem \hyperref[var_form_thm_c]{\ref*{var_form_thm}\ref*{var_form_thm_c}} implies $\RR^\bxi\subset\Sigma_\Lambda$.
Now consider any countable family $\FFF$ of measurable functions $\psi \colon [p_0,1]\to[0,\infty)$ satisfying the following two conditions:
\begin{enumerate}[label=\textup{(\roman*)}]

\item \label{family_2} 
$\|\psi\|_\infty \leq 1$ for all $\psi\in\FFF$. 

\item \label{family_3}
If $\sigma_1,\sigma_2\in\Sigma_{\Lambda}$ are such that $\langle\bar\psi,\sigma_1\rangle=\langle\bar\psi,\sigma_2\rangle$ for every $\psi\in\FFF$, then $\sigma_1|_{[p_0,1]}=\sigma_2|_{[p_0,1]}$.
\end{enumerate}

\begin{eg} \label{choice_eg}
One suitable choice for $\FFF$ is the set of indicator functions for intervals of the form $[q,1]$, where $q\in[p_0,1]$ is a rational number.
Another possibility is any countable, dense subset of $C^0([p_0,1],[0,\infty))$ with each element scaled to satisfy (ii).
\end{eg}

Let $\ell^1(\N)$ denote the Banach space of sequences $\vc  t = ( t_1, t_2,\dots)$ such that
$\|\vc  t\|_1 \coloneqq \sum_{i=1}^\infty |t_i| < \infty$.
We will write $\ell^1(\N)_{+}$ \label{ell1_banach_space}
for the subset of those $\vc t$ with $ t_i\geq 0$ for all $i$.
For a sequence of functions $\psi_1,\psi_2,\dots$, we use the notation $\vc\psi = (\psi_1,\psi_2,\dots)$ and
$\vc t\cdot\bar{\vc\psi} = \sum_{i=1}^\infty  t_i\bar\psi_i$.
Note that when $\vc t\in\ell^1(\N)$ and each $\psi_i$ belongs to an $\FFF$ satisfying condition \ref{family_2}, we have
\eeq{ \label{1_to_infty_bd}
\|\vc t\cdot\bar{\vc\psi}\|_\infty \leq \|\vc t\|_1.
}

\begin{thm} \label{bounded_emp}
Fix $\bxi\in\S^{d-1}$ and $p_0\in[0,p_\cc(\Z^d))$.
Let $\FFF = (\psi_i)_{i\geq1}$ be a sequence of measurable functions $[0,p_0]\to[0,\infty)$ satisfying \ref{family_2} and \ref{family_3} shown above.
For any measurable $\tau \colon  [0,1]\to[0,\infty)$ with $\{u:\,  \tau(u)=0\} = [0,p_0]$, there is a convex, open set $\CCC=\CCC(\LL,\bxi,p_0,\FFF,\tau)\subset \ell^1(\N)$ that contains $\ell^1(\N)_+$ and is such that:
\begin{enumerate}[label=\textup{(\alph*)}]

\item \label{bounded_emp_a}
The map $\CCC\to(0,\infty)$ given by $\vc t\mapsto f(\vc t)=\mu_\bxi(\tau+\vc t\cdot\bar{\vc\psi})$ is well-defined and satisfies
\eq{
\mu_\bxi(\tau+\vc t\cdot\vc\psi) = \inf_{\sigma\in\RR^\bxi}\langle \tau+\vc t\cdot\vc\psi,\sigma\rangle.
}
Moreover, the set $\RR^\bxi_{\tau+\vc t\cdot\vc\psi} = \{\sigma\in\RR^\bxi :\,\langle \tau+\vc t\cdot\vc\psi,\sigma\rangle = \mu_\bxi(\tau+\vc t\cdot\vc\psi)\}$ is nonempty for every $\vc t\in\CCC$.

\item  \label{bounded_emp_b}
There exists a dense $G_\delta$ subset $\DDD\subset\CCC$ such that 
%
\eq{
(\tau+\vc t\cdot\bar{\vc\psi})_*(\Lambda)\in\PPP_\mathrm{emp}(\bxi)\cap\PPP_\mathrm{emp}^\infty(\bxi) \quad
\text{for every $\vc t\in\DDD\cap\ell^1(\N)_+$}.
}

\end{enumerate}
\end{thm}

\begin{proof}
If $\LL = \tau_*(\Lambda)$, then the assumption on $\tau$ gives $F(0)<p_\cc(\Z^d)$.
Therefore, the existence of $\CCC\supset\ell^1(\N)_+$ satisfying part \ref{bounded_emp_a} is immediate from Proposition \ref{slightly_negative_lemma} and \eqref{1_to_infty_bd}.
For instance, with $\af>0$ denoting the constant from Lemma \ref{slightly_negative_lemma}, we can take
\eq{
\CCC \coloneqq \bigg\{\vc t \in \ell^1(\N) :\,\sum_{i = 1}^\infty (- t_i\vee0) < \af\bigg\}.
}
We can now mostly repeat the proof of Theorem \ref{continuous_emp} to obtain part \ref{bounded_emp_b}.

As before, concavity of $f$ is immediate from Lemma \ref{concavity_lemma}.
Meanwhile, the continuity of $f$ is verified as follows. 
Whenever $\|\vc t^{(j)}-\vc t\|_1\to0$ as $j\to\infty$, we have $\|(\tau+\vc t^{(j)}\cdot\vc\psi)-(\tau+\vc t\cdot\vc\psi)\|_\infty\to0$ by \eqref{1_to_infty_bd}.
Now appeal to \cite[Thm.~6.9]{kesten86} once more to obtain continuity of $f$.\footnote{Technically, the theorem referenced applies only for $\vc t\in\ell^1(\N)_+$, but it can be modified to work for $\vc t\in\CCC$, provided one has Lemma \ref{eta_choice_lemma}.
No circular logic is created here, as Lemma \ref{eta_choice_lemma} is needed for the proofs of Theorems \ref{var_form_thm} and \ref{minimizers_thm}, as well as Proposition \ref{slightly_negative_lemma}. See also Footnote \ref{also_for_negative_fn}.}
Given concavity and continuity, Mazur's Theorem again provides a dense $G_\delta$ subset $\DDD\subset\CCC$ such that $f$ is G\^{a}teaux differentiable at every $\vc t\in\DDD$.

Now consider any $\vc t\in\ell^1(\N)_+$ that is a point of differentiability.
For every $\sigma\in\RR^\bxi_{\tau+\vc t\cdot\bar{\vc\psi}}$ and $\vc s\in\CCC$, we have
\eeq{ \label{unique_majorization}
f(\vc s)-f(\vc t) &= \inf_{\sigma'\in\RR^\bxi}\langle\tau+\vc s\cdot\bar{\vc\psi},\sigma'\rangle - \langle \tau+\vc t\cdot\bar{\vc\psi},\sigma\rangle 
\leq \langle(\vc s-\vc t)\cdot\bar{\vc\psi},\sigma\rangle.
}
Observe that $\vc s\mapsto\langle\vc s\cdot\bar{\vc\psi},\sigma\rangle$ is linear on $\ell^1(\N)$, and continuous by \eqref{1_to_infty_bd}.
Because of \eqref{unique_majorization}, the assumption of differentiability implies that this map is equal to $Df[\vc t]$ for every $\sigma\in\RR^\bxi_{\tau+\bar\vphi}$.
Because $\FFF$ satisfies condition \ref{family_3}, it follows that $\sigma$ restricted to $[p_0,1]$ is the same measure for every $\sigma\in\RR^\bxi_{\tau+\vc t\cdot\bar{\vc\psi}}$.
Given the hypothesis $(\tau+\vc t\cdot\bar{\vc\psi})(u) = 0$ for all $u\in[0,p_0)$, we deduce that 
$(\tau+\vc t\cdot\bar{\vc\psi})_*^+(\sigma)$ is the same measure for every $\sigma\in\RR^\bxi_{\tau+\vc t\cdot\bar{\vc\psi}}$.
In light of Theorem \ref{minimizers_thm}, we conclude $(\tau+\vc t\cdot\bar{\vc\psi})_*(\Lambda)\in\PPP_\mathrm{emp}(\bxi)\cap\PPP_\mathrm{emp}^\infty(\bxi)$.
\end{proof}

\begin{eg}[\textit{Continuous distribution with full support}] \label{better_choice_eg}
Consider the following modification of the first choice for $\FFF$ in Example \ref{choice_eg}.
To begin, enumerate the rationals in $(p_0,1)$ as $q_1$, $q_2$, and so on.
Next define functions $\Psi, \Psi^{(k)} \colon [0,1]\to[0,1]$ by
\eq{
\Psi(u) \coloneqq \begin{cases} 0 &\rlap{$0$}{\hphantom{p_0}}\leq u<p_0, \\
\frac{u-p_0}{1-p_0} &p_0\leq u\leq 1,
\end{cases} \qquad
\Psi^{(k)}(u) \coloneqq \begin{cases} 0 &\rlap{$0$}{\hphantom{p_0}}\leq u<q_i, \\
\frac{q_k-p_0}{1-p_0} &\rlap{$q_k$}{\hphantom{p_0}}\leq u\leq 1.
\end{cases}
}
For each $i$, let $\psi^{(k)}_1,\psi^{(k)}_2,\dots$ be a sequence of continuous, nonnegative, strictly increasing functions on $[p_0,1]$ such that $\psi^{(k)}_j\leq \Psi$ for all $j$, and $\psi^{(k)}_j$ converges pointwise almost everywhere on $[p_0,1]$ to $\Psi^{(k)}$ as $j\to\infty$.
For an example illustration, see Figure \ref{better_choice_fig}.
We claim that $\FFF = \{\psi_j^{(k)} :\,j,k\geq1\}$ satisfies conditions \ref{family_2} and \ref{family_3}.

\begin{figure}[b]
\centering
\includegraphics[width=0.7\textwidth]{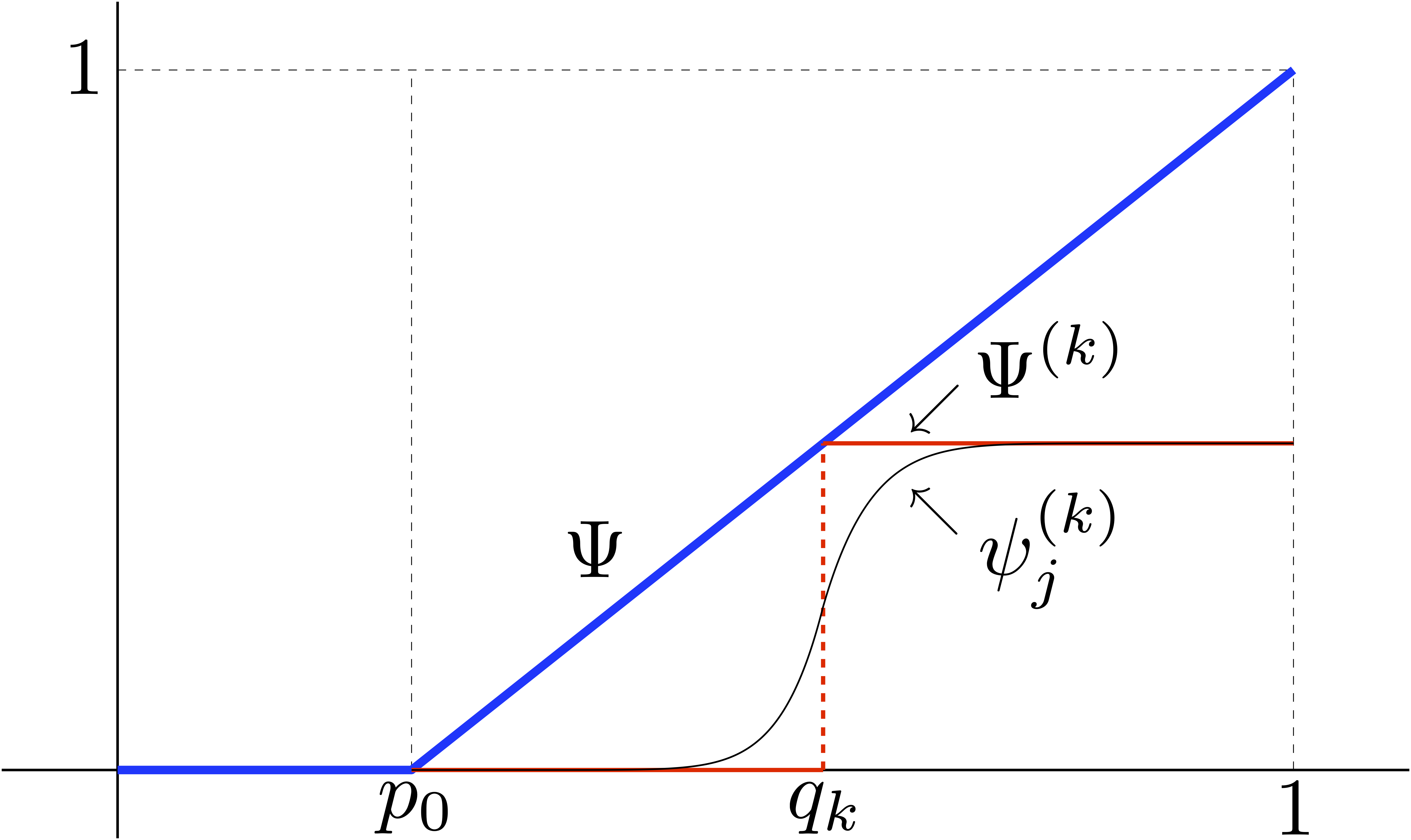}
\caption{Illustration for Example \ref{better_choice_eg}.
The thick curve on the diagonal is $\Psi$, the piecewise constant function is $\Psi^{(k)}$, and the strictly increasing curve approximating $\Psi^{(k)}$ is $\psi_j^{(k)}$.
As $j\to\infty$, we demand $\psi_j^{(k)}(u)\to\Psi^{(k)}(u)$ for Lebesgue-almost every $u\in[p_0,1]$.}
\label{better_choice_fig}
\end{figure}

Since $\psi^{(k)}_j\leq \Psi$ and $\|\Psi\|_\infty = 1$, condition \ref{family_2} is clear.
For \ref{family_3}, recall that every $\sigma\in\Sigma_\Lambda$ admits a density $\dd\sigma/\dd\Lambda\in L^1(\Lambda)$.
By dominated convergence (the dominating function being $\dd\sigma/\dd\Lambda$), we have
\eq{
\lim_{j\to\infty} \langle \bar\psi^{(k)}_j,\sigma\rangle 
= \lim_{j\to\infty} \int_0^1 \bar\psi^{(k)}_j(u)\frac{\dd\sigma}{\dd\Lambda}(u)\ \dd u 
&= \int_0^1 \Psi^{(k)}(u)\frac{\dd\sigma}{\dd\Lambda}(u) \\ 
&= \frac{q_k-p_0}{1-p_0}\langle \one_{[q_k,1]},\sigma\rangle.
}
Therefore, if $\langle \bar\psi_j^{(k)},\sigma_1\rangle = \langle\bar \psi_j^{(k)},\sigma_2\rangle$ for all $j$, then $\langle \one_{[q_k,1]},\sigma_1\rangle = \langle \one_{[q_k,1]},\sigma_2\rangle$.
In turn, if the latter equality holds for all $k$, then $\sigma_1|_{[p_0,1]} =\sigma_2|_{[p_0,1]}$ as desired.

Now let $\tau$ be any continuous function that vanishes on $[0,p_0]$, is strictly increasing on $[p_0,1)$, and diverges to $\infty$ at $1$.
Enumerate the elements of $\FFF$ as $\psi_1,\psi_2,\dots$, and consider any $\vc t$ belonging to the set $\DDD\cap\ell^1(\N)_+$ guaranteed by Theorem \ref{bounded_emp}.
Because each $\psi_i$ is continuous and strictly increasing on $[p_0,1]$, so too is $\tau+\vc t\cdot\bar{\vc\psi}$.
Furthermore, for each $u\in[p_0,1]$, we have
\eq{
\tau(u) \leq (\tau+\vc t\cdot\bar{\vc\psi})(u) \leq \tau(u) + \|\vc t\|_1\Psi(u) = \tau(u) + \frac{u-p_0}{1-p_0}.
}
Consequently, we have
\eq{
(\tau+\vc t\cdot\bar{\vc\psi})\big|_{[0,p_0]} \equiv 0, \quad
\lim_{u\searrow p_0}(\tau+\vc t\cdot\bar{\vc\psi})(u)=0, \quad 
\lim_{u\nearrow1}(\tau+\vc t\cdot\bar{\vc\psi})(u) = \infty,
} 
which means $(\tau+\vc t\cdot\bar{\vc\psi})_*(\Lambda)$ is an element of $\PPP_\mathrm{emp}(\bxi)\cap\PPP_\mathrm{emp}^\infty(\bxi)$ that has an atom at $0$ of mass $p_0$, and is otherwise continuous and supported on all of $[0,\infty]$.
\end{eg}

\begin{remark}[\textit{Distributions with mixed properties}] \label{mix_remark}
The various theorems in the infinite-dimensional setting could be combined to generate examples of $\LL\in\PPP_\mathrm{emp}(\bxi)\cap\PPP_\mathrm{emp}^\infty(\bxi)$ containing both discrete and continuous components.
More precisely, we could take as given any probabilities $p_0+p_\mathrm{bv}+p_{C^k}+p_{\ell^1}=1$ and any function $\tau \colon  [0,1]\to[0,\infty)$ that vanishes on $[0,p_0+p_\mathrm{bv})$ and is strictly positive on $\big((1+p_\cc(\Z^d))(1-p_{\ell^1}),1\big]$.
Now perturb $\tau$ in the several ways we have seen, depending on the subinterval of $[0,1]$. 
Within $[p_0,p_0+p_\mathrm{bv})$, add the scaled indicator functions from the proof of Theorem \ref{countable_emp}; on $[p_0+p_\mathrm{bv},p_0+p_\mathrm{bv}+p_{C^k}]$, use a strictly positive $\vphi\in C^k$ as in Theorem \ref{continuous_emp}; and on $(1-p_{\ell^1},1]$, perturb by a sum of functions in $\FFF$ as in Theorem \ref{bounded_emp}.  
Applying the relevant convex function theory in the single Banach space $\BBB_\mathrm{bv}\times C^k([p_0+p_\mathrm{bv},1-p_{\ell^1}]) \times \ell^1(\N)$, we can obtain $\LL\in\PPP_\mathrm{emp}(\bxi)\cap\PPP_\mathrm{emp}^\infty(\bxi)$ which, for instance, exhibits the characteristics of Examples \ref{dense_discrete}, \ref{continuous_eg}, \ref{better_choice_eg} simultaneously, each one on the corresponding subinterval.
\end{remark}

\chapter{First-Passage Percolation on $d$-ary Tree} \label{tree_sec}

It is possible to strengthen our results regarding empirical measures when $\Z^d$ is replaced by the infinite (complete) $d$-ary tree $\T_d$, $d\geq2$. \label{tree_d_def}
That is, $\T_d$ is the rooted tree in which every node has exactly $d$ children.
For $x,y\in\T_d$, let us write $y\leq x$ if $y$ is an ancestor of $x$ (allowing $y=x$).
Also, $|x|$ will denote the generation number of $x$, i.e.~the graph distance from $x$ to the root, \label{gen_number_def}
which we denote by $0$.
Once we assign each $x\in\T_d$ with an i.i.d.~random variable $\tau_x$ (not necessarily nonnegative) \label{vertex_weight_def}
having law $\LL$, \label{weight_law_def_2}
we define
\eeq{ \label{tree_FPP_model}
T(x) \coloneqq \sum_{0<y\leq x}\tau_y, \qquad T_n \coloneqq \inf_{x:\, |x|=n} T(x).
}
(Here $n$ is an integer.) 
We will write 
\eeq{ \label{geo_n_def} 
\Geo_n = \{x\in \T_d :\,|x|=n,\, T(x) = T_n\}.
}
Under a suitable moment assumption, there is once again a time constant.

\begin{theirthm} \label{tree_time_constant_thm}
\textup{\cite[Thm.~1.3]{shi15}}
If $\E(\e^{-\alpha\tau_x})<\infty$ for some $\alpha>0$, then
\eeq{ \label{tree_time_constant_def}
\lim_{n\to\infty} \frac{T_n}{n} = \mu_{\T_d} \quad \mathrm{a.s.},
}
where
\eq{ 
\mu_{\T_d} 
= -\inf_{\alpha>0}\frac{\log d + \log \E(\e^{-\alpha\tau_x})}{\alpha}
\in \R.
}
\end{theirthm}

As in the lattice case, we assume $\tau_x = \tau(U_x)$ for some measurable $\tau \colon  [0,1]\to\R$, \label{tau_measurable_def_2}
where $(U_x)_{x\in\T_d}$ is a family i.i.d.~uniform random variables on $[0,1]$. \label{U_x_def_1}
We then define, analogous to \eqref{hat_nu_def} and \eqref{sigma_gamma},
\eeq{ \label{empirical_measures_tree} 
\hat\nu_x = \frac{1}{|x|}\sum_{0<y\leq x} \delta_{\tau_y}, \qquad \hat\sigma_x = \frac{1}{|x|}\sum_{0<y\leq x}\delta_{U_y}.
}
Let $\hat\Sigma$ denote the set of all \mbox{Borel} probability measures on $[0,1]$. \label{hat_Sigma_def}
For $\hat\sigma,\pp\in\hat\Sigma$, let us write $\hat\sigma\ll\pp$ if $\hat\sigma$ is absolutely continuous with respect to $\pp$.
Recall that the \textit{relative entropy} or \textit{Kullback--Leibler divergence} of $\hat\sigma$ with respect to $\pp$ is
\eeq{ \label{KL_div_def} 
\KL{\hat\sigma}{\pp} \coloneqq \begin{cases}
\displaystyle\int_0^1\log\frac{\dd\hat\sigma}{\dd\pp}(t)\ \hat\sigma(\dd t) &\text{if $\hat\sigma\ll\pp$}, \\
\infty &\text{otherwise},
\end{cases}
}
where $\dd\hat\sigma/\dd\pp$ is the Radon--Nikodym derivative.
We then have the following result, which assumes a slightly stronger moment assumption than does Theorem \ref{tree_time_constant_thm}.
Recall that $\tau^- = -\tau\vee0$ denotes the negative part of $\tau$.

\begin{thm} \label{tree_thm}
Assume $\LL = \tau_*(\Lambda)$, where $\tau \colon  [0,1]\to\R$ satisfies 
\eq{ 
\langle\e^{\alpha\tau^-(u)^\beta},\Lambda\rangle<\infty \quad \text{for some $\alpha>0$ and $\beta>1$}.
}
\begin{enumerate}[label=\textup{(\alph*)}]
\item \label{tree_thm_a}
The time constant from \eqref{tree_time_constant_def} is given by the variational formula
\eeq{ \label{linear_functional_tree}
\mu_{\T_d}(\tau) = \inf_{\hat\sigma :\, \KL{\hat\sigma}{\Lambda} \leq \log d}\langle \tau,\hat\sigma\rangle.
}

\item \label{tree_thm_b}
The set of minimizers is nonempty: 
\eeq{ \label{minimizers_tree_def} 
\RR_\tau^{\T_d} \coloneqq \{\hat\sigma\in\hat\Sigma :\, \KL{\hat\sigma}{\Lambda}\leq\log d,\, \langle\tau,\hat\sigma\rangle=\mu_{\T_d}(\tau)\} \neq \varnothing.
}

\item \label{tree_thm_c} 
Almost surely, for any increasing sequence of integers $(n_k)_{k\geq1}$ and any sequence $(x_{n_k})_{k\geq1}$ in $\T_d$ such that $|x_{n_k}|=n_k$ and $T(x_{n_k})/n_k\to\mu_{\T_d}$ as $k\to\infty$,
there is a subsequence $(x_{n_{k_j}})_{j\geq1}$ such that
\eq{
\hat\sigma_{x_{n_{k_j}}}\Rightarrow \hat\sigma \quad \text{as $j\to\infty$}, \quad \text{for some $\hat\sigma\in\RR_\tau^{\T_d}$},
}
in which case
\eq{
\hat\nu_{x_{n_{k_j}}}\Rightarrow \tau_*(\hat\sigma).
}


\item \label{tree_thm_d}
Let $\mathfrak{b}\in[-\infty,\infty)$ \label{bfrak_def}
denote the essential infimum of $\LL$.
If $\P(\tau_x = \mathfrak{b})\geq 1/d$, then $\tau_*(\hat\sigma) = \delta_{\mathfrak{b}}$ for every $\hat\sigma\in\RR_\tau^{\T_d}$.
If instead $\P(\tau_x = \mathfrak{b}) < 1/d$, then $\RR_\tau^{\T_d} = \{\hat\sigma_\star\}$, where $\hat\sigma_\star$ is the unique solution to
\eeq{ \label{minimum_RE}
\min \KL{\hat \sigma}{\Lambda} \quad \text{subject to} \quad \langle \tau,\hat\sigma\rangle = \mu_{\T_d}(\tau).
}

\end{enumerate}
\end{thm}

In particular, statements \ref{tree_thm_c} and \ref{tree_thm_d} combine to show that every sequence of empirical measures $(\hat\nu_{x_n})_{n\geq1}$, with $x_n\in\Geo_n$, converges to a deterministic limit.
The formulation of \eqref{minimum_RE} is meant to emphasize that $\sigma_*$ can be estimated numerically, for instance using the method of \cite{kruk04}.
The proof of Theorem \ref{tree_thm} hinges on the large deviations principle for the empirical measure of i.i.d.~samples from a given distribution.
The details are presented in Chapter \ref{tree_proof}.

\chapter{Negative Weights and Passage Times along Geodesics} \label{modification_sec}

This lengthy section is included primarily for two reasons: (1) so that in the subcritical case $F(0)<p_\cc(\Z^d)$, we can allow slightly negative edge-weights; and (2) so that nowhere else in the manuscript do we require any moment assumption.
The inputs and key ideas come mostly from percolation theory and are largely independent from this monograph's more central themes.
Before stating the results, let us say a word about why they are needed.

Regarding goal (1), we have yet to prove Proposition \ref{slightly_negative_lemma}, which was needed in the proof of Theorem \ref{bounded_emp} to extend the map $\tau\mapsto\mu_\bxi(\tau)$ to a suitable open set.
Only then could we appeal to differentiability results for concave functions on Banach spaces.
Proposition \hyperref[slightly_negative_lemma_a]{\ref*{slightly_negative_lemma}\ref*{slightly_negative_lemma_a}} is implied by Proposition \ref{negative_thm}, while \hyperref[slightly_negative_lemma_c]{\ref*{slightly_negative_lemma}\ref{slightly_negative_lemma_c}} will be included in Theorem \ref{var_formula_thm_sub}.
The main difficulty of negative edge-weights is the loss of subadditivity; that is, $T$ no longer satisfies the triangle inequality.
While a natural way of bypassing this issue is to instead consider so-called cylinder passage times (see \cite[pg.~136]{kesten86}), doing so would not let us say anything about geodesics.
Instead, we will update the construction of ``shell passage times" due to Cox and \mbox{Durrett} \cite{cox-durrett81} and \mbox{Kesten} \cite{kesten86}, with Lemma \ref{subadditive_lemma} serving as the necessary surrogate for subadditivity. 
Tail bounds such as Proposition \hyperref[negative_thm_a]{\ref*{negative_thm}\ref*{negative_thm_a}}, Lemma \ref{chemical_lemma_2}, and Proposition \ref{geo_length_cor}, as well as the shape theorem \eqref{shape_thm}, may be useful inputs in other studies. 
In particular, they go beyond just saying that shell passage times have finite moments.

Concerning (2): without any moment assumption, we may not have the almost sure convergence $T(0,n\bxi)/n\to\mu_\bxi$, which is crucial to proving the variational formula \eqref{var_form_eq} and the convergence results of Theorem \ref{minimizers_thm}.
Fortunately, Proposition \ref{replacement_thm} says we can delete $o(n)$ many edges in order to recover almost sure convergence to the time constant.
Our proof will make use of the shell construction developed for Proposition \ref{negative_thm}.
Arguments of a similar nature appear in \cite[Sec.~6]{auffinger-damron-hanson15II}.
The principal complication we need to resolve is that honest geodesics need not coincide with geodesics in the shell environment.

\section{Notation and statements of key results}

Consider edge-weights of the form $\tilde\tau_e \coloneqq \tau_e + \tau^\pert_e$, $e\in E(\Z^d)$, \label{tilde_edge_weights_def}
where $\tau_e$ and $\tau_e^\pert$ \label{pert_edge_weights_def}
are random variables on the complete probability space $(\Omega,\FF,\P)$.
For $z\in\Z^d$, we write $\mathfrak{T}_z \colon \Omega\to\Omega$
to denote translation of the environment: 
\eeq{ \label{shift_operator} 
\tau_e(\mathfrak{T}_z\omega) = \tau_{e-z}(\omega), \qquad
\tau_e^\pert(\mathfrak{T}_z\omega) = \tau_{e-z}^\pert(\omega),
}
where the translation of the edge $e = \{x,y\}$ has been denoted by $e-z = \{x-z,y-z\}$.
We assume that
\eeqs{perturbed_general_assumptions}{
\tau_e, {e\in E(\Z^d)}, \text{ are i.i.d.~and nonnegative}, \label{perturbed_general_assumptions_a} \\
(\tilde\tau_e)_{e\in E(\Z^d)} \text{ is stationary and ergodic with respect to every $\mathfrak{T}_{z}$}, \label{perturbed_general_assumptions_b} \\
\P\big(|\tau_e^\pert|\leq \af \text{ for all $e\in E(\Z^d)$}\big) = 1, \label{perturbed_general_assumptions_c}
}
where $\af$ is a nonnegative constant whose purpose is explained by Proposition \ref{negative_thm}.
Notice that these assumptions allow for flexibility in the joint law of $(\tau_e,\tau^\pert_e)$.
For instance, $(\tau_e)_{e\in E(\Z^d)}$ and $(\tau_e^\pert)_{e\in E(\Z^d)}$ could be independent families; in the other extreme, we could have $\tau_e = \tau(U_e)$ and $\tau_e^\pert = \tau^\pert(U_e)$ as in Proposition \ref{slightly_negative_lemma}.
The law $\LL$ and the distribution function $F$ are always associated to just $\tau_e$ as in \eqref{distribution_fnc_def}, rather than $\tilde\tau_e$.
So that we can still on occasion write $T$ as in \eqref{fpp_def} for passage times with respect to the $\tau_e$'s, we define separate notation for passage times with respect the $\tilde\tau_e$'s:
\eeq{ \label{fpp_def_perturbed}
\wt T(x,y) \coloneqq \inf_{\gamma \in \PP(x,y)} \wt T(\gamma), \quad \text{where} \quad \wt T(\gamma) \coloneqq \sum_{e\in\gamma} \tilde \tau_e.
}
There will be no need for us to decorate other notations such as $\mu_\bxi$ and $\Geo$, as \underline{these will always be taken with respect to the $\tilde\tau_e$'s}.

Throughout this section, $C$ and $c$ will denote positive constants depending only on the edge-weight distribution $\LL$ and the dimension $d$. \label{C_c_def}
In particular, these constants will never depend on $n$, the direction $\bxi$, locations in $\Z^d$, or the perturbation variables $(\tau_e^\pert)_{e\in E(\Z^d)}$. 
For notational simplicity, though, we do allow the values of $C$ and $c$ to change from line to line.

Our first result is that if $\tau^\pert_e$ is not allowed to be too large (so that $\tilde\tau_e$ is not too negative), then the FPP model \eqref{fpp_def_perturbed} with self-avoiding paths is still well-behaved.
The following proposition generalizes \cite[Thm.~2.26]{kesten86} and, for our purposes, replaces Theorem \ref{time_constant_thm}.

\begin{prop} \label{negative_thm}
If $F(0)<p_\cc(\Z^d)$, then there exists $\af = \af(\LL,d)>0$ such that whenever \eqref{perturbed_general_assumptions} holds, there is a collection of random variables $\{\wh T(x,y) :\, x,y\in\Z^d\}$, defined in \eqref{wh_T_def}, with the following properties.
(As before, we declare $\wh T(x,y) = \wh T([x],[y])$ for $x,y\in\R^d$.)
\begin{enumerate}[label=\textup{(\alph*)}]

\item \label{negative_thm_a}
For all $x,y\in\Z^d$, we have
\eq{
\P(\wh T(x,y) \leq -s) \leq C\e^{-cs} \quad \text{for all $s\geq0$},
}
as well as
\eq{
\P(\wh T(x,y)\geq s\|x-y\|_\infty) \leq C\e^{-c(s\|x-y\|_\infty)^{1/d}} \quad \text{for all $s\geq \mathfrak{s} =\mathfrak{s} (\LL,d)$}.
}

\item \label{negative_thm_b}
The family $\{\wh T(x,y)-\wt T(x,y) :\, x,y\in\Z^d\}$ is tight. 

\item \label{negative_thm_c}
For every ${\bxi}\in\S^{d-1}$, there is a constant $\mu_\bxi\in[\af,\infty)$ such that
\eeq{ \label{shell_to_mu}
\lim_{n\to\infty} \frac{\wh T(0,n\bxi)}{n} = \mu_\bxi \quad \mathrm{a.s.}\text{ and in $L^p$, $p\in[1,\infty)$}.
}

\item \label{negative_thm_c2}
The map $\bxi\mapsto\mu_{\bxi}$ is continuous, and
\eeq{ \label{shape_thm}
\lim_{n\to\infty}\sup_{\bxi\in\S^{d-1}}\Big|\frac{\wh T(0,n\bxi)}{n} -\mu_\bxi\Big| =0 \quad \mathrm{a.s.}\text{ and in $L^p$, $p\in[1,\infty)$}.
}

\end{enumerate}
If $F(0)\geq p_\cc(\Z^d)$, then the same statements hold with $\af=0$ and $\mu_\bxi\equiv0$.
\end{prop}

\begin{remark}
Part \ref{negative_thm_b} ensures that when $\tau^\pert_e=0$ for all $e\in E(\Z^d)$, the time constant $\mu_{\bxi}$ coincides with the one from Theorem \ref{time_constant_thm}.
Also, the notation $\wh T$ is not related to the hat decoration in $\hat\nu$, which denotes the normalization of a measure $\nu$.
\end{remark}

For completeness (and eventually to prove Theorem \ref{derivative_ineq_thm} in Section \ref{length_lemma_proof}), we check the optimality of Proposition \ref{negative_thm} in the critical and supercritical cases.
Unsurprisingly, if $F(0) \geq p_\cc(\Z^d)$, then no positive value for $\af$ can be allowed.
The following is a restatement of Proposition \ref{slightly_negative_lemma_2}.

\begin{prop} \label{negative_thm_2}
If $F(0)\geq p_\cc(\Z^d)$ and $\tau^\pert_e= -h$ for all $e \in E(\Z^d)$, with $h>0$, then almost surely $\wt T(x,y)=-\infty$ for all $x\neq y$.
\end{prop}

In preparation for our next result, we introduce the following notation.
Given a path $\gamma\in\PP(x,y)$, consider the natural ordering of its edges $e_1,\dots,e_{|\gamma|}$ with initial vertex $x\in e_1$ and terminal vertex $y\in e_{|\gamma|}$.
For positive integers $a_0$ and $a_1$ such that $a_0+a_1\leq|\gamma|$, define 
\eeq{ \label{delete_edges_def} 
\gamma^{(a_0,a_1)}\coloneqq\{e_{1+a_0},\dots,e_{|\gamma|-a_1}\}.
}
That is, $\gamma^{(a_0,a_1)}$ is the subpath of $\gamma$ obtained by deleting its first $a_0$ edges and its last $a_1$ edges.

\begin{prop} \label{replacement_thm}
Assume \eqref{perturbed_general_assumptions}, where in \eqref{perturbed_general_assumptions_c} the constant $\af\geq0$ is from Proposition \ref{negative_thm}.
There exist random nonnegative integers $A_0$ and $(A_n)_{n\in[1,\infty)}$ for which the following statements hold:
\begin{enumerate}[label=\textup{(\alph*)}]

\item \label{replacement_thm_a}
For any $\eps>0$, we almost surely have $A_n/n^\eps\to0$ as $n\to\infty$.

\item \label{replacement_thm_b}
For each $x\in\Z^d\setminus\{0\}$ and $\gamma\in\Geo(0,x)$, there are nonnegative integers $a_0(\gamma)$ and $a_1(\gamma)$ such that 
\eq{
a_0(\gamma)\leq A_0 \quad \text{and} \quad a_1(\gamma)\leq A_n \quad \text{whenever $x = [n\bxi]$ for some $\bxi\in\S^{d-1}$},
}
and the following limit holds almost surely and in $L^p$, $p\in[1,\infty)$:
\eeq{ \label{adjusted_convergence}
\lim_{n\to\infty}\sup_{\substack{\bxi\in\S^{d-1}\\ \gamma\in\Geo(0,n\bxi)}}\Big|\frac{\wt T(\gamma^{(a_0(\gamma),a_1(\gamma))})}{n} - \mu_\bxi\Big| = 0.
}

\item \label{replacement_thm_c}
The random integers $A_0$ and $A_n$ can be chosen to satisfy the bounds
\eq{
\P(A_0\geq r) \leq C\e^{-cr^{1/d}},\ \P(A_n \geq r) \leq Cn^d\e^{-cr^{1/d}} \text{ for all $r\geq1$, $n\in[1,\infty)$}.
}

\item \label{replacement_thm_d}
Finally, if $\esssup\tau_e<\infty$, then part \ref{replacement_thm_b} holds even if one takes $A_0=0$ and $A_n=0$ for all $n\in[1,\infty)$.
\end{enumerate}
\end{prop}

\begin{remark}
If we only need that \ref{replacement_thm_b} holds for a single deterministic direction $\bxi$, then the factor of $n^d$ can be removed from the second tail bound in \ref{replacement_thm_c}, meaning $(A_n)_{n\in[1,\infty)}$ is tight in this restricted setting.
Irrespective of this comment, if we only demand almost sure convergence in \eqref{adjusted_convergence}, then we can take $A_0=0$.
This is because $A_0$ is only needed to remove high-weight edges near $0$, and
although these edges will not change the almost sure convergence, they can spoil integrability.
\end{remark}

When dealing with infinite geodesics, we will also make use of the following sister result.
I do not believe it has appeared in the literature before.
It says that infinite geodesics almost surely achieve the correct time constant even without any moment assumption on the edge-weights.
For any infinite path, we label its vertices in the order traversed as $x_0,x_1,\dots$.
Also recall the subpath notation $\Gamma^{(\ell)}$ from \eqref{sub_of_infinite}.

\begin{prop} \label{addition_thm}
Assume \eqref{perturbed_general_assumptions}, where in \eqref{perturbed_general_assumptions_c} the constant $\af\geq0$ is from Proposition \ref{negative_thm}.
For any $x\in\Z^d$, we have
\eeq{ \label{sup_infinite}
\lim_{n\to\infty}\sup_{\bxi\in\S^{d-1}}\sup_{\substack{\Gamma\in\Geo_\infty\\
x_0=x,\, x_\ell=x+[n\bxi]}}\Big|\frac{\wt T(\Gamma^{(\ell)})}{n}- \mu_\bxi\Big| = 0 \quad \mathrm{a.s.}
}
In particular, if $\Gamma\in\Geo_\infty(\bxi)$, then
\eeq{ \label{single_infinite}
\lim_{n\to\infty} \frac{\wt T(\Gamma^{(\ell)})}{x_\ell\cdot\bxi} = \mu_\bxi.
}
\end{prop}

The limit in \eqref{sup_infinite} could be made valid in $L^p$, $p\in[1,\infty)$, if we allowed ourselves to delete a random, unit-order number of edges from the beginning of $\Gamma$.
We have not stated the result in this way for two reasons.
First, the main distinction between Propositions \ref{replacement_thm} and \ref{addition_thm} is that in the latter, we do not need to delete \textit{any} edges, even near the vertex $[n\bxi]$.
Second, the statement \eqref{sup_infinite} holds for a \textit{fixed} starting vertex $x_0=x$.
Therefore, the utility of Proposition \ref{addition_thm} comes from the fact that, on a probability-one event, \eqref{sup_infinite} holds simultaneously for every $x\in\Z^d$.
Consequently, \eqref{single_infinite} is valid for any $\Gamma\in\Geo_\infty(\bxi)$, regardless of its starting point.

Sections \ref{shell_review}, \ref{coupling_bernoulli}, \ref{chemical} accumulate a set of preliminary results that will allow us prove Propositions \ref{negative_thm}, \ref{negative_thm_2}, \ref{replacement_thm}, \ref{addition_thm} in Section \ref{reckoning}.

\section{The Cox--Durrett--Kesten shell passage time} \label{shell_review}
Here we review only the essential features of the shell construction; a full treatment can found in \cite[pg.~137--153]{kesten86}.
Fix a constant $\tf $ \label{tfrak_def}
that is large enough (depending on $d$ and $\LL$).
Each site $x\in\Z^d$ is assigned a color. 
We take $Y_x=1$ (\textit{white}) if all edges $e$ containing $x$ have weight $\tau_e\leq \tf $, and $Y_x=0$ (\textit{black}) otherwise.
Let $\WW$ denote the union of all infinite components of the subgraph of $\Z^d$ induced by the white vertices.
\label{infinite_white_cluster}
For each $x\in\Z^d$, a random \textit{shell} $\SS(x) = \SS(x,\omega)\subset\Z^d$ \label{shell_def}
around $x$ can be identified such that:
\begin{enumerate}[label=\textup{(\roman*)}]
\item \label{condition_1}
\cite[disp.~(2.22) and pg.~141]{kesten86} The shell $\SS(x)$ contains only white vertices, does not contain $x$, and almost surely contains some $w\in \WW$.

\item \label{condition_2}
\cite[Lem.~2.23 and 2.24]{kesten86} Almost surely, $\SS(x)$ is finite and connected, in which case every infinite self-avoiding path starting at $x$ must intersect $\SS(x)$.

\item \label{condition_3} \cite[Lem.~2.24]{kesten86}
The following two inequalities hold for all $r\geq1$:
\eeqs{condition_2_eq}{
\P\Big(\inf_{w \in \SS(x)} \|x-w\|_\infty \geq r\Big) &\leq C\e^{-cr}, \\ 
\P\Big(\sup_{w,z\in \SS(x)} \|w-z\|_\infty \geq r\Big) &\leq C\e^{-cr}. \label{x}
}

\item \label{condition_4}
The value of $\SS(x)$ depends only on $(\tau_e)_{e\in E(\Z^d)}$.
Furthermore, for any $x,z\in\Z^d$, we have $\SS(x,\mathfrak{T}_z\omega) = \SS(x-z,\omega)$.

\item \label{condition_5}
If $\tau_e$ is bounded, then we assume $\tf\geq\esssup\tau_e$ so that \ref{condition_1}--\ref{condition_4} are satisfied by $\SS(x) = \{y\in\Z^d:\, \|x-y\|_1 = 1\}$.
\end{enumerate}

We will repeatedly refer back to these properties, in order to demonstrate that the following random variables satisfy Proposition \ref{negative_thm}:
\eeq{ \label{wh_T_def}
\wh T(x,y) \coloneqq \wt T(\SS(x),\SS(y)), \quad x,y\in\Z^d,
}
where $\wt T(S_1,S_2) \coloneqq \inf_{w \in S_1,z\in S_2} \wt T(w,z)$ for $S_1,S_2\subset\Z^d$.
Note that properties \ref{condition_1} and \ref{condition_2} together imply 
\eeq{ \label{S_in_W}
\SS(x)\subset \WW \quad \text{for all $x\in\Z^d$, almost surely.}
}
The following observation is essentially contained in \cite{kesten86}, but we include a proof for completeness.

\begin{lemma} \label{single_cluster_lemma}
Almost surely, $\WW$ has a single  component.
\end{lemma}

\begin{proof}
By \cite[disp.~(2.30)]{kesten86}, if $x$ and $y$ are adjacent, then there exists $z\in\Z^d$ with the following property:
There are nearest-neighbor paths from $z$ to $\SS(x)$ and from $z$ to $\SS(y)$ 
that use only white vertices.
Therefore, $\SS(x)$ and $\SS(y)$ belong to the same  component of $\WW$.
Since this holds for every pair $x,y$ of adjacent vertices, we deduce that $\bigcup_{x\in\Z^d} \SS(x)$ is contained in a single  component of $\WW$.
Finally, if an infinite self-avoiding path passes through $x\in\Z^d$, then it is also passes through $\SS(x)$ by property \ref{condition_2}.
Consequently, there can be no  component of $\WW$ other than the one containing $\bigcup_{x\in\Z^d} \SS(x)$.
\end{proof}

Meanwhile, property \ref{condition_3} can be leveraged in the following manner.
We denote a $d$-dimensional box centered at $x\in\Z^d$ by
\begin{subequations} \label{box_defs}
\begin{align} \label{box_def_a}
\Bbf_r(x) &\coloneqq \{y \in \Z^d :\, \|x-y\|_\infty \leq r\},
\end{align}
and its boundary by
\begin{align} \label{box_def_b}
\partial \Bbf_r(x) &\coloneqq \{y \in \Z^d :\, \|x-y\|_\infty = r\}.
\end{align}
\end{subequations}
It will be useful to define the quantity
\eeq{ \label{radius_x_def} 
\Rvc(x) \coloneqq \sup_{z\in \SS(x)} \|x-z\|_\infty, 
}
so that $\SS(x)\subset \Bbf_{\Rvc(x)}(x)$ and thus $|\SS(x)|\leq (2\Rvc(x)+1)^d$.
Considering that
\eq{
\Rvc(x) \leq \inf_{w\in \SS(x)} \|x - w\|_\infty + \sup_{w,z\in \SS(x)} \|w - z\|_\infty,
}
the inequalities in \eqref{condition_2_eq} give
\eeq{ \label{condition_2_eq_c}
\P(\Rvc(x)>r) 
\leq C\e^{-cr}. 
}
Let $\PP(x)$ denote the set of all self-avoiding paths starting at $x\in\Z^d$. \label{path_x_def}
As demonstrated in the next two lemmas, \eqref{condition_2_eq_c} leads to tail bounds for the following quantities, which are defined for $x\in\Z^d$, $h\in\R$, and $q>0$:
\eeq{ \label{LQ_def}
\Avc(x) &\coloneqq \sup\big\{|\gamma| :\, \text{$\gamma\in\PP(x)$, $\gamma$ avoids $\SS(x)$}\big\}+1, \\
\Lvc(x;h) &\coloneqq \sup\big\{|\gamma| :\, \gamma\in\PP(x),\, \wt T(\gamma)<h|\gamma|\big\}\vee0, \\
\Qvc(x;h,q) &\coloneqq \sup\big\{|\gamma| :\, \gamma\in\PP(x),\, |\{e\in\gamma :\, \tilde\tau_e\leq h\}|\geq q|\gamma|\big\}.
}

\begin{lemma} \label{L_tail_bd_lemma}
For every $x\in\Z^d$ and $r\geq0$, we have
\eeq{ \label{Lx_tail}
\P(\Avc(x) \geq r ) \leq C\e^{-cr^{1/d}}.
}
\end{lemma}

\begin{proof}
The claimed bound is immediate from \eqref{condition_2_eq_c} once we show that
\eq{
\Avc(x) \leq (2\Rvc(x)+1)^d \quad \text{on the event that $\SS(x)$ is finite and connected}.
}
Indeed, let $\gamma$ be any self-avoiding path starting at $x$ and having length $|\gamma|$ equal to $(2\Rvc(x)+1)^d = |\Bbf_{\Rvc(x)}(x)|$.
Because of this length assumption, $\gamma$ must reach some vertex $y\notin \Bbf_{\Rvc(x)}(x)$.
Since $\SS(x)\subset \Bbf_{\Rvc(x)}(x)$, it is possible to construct an infinite self-avoiding path that starts at $y$ and never passes through $\SS(x)$.
Therefore, by property \ref{condition_2} of the shell $\SS(x)$, the path $\gamma$ must hit $\SS(x)$ prior to reaching $y$.
We conclude $\Avc(x) \leq |\gamma| = (2\Rvc(x)+1)^d$, as needed.
\end{proof}

The next lemma explains the origin of $\af$ in Proposition \ref{negative_thm}.

\begin{lemma} \label{eta_choice_lemma}
If $F(0)<p_\cc(\Z^d)$, then exist constants $\af = \af(\LL,d)>0$ and $\qf  = \qf (\LL,d)<1$ such that whenever \eqref{perturbed_general_assumptions_a} and \eqref{perturbed_general_assumptions_c} hold, the following inequalities hold for all $x\in\Z^d$, $h\leq \af$, $q\geq \qf $, and $r\geq0$:
\eeq{ \label{Qx_tail}
\P(\Lvc(x;h) \geq r) \leq C\e^{-cr} \quad \text{and} \quad\ \
\P(\Qvc(x;h,q) \geq r) \leq C\e^{-cr}.
}
If $F(0)\geq p_\cc(\Z^d)$, then we take $\af=0$ so that $\Lvc(x;h)=0$ for all $h\leq0$.
\end{lemma}

\begin{proof}
By \cite[Prop.~5.8]{kesten86}, there exists a constant $h_0 = h_0(\LL,d)>0$ such that
\eeq{ \label{kesten_tail}
\P\Big(\exists\,\gamma\in\PP(x) \text{ such that } |\gamma| \geq r,\, \sum_{e\in\gamma}(\tau_e\wedge1) \leq h_0|\gamma|\Big) \leq C\e^{-cr}.
}
Upon taking $\af = (h_0\wedge1)/4$ and assuming $h\leq \af$ and $|\tau_e^\pert|\leq \af$ for all $e\in E(\Z^d)$, we have
the following for all $r\geq1$:
\eq{
\Lvc(x;h) \geq r \ &\implies \
\exists\,  \gamma\in\PP(x) \text{ s.t.~} |\gamma| \geq r,\, \wt T(\gamma) < h |\gamma|\leq \af|\gamma| \\
&\implies\  
\sum_{e\in\gamma}(\tau_e\wedge 1) \leq T(\gamma) \leq \wt T(\gamma) + \af|\gamma| \leq 2\af|\gamma| \leq h_0|\gamma|.
}
If in addition we set $\qf  = (1-3\af)/(1-2\af)$ and consider $q\geq \qf $, then
\eq{
\Qvc(x;h,q) \geq r \ 
&\implies\ \exists\,  \gamma\in\PP(x) \text{ s.t.~} |\gamma| \geq r,\, |\{e\in\gamma :\, \tilde\tau_e\leq h\}| \geq q|\gamma| \\
&\implies\ |\{e\in\gamma :\, \tau_e\leq2\af\}| \geq \qf |\gamma| \\
&\implies\ \sum_{e\in\gamma}(\tau_e\wedge 1) \leq 2\af  \qf |\gamma| + (1-\qf )|\gamma|
=3\af |\gamma| \leq h_0|\gamma|.
}
Therefore, \eqref{Qx_tail} follows from \eqref{kesten_tail} and \eqref{perturbed_general_assumptions_c}.
We assume $C\geq1$ so that  \eqref{Qx_tail} holds for $r=0$.
\end{proof}

\begin{lemma} \label{Q_trivial_lemma}
Assume \eqref{perturbed_general_assumptions_c}.
Almost surely we have
\eeq{ \label{trivial_neg_bd}
-\wt T(\gamma) \leq \af \min\{\Lvc(x;0),\Lvc(y;0)\} \quad \text{for any $x,y\in\Z^d$, $\gamma\in\PP(x,y)$}.
}
\end{lemma}

\begin{proof}
We have $\wt T(\gamma)< 0$ only if $\Lvc(x;0)\geq|\gamma|$.
In addition, since $\tau_e\geq0$, we trivially have $\wt T(\gamma) \geq -\af |\gamma|$ by \eqref{perturbed_general_assumptions_c}.
These two observations collectively yield $-\wt T(\gamma) \leq \af \Lvc(x;0)$.
By symmetry, we also have $-\wt T(\gamma) \leq \af \Lvc(y;0)$, hence \eqref{trivial_neg_bd}.
\end{proof}

\begin{lemma} \label{T_hat_lower}
Assume \eqref{perturbed_general_assumptions_a} and \eqref{perturbed_general_assumptions_c}, where $\af \geq0$ is the constant from Lemma \ref{eta_choice_lemma}.
For all $x,y\in\Z^d$, we have
\eeq{ \label{T_hat_lower_eq}
\P(\wh T(x,y)\leq-s) \leq C\e^{-cs} \quad \text{for all $s\geq0$}.
}
\end{lemma}

\begin{proof}
If $F(0)\geq p_\cc(\Z^d)$ so that $\af =0$, then $\wh T(x,y)\geq0$.
Consequently, the claim holds so long as $C\geq1$.
Therefore, let us assume $F(0)<p_\cc(\Z^d)$ so that $\af >0$.
By Lemma \ref{Q_trivial_lemma}, we have
\eq{
\wh T(x,y) = \wt T(\SS(x),\SS(y)) \geq -\af \sup_{w\in \SS(x)} \Lvc(w;0) \quad \mathrm{a.s.}
}
Therefore, for any positive integer $s$, a union bound gives
\eq{
\P(\wh T(x,y)\leq -s)
&\stackrefp{condition_2_eq_c}{\leq} \P(\Rvc(x)>s) 
+ \P\Big(\sup_{w\in \Bbf_s(x)} \Lvc(w;0) \geq \frac{s}{\af }\Big) \\
&\stackref{condition_2_eq_c}{\leq} C\e^{-cs} + 2(2s+1)^d\P\Big(\Lvc(x;0) \geq \frac{s}{ \af }\Big) 
\stackref{Qx_tail}{\leq} C\e^{-cs}.
}
By suitably adjusting $C$, the upper bound trivially extends to all $s\geq0$.
\end{proof}

\section{Coupling with Bernoulli percolation} \label{coupling_bernoulli}

Let $p\in(p_\cc(\Z^d),1)$ be a parameter (sufficiently close to $1$) to be chosen later in Lemma \ref{from_adh_lemma}.
We color each edge $e=\{x,y\}\in E(\Z^d)$ according to the rule
\eq{
Y_e \coloneqq \begin{cases} 1 &\text{if $Y_x=Y_y=1$,} \\
0 &\text{otherwise}.
\end{cases}
}
Note that $\P(Y_e = 1) = \P(\tau_e\geq \tf )^{4d-1}$, and that $Y_{e}$ and $Y_{e'}$ are independent whenever $e$ and $e'$ are distance at least $2$ apart (by which we mean the minimum graph distance between a vertex of $e$ and a vertex of $e'$ is at least $2$).
In other words, $(Y_e)_{e\in E(\Z^d)}$ is a $1$-dependent process, and $\P(Y_e=1)$ can be made arbitrarily close to $1$ by choosing $\tf $ sufficiently large.
It is thus possible, by either \cite[Thm.~0.0]{liggett-schonmann-stacey97} or \cite[Thm.~7.65]{grimmett99}, to take $\tf $ large enough that $(Y_e)_{e\in E(\Z^d)}$ stochastically dominates the Bernoulli$(p)$ product measure on $E(\Z^d)$.\footnote{Strictly speaking, these references deal with $\Z^d$-indexed $k$-dependent processes, but the results can be understood equally well for $E(\Z^d)$-indexed processes.
For example, one could embed $E(\Z^d)\hookrightarrow\Z^{d+1}$ via $\{x,x+\mathbf{e}_i\} \mapsto (x,i)$, extend $Y$ to the remainder of $\Z^{d+1}$ as independent Bernoulli($p$) random variables, apply the result for $\Z^{d+1}$-indexed $(k+d)$-dependent processes, and then reverse the embedding.
}
To be precise, let $(X_e)_{e\in E(\Z^d)}$ be a collection of independent Bernoulli$(p)$ random variables.
In the usual percolation parlance, we say $e$ is \textit{open} if $X_e=1$, and \textit{closed} if $X_e=0$.
An \textit{open cluster} is a component of the subgraph of $\Z^d$ induced by the open edges.
The stochastic domination discussed above means we can couple $(X_e)_{e\in E(\Z^d)}$ and $(Y_e)_{e\in E(\Z^d)}$ in such a way that
\eeq{ \label{open_is_white}
X_e \leq Y_e \quad \text{for all $e\in E(\Z^d)$}.
}
In particular, if $\OO\subset \Z^d$ denotes the infinite open cluster \label{infinite_open_cluster}
(for uniqueness, see \cite[Thm.~8.1]{grimmett99}), then $\OO\subset\WW$.
Our choice of $p$ is dictated by the following lemma.

\begin{lemma} \label{from_adh_lemma}
\textup{\cite[Lem.~6.3]{auffinger-damron-hanson15II}}
There exists $p_0\in(p_\cc(\Z^d),1)$ such that for every $p\in[p_0,1]$, there is some constant $c_p>0$ satisfying
\eeq{ \label{from_adh}
\P(\text{some path $x\to\partial \Bbf_r(x)$ avoids $\OO$}) \leq \e^{-c_pr} \quad \text{$\forall$ $x\in\Z^d$, $r\geq1$}.
}
\end{lemma}

We henceforth assume $p\in[p_0,1)$.

\section{Chemical distances in the clusters} \label{chemical}

Call a path \textit{open} if all of its edges are open.
For $x,y\in\Z^d$, let us write $x\leftrightarrow y$ if $x$ and $y$ are connected via some open path.
The minimum length of such a path is called the \textit{chemical distance}: 
\eeq{ \label{chemical_def_1} 
D_\oo(x,y) \coloneqq \inf\{|\gamma|:\, \gamma\in\PP(x,y), \gamma \text{ is open}\}, \quad x,y\in\Z^d.
}

\begin{lemma} \label{chemical_lemma}
There exists a constant $C_1 = C_1(\LL,d)>0$ such that 
\eeq{ \label{exp_decay_5}
\P\Big(D_\oo(x,y) \geq C_1r \text{ for some $x,y\in\OO\cap \Bbf_{r}(0)$}\Big) \leq C\e^{-cr} \quad \text{$\forall$ $r\geq1$}.
}
\end{lemma}

\begin{proof}
By \cite[Thm.~1.1]{antal-pisztora96}, there is a constant $C_0 = C_0(p)>0$ such that
\eeq{ \label{exp_decay_4}
\P\big(0\leftrightarrow y, D_\oo(0,y) \geq C_0\|y\|_1\big)\leq C\e^{-c\|y\|_1} \quad \text{for all $y\in\Z^d$}.
} 
Suppose there exists some $z\in\OO$ such that $\|z\|_\infty = 2r$.
Then every $x\in \Bbf_r(0)$ satisfies
\eq{
r \leq \|x-z\|_\infty\leq \|x-z\|_1\leq d\|x-z\|_\infty \leq d(\|x\|_\infty+\|z\|_\infty) \leq 3dr.
}
Moreover, if $x,y\in\OO\cap \Bbf_r(0)$, then
\eq{
D_\oo(x,y) \leq D_\oo(x,z) + D_\oo(z,y) \leq 2\sup_{x\in\OO\cap \Bbf_r(0)} D_\oo(x,z).
}
From these trivial inequalities, we deduce the following implication:
\eq{
x,y\in\OO\cap \Bbf_r(0), D_\oo(x,y) \geq 6C_0dr \  &\implies \  \sup_{x\in\OO\cap \Bbf_r(0)} D_\oo(x,z) \geq 3C_0dr \\
&\implies \  \sup_{x\in \OO\cap \Bbf_r(0)} \frac{D_\oo(x,z)}{\|x-z\|_1} \geq C_0.
}
In summary, we have
\eeq{ \label{two_terms}
&\P\Big(D_\oo(x,y) \geq 6C_0dr \text{ for some $x,y\in\OO\cap \Bbf_{r}(0)$}\Big) \\
&\leq \P(\OO\cap \Bbf_{2r}(0)=\varnothing) 
+\P\bigg(\bigcup_{\|z\|_\infty = 2r}\bigcup_{x\in \Bbf_r(0)} 
\Big\{x\leftrightarrow z,\frac{D_\oo(x,z)}{\|x-z\|_1} \geq C_0\Big\}\bigg). 
}
Concerning the first term on the right-hand side, observe that if some path $0\to\partial \Bbf_{2r}(0)$ hits $\OO$ before or upon reaching $\partial \Bbf_{2r}(0)$, then $\Bbf_{2r}(0)$ necessarily intersects $\OO$.
Hence
\eq{
\P(\Bbf_{2r}(0)\cap\OO=\varnothing) &\leq \P(\text{every path $0\to\partial \Bbf_{2r}(0)$ avoids $\OO$}) \\
&\leq \P(\text{some path $0\to\partial \Bbf_{2r}(0)$ avoids $\OO$}) 
\stackref{from_adh}{\leq} \e^{-cr}.
}
Meanwhile, the second term on the right-hand side of \eqref{two_terms} is controlled by a simple union bound:
\eq{
&\P\bigg(\bigcup_{\|z\|_\infty = 2r}\bigcup_{x\in \Bbf_r(0)} 
\Big\{x\leftrightarrow z,\frac{D_\oo(x,z)}{\|x-z\|_1} \geq C_0\Big\}\bigg) \\
&\stackref{exp_decay_4}{\leq} |\partial \Bbf_{2r}(0)|\cdot|\Bbf_r(0)|\cdot C \e^{-cr}
\leq C\e^{-cr}.
} 
Using the two previous displays in \eqref{two_terms}, we obtain the desired result with $C_1 = 6C_0d$.
\end{proof}

Recall that an edge $e$ is colored white if and only if both of its vertices are white.
Let us say that a path is \textit{white} if all of its edges are white.
A chemical distance can be defined with respect to these paths: 
\eeq{ \label{chemical_def_2} 
D_\ww(w,z) \coloneqq \inf\{|\gamma| :\, \gamma\in\PP(w,z), \gamma \text{ is white}\}, \quad w,z\in\Z^d.
}
Because of \eqref{open_is_white}, all open paths are necessarily white.
Hence $D_\ww\leq D_\oo$.

\begin{lemma} \label{chemical_lemma_2}
There exists a constant $C_2 = C_2(\LL,d)>0$ such that 
\eeq{ \label{exp_decay_6}
\P\Big(D_\ww(w,z) \geq C_2r \text{ for some $w,z\in\WW\cap \Bbf_{r}(0)$}\Big) \leq C\e^{-cr^{1/d}} \quad \text{$\forall$ $r\geq1$}.
}
\end{lemma}

\begin{proof}
Let $C_1$ be as in Lemma \ref{chemical_lemma}, and suppose $D_\oo(x,y) \leq 2C_1r$ for every $x,y\in\OO\cap \Bbf_{2r}(0)$.
Suppose further that the following event occurs:
\eq{
\Asf_r \coloneqq \bigcap_{w\in \Bbf_r(0)} \Big\{\text{every path $w\to\partial \Bbf_{\lfloor\frac{r^{1/d}-1}{2}\rfloor}(w)$ intersects $\OO$}\Big\}.
}
Since $|\Bbf_{\lfloor(r^{1/d}-1)/2\rfloor}(w)|\leq r$, this event implies that every self-avoiding path $\gamma$ which starts  in $\Bbf_r(0)$ and has length $|\gamma| \geq r$, must intersect $\OO$.
Now consider any $w,z\in\WW\cap \Bbf_r(0)$.
By Lemma \ref{single_cluster_lemma}, there is a white path $\gamma\in\PP(w,z)$ of length $|\gamma| = D_\ww(w,z)<\infty$.
If $|\gamma| \geq r$, then the discussion from above tells us that $\gamma$ reaches some $x\in\OO$ within $r$ edges of $w$.
Similarly, $\gamma$ reaches some $y\in\OO$ within $r$ edges of $z$.
That is,
\eq{
D_\ww(w,x) \leq r \quad \text{and} \quad D_\ww(y,z) \leq r.
}
Since $x$ and $y$ necessarily belong to $\Bbf_{2r}(0)$, we also have
\eq{
D_\ww(x,y)\leq D_\oo(x,y) \leq 2C_1r.
}
Together, the two previous displays yield $D_\ww(w,z) \leq 2(C_1+1)r$.
We deduce from this argument that
\eq{
&\P\Big(D_\ww(w,z) > 2(C_1+1)r \text{ for some $w,z\in\WW\cap \Bbf_{r}(0)$}\Big) \\
&\stackrefp{exp_decay_5,from_adh}{\leq} \P\Big(D_\oo(x,y) \geq 2C_1r \text{ for some $x,y\in\OO\cap \Bbf_{2r}(0)$}\Big)  + \P(\Omega\setminus\Asf_r) \\
&\stackref{exp_decay_5,from_adh}{\leq}
C\e^{-cr} + (2r+1)^dC\e^{-cr^{1/d}} \leq C\e^{-cr^{1/d}}.
}
Hence $C_2 = 3(C_1+1)$ suffices for \eqref{exp_decay_6}.
\end{proof}

We can now obtain an upper-tail bound on $\wh T(x,y)$ to complement the lower-tail bound from Lemma \ref{T_hat_lower}.

\begin{lemma} \label{T_hat_upper}
Assume \eqref{perturbed_general_assumptions_a} and \eqref{perturbed_general_assumptions_c}, where $\af \geq0$ is the constant from Lemma \ref{eta_choice_lemma}.
There exists $\mathfrak{s}  = \mathfrak{s} (\LL,d)>0$ such that
\eq{ 
\P(\wh T(x,y) \geq s\|x-y\|_\infty) \leq C\e^{-c(s\|x-y\|_\infty)^{1/d}} \quad \text{$\forall$ $x,y\in\Z^d$ and $s\geq \mathfrak{s} $}.
}
\end{lemma}

\begin{proof}
Let $C_2$ be the constant from Lemma \ref{chemical_lemma_2}.
Observe that
\eq{
\wh T(x,y) = \wt T(\SS(x),\SS(y)) \stackref{perturbed_general_assumptions_c}{\leq} (\tf + \af )\inf_{w\in \SS(x), z\in \SS(y)} D_\ww(w,z) \quad \mathrm{a.s.},
}
and recall from \eqref{S_in_W} that $\SS(x),\SS(y)\subset\WW$ almost surely.
Furthermore, if $\Rvc(x)\leq r$ and $\Rvc(y)\leq r$, then $\SS(x),\SS(y)\subset \Bbf_{\|x-y\|_\infty+2r}(x)$.
Therefore, for any positive integer $r$ and positive numbers $s$ and $\kappa$, a union bound shows
\eeq{ \label{upper_tail_0}
&\P(\wh T(x,y) \geq s\|x-y\|_\infty) \\
&\leq 2\P(\Rvc(x)>r)
+ \P\bigg(\bigcup_{w,z\in\WW\cap \Bbf_{\|x-y\|_\infty+2r}(x)}\Big\{D_\ww(w,z) \geq \frac{s\|x-y\|_\infty}{\tf + \af }\Big\}\bigg), 
}
where we have used the translation invariance from property \ref{condition_4} and \eqref{perturbed_general_assumptions_a} to write $\P(\Rvc(y)>r)=\P(\Rvc(x)>r)$.
This first term on the right-hand side of \eqref{upper_tail_0} is controlled by \eqref{condition_2_eq_c}.  
Turning our attention to the second term, we set $\mathfrak{s}  = 6C_2(\tf + \af )$ and $\kappa = 1/\mathfrak{s} $, so that the following inequalities hold for all $s\geq \mathfrak{s} $:
\eq{
\frac{1}{\tf +\af } &\geq 3C_2(s^{-1}+\kappa) \\
\implies \frac{s\|x-y\|_\infty}{\tf +\af } &\geq 3C_2(1+\kappa s)\|x-y\|_\infty \\
&\geq C_2(\|x-y\|_\infty + 2\lceil \kappa s\|x-y\|_\infty\rceil).
}
Taking $r = \lceil \kappa s\|x-y\|_\infty\rceil$, the tail bound \eqref{exp_decay_6} now shows
\eq{ 
&\P\Big(D_\ww(w,z) \geq \frac{s\|x-y\|_\infty}{\tf + \af }\text{ for some $w,z\in\WW\cap \Bbf_{\|x-y\|_\infty + 2r}(x)$}\Big) \\
&\leq C\e^{-c((1+2\kappa s)\|x-y\|_\infty)^{1/d}} \leq
C\e^{-c(\kappa s\|x-y\|_\infty)^{1/d}} \quad \text{for all $s\geq \mathfrak{s} $}.
}
Upon inserting the same value of $r$ into \eqref{upper_tail_0} and then absorbing $\kappa$ into the constant $c$, we arrive at
\eeq{
\label{T_hat_upper_eq}
\P(\wh T(x,y) \geq s\|x-y\|_\infty) &\leq C\e^{-c\kappa s\|x-y\|_\infty} + C\e^{-c(\kappa s\|x-y\|_\infty)^{1/d}} \\
&\leq C\e^{-c(s\|x-y\|_\infty)^{1/d}} \quad \text{for all $s\geq \mathfrak{s} $}. 
}
\end{proof}

\begin{remark}
I expect that the correct exponent on $s$ in \eqref{T_hat_upper_eq} is $1$ rather than $1/d$, in agreement with \cite[Thm.~1.2]{antal-pisztora96}.
With sufficient effort, one may be able to adapt the methods of that paper to better understand the geometry of the set $\WW$ and make this improvement.
Absent a present necessity, however, we do not pursue this line of inquiry.
\end{remark}

\section{Final arguments} \label{reckoning}

Proposition \hyperref[negative_thm_a]{\ref*{negative_thm}\ref*{negative_thm_a}} has already been established in Corollaries \ref{T_hat_lower} and \ref{T_hat_upper}.

\subsection{Tightness}
For Propositions \hyperref[negative_thm_b]{\ref*{negative_thm}\ref*{negative_thm_b}}, \ref{replacement_thm}, \ref{addition_thm}, we will use the following lemma.
Let $\SS_\mathrm{int}(x)$ denote the ``interior" of $\SS(x)$, \label{interior_shell_def}
i.e.~the set of all $y\in\Z^d\setminus\SS(x)$ such that every infinite self-avoiding path starting at $y$ must intersect $\SS(x)$.
In particular, since $x\notin\SS(x)$ by property \ref{condition_1}, we have $x\in\SS_\mathrm{int}(x)$ by property \ref{condition_2}.
Let $E_\mathrm{int}(x)$ be the set of all edges with at least one vertex in $\SS_\mathrm{int}(x)$.
Recall the quantity $\Lvc(x;h)$ from \eqref{LQ_def}, and define
\eeq{ \label{deltas_def}
\Delta_{\mathrm{int}}(x) &\coloneqq \sum_{e\in E_\mathrm{int}(x)}|\tilde \tau_e|, \\
\Delta_{\mathrm{ext}}(x) &\coloneqq (\tf + \af )|\SS(x)| +  2\af \sup_{w\in\SS(x)}\Lvc(w;0).
}

\begin{lemma} \label{for_tightness_lemma}
Assume \eqref{perturbed_general_assumptions_c}, where $\af \geq0$ is the constant from Lemma \ref{eta_choice_lemma}.
The following statements hold.

\begin{enumerate}[label=\textup{(\alph*)}]

\item \label{for_tightness_lemma_a}
If $x,y\in\Z^d$ and $\gamma\in\PP(x,y)$ satisfy
$\Avc(x)+\Avc(y) < |\gamma|$,
then $\gamma$ must intersect both $\SS(x)$ and $\SS(y)$.
Furthermore, if we let $\gamma'$ denote the (nonempty) portion of $\gamma$ between its first intersection with $\SS(x)$ and its last intersection with $\SS(y)$, as in Figure \ref{shell_times_1_fig}, then
\begin{align}
\label{for_tightness_2}
|\wt T(\gamma') - \wh T(x,y)| &\leq \wt T(\gamma)-\wt T(x,y) +\Delta_{\mathrm{ext}}(x)+\Delta_{\mathrm{ext}}(y) \quad \mathrm{a.s.}
\end{align}

\item \label{for_tightness_lemma_b}
If $x,y\in\Z^d$ satisfy
\eeq{ \label{for_tightness_condition}
\Rvc(x)+\Rvc(y) < \|x-y\|_\infty,
}
then
\eeq{
\label{for_tightness_3}
|\wt T(x,y) - \wh T(x,y)| &\leq \Delta_{\mathrm{int}}(x)+\Delta_{\mathrm{int}}(y) + \Delta_{\mathrm{ext}}(x)+\Delta_{\mathrm{ext}}(y) \quad \mathrm{a.s.}
}

\end{enumerate}
\end{lemma}

\begin{proof}
Assume $\Avc(x)+\Avc(y) < |\gamma|$. 
By the definition of $\Avc(\cdot)$ from \eqref{LQ_def}, the first intersection of $\gamma$ with $\SS(x)$, say at vertex $w'$, must occur within $\Avc(x)$ edges of $x$. 
In particular, $\gamma$ reaches $w'$ \textit{before} its last intersection with $\SS(y)$, say at $z'$, which occurs within $\Avc(y)$ edges of $y$. 
Hence $\gamma'\in\PP(w',z')$ described in the statement of the lemma is well-defined.
Let $\gamma_\mathrm{pre}$ be the portion of $\gamma$ before reaching $w'$, and $\gamma_\mathrm{post}$ the portion after reaching $z'$.
By choice of $w'$ and $z'$, we have that
\eeqs{vertices_okay_1}{
\text{all vertices of $\gamma_\mathrm{pre}$ belong to $\SS_\mathrm{int}(x)\cup\{w'\}$}, \\
\text{all vertices of $\gamma_\mathrm{post}$ belong to $\SS_\mathrm{int}(y)\cup\{z'\}$}, 
}
as well as
\eq{
|\wt T(\gamma') - \wt T(\gamma)|  \leq \Delta_{\mathrm{int}}(x)+\Delta_{\mathrm{int}}(y).
}
Since $\wt T(\gamma)\geq\wt T(x,y)$ by definition of $\wt T$, a triangle inequality now shows
\eeq{ \label{for_tightness_1}
|\wt T(\gamma') - \wt T(x,y)| &\leq \wt T(\gamma) - \wt T(x,y) + \Delta_{\mathrm{int}}(x)+\Delta_{\mathrm{int}}(y).
}

To reach \eqref{for_tightness_2}, we must work a bit harder.
Let $\hat w\in\SS(x)$ and $\hat z\in\SS(z)$ be such that $\wt T(\hat w,\hat z) = \wh T(x,y)$; this is possible because $\SS(x)$ and $\SS(y)$ all almost surely finite by property \ref{condition_2}.
Given any $\eps>0$, we can choose $\hat\gamma\in\PP(\hat w,\hat z)$ such that $\wt T(\hat \gamma) \leq\wh T(x,y)+\eps$.
Following $\hat\gamma$ from $\hat w$ to $\hat z$, let $\hat w_\mathrm{last}$ be the last intersection with $\SS(x)$.
Next let $\hat z_\mathrm{first}$ be the first intersection with $\SS(y)$ that occurs \textit{after} $\hat w_\mathrm{last}$ (although this may be $\hat w_\mathrm{last}$ itself). 
These choices partition $\hat\gamma$ into three subpaths: $\hat\gamma_\mathrm{beg}\in\PP(\hat w,\hat w_\mathrm{last})$, $\hat\gamma_\mathrm{mid}\in\PP(\hat w_\mathrm{last},\hat z_\mathrm{first})$, and $\hat\gamma_\mathrm{end}\in\PP(\hat z_\mathrm{first},\hat z)$.
By construction, we have that
\eeq{ \label{vertices_okay_2}
\parbox{0.7\textwidth}{\centering no vertices of $\hat\gamma_\mathrm{mid}$ other than $\hat w_\mathrm{last}$ and $\hat z_\mathrm{first}$ \\ belong to $\SS(x)\cup\SS_\mathrm{int}(x)\cup\SS(y)\cup\SS_\mathrm{int}(y)$.}
}
Now take any self-avoiding path $\gamma_1$ from $w'$ to $\hat w_\mathrm{last}$ using only vertices in $\SS(x)$ (recall that $\SS(x)$ is connected by property \ref{condition_2}).
Similarly, let $\gamma_2$ be any self-avoiding path from $\hat z_\mathrm{first}$ to $z'$ confined to $\SS(y)$.
By \eqref{vertices_okay_1} and \eqref{vertices_okay_2}, the following path between $x$ and $y$ is self-avoiding (see Figure \ref{shell_times_1_fig}):
\eq{ \tilde\gamma \coloneqq \gamma_\mathrm{pre}\cup\gamma_1\cup\hat\gamma_\mathrm{mid}\cup\gamma_2\cup\gamma_\mathrm{post}.
}
Hence $\tilde\gamma$ is a candidate path for $\wt T(x,y)$, which means
\eq{
\wt T(x,y) &\leq \wt T(\tilde\gamma)
= \wt T(\gamma_\mathrm{pre}) + \wt T(\gamma_1) + \wt T(\hat\gamma_\mathrm{mid})+\wt T(\gamma_2)+\wt T(\gamma_\mathrm{post}) \\
&= \wt T(\gamma_1) + \wt T(\gamma_2) + \wt T(\gamma)-\wt T(\gamma') + \wt T(\hat\gamma)-\wt T(\hat \gamma_\mathrm{beg})-\wt T(\hat\gamma_\mathrm{end}) \\
&\leq \wt T(\gamma_1) + \wt T(\gamma_2) + \wt T(\gamma)-\wt T(\gamma') +\wh T(x,y)+\eps -\wt T(\hat \gamma_\mathrm{beg})-\wt T(\hat\gamma_\mathrm{end}).
}
We can rewrite this inequality as
\eq{
\wt T(\gamma') - \wh T(x,y) \leq \wt T(\gamma_1)+\wt T(\gamma_2) + \wt T(\gamma)-\wt T(x,y) + \eps -\wt T(\hat \gamma_\mathrm{beg})-\wt T(\hat\gamma_\mathrm{end}).
}
In addition, we trivially have $\wt T(\gamma') - \wh T(x,y) \geq 0$ because $w'\in\SS(x)$ and $z'\in\SS(y)$.
We now work to bound from above the first and last pairs of terms on the right-hand side.
For every $e$ connecting two vertices in $\SS(x)$, we have $\tau_e\leq \tf $ because all vertices in $\SS(x)$ are white.
Since we almost surely have $|\tau_e^\pert|\leq \af $ for all $e \in E(\Z^d)$ by \eqref{perturbed_general_assumptions_c}, this observation results in
\eq{
\wt T(\gamma_1) \leq (\tf + \af )|\SS(x)| \quad \text{and} \quad \wt T(\gamma_2) \leq (\tf + \af )|\SS(y)| \quad \mathrm{a.s.}
}
Next observe that Lemma \ref{Q_trivial_lemma} implies
\eq{
-\wt T(\hat\gamma_\mathrm{beg}) \leq \af \Lvc(\hat w;0)\leq\af\sup_{w\in\SS(x)}\Lvc(w;0) \quad \mathrm{a.s.}
}
as well as
\eq{
-\wt T(\hat\gamma_\mathrm{end}) \leq \af\Lvc(\hat z;0)\leq \af \sup_{z\in\SS(y)} \Lvc(z;0) \quad \mathrm{a.s.}
}
The inequality \eqref{for_tightness_2} can now be read from the previous four displays, once $\eps$ is taken to $0$.
This completes the proof of part \ref{for_tightness_lemma_a}.

For \ref{for_tightness_lemma_b}, the assumption \eqref{for_tightness_condition} ensures that $\SS(x)\cup\SS_\mathrm{int}(x)$ is disjoint from $\SS(y)\cup\SS_\mathrm{int}(x)$.
In this case, property \ref{condition_2} of the shell construction forces every $\gamma\in\PP(x,y)$ to intersect both $\SS(x)$ and $\SS(y)$, where again the first intersection with $\SS(x)$ occurs before the last intersection with $\SS(y)$.
Therefore, \eqref{for_tightness_3} is immediate from \eqref{for_tightness_1} and \eqref{for_tightness_2} by choosing $\gamma$ so that $\wt T(\gamma)$ is arbitrarily close to $\wt T(x,y)$.
\end{proof}

\begin{figure}
\centering
\includegraphics[width=0.9\textwidth]{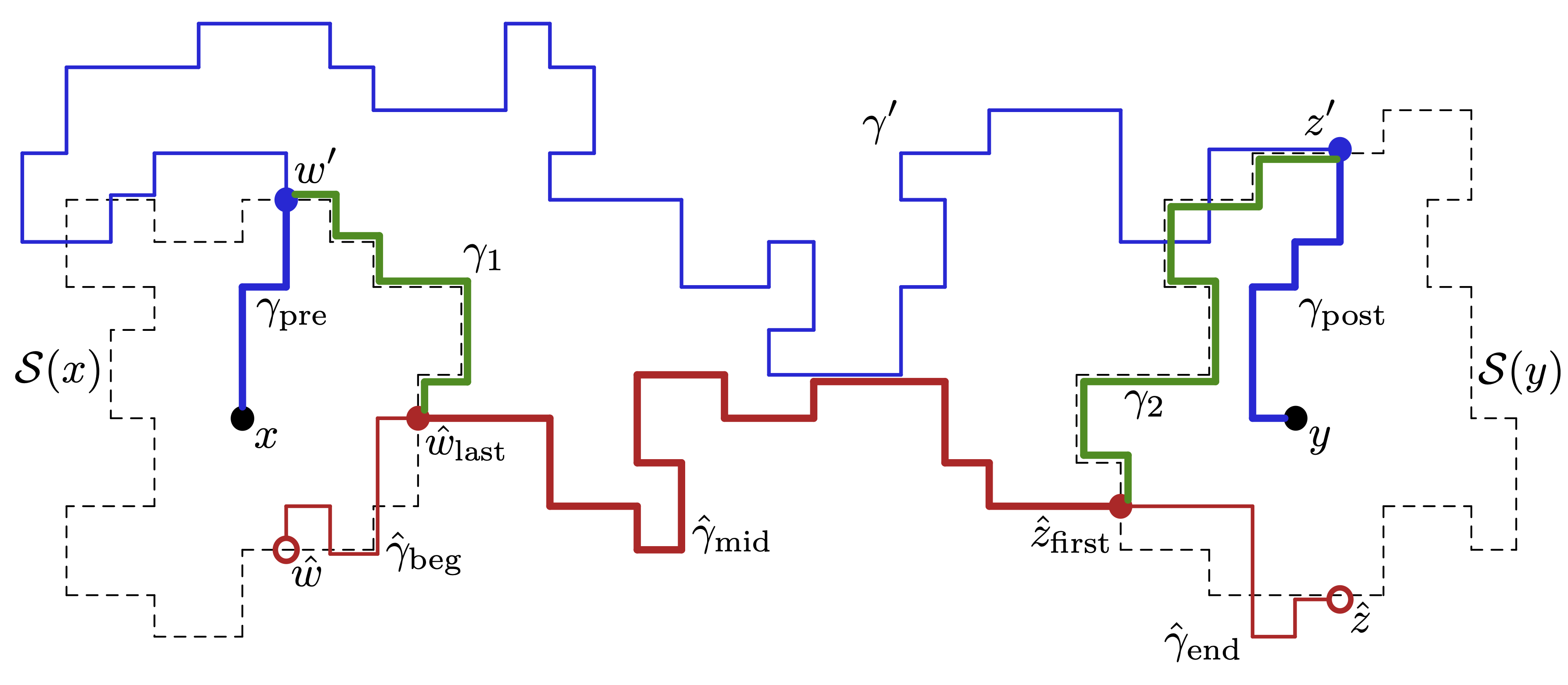}
\caption{Illustration for the proof of Lemma \ref{for_tightness_lemma}. If the top curve $\gamma = \gamma_\mathrm{pre}\cup\gamma'\cup\gamma_\mathrm{post}$ is a geodesic for $\wt T(x,y)$, and the bottom curve $\hat\gamma = \hat\gamma_\mathrm{beg}\cup\hat\gamma_\mathrm{mid}\cup\hat\gamma_\mathrm{end}$ is a geodesic for $\wh T(x,y)$, then the concatenated curve $\tilde\gamma = \gamma_\mathrm{pre}\cup\gamma_1\cup\hat\gamma_\mathrm{mid}\cup\gamma_2\cup\gamma_\mathrm{post}$ is a near-geodesic for $\wh T(x,y)$.}
\label{shell_times_1_fig}
\end{figure}

\noindent \textit{Proof of Proposition \textup{\hyperref[negative_thm_b]{\ref*{negative_thm}\ref*{negative_thm_b}}}}.
Let $\af$ be as in Lemma \ref{eta_choice_lemma}, so that $h>0$ if $F(0)<p_\cc(\Z^d)$, and $h=0$ otherwise.
In either case, we assume from \eqref{perturbed_general_assumptions_c} that $|\tau_e^\pert|\leq \af $ for all $e\in E(\Z^d)$.
The tightness of the collection $\{\wt T(x,y)-\wh T(x,y):\, x,y\in\Z^d\}$ is now deduced as follows.
Let $\eps>0$ be given.
First, \eqref{condition_2_eq_c} allows us to choose $r_1 = r_1(\LL,d)$ sufficiently large that
\eq{ 
\P(\Bsf_x) \geq 1-\eps/6, \quad \text{where}\quad
\Bsf_x \coloneqq \{\Rvc(x)\leq r_1\}. 
}
By \eqref{Qx_tail}---or the fact that $ \af =0$ when $F(0)\geq p_\cc(\Z^d)$---there is $r_2=r_2(\LL,d)$ such that
\eeq{ \label{conclude_tightness_2}
\P(\Csf_x) \geq 1-\eps/6, \quad \text{where} \quad \Csf_x\coloneqq \bigg\{\af \sup_{w\in \Bbf_{r_1}(x)}\Lvc(w;0) \leq r_2\bigg\}.
}
Finally, since $r_1$ has been fixed, there is some $r_3$ such that
\eeq{ \label{conclude_tightness_3}
\P(\Dsf_x) \geq 1-\eps/6, \quad \text{where} \quad \Dsf_x \coloneqq \bigg\{\sum_{E(\Bbf_{r_1}(x))} (\tau_e+ \af )\leq r_3\bigg\},
}
where $E(\Bbf_{r}(x))$ is the set of all edges in which both vertices belong to $\Bbf_{r}(x)$.
By the translation invariance from property \ref{condition_4} and \eqref{perturbed_general_assumptions_a}, the values of $\P(\Bsf_x)$ and $\P(\Dsf_x)$ do not depend on $x$. 
Similarly, the value of $\P(\Dsf_x)$ does not depend on $x$ because of \eqref{perturbed_general_assumptions_a}.
Meanwhile, the lower bound in \eqref{conclude_tightness_2} is valid for all $x$, since the constants in \eqref{Qx_tail} do not depend on $x$.

Consider any $x,y\in\Z^d$ such that $\|x-y\|_\infty > 2r_1$. 
If $\Bsf_x$ and $\Bsf_y$ both occur, then \eqref{for_tightness_condition} holds, in which case we can apply \eqref{for_tightness_3} in conjunction with the events considered in \eqref{conclude_tightness_2} and \eqref{conclude_tightness_3}.
The resulting conclusion is that on the event $\Bsf_x\cap\Bsf_y\cap\Csf_x\cap\Csf_y\cap\Dsf_x\cap\Dsf_y$, which occurs with probability at least $1-\eps$, we have
\eq{
|\wt T(x,y)-\wh T(x,y)| \leq 2(\tf + \af )(2r_1+1)^d + 2r_2 + 2r_3.
}
As $\eps>0$ is arbitrary, we have shown $\{\wt T(x,y)-\wh T(x,y):\, \|x-y\|_\infty>2r_1\}$ is tight.
On the other hand, because of \eqref{perturbed_general_assumptions_b}, the complementary family
$\{\wt T(x,y)-\wh T(x,y):\, \|x-y\|_\infty\leq 2r_1\}$
is also tight, as there are only finitely many distinct distributions represented. \hfill \qedsymbol \\

Another consequence of Lemma \ref{for_tightness_lemma} is the following result on the length of geodesics with respect to $\wt T$.
It is an extension of \cite[Thm.~4.9]{auffinger-damron-hanson17} and will be needed in Chapter \ref{var_form_proof}.

\begin{prop} \label{geo_length_cor}
Assume $F(0)<p_\cc(\Z^d)$, \eqref{perturbed_general_assumptions_a}, and \eqref{perturbed_general_assumptions_c}, where $\af >0$ is the constant from Lemma \ref{eta_choice_lemma}.
There is a constant $\mf  = \mf (\LL,d)$ such that
\eq{
\P\Big(\sup_{\gamma\in{\Geo}(x,y)}|\gamma|\geq \mf \|x-y\|_\infty\Big) \leq C\e^{-c\|x-y\|_\infty^{1/d}} \quad \text{for all $x,y\in\Z^d$}.
}
\end{prop}

\begin{proof}
Recall from \eqref{condition_2_eq_c} that
\eq{
\P\Big(\Rvc(x) \geq \frac{\|x-y\|_\infty}{2}\Big) \leq C\e^{-c\|x-y\|_\infty},
}
and so the probability that \eqref{for_tightness_condition} holds is at least $1-C\e^{-c\|x-y\|_\infty}$.
Suppose that \eqref{for_tightness_condition} does hold, so that---using the notation from Lemma \ref{for_tightness_lemma}---we can invoke \eqref{for_tightness_2} to write the following inequality for any $\gamma\in\Geo(x,y)$:
\eeq{ \label{for_tightness_2_cor}
\wt T(\gamma')-\wh T(x,y) \leq \Delta_{\mathrm{ext}}(x)+\Delta_{\mathrm{ext}}(y) \quad \mathrm{a.s.}
}
Recall the constant $\mathfrak{s} $ from Lemma \ref{T_hat_upper}, and suppose further that
\eeq{ \label{length_cor_eq_prep}
\wh T(x,y) \leq \mathfrak{s} \|x-y\|_\infty.
}
In order to simplify the coming notation, we define the quantity
\eq{
{R_0} = \left\lfloor \frac{1}{2}\Big(\frac{\mathfrak{s} \|x-y\|_\infty}{2(\tf +\af )}\Big)^{1/d} - \frac{1}{2}\right\rfloor.
}
Suppose even further that 
\eeq{ \label{supposing_Rs}
\Rvc(x)\leq {R_0},\quad \Rvc(y)\leq {R_0},
}
and
\eeq{ \label{supposing_sups}
\sup_{w\in \Bbf_{R_0}(x)\cup \Bbf_{R_0}(y)} \Lvc(w;\af ) < \frac{\mathfrak{s} }{\af }\|x-y\|_\infty. 
}
Now we make use of these several suppositions.
First, since every edge in $\gamma\setminus\gamma'$ contains a vertex in either $\SS(x)$ or $\SS(y)$, there is the trivial inequality
\eeq{ \label{length_in_3}
|\gamma| &\stackrefp{supposing_Rs}{\leq} |\gamma'| + |\SS(x)| + |\SS(y)| \\
&\stackrefp{supposing_Rs}{\leq} |\gamma'| + (2\Rvc(x)+1)^d + (2\Rvc(y)+1)^d \\
&\stackref{supposing_Rs}{\leq} |\gamma'| + \frac{\mathfrak{s} \|x-y\|_\infty}{\tf +\af } \leq |\gamma'| + \frac{\mathfrak{s} }{\af }\|x-y\|_\infty.
}
Second, we have
\eeq{ \label{two_delta_ones}
&\Delta_{\mathrm{ext}}(x)+\Delta_{\mathrm{ext}}(y)
= (\tf +\af )\big(|\SS(x)|+|\SS(y)|\big) + \af \sup_{w\in\SS(x)\cup\SS(y)}\Lvc(w;0) \\
&\stackref{length_in_3,supposing_Rs}{\leq} \mathfrak{s} \|x-y\|_\infty + \af \sup_{w\in \Bbf_{R_0}(x)\cup \Bbf_{R_0}(y)} \Lvc(w;\af ) \\
&\stackrefpp{supposing_sups}{length_in_3,supposing_Rs}{\leq} 3\mathfrak{s} \|x-y\|_\infty.
}
Third, by \eqref{supposing_Rs} and \eqref{supposing_sups}, every $\gamma\in\PP(x,y)$ with $|\gamma'|\geq 4(\mathfrak{s} /\af )\|x-y\|_\infty$ must have
\eq{
\wt T(\gamma') > \af |\gamma'| \geq 4\mathfrak{s} \|x-y\|_\infty.
}
On the other hand, combining \eqref{for_tightness_2_cor}, \eqref{length_cor_eq_prep}, and \eqref{two_delta_ones} shows
\eq{
\sup_{\gamma\in\Geo(x,y)}\wt T(\gamma') \leq 4\mathfrak{s} \|x-y\|_\infty.
}
To reconcile these two observations, we must have
\eq{
\sup_{\gamma\in\Geo(x,y)}|\gamma'| \leq \frac{4\mathfrak{s} }{\af }\|x-y\|_\infty.
}
Combining this bound with \eqref{length_in_3}, we arrive at
\eq{
\sup_{\gamma\in\Geo(x,y)}|\gamma| \leq \frac{5\mathfrak{s} }{\af }\|x-y\|_\infty.
}
In summary, we have shown
\eq{
\P\Big(\sup_{\gamma\in\Geo(x,y)}|\gamma| \leq \frac{5\mathfrak{s} }{\af }\|x-y\|_\infty\Big)
\geq \P(\eqref{for_tightness_condition},\eqref{length_cor_eq_prep},\eqref{supposing_Rs},\eqref{supposing_sups} \text{ all hold}).
}
We saw at the beginning of the proof that \eqref{for_tightness_condition} occurs with probability least $1-C\e^{-c\|x-y\|_\infty}$.
By Lemma \ref{T_hat_upper} and \eqref{condition_2_eq_c}, each of \eqref{length_cor_eq_prep} and \eqref{supposing_Rs} occurs with probability at least $1-C\e^{-c\|x-y\|_\infty^{1/d}}$.
Finally, by applying \eqref{Qx_tail} with a union bound over $\Bbf_{R_0}(x) \cup \Bbf_{R_0}(y)$, the probability of \eqref{supposing_sups} is seen to be at least $1-C\e^{-c\|x-y\|_\infty}$.
We conclude that $\mf  = 5\mathfrak{s} /\af $ suffices.
\end{proof}

\subsection{Subadditivity}
The next lemma shows that both $\wh T$ and $\wt T$ are \textit{almost} subadditive.
The inequality \eqref{final_subadditive_tilde} also appeared in \cite[disp.~(5.34)]{smythe-wierman78}.
Recall the quantity $\Delta_\mathrm{ext}(x)$ from \eqref{deltas_def}, and note that the following tail bound is immediate from \eqref{condition_2_eq_c}, Lemma \ref{eta_choice_lemma}, and the fact that $\SS(x)\subset\Bbf_{\Rvc(x)}(x)$:
\eeq{ \label{Delta_2_tail}
\P(\Delta_{\mathrm{ext}}(x)\geq s) \leq C\e^{-cs^{1/d}} \quad \text{for all $x\in\Z^d$ and $s\geq0$}.
}


\begin{lemma} \label{subadditive_lemma}
Assume \eqref{perturbed_general_assumptions_c}.
For any $x,y,z\in\Z^d$, we have
\eeq{ \label{final_subadditive}
\wh T(x,y) \leq \wh T(x,z)+\wh T(z,y) + \Delta_{\mathrm{ext}}(z) \quad \mathrm{a.s.},
}
as well as
\eeq{ \label{final_subadditive_tilde}
\wt T(x,y) \leq \wt T(x,z)+\wt T(z,y)+ 2 \af  \Lvc(z;0) \quad \mathrm{a.s.}
}
\end{lemma}

\begin{proof}
We will argue \eqref{final_subadditive}, and then specify at the end the very minor modification needed to prove \eqref{final_subadditive_2}.
Consider any paths $\gamma_1\in\PP(w_0,w_1)$ and $\gamma_2\in\PP(w_2,w_3)$,
where $w_0\in\SS(x)$, $w_1,w_2\in\SS(z)$, and $w_3\in\SS(y)$.
\begin{figure}
\subfloat[If $\gamma_1$ and $\gamma_2$ intersect, then we set $\gamma = \gamma_1^\mathrm{pre}\cup\gamma_2^\mathrm{post}$]{
\label{subadditive_figure_a}
\centering
\includegraphics[width=0.9\textwidth]{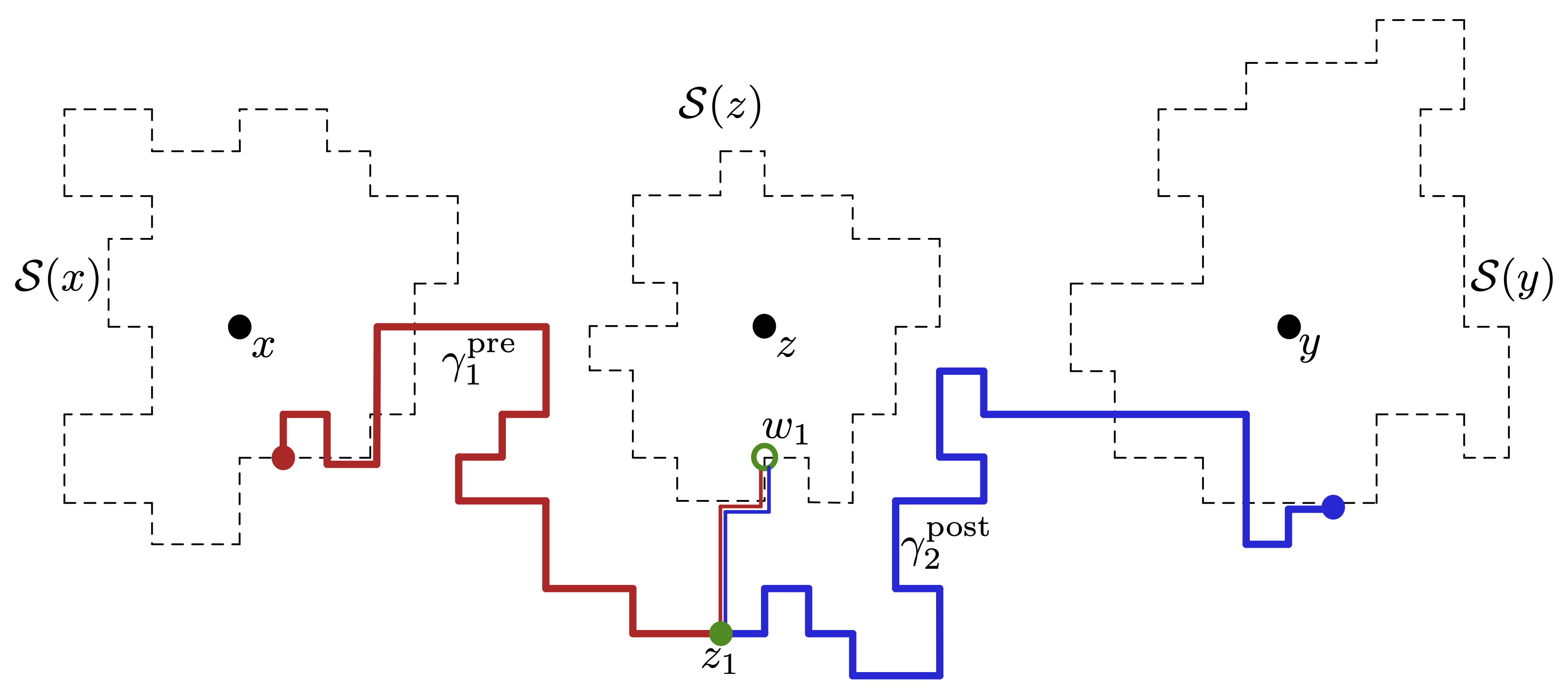}
}
\\
\subfloat[If $\gamma_1$ and $\gamma_2$ do not intersect, then we set $\gamma = \gamma_1^\mathrm{pre}\cup\gamma_\mathrm{bridge}\cup\gamma_2^\mathrm{post}$.  The dotted curve along $\SS(x)$ is the portion of $\gamma_1^+$ not included in $\gamma_\mathrm{bridge}$.]{
\label{subadditive_figure_b}
\centering
\includegraphics[clip, trim = 0in 0.4in 0 0, width=0.9\textwidth]{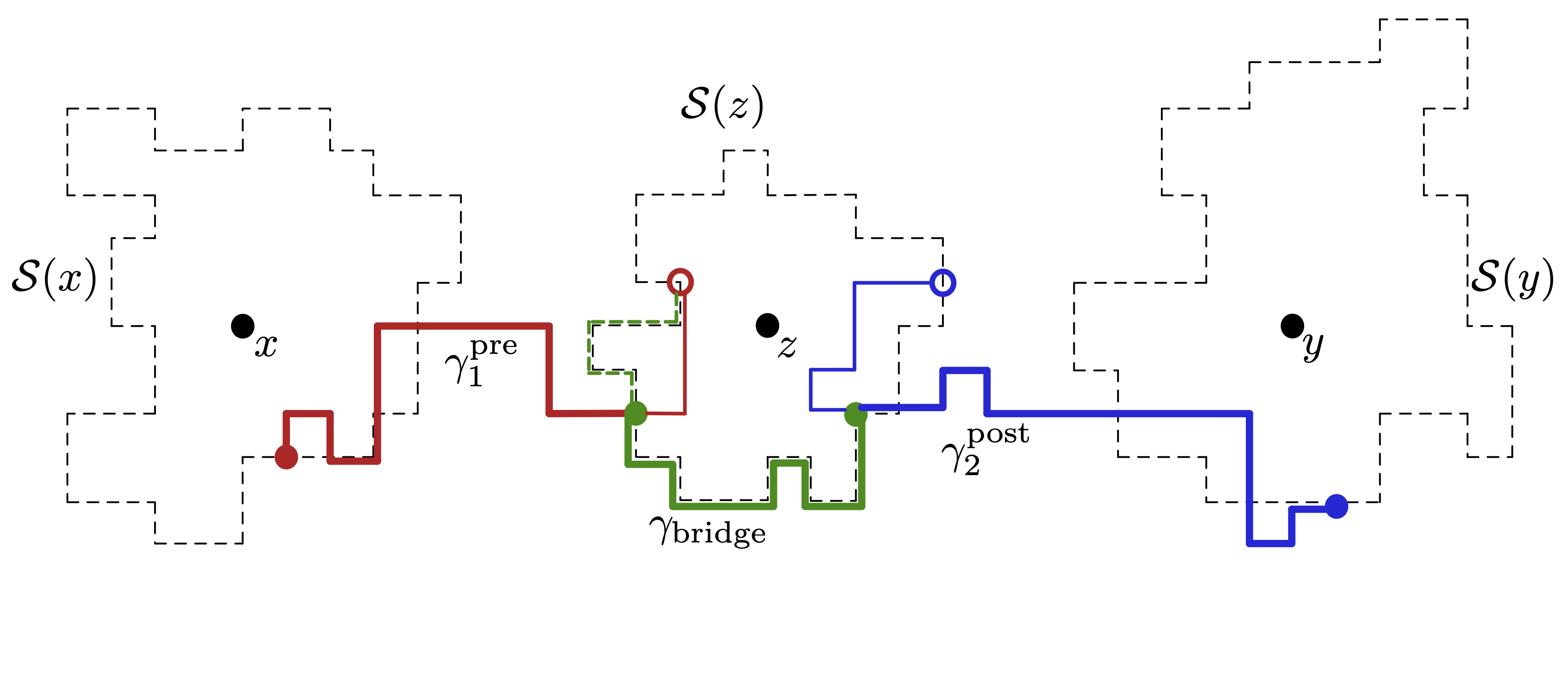}
}
\captionsetup{width=.9\linewidth}
\caption{Illustrations for the proof of Lemma \ref{subadditive_lemma}. The goal is to use paths $\gamma_1$ from $\SS(x)$ to $\SS(z)$ and $\gamma_2$ from $\SS(z)$ to $\SS(y)$, to create a single path $\gamma$ between $\SS(x)$ and $\SS(y)$ whose passage time is close to the sum of passage times for $\gamma_1$ and $\gamma_2$.}
\label{subadditive_figure}
\end{figure}
First suppose that $\gamma_1$ and $\gamma_2$ intersect (in the sense of sharing a vertex), as shown in Figure \ref{subadditive_figure_a}.
Let $\gamma_1^\mathrm{pre}$ be the portion of $\gamma_{1}$ between its initial location and its first intersection with $\gamma_{2}$; call this intersection point $z_1$.
Let $\gamma_1^\mathrm{post}$ be the remaining portion of $\gamma_{1}$ between $z_1$ and its final location $w_1\in \SS(z)$.
We next write $\gamma_2^\mathrm{pre}$ to denote the subpath of $\gamma_2$ between its initial location in $\SS(z)$ and $z_1$, while $\gamma_2^\mathrm{post}$ will denote the portion between $z_1$ and the final location in $\SS(y)$.
Clearly $\gamma_1^\mathrm{pre}\cup\gamma_2^\mathrm{post}$ is a self-avoiding path between $\SS(x)$ and $\SS(y)$, and so
\eeq{ \label{put_together_1}
\wh T(x,y) &\leq \wt T(\gamma_1^\mathrm{pre})+\wt T(\gamma_2^\mathrm{post}) 
= \wt T(\gamma_1)-\wt T(\gamma_1^\mathrm{post})+\wt T(\gamma_2)-\wt T(\gamma_2^\mathrm{pre}).
}
Because $\gamma_1$ terminates at $w_1\in\SS(z)$, we have
\begin{subequations} \label{post_and_pre}
\begin{align}
-\wt T(\gamma_1^\mathrm{post}) 
\stackref{trivial_neg_bd}{\leq} \af  \Lvc(w_1;0) 
&\leq  \af \sup_{w\in\SS(z)}\Lvc(w;0) \quad \mathrm{a.s.}
\end{align}
Similarly, because $\gamma_2$ begins in $\SS(z)$, we have
\begin{align}
-\wt T(\gamma_2^\mathrm{pre}) \leq  \af \sup_{w\in\SS(z)}\Lvc(w;0) \quad \mathrm{a.s.}
\end{align}
\end{subequations}
Together, \eqref{put_together_1} and \eqref{post_and_pre} yield
\eeq{ \label{final_subadditive_1}
\wh T(x,y) &\leq \wt T(\gamma_1) + \wt T(\gamma_2) + 2 \af \sup_{w\in\SS(z)}\Lvc(w;0) \\
&\leq \wt T(\gamma_1)+\wt T(\gamma_2)+\Delta_{\mathrm{ext}}(z) \quad \mathrm{a.s.}
}

If instead $\gamma_1$ and $\gamma_2$ do not intersect, as in Figure \ref{subadditive_figure_b}, then we attach to the end of $\gamma_1$ some path $\gamma_1^+$ which is confined to $\SS(z)$ and terminates at its first intersection with $\gamma_2$; such an intersection exists because $\gamma_2$ begins in $\SS(z)$, and $\SS(z)$ is connected.
Allowing for the possibility that $\gamma_1^+$ intersects $\gamma_1$ again before reaching $\gamma_2$, we let $\gamma_\mathrm{bridge}$ be the portion of $\gamma_1^+$ between its \textit{last} intersection with $\gamma_1$ and its terminal point shared with $\gamma_2$.
Now let $\gamma_1^\mathrm{pre}$ be the portion of $\gamma_1$ between its starting point in $\SS(x)$ and its unique point of intersection with $\gamma_\mathrm{bridge}$, and let $\gamma_1^\mathrm{post}$ be the remaining portion of $\gamma_1$ beyond this point of intersection.
Similarly, let $\gamma_2^\mathrm{pre}$ be the subpath of $\gamma_2$ between its starting point and its unique intersection with $\gamma_\mathrm{bridge}$, while $\gamma_2^\mathrm{post}$ will denote the remaining portion of $\gamma_2$ past this intersection.
Since $\gamma_1$ and $\gamma_2$ were assumed to not intersect, the concatenated path $\gamma_1^\mathrm{pre}\cup\gamma_\mathrm{bridge}\cup\gamma_2^\mathrm{post}$ is self-avoiding, starts in $\SS(x)$, and ends in $\SS(y)$.
Therefore, it is a candidate path for $\wh T(x,y)$, which means
\eq{
\wh T(x,y) &\leq \wt T(\gamma_1^\mathrm{pre})+\wt T(\gamma_\mathrm{bridge})+\wt T(\gamma_2^\mathrm{post}) \\
&= \wt T(\gamma_1)-\wt T(\gamma_1^\mathrm{post})+\wt T(\gamma_\mathrm{bridge})+\wt T(\gamma_2)-\wt T(\gamma_2^\mathrm{pre}).
}
As before, $-\wt T(\gamma_1^\mathrm{post})$ and $-\wt T(\gamma_2^\mathrm{pre})$ are controlled by \eqref{post_and_pre}.
In addition, because $\gamma_\mathrm{bridge}$ traverses only vertices in $\SS(z)$, \eqref{S_in_W} and \eqref{perturbed_general_assumptions_c} together imply
\eq{
\wt T(\gamma_\mathrm{bridge}) 
\leq (\tf + \af )|\SS(x)| \quad \mathrm{a.s.}
}
Together, the two previous displays yield
\eeq{ \label{final_subadditive_2}
\wh T(x,y) \leq \wt T(\gamma_1)+\wt T(\gamma_2) +\Delta_{\mathrm{ext}}(z) \quad \mathrm{a.s.}
}
The claimed inequality \eqref{final_subadditive} now follows from whichever of \eqref{final_subadditive_1} and \eqref{final_subadditive_2} applies, once $\gamma_1$ and $\gamma_2$ are chosen so that $\wt T(\gamma_1)$ and $\wt T(\gamma_2)$ become arbitrarily close to $\wh T(x,z)$ and $\wh T(z,y)$, respectively.

To obtain \eqref{final_subadditive_tilde} instead of \eqref{final_subadditive},
one would simply note that now the endpoint of $\gamma_1\in\PP(x,z)$ is \textit{equal} to the starting point of $\gamma_2\in\PP(z,y)$.
Hence the upper bounds in \eqref{post_and_pre} become $ \af  \Lvc(z;0)$, and $\gamma_\mathrm{bridge}$ is empty.
\end{proof}

To prove Propositions \hyperref[negative_thm_c]{\ref*{negative_thm}\ref*{negative_thm_c}} and \hyperref[negative_thm_c2]{\ref*{negative_thm}\ref*{negative_thm_c2}}, we will need a continuity condition.

\begin{lemma} \label{continuity_lemma}
Assume \eqref{perturbed_general_assumptions_a} and \eqref{perturbed_general_assumptions_c}, where $\af \geq0$ is the constant from Lemma \ref{eta_choice_lemma}.
For any $\eps>0$, there exists $\delta = \delta(\eps,\LL,d)>0$ such that 
\eeq{ \label{sup_approx_moment}
\limsup_{n\to\infty} \E\bigg[\sup_{\substack{\bxi,\zetab\in\S^{d-1}\\ \|\bxi-\zetab\|_2\leq\delta}} \Big|\frac{\wh T(0,n\bxi)-\wh T(0,n\zetab)}{n}\Big|^p\bigg]  
\leq p\eps^p \quad \text{$\forall$ $p\in[1,\infty)$},
}
and
\eeq{ \label{sup_approx_as}
\limsup_{n\to\infty}\bigg[\sup_{\substack{\bxi,\zetab\in\S^{d-1}\\ \|\bxi-\zetab\|_2\leq\delta}} \Big|\frac{\wh T(0,n\bxi)-\wh T(0,n\zetab)}{n}\Big|\bigg] \leq \eps \quad \mathrm{a.s.}
}
\end{lemma}

\begin{proof}
Let $\eps>0$ be given, and recall the constant $\mathfrak{s} $ from Lemma \ref{T_hat_upper}.
Set $\delta=\frac{1}{12}(1\wedge \eps/\mathfrak{s} )$ so that the following implications hold for all $\bxi,\zetab\in\S^{d-1}$, $s\geq0$, and large enough $n$ (depending on $\delta$):
\eeq{ \label{delta_large_n}
&\|\bxi-\zetab\|_2 \in [\delta/2,2\delta]
 \implies  \|[n \bxi]-[n\zetab]\|_2 \in [\delta n/3,3\delta n] \\
  &\implies  (\mathfrak{s} +s)\|[n \bxi]-[n\zetab]\|_\infty \leq (\mathfrak{s} +s)\|[n\bxi]-[n\zetab]\|_2 \leq (\eps+s)n/4.
}
For $n,s\in[0,\infty)$, define the events
\eq{
\Usf_{n,s} &\coloneqq \bigcup_{\substack{\bxi,\zetab\in\S^{d-1} \\ \frac{\delta}{2}\leq\|\bxi-\zetab\|_2\leq 2\delta }}\Big\{|\wh T(n\bxi,n\zetab)| \geq \frac{\eps+s}{4}n\Big\}, \\
\Vsf_{n,s} &\coloneqq \bigcup_{\bxi\in\S^{d-1}} \Big\{\Delta_{\mathrm{ext}}([n\bxi]) \geq \frac{\eps+s}{8}n\Big\}.
}
Note that for $n\in[1,\infty)$, the cardinality of the set 
\eeq{ \label{Sn_def}
\mathbf{S}_n \coloneqq \{[n\bxi]:\, \bxi\in\S^{d-1}\}\subset\Z^d
} 
is at most $Cn^{d-1}$.
Therefore, by a union bound together with \eqref{T_hat_upper_eq} and \eqref{T_hat_lower_eq}, we have
\eeq{ \label{U_event_bd}
\P(\Usf_{n,s}) \leq Cn^{2d-2}\e^{-c((\mathfrak{s} +s)\delta n/(6d))^{1/d}} \quad \text{for all large $n$, $s\in[0,\infty)$},
}
where we have the used the fact from \eqref{delta_large_n} that 
\eq{ 
\|\bxi-\zetab\|_2 \geq \delta/2 \implies \|[n \bxi]-[n\zetab]\|_\infty \geq \|[n\bxi]-[n\zetab]\|_2/(2d) \geq \delta n/(6d).
}
Similarly, by a union bound over $\mathbf{S}_n$ together with \eqref{Delta_2_tail}, we have
\eeq{ \label{V_event_bd}
\P(\Vsf_{n,s}) \leq Cn^{d-1}\e^{-c((\eps+s)n/8)^{1/d}} \quad \text{for all $n\in[1,\infty)$, $s\in[0,\infty)$}.
}
Putting \eqref{U_event_bd} and \eqref{V_event_bd} together yields
\eeq{ \label{UV_events_bd}
\P(\Usf_{n,s}\cup\Vsf_{n,s}) \leq Cn^{2d-2}\e^{-c(\delta s n)^{1/d}}.
}
Now suppose that neither $\Usf_{n,s}$ nor $\Vsf_{n,s}$ occur. 
Given any $\bxi_1,\bxi_2\in\S^{d-1}$ with $\|\bxi-\zetab\|_2\leq \delta$, choose $\zetab\in\S^{d-1}$ such that
both $\|\bxi_1-\zetab\|_2$ and $\|\bxi_2-\zetab\|_2$ lie in the interval $[\delta/2,2\delta]$.
We then have
\eq{
&|\wh T(0,n\bxi_1) - \wh T(0,n\bxi_2)|
\leq|\wh T(0,n\bxi_1) - \wh T(0,n\zetab)| + |\wh T(0,n\bxi_2) - \wh T(0,n\zetab)| \\
&\stackref{final_subadditive}{\leq} |\wh T(n\bxi_1,n\zetab)| + \Delta_{\mathrm{ext}}([n\bxi_1])  + |\wh T(n\bxi_2,n\zetab)| + \Delta_{\mathrm{ext}}([n\bxi_2]) +2 \Delta_{\mathrm{ext}}([n\zetab]) \\
&\stackrefp{final_subadditive}{<} (\eps+s)n \quad \mathrm{a.s.}
}
We have thus argued that
\eeq{ \label{event_with_s}
\sup_{\substack{\bxi_1,\bxi_2\in\S^{d-1}\\ \|\bxi-\zetab\|_2\leq\delta}} |\wh T(0,n\bxi_1)-\wh T(0,n\bxi_2)| \leq (\eps+s)n
\quad \text{on $\Omega\setminus(\Usf_{n,s}\cup \Vsf_{n,s})$}.
} 
Therefore, once we integrate the tail, our bound from \eqref{UV_events_bd} gives
\eq{
&\E\bigg[\sup_{\substack{\bxi,\zetab\in\S^{d-1}\\ \|\bxi-\zetab\|_2\leq\delta}} \Big|\frac{\wh T(0,n\bxi)-\wh T(0,n\zetab)}{n}\Big|^p\bigg]  \\
&\leq p\eps^p + Cn^{2d-2}p\int_0^\infty (\eps+s)^{p-1}\e^{-c(\delta s n)^{1/d}}\, \dd s
}
for all large $n$.
Since the integral vanishes as $n\to\infty$, this establishes \eqref{sup_approx_moment}.

To next prove \eqref{sup_approx_as}, we make the following discretization argument.
As $n$ varies continuously from $0$ to $\infty$, the set $\mathbf{S}_n$ increases in size at only countably many values of $n$.
More precisely, for $n \in [0,r]$ and $r\geq1$, the set $\mathbf{S}_n$ is contained in $\Bbf_{r+\sqrt{d}}(0)$ and thus assumes one of at most $|\Bbf_{r+\sqrt{d}}(0)|\leq Cr^d$ different values.
Consequently, for every large integer $r$, \eqref{U_event_bd} and \eqref{V_event_bd} now show
\eq{ 
\P\bigg(\bigcup_{n\in[r,r+1]}(\Usf_{n,0}\cup\Vsf_{n,0})\bigg) &\leq Cr^{3d-1}\e^{-c(\delta r)^{1/d}}+Cr^{2d-1}\e^{-c(\eps r)^{1/d}}.
}
Applying \mbox{Borel}--\mbox{Cantelli}, we conclude that almost surely for every large integer $r$, the event $\bigcup_{n\in[r,r+1]}(\Usf_{n,0}\cup\Vsf_{n,0})$ does not occur.
In light of \eqref{event_with_s}, the statement \eqref{sup_approx_as} follows.
\end{proof}

\noindent \textit{Proof of Proposition \textup{\hyperref[negative_thm_c]{\ref*{negative_thm}(c,d)}}}.
First consider any $\zetab\in\S^{d-1}$ with rational coordinates, say $\zetab = (a_1/b,\dots,a_d/b)$ with $a_1,\dots,a_d,b\in\Z$.
Let $z = b\zetab \in \Z^d$.
By Lemma \ref{subadditive_lemma}, the family $( X_{j,k})_{0\leq j<k}$ given by
$X_{j,k} = \wh T(jz,+kz)+\Delta_{\mathrm{ext}}(kz)$
is subadditive: $X_{0,k} \leq  X_{0,j}+ X_{j,k}$ whenever $0<j<k$.
In addition, by property \ref{condition_4} together with \eqref{perturbed_general_assumptions_a} and \eqref{perturbed_general_assumptions_b}, we have that
\eq{
(X_{j,j+k})_{k\geq1}\stackrel{\mathrm{dist}}{=} (X_{j+1,j+k+1})_{k\geq1} \quad
&\text{for all $j\geq0$; and} \\
(X_{j\ell,(j+1)\ell})_{j\geq1} \quad &\text{is stationary and ergodic for each $\ell\geq1$}.
} 
Meanwhile, Proposition \hyperref[negative_thm_a]{\ref*{negative_thm}\ref*{negative_thm_a}} leads to the moment bound
\eeq{ \label{T_hat_moment_bound}
\E|\wh T(x,y)|^p \leq C_p\|x-y\|^p_\infty \quad \text{for all $x,y\in\Z^d$ and $p\in[1,\infty)$},
} 
where $C_p$ depends only on $\LL$, $d$, and $p$.
Furthermore, since Lemma \ref{Q_trivial_lemma} gives
\eq{
\wh T(x,y) \geq -\af \sup_{w\in\SS(x)} \Lvc(w;0) \quad \mathrm{a.s.},
}
and $\SS(x)\subset\Bbf_{\Rvc(x)}(x)$, the estimates \eqref{condition_2_eq_c} and \eqref{Qx_tail} can be used to show
\eq{ 
\E[\wh T(x,y)] 
\geq -C \quad \text{for all $x,y\in\Z^d$}.
}
Meanwhile, the tail bound \eqref{Delta_2_tail} implies
\eeq{ \label{Delta_to_0}
\lim_{n\to\infty} \frac{\Delta_{\mathrm{ext}}([nz])}{n} = 0 \quad \mathrm{a.s.}\text{ and in $L^p$ for every $p\in[1,\infty)$}.
}
We have now verified all the hypotheses of the subadditive ergodic theorem \cite[Thm.~1.10]{liggett85}, which gives the existence of some $\mu_\zetab$ such that
\eeq{ \label{X_convergence}
\lim_{k\to\infty}\frac{X_{0,k}}{k} = \frac{\mu_\zetab}{b} \quad \mathrm{a.s.}\text{ and in $L^1$}.
}
It now follows from \eqref{T_hat_moment_bound}, \eqref{Delta_to_0}, and \eqref{X_convergence} that
\eeq{ \label{X_convergence_with_p}
\lim_{k\to\infty} \frac{\wh T(0,kz)}{k} = \frac{\mu_\zetab}{b} \quad\mathrm{a.s.}\text{ and in $L^p$ for every $p\in[1,\infty)$}.
}
To remove the factor of $b$, we observe that for any $n\in\R$,
\eq{
|\wh T(0,\lfloor n\rfloor z)-\wh T(0,nz)|
\stackref{final_subadditive}{\leq}
|\wh T(\lfloor n\rfloor z,nz)| + \Delta_\mathrm{ext}(\lfloor n\rfloor z) + \Delta_\mathrm{ext}([nz]).
}
By \eqref{T_hat_moment_bound} and \eqref{Delta_to_0}, the right-hand side tends to $0$ almost surely and in $L^p$, $p\in[1,\infty)$. 
Therefore, we can upgrade \eqref{X_convergence_with_p} to
\eeq{ \label{X_minus_convergence}
\lim_{n\to\infty}\frac{\wh T(0,n\zetab)}{n} = \mu_\zetab \quad \mathrm{a.s.}\text{ and in $L^p$ for every $p\in[1,\infty)$}.
}
In the case $F(0)<p_\cc(\Z^d)$, Lemma \ref{eta_choice_lemma} can be used with \mbox{Borel}--\mbox{Cantelli} to show $\mu_\zetab\geq\af$.
If instead $F(0)\geq p_\cc(\Z^d)$, then we know $\mu_\zetab=0$ from Theorem \ref{time_constant_thm} together with Proposition \hyperref[negative_thm_b]{\ref*{negative_thm}\ref*{negative_thm_b}}.

To complete the proof, we appeal to \eqref{sup_approx_as} to see that $\zetab\mapsto\mu_{\zetab}$ is continuous on $\S^{d-1}\cap\Q^d$.
Therefore, we can extend the map continuously to all of $\S^{d-1}$.
Upon making this extension, 
we use \eqref{sup_approx_as} once more to conclude that for any $\bxi\in\S^{d-1}$, we have
\eq{ 
\lim_{n\to\infty} \frac{\wh T(0,n\bxi)}{n} = \mu_{\bxi} \coloneqq \lim_{\substack{\zetab\to\bxi\\ \zetab\in\S^{d-1}\cap\Q^d}}\mu_\zetab
\quad\  \mathrm{a.s.}
}
To obtain the uniform statement \eqref{shape_thm}, let $\eps>0$ be arbitrary and take $\delta>0$ as in Lemma \ref{continuity_lemma}.
Let us also assume $\delta>0$ is sufficiently small that $|\mu_{\bxi}-\mu_{\zetab}|\leq\eps$ whenever $\|\bxi-\zetab\|_2\leq\delta$.
Now choose $\zetab_1,\dots,\zetab_m\in\S^{d-1}\cap\Q^d$ such that every $\bxi\in\S^{d-1}$ is within $\ell^2$-distance $\delta$ of some $\zetab_i$, $1\leq i\leq m$.
We then have
\eq{
&\sup_{\bxi\in\S^{d-1}} \Big|\frac{\wh T(0,n\bxi)}{n}-\mu_{\bxi}\Big| \\
&\leq \sup_{\bxi\in\S^{d-1}}\inf_{1\leq i\leq m}\Big(\frac{|\wh T(0,n\bxi) - \wh T(0,n\zetab_i)|}{n} + \Big|\frac{\wh T(0,n\zetab_i)}{n} - \mu_{\zetab_i}\Big| + |\mu_{\zetab_i} - \mu_{\bxi}|\Big) \\
&\leq\sup_{\substack{\bxi,\zetab\in\S^{d-1}\\\|\bxi-\zetab\|_2\leq\delta}} \frac{|\wh T(0,n\bxi)-\wh T(0,n\zetab)|}{n}
+ \sup_{1\leq i\leq m} \Big|\frac{\wh T(0,n\zetab_i)}{n}-\mu_{\zetab_i}\Big| + \eps.
}
By \eqref{sup_approx_as}, the first supremum in the final line is almost surely bounded from above by $\eps$ in the limit $n\to\infty$.
And because of \eqref{sup_approx_moment}, its limiting $L^p$ norm is most $p^{1/p}\eps$.
Meanwhile, by \eqref{X_minus_convergence} the second supremum tends to $0$ almost surely and in $L^p$ for every $p\in[1,\infty)$.
The aggregate conclusion is then
\eq{
\limsup_{n\to\infty}\sup_{\bxi\in\S^{d-1}} \Big|\frac{\wh T(0,n\bxi)}{n}-\mu_{\bxi}\Big| \leq 2\eps \quad \mathrm{a.s.},
}
and also
\eq{
\limsup_{n\to\infty}\E \bigg[\sup_{\bxi\in\S^{d-1}} \Big|\frac{\wh T(0,n\bxi)}{n}-\mu_{\bxi}\Big|^{p}\bigg]^{1/p} \leq (p^{1/p}+1)\eps \quad \text{for every $p\in[1,\infty)$}.
}
As $\eps$ is arbitrary, we indeed have \eqref{shape_thm}. \hfill \qedsymbol

\begin{proof}[Proof of Proposition \ref{replacement_thm}]

Take $A_0 = \Avc(0)$ and $A_n = \sup_{\bxi\in\S^{d-1}}\Avc([n\bxi])$, so that the claimed tail bounds in part \ref{replacement_thm_c} come from \eqref{Lx_tail} (together with a union bound over the set $\mathbf{S}_n$ from \eqref{Sn_def}, in the case of $A_n$).
Part \ref{replacement_thm_a} of the proposition follows from these bounds by Borel--Cantelli, once we note---as in the proof of Lemma \ref{continuity_lemma}---that $\mathbf{S}_n$ changes value at only countably many values of $n$.
In particular, for any $\eps>0$, it is almost surely the case that the event
\eq{
\Esf_n \coloneqq \{\Avc(0) + \Avc([n\bxi]) < n \text{ for all $\bxi\in\S^{d-1}$}\}.
} 
occurs for all sufficiently large $n$.

Now consider any $\bxi\in\S^{d-1}$ and any $\gamma\in\Geo(0,n\bxi)$.
If $|\gamma| \leq A_0+A_n$, which can only occur if $\Esf_n$ fails, then set $a_0(\gamma) = A_0$ and $a_1(\gamma) = |\gamma|-A_0$.
That is, $\gamma^{(a_0(\gamma),a_1(\gamma))}$ is empty, and so $\wt T(\gamma^{(a_0(\gamma),a_1(\gamma))}) = 0$.
On the other hand, if $|\gamma| > A_0+A_n$, then set $a_0(\gamma)$ equal to the number of edges used by $\gamma$ until reaching $\SS(0)$, and set $a_1(\gamma)$ equal to the number of edges remaining when $\gamma$ lasts intersects $\SS([n\bxi])$.
(By definition of $\Avc(\cdot)$ from \eqref{LQ_def}, we have $a_0(\gamma)\leq A_0$ and $a_1(\gamma)\leq A_n$.)
That is, $\gamma^{(a_0(\gamma),a_1(\gamma))}$ is equal to $\gamma'$ as defined in Lemma \hyperref[for_tightness_lemma_a]{\ref*{for_tightness_lemma}\ref*{for_tightness_lemma_a}}, and so \eqref{for_tightness_2} gives
\eq{
\Big|\frac{\wt T(\gamma^{(a_0(\gamma),a_1(\gamma))})}{n} - \mu_\bxi\Big| \leq
\frac{\Delta_{\mathrm{ext}}(0)+\Delta_{\mathrm{ext}}([n\bxi])}{n} + \Big|\frac{\wh T(0,n\bxi)}{n}-\mu_\bxi\Big|.
}
In summary, we have
\eq{
&\sup_{\bxi\in\S^{d-1}}\sup_{\gamma\in\Geo(0,n\bxi)}\Big|\frac{\wt T(\gamma^{(a_0(\gamma),a_1(\gamma))})}{n} - \mu_\bxi\Big| \\
&\leq \Big(\one_{\Omega\setminus\Esf_n}\sup_{\bxi\in\S^{d-1}}\mu_\bxi\Big)+\frac{\Delta_\mathrm{ext}(0)}{n} + \sup_{x\in\mathbf{S}_n} \frac{\Delta_\mathrm{ext}(x)}{n}
+\sup_{\bxi\in\S^{d-1}}\Big|\frac{\wh T(0,n\bxi)}{n}-\mu_\bxi\Big|.
}
By the fact that $\Esf_n$ occurs for all large $n$, the first term on the right-hand side tends to $0$ almost surely and hence trivially in $L^p$, $p\in[1,\infty)$.
The estimate \eqref{Delta_2_tail}, again with a union bound over $\mathbf{S}_n$, can be used to show that the second and third terms also converge to $0$ almost surely and in $L^p$.
The fourth term is covered by \eqref{shape_thm}.
We have thus proved \eqref{adjusted_convergence}.

Finally, if $\tau_e$ is bounded, then we assume $\tf$ is at least $\esssup\tau_e$.
Property \ref{condition_5} implies that $\Avc(x) = 1$ for all $x\in\Z^d$.
In this case, by \eqref{perturbed_general_assumptions_c} we have 
\eq{
|\wt T(\gamma) - \wt T(\gamma^{(a_0(\gamma),a_1(\gamma))})| \leq 2\sup_{e\in E(\Z^d)}|\tilde\tau_e|\leq \af +\tf  \quad \mathrm{a.s.} 
}
Consequently, \eqref{adjusted_convergence} still holds with $a_0(\gamma)=a_1(\gamma)=0$ for all $\gamma$.
\end{proof}

\begin{proof}[Proof of Proposition \ref{addition_thm}]
By the translation invariance from \eqref{perturbed_general_assumptions}, it suffices to prove \eqref{sup_infinite} in the case $x=0$.
First note that $|\wt T(0,\SS([n\bxi])) - \wh T(0,n\bxi)| \leq \Delta_{\mathrm{int}}(0)+\Delta_\mathrm{ext}(0)$, where $\Delta_{\mathrm{int}}(\cdot)$ and $\Delta_\mathrm{ext}(\cdot)$ are defined in \eqref{deltas_def}.
Since $\Delta_{\mathrm{int}}(0)+\Delta_\mathrm{ext}(0)$ is almost surely finite by property \ref{condition_2}, this inequality allows us to rewrite \eqref{shape_thm} as
\eeq{ \label{shell_to_mu_modified}
\lim_{n\to\infty} \sup_{\bxi\in\S^{d-1}}\Big|\frac{\wt T(0,\SS([n\bxi]))}{n} - \mu_\bxi\Big| = 0 \quad \mathrm{a.s.}
}
Next recall from the proof of Proposition \ref{replacement_thm} that
\eq{
\lim_{n\to\infty} \sup_{\bxi\in\S^{d-1}} \frac{\Avc([n\bxi])}{n} = 0 \quad \mathrm{a.s.}
}
Therefore, we may assume $n$ is sufficiently large that $\Avc(0)+\Avc([n\bxi]) < \|[n\bxi]\|_1$ for all $\bxi\in\S^{d-1}$.
Consider any $\Gamma\in\Geo_\infty$ with initial vertex $x_0=0$.
Suppose that $x_\ell = [n\bxi]$ for $\bxi\in\S^{d-1}$, where $\ell$ necessarily satisfies $\ell\geq \|[n\bxi]\|_1>\Avc(0)+\Avc([n\bxi])$.
Consequently, Lemma \hyperref[for_tightness_lemma_a]{\ref*{for_tightness_lemma}\ref*{for_tightness_lemma_a}} implies the existence of $j\in\{0,\dots,\ell-1\}$ such that $x_j\in\SS(x_\ell)$.
Following  the path $\Gamma$ in the forward direction, we can also find $k>\ell$ such that $x_k\in\SS(x_\ell)$, by property \ref{condition_2}.

Now let $w\in \SS(x_\ell)$ be such that $\wt T(0,w) = \wt T(0,\SS(x_\ell))$; here we are again using finiteness of shells. 
Given that $\SS(x_\ell)$ is connected and contains only white vertices, we almost surely have $\wt T(x_{j},x_{k}) \leq (\tf + \af )|\SS(x_\ell)|$ by \eqref{perturbed_general_assumptions_c}, and also $\wt T(x_j,w) \leq (\tf +\af )|\SS(x_\ell)|$ for the same reason.
Since $\Gamma$ is a geodesic, we also know $\wt T(x_{j},x_{k}) = \wt T(x_{j},x_\ell)+\wt T(x_\ell,x_{k})$.
Putting these two observations together, we have
\begin{subequations} \label{w_connection}
\begin{align}
\wt T(w,x_\ell) &\stackref{final_subadditive_tilde}{\leq} \wt T(w,x_j) + \wt T(x_j,x_\ell) + 2\af  \Lvc(x_j;0) \notag \\
&\stackrefp{trivial_neg_bd}{=}  \wt T(w,x_j) + \wt T(x_j,x_k) - \wt T(x_\ell,x_k) + 2\af  \Lvc(x_j;0) \label{w_connection_1}\\
&\stackref{trivial_neg_bd}{\leq} 2(\tf +\af )|\SS(x_\ell)| + \af \Lvc(x_k;0) + 2\af  \Lvc(x_j;0). \notag
\end{align}
In addition, we trivially have
\begin{align} \label{w_connection_2}
-\wt T(w,x_\ell) \stackref{trivial_neg_bd}{\leq} \af  \Lvc(w;0).
\end{align}
\end{subequations}
We now see
\eq{
|\wt T(0,w) - \wt T(0,x_\ell)| 
&\stackref{final_subadditive_tilde}{\leq}  |\wt T(w,x_\ell)| + 2\af \Lvc(w;0) + 2\af \Lvc(x_\ell;0) \\
&\stackref{w_connection}{\leq} 3\Delta_\mathrm{ext}(x_\ell) + 2\af \Lvc(x_\ell;0).
}
Since $\wt T(0,w) = \wt T(0,\SS(x_\ell))$, we have thus shown
\eq{
\sup_{\bxi\in\S^{d-1}}\sup_{\substack{\Gamma\in\Geo_\infty \\
x_0=0,\, x_\ell=[n\bxi]}}\Big|\frac{\wt T(\Gamma^{(\ell)})}{n}- \frac{\wt T(0,\SS([n\bxi]))}{n}\Big|
\leq \sup_{x\in \mathbf{S}_n}\big(3\Delta_\mathrm{ext}(x) + 2\af \Lvc(x;0)\big),
}
where $\mathbf{S}_n$ was defined in \eqref{Sn_def}.
As in the proof of \eqref{sup_approx_as}, we can use the estimates \eqref{Delta_2_tail} and \eqref{Qx_tail}---together with a union bound over $\mathbf{S}_n$ followed by an application of Borel--Cantelli---to show that the right-hand side displayed above tends to $0$ almost surely.
Therefore, \eqref{sup_infinite} follows from \eqref{shell_to_mu_modified}.

To deduce \eqref{single_infinite} from \eqref{sup_infinite}, all that remains to show is the following implication:
\eeq{ \label{direction_implication}
\Gamma\in\Geo_\infty(\bxi) \quad \implies \quad \lim_{\ell\to\infty} \frac{x_\ell\cdot\bxi}{n_\ell} = 1,
}
where $n_\ell$ is any number such that $x_\ell = x_0+[n_\ell\zetab_\ell]$ for some $\zetab_\ell\in\S^{d-1}$.
Indeed, the fact that $x_\ell/\|x_\ell\|_2$ converges to $\bxi$ implies that $\zetab_\ell$ converges to $\bxi$.
For any $\ell$, the Cauchy--Schwarz and the triangle inequalities produce the estimate
\eq{
|[n_\ell\zetab_\ell] \cdot \bxi - n_\ell|
= \Big|\big([n_\ell\zetab_\ell]- (n_\ell\bxi)\big)\cdot\bxi\Big|
&\leq \|[n_\ell\zetab_\ell]-(n_\ell\bxi)\|_2 \\
&\leq \|[n_\ell\zetab_\ell]-n_\ell\zetab_\ell\|_2+n_\ell\|\zetab_\ell-\bxi\|_2 \\
&\leq \sqrt{d} + n_\ell\|\zetab_\ell-\bxi\|_2.
}
Upon dividing by $n_\ell$ and sending $\ell\to\infty$, we conclude that $x_\ell\cdot\bxi/n_\ell\to1$.
\end{proof}

This section's final proof 
is independent of what has come before.

\begin{proof}[Proof of Proposition \ref{negative_thm_2}]
We require two separate arguments: one for $d=2$, and another for $d\geq3$. \\

\noindent \textbf{Case 1: $d=2$.}
First assume $F(0)>p_\cc(\Z^2)=1/2$.
As usual, call an edge $e$ \textit{open} if $\tau_e = 0$; otherwise $e$ is \textit{closed}.
As before, let $\OO$ denote the unique infinite open cluster.  \label{infinite_open_cluster_2}

\begin{claim}
There exists a doubly-infinite self-avoiding path in $\OO$.
\end{claim}

\renewcommand{\qedsymbol}{$\square$ (Claim)}
\begin{proof}
Let $\theta=\P(0\in\OO)$, which is positive because $F(0)>p_\cc(\Z^2)$.
By ergodicity, we know
\eeq{ \label{ergodic_conv}
\lim_{N\to\infty}\frac{1}{N}\sum_{n=0}^{N-1}\one_{\{n\mathbf{e}_1\in\OO\}} = \theta \quad \mathrm{a.s.}
}
In particular, for any $\eps>0$, there almost surely exists $N_0$ sufficiently large that the following implication holds for all $N_1,N_2$:
\eeq{ \label{guaranteed_points}
N_0 \leq N_1 \leq \frac{N_2}{1+\eps} \implies n\mathbf{e}_1\in\OO \text{ for some $n \in \{N_1,\dots,N_2-1\}$}.
}
To see this, assume without loss of generality that $\eps\in(0,1)$, and take $N_0$ large enough that
\eq{
N \geq N_0 \quad \implies \quad \Big(1 - \frac{\eps}{2}\Big)\theta < \frac{1}{N}\sum_{n=0}^{N-1}\one_{\{n\mathbf{e}_1\in\OO\}} < (1+\eps)\Big(1 - \frac{\eps}{2}\Big)\theta.
}
Then whenever $N_0 \leq N_1 \leq N_2(1+\eps)^{-1}$, we must have $n\mathbf{e}_1\in\OO$ for some $n\in\{N_1,\dots,N_2-1\}$, since otherwise we would obtain the following contradiction:
\eq{
\Big(1 - \frac{\eps}{2}\Big)\theta < \frac{1}{N_2}\sum_{n=0}^{N_2-1} \one_{\{n\mathbf{e}_1\in\OO\}} 
&= \frac{1}{N_2}\sum_{n=0}^{N_1-1} \one_{\{n\mathbf{e}_1\in\OO\}} \\
&\leq\frac{1}{(1+\eps)N_1}\sum_{n=0}^{N_1-1} \one_{\{n\mathbf{e}_1\in\OO\}} 
< \Big(1 - \frac{\eps}{2}\Big)\theta.
}
Now recall that \cite[Thm.~1.1]{antal-pisztora96} gives a constant $C_0 = C_0(F(0),d)$ such that
\eq{
\P\big(n\mathbf{e}_1\leftrightarrow N\mathbf{e}_1, D_\oo(n\mathbf{e}_1,N\mathbf{e}_1) \geq C_0(n-N)\Big)\leq C\e^{-c(n-N)} \quad 0\leq N<n,
}
where $x\leftrightarrow y$ means $x$ and $y$ are connected by a path of open edges, and $D_\oo(x,y)$ is the minimum length of such a path.
Set $\eps = 1/C_0$ so that for every $n\geq(1+\eps)N$, we have $C_0(n-N)\geq N$.
By taking a union bound and applying the tail bound shown above, we obtain
\eq{
\P\bigg(\bigcup_{n=\lceil (1+\eps)N\rceil}^{\big\lceil(1+\eps)\lceil (1+\eps)N\rceil\big\rceil} \{n\mathbf{e}_1\leftrightarrow N\mathbf{e}_1, D_\oo(n\mathbf{e}_1,N\mathbf{e}_1) \geq N\}\bigg) \leq C\eps N\e^{-c\eps N}
}
for all $N>0$.
By \mbox{Borel}--\mbox{Cantelli}, it is almost surely the case that the event in the display occurs for only finitely many $N>0$.
Yet \eqref{guaranteed_points} shows
\eq{
\bigcup_{n=\lceil (1+\eps)N\rceil}^{\big\lceil(1+\eps)\lceil (1+\eps)N\rceil\big\rceil} \{n\mathbf{e}_1\leftrightarrow N\mathbf{e}_1\} \quad \text{occurs whenever $N\mathbf{e}_1\in\OO_1$, $N\geq N_0$}.
}
Therefore, there is almost surely some $N_0'>0$ such that
\eq{
N\mathbf{e}_1\in\OO,\, N \geq N_0' \quad \implies \quad D_\oo(n\mathbf{e}_1,N\mathbf{e}_1) < N \quad \text{for some $n\geq(1+\eps)N$}.
}
Since \eqref{ergodic_conv} implies $N\mathbf{e}_1\in\OO$ for infinitely many $N\geq N_0'$ (in fact, we just need to know this for a single $N\geq N_0'$), it follows that there is almost surely some infinite open path in $\OO$ that traverses $n\mathbf{e}_1$ for infinitely many positive $n\in\Z$, yet intersects the vertical axis only finitely many times.
By symmetry, there also exists an open path traversing $n\mathbf{e}_1$ for infinitely many negative $n\in\Z$, but again intersecting the vertical axis only finitely often.
These two paths obviously intersect each other only finitely many times, and so they can be connected to form a doubly-infinite open path.
By omitting loops, this doubly-infinite path can be trimmed to a self-avoiding one.
\end{proof}
\renewcommand{\qedsymbol}{$\square$}

Let $\Gamma$ be any doubly-infinite self-avoiding path in $\OO$.
Consider any two distinct vertices $x,y\in\Z^d$, and let $x_0$ and $y_0$ be any two distinct vertices traversed by $\Gamma$.
Take any $\gamma_1\in\PP(x,x_0)$ and $\gamma_2\in\PP(y,y_0)$ such that $\gamma_1$ and $\gamma_2$ do not intersect.
If either $\gamma_1$ or $\gamma_2$ pass through a vertex common with $\Gamma$ before reaching $x_0$ or $y_0$, respectively, then replace $x_0$ or $y_0$ by the first point of intersection.
In this way, $\gamma_1$ and $\gamma_2$ may be assumed to avoid $\Gamma$ expect at the terminal points $x_0$ and $y_0$, as shown in Figure \ref{2d_fig}.

\begin{figure}
\includegraphics[width=0.9\textwidth]{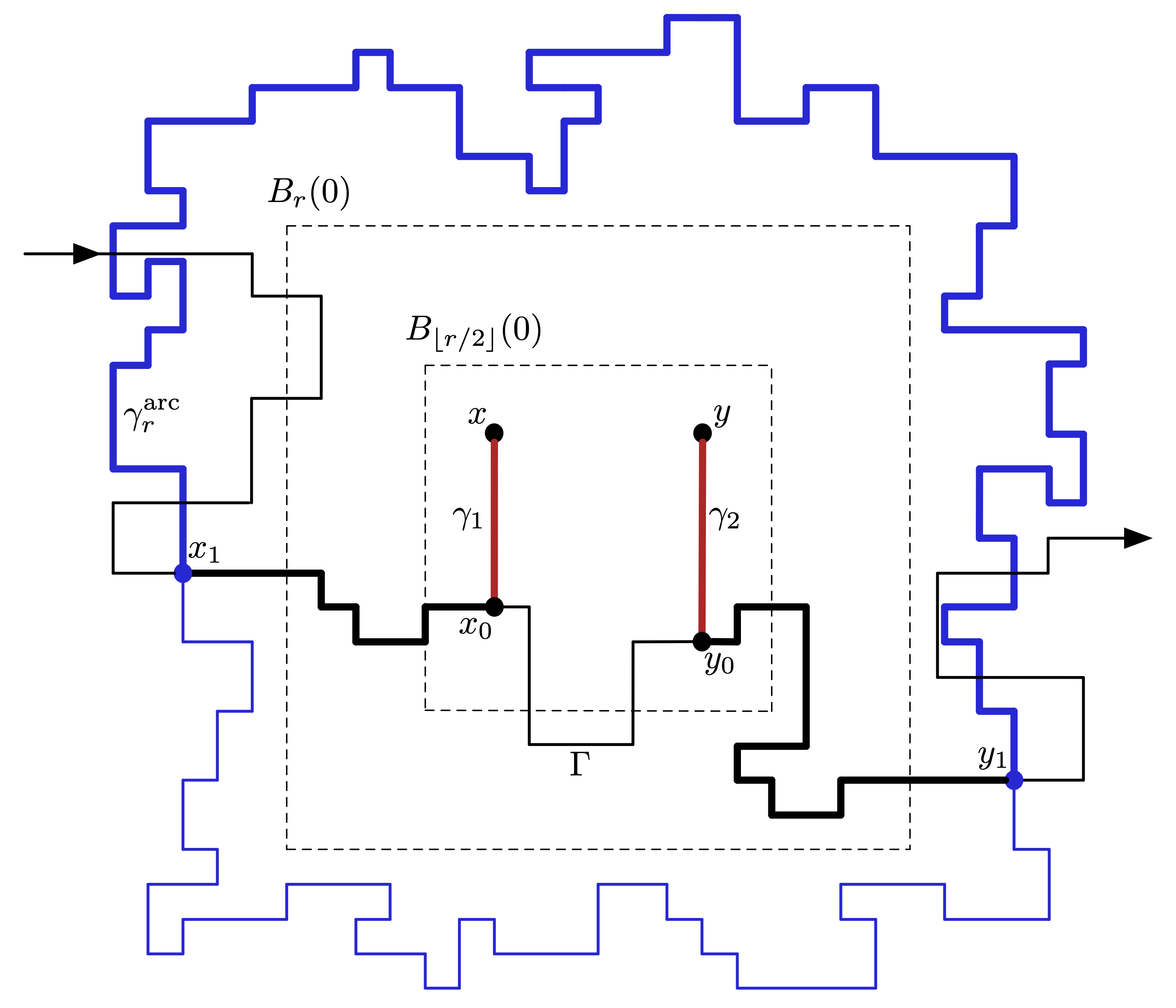}
\caption{Path construction in proof of Proposition \ref{negative_thm_2}, $d=2$.
The desired path $\gamma\in\PP(x,y)$ defined in \eqref{final_path} by concatenation is indicated by thicker curves.
All edges in $\Gamma$ and $\gamma_r^\mathrm{arc}$ have $\tau_e=0$, and $\gamma_1,\gamma_2$ do not vary with $r$.
Therefore, $\gamma$ can be made arbitrarily long while maintaining a fixed passage time $T(\gamma) = T(\gamma_1)+T(\gamma_2)$.
When $h$ is subtracted from each edge-weight, this leads to $\wt T(x,y)=-\infty$.}
\label{2d_fig}
\end{figure}

It is well-known that for supercritical (and critical) Bernoulli percolation in $\Z^2$, there almost surely exists for every $r\geq0$ some loop (or ``circuit") $\gamma_r$ of open edges which contains $\Bbf_r(0)$ in its interior; for example, see \cite[Section 11.7]{grimmett99}.
So let $r$ be large enough that $x_0,y_0\in \Bbf_{\lfloor r/2\rfloor}(0)$, which implies that any such $\gamma_r$ intersects $\Gamma$ at least twice: once as $\Gamma$ enters the interior of the circuit, and once more upon exit.
Here we have chosen an orientation for $\Gamma$; without loss of generality, we assume this orientation is such that $\Gamma$ reaches $x_0$ before $y_0$.
We also assume that $r$ is large enough that the portion of $\Gamma$ between $x_0$ and $y_0$ remains in the interior of $\Bbf_r(0)$.
Now, with $\gamma_r$ as above, let $x_1$ be the last point of intersection between $\gamma_r$ and $\Gamma$ before $\Gamma$ reaches $x_0$. 
Similarly, let $y_1$ be the first point of intersection between $\gamma_r$ and $\Gamma$ after $\Gamma$ has left $y_0$.
By our choice of $r$, the path $\Gamma$ does not intersect $\gamma_r$ at any point between $x_0$ and $y_0$.
Let $\gamma_r^\mathrm{arc}$ be the larger of the two arcs in $\gamma_r$ connecting $x_1$ to $y_1$, so that $|\gamma_r^\mathrm{arc}| \geq 4r$.

Now consider the following concatenation of paths, as shown in Figure \ref{2d_fig}:
\eeq{ \label{final_path}
\gamma:\, x \xrightarrow{\gamma_1} x_0 \xrightarrow{\Gamma} x_1 \xrightarrow{\gamma_r^\mathrm{arc}} y_1 \xrightarrow{\Gamma} y_0 \xrightarrow{\gamma_2} y.
}
With $\tau^\pert_e= -h<0$ for every edge $e$, it is now apparent that
\eeq{ \label{supercritical_tilde}
\wt T(x,y) \leq \wt T(\gamma) = T(\gamma_1) + T(\gamma_2) - h|\gamma| \leq T(\gamma_1) + T(\gamma_2) - 4rh.
}
Since $\gamma_1$ and $\gamma_2$ are fixed, sending $r\to\infty$ shows $\wt T(x,y) = -\infty$.

Now consider the critical case $F(0)=1/2$.
With $\eps\in(0,1/4)$ a small number to be specified below, we define a modified environment $\tau_e^\downarrow$ given by
\eq{
\tau_e^\downarrow = \begin{cases} 0 &\text{if $F(\tau_e)<\frac{1}{2}+\eps$},\\
\tau_e &\text{otherwise}.
\end{cases}
}
In other words, $\tau_e^\downarrow$ replaces certain instances of $\tau_e\neq0$ with the value $0$ in order to create a supercritical environment.
Yet our assumption $\eps<1/4$ ensures
\eq{
\tau_e-\tau_e^\downarrow \leq t' \coloneqq F^{-1}(3/4).
}
With $h>0$ already fixed, we choose $\eps$ be sufficiently small that there exist constants $C,c>0$ satisfying
\eq{
\P\Big(\exists\, \gamma\in\PP(x),\,|\gamma|\geq \ell,\, \sum_{e\in\gamma}\one_{\{\frac{1}{2}<F(\tau_e)<\frac{1}{2}+\eps\}} \geq \frac{h|\gamma|}{2t'}\Big) \leq C\e^{-c\ell} \quad \text{for all $\ell\geq1$}.
}
(This is possible via a standard Chernoff bound for the sum of $\ell$ independent Bernoulli($\eps)$ random variables, using the fact that there are at most $4\cdot 3^{\ell-1}$ self-avoiding paths of length $\ell$ starting at $x$.
For example, see \cite[pg.~23, 24]{boucheron-lugosi-massart13} or refer to \eqref{binomial_chernoff} with $p=h/(2t')$, observing that $h_\eps(p)\to\infty$ as $\eps\searrow0$).
By \mbox{Borel}--\mbox{Cantelli}, it follows that 
\eq{
\sum_{e\in\gamma}(\tau_e-\tau_e^\downarrow) \leq \frac{h}{2\tf }|\gamma| \quad \text{for all $\gamma\in\PP(x)$, $|\gamma|$ sufficiently large}.
}
Combined with \eqref{supercritical_tilde}, which is valid for the edge-weights $(\tau_e^\downarrow)_{e\in E(\Z^d)}$ because $\P(\tau_e^\downarrow=0)>1/2$, this observation shows we still have $\wt T(x,y) = -\infty$. \\

\noindent \textbf{Case 2: $d\geq3$.}
We again begin by assuming $F(0)>p_\cc(\Z^d)$.
Let $x$ and $y$ be any two distinct vertices.
Since the FPP model is invariant under translations and symmetries of the lattice, we may assume without loss generality that $y_3 < 0\leq x_3$.
As in \cite{grimmett-marstrand90}, we define ``slices of thickness $k$":
\eq{
S_+(k) &\coloneqq \{z\in\Z^d :\, 0 \leq z_i \leq k \text{ whenever $3\leq i\leq d$}\}, \\
S_-(k) &\coloneqq \{z\in\Z^d :\, 0 > z_i \geq -(1+k) \text{ whenever $3\leq i\leq d$}\}.
}
The main result of \cite{grimmett-marstrand90} is that if $k$ is sufficiently large, then $S_+(k)$ almost surely contains an infinite cluster $\OO_+$ of sites connected by open edges within $S_+(k)$.
Analogously, there is an infinite open cluster $\OO_-\subset S_-(k)$.\footnote{The article \cite{grimmett-marstrand90} is phrased in terms of site percolation, but its methods work equally well for bond percolation; see \cite[pg.~2]{grimmett-marstrand90}. Another possible reference is \cite[Thm.~7.2]{grimmett99}.}
With $\mathbf{e}_1$ and $\mathbf{e}_3\in\Z^d$ denoting the first and third standard basis vectors, and $\theta \coloneqq \P(0\in\OO_+)=\P(-\mathbf{e}_3\in\OO_-)>0$, independence and ergodicity imply
\eq{
\lim_{N\to\infty} \frac{1}{N}\sum_{n=0}^{N-1} \one_{\{n\mathbf{e}_1\in\OO_+,\, n\mathbf{e}_1-\mathbf{e}_3\in\OO_-,\,\, \{n\mathbf{e}_1,n\mathbf{e}_1-\mathbf{e}_3\}\text{ is open}\}} = \theta^2F(0) \quad \mathrm{a.s.}
}
Now fix any $x_0\in\OO_+$, $y_0\in\OO_-$, and paths $\gamma_1\in\PP(x,x_0)$, $\gamma_2\in\PP(y,y_0)$ that are confined to $S_+(k)$ and $S_-(k)$, respectively; here we assume $k$ is large enough to ensure $x\in S_+(k)$ and $y\in S_-(k)$.
As in the two-dimensional case, we also assume that $x_0$ and $y_0$ are the \textit{first} intersections of $\gamma_1$ and $\gamma_2$ with $\OO_+$ and $\OO_-$, respectively.
By the limit displayed above, almost surely the following holds for infinitely many nonnegative integers $n$:
We can connect $x_0$ to $n\mathbf{e}_1$ via an open path remaining in $ S_+(k)$, then join $n\mathbf{e}_1$ to $n\mathbf{e}_1-\mathbf{e}_3$ by an open edge, and finally connect $n\mathbf{e}_1-\mathbf{e}_3$ to $y_0$ with an open path remaining in $ S_-(k)$.
(See Figure \ref{3d_fig} for an illustration.)
By appending $\gamma_1$ and $\gamma_2$ to the beginning and end of this path, we obtain $\gamma\in\PP(x,y)$ such that
\eeq{ \label{supercritical_tilde_2}
\wt T(x,y) \leq \wt T(\gamma) = T(\gamma_1)+T(\gamma_2) - h|\gamma|.
}
Letting $n\to\infty$ so that $|\gamma|\to\infty$, we conclude $\wt T(x,y) = -\infty$.

\begin{figure}
\subfloat[Connecting $x$ to $n\mathbf{e}_1$; all edges except those of $\gamma_1$ remain in $\OO_+$]{
\label{3d_fig_a}
\includegraphics[width=0.9\textwidth]{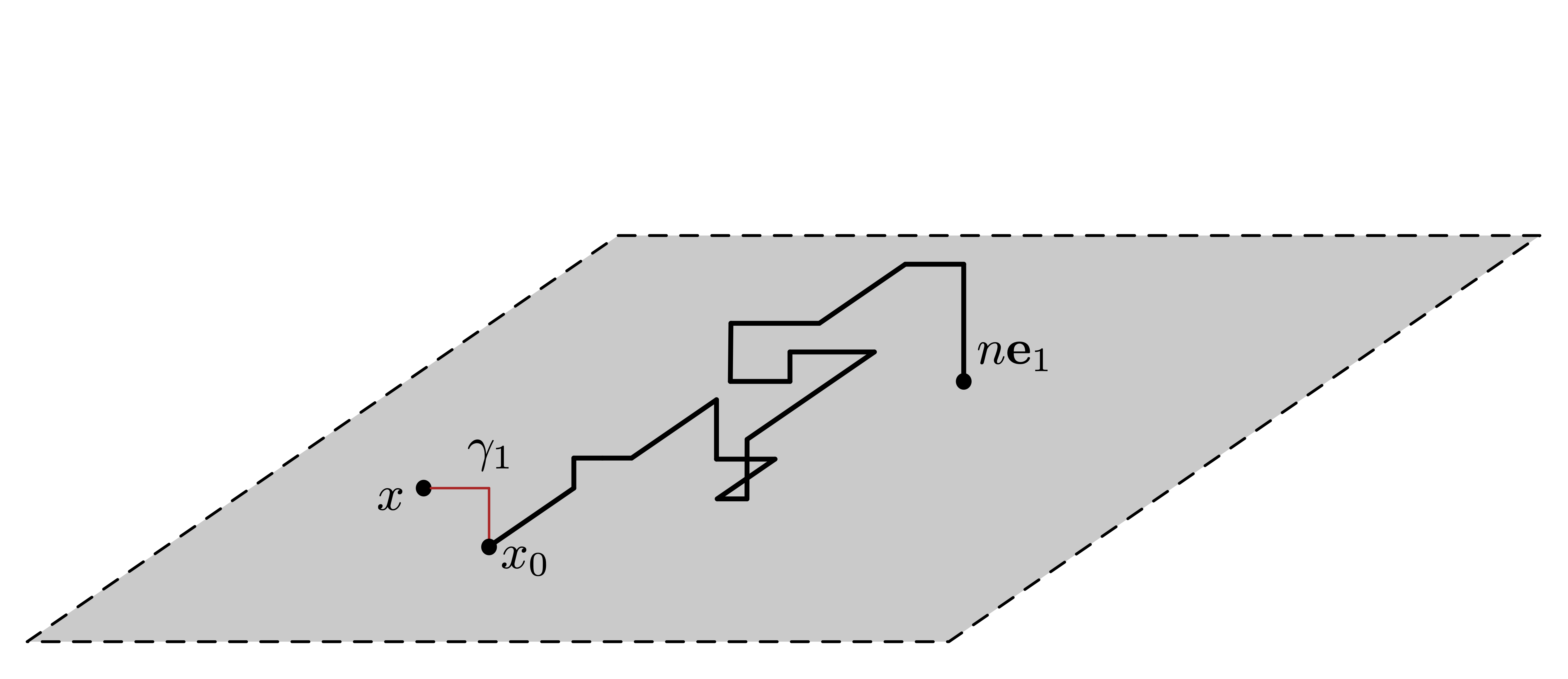}
} \\
\subfloat[Connecting $y$ to $n\mathbf{e}_1-\mathbf{e}_3$; all edges except those of $\gamma_2$ remain in $\OO_-$]{
\label{3d_fig_b}
\includegraphics[width=0.9\textwidth]{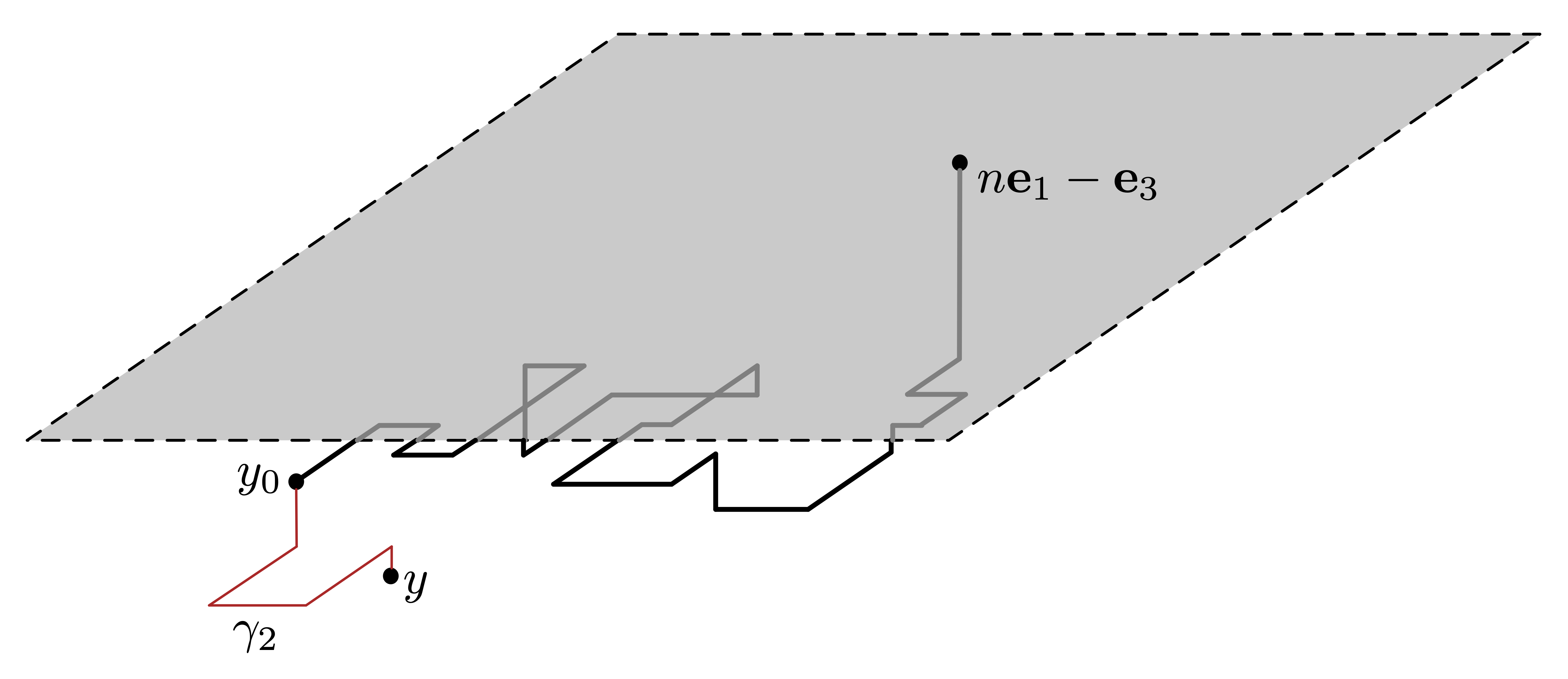}
}\\
\subfloat[Joining the two paths with an open edge between $n\mathbf{e}_1$ and $n\mathbf{e}_1-\mathbf{e}_3$]{
\label{3d_fig_c}
\includegraphics[width=0.9\textwidth]{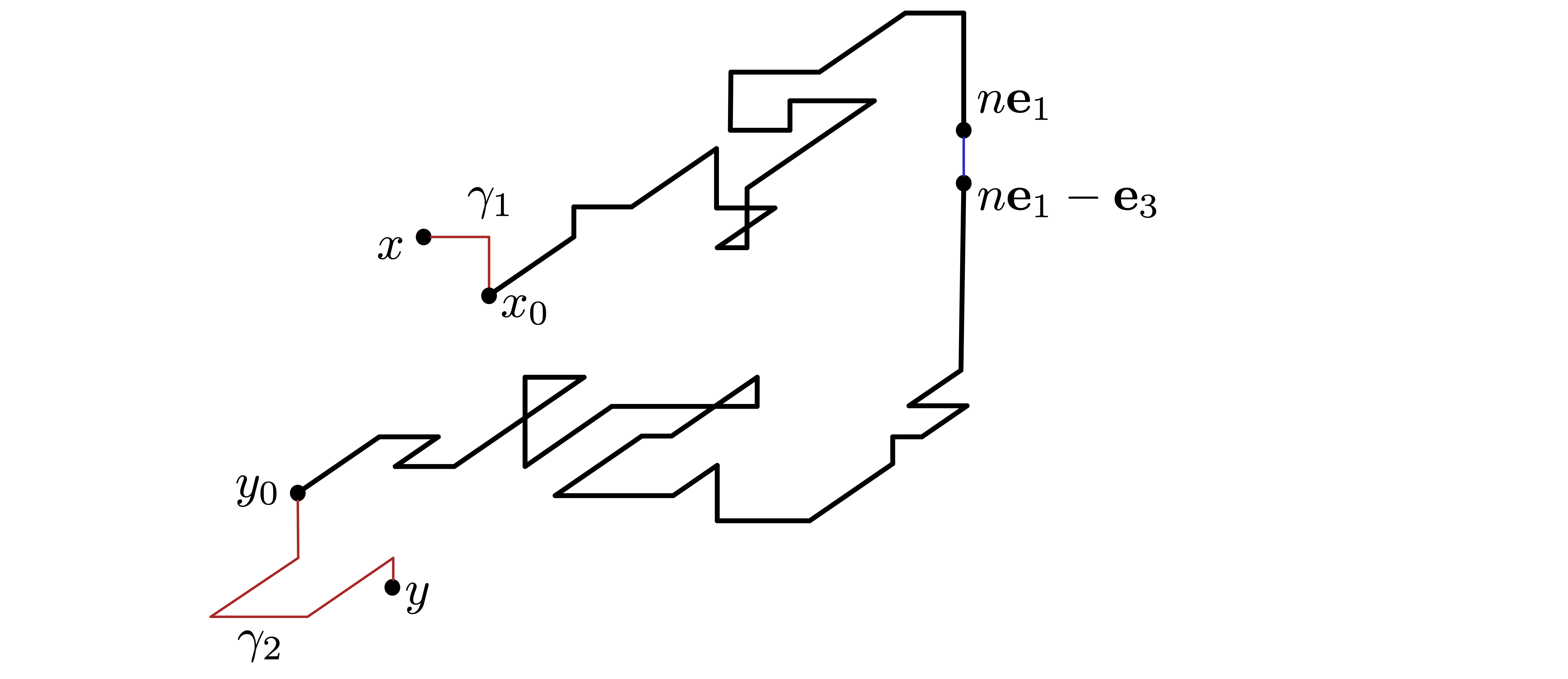}
}
\caption{Path construction in the proof of Proposition \ref{negative_thm_2}, $d\geq3$.
Diagram (a) takes place in $S_+(k)$, and (b) in $S_-(k)$.
Each slice admits an infinite open cluster that is used to create long paths starting at $x$ and $y$, eventually consisting of all zero-weight edges.
Since $S_+(k)\cap S_-(k)=\varnothing$, these paths remain disjoint until they are strategically connected in (c) at adjacent points on the boundaries.}
\label{3d_fig}
\end{figure}

Finally, we can handle the critical case $F(0)=p_\cc(\Z^d)$ exactly as we did in the two-dimensional case, with \eqref{supercritical_tilde_2} replacing \eqref{supercritical_tilde}.
\end{proof}

\chapter{Construction of the Constraint Set} \label{constraint_sec}
The goal of this section is to construct the constraint set $\RR^\bxi$ found in the variational formula \eqref{var_form_eq}.
While that particular set is defined using Lebesgue measure $\Lambda$, here we pursue a slightly more general construction.
Namely, in this section we allow $\Lambda$ to be replaced by any Radon probability measure $\pp$ on $[0,1]$. \label{radon_meas_def}
We maintain this level of generality in order to later prove Theorem \ref{weak_to_strong_thm}, although it will not make any difference in the arguments.

\section{Topological preliminaries} \label{topological_preliminaries}
Recall that $\Sigma$ denotes the set of finite, positive \mbox{Borel} measures on $[0,1]$ having total mass at least $1$.
Let $\hat\Sigma\subset\Sigma$ be the subset consisting of probability measures, i.e.~with mass exactly $1$.
A standard metric on $\hat\Sigma$ is the \textit{Wasserstein distance}, 
\eeq{ \label{wasserstein_def} 
 W(\hat\sigma_1,\hat\sigma_2) \coloneqq \inf_{P\in\hat\Sigma(\hat\sigma_1,\hat\sigma_2)}\int_{[0,1]^2} |s - t|\, P(\dd s,\dd t), \quad \hat\sigma_1,\hat\sigma_2\in\hat\Sigma,
}
where $\hat\Sigma(\hat\sigma_1,\hat\sigma_2)$ is the set of probability measures $P$ on $[0,1]^2$ having $\hat\sigma_1$ and $\hat\sigma_2$ as marginals.
A well-known fact is that if $\hat\sigma_1$ and $\hat\sigma_2$ have distribution functions $G_1$ and $G_2$, then
\eeq{ \label{W_explicit}
 W(\hat\sigma_1,\hat\sigma_2) = \int_0^1 |G_1^{-1}(u)-G_2^{-1}(u)|\, \dd u,
}
where the inverse functions are given by \eqref{inverse_cdf}.
Also, 
\eq{
\lim_{k\to\infty} W(\hat\sigma_k,\hat\sigma)=0 \quad \iff \quad \hat\sigma_k\Rightarrow \sigma.
}
That is, $ W$ metrizes the topology of weak convergence on $\hat\Sigma$; see \cite[Thm.~6.9]{villani09}.
The key feature here is that $[0,1]$ is compact.

It will sometimes be more convenient to control the Wasserstein distance from above by \textit{total variation distance}. 
The following, standard bound can be found in \cite[Thm.~6.15]{villani09}:
\eeq{ \label{w_tv_1}
 W(\hat\sigma_1,\hat\sigma_2)
\leq \TV(\hat\sigma_1,\hat\sigma_2) &\coloneqq \sup_{\text{measurable }B\subset[0,1]} |\hat\sigma_1(B)-\hat\sigma_2(B)|.
}
Also, it is a standard exercise that if $\hat\sigma_1,\hat\sigma_2\in\hat\Sigma$ are each supported on $\{u_1,\dots,u_L\}$, then
\eeq{ \label{w_tv_2}
\TV(\hat\sigma_1,\hat\sigma_2)
&=\frac{1}{2}\sum_{\ell=1}^L |\hat\sigma_1(\{u_\ell\}) - \hat\sigma_2(\{u_\ell\})|.
}
For this last quantity, we have the following bound.

\begin{lemma} \label{tv_bound_lemma}
Let $v_1,\dots,v_{I},w_1,\dots,w_{J},x_1,\dots,x_K\in[0,1]$, and consider the probability measures
\eq{
\hat\sigma_1 = \frac{1}{I+K}\bigg(\sum_{i=1}^{I} \delta_{v_i}+\sum_{k=1}^K \delta_{x_k}\bigg), \quad
\hat\sigma_2 = \frac{1}{J+K}\bigg(\sum_{j=1}^{J} \delta_{w_j}+\sum_{k=1}^K \delta_{x_k}\bigg).
}
We have
\eq{
2\TV(\hat\sigma_1,\hat\sigma_2) \leq K\Big|\frac{1}{I+K}-\frac{1}{J+K}\Big| + \frac{I}{I+K} + \frac{J}{J+K}.
}
\end{lemma}

\begin{proof}
Let $u_1,\dots,u_L$ be the distinct elements among the concatenated list $v_1,\dots,v_{I}$, $w_1,\dots,w_{J}$, $x_1,\dots,x_K$.
Using \eqref{w_tv_2}, we have
\eq{ 
&2\TV(\hat\sigma_1,\hat\sigma_2) 
= \sum_{\ell=1}^L\bigg|\frac{1}{I+K}\bigg(\sum_{i=1}^I \one_{\{v_i=u_\ell\}}
+\sum_{k=1}^K\one_{\{x_k=u_\ell\}} \bigg) \\
&\hphantom{2\TV(\hat\sigma^1,\hat\sigma^2)= \sum_{\ell=1}^L\bigg|}- \frac{1}{J+K}\bigg(\sum_{j=1}^J\one_{\{w_j=u_\ell\}}
-\sum_{k=1}^K\one_{\{x_k=u_\ell\}}\bigg)\bigg| \\
&\leq \sum_{\ell=1}^L\bigg(\sum_{k=1}^K\Big|\frac{\one_{\{x_k=u_\ell\}}}{I+K}-\frac{\one_{\{x_k=u_\ell\}}}{J+K}\Big| + \sum_{i=1}^I \frac{\one_{\{v_i=u_\ell\}}}{I+K} + \sum_{j=1}^J \frac{\one_{\{w_j=u_\ell\}}}{J+K}\bigg)\\
&= \Big(K\Big|\frac{1}{I+K}-\frac{1}{J+K}\Big| + \frac{I}{I+K} + \frac{J}{J+K}\Big). 
}
\end{proof}

We can extend $ W$ to all of $\Sigma$ by defining
\eq{
W(\sigma_1,\sigma_2) \coloneqq |\langle 1,\sigma_1-\sigma_2\rangle| +  W(\hat\sigma_1,\hat\sigma_2), \quad \sigma_1,\sigma_2\in\Sigma.
}
It is easy to see that $W$ is a valid metric.
For $M\geq1$, let us define
\eeq{ \label{Sigma_M_def}
\Sigma_M \coloneqq \{\sigma\in\Sigma :\, \langle 1,\sigma\rangle\leq M\}.
}

\begin{lemma} \label{compactness_lemma}
For every $M\geq 1$, $(\Sigma_M,W)$ is compact.
\end{lemma}

\begin{proof}
Because $[0,1]$ is compact, the space $(\hat\Sigma, W)$ is also compact (see e.g. \cite[Rmk.~6.19]{villani09}).
Now consider any sequence $(\sigma_k)_{k\geq1}$ in $\Sigma_M$.
By passing to a subsequence, we may assume that $\langle 1,\sigma_k\rangle$ converges to some $m\in[1,M]$.
By passing to a further subsequence, we may also assume that $\hat\sigma_k$ converges to some $\hat\sigma\in\hat\Sigma$ under $ W$.
It is then clear that $\sigma_k$ converges to $m\hat\sigma$ under $W$.
The (sequential) compactness of $(\Sigma_M,W)$ has been verified.
\end{proof}

We close this preliminary section by considering the space of nonempty, closed subsets of $\Sigma$, which we denote by $\KK(\Sigma)$. \label{KK_def}
This space comes equipped with the usual \textit{Hausdorff metric}, 
\eeq{ \label{hausdorff_def} 
\HH(\AA_1,\AA_2) &\coloneqq \max\Big\{\sup_{\sigma\in \AA_1}W(\sigma,\AA_2),\sup_{\sigma\in \AA_2} W(\sigma,\AA_1)\Big\},
}
where
\eq{
W(\sigma,\AA) \coloneqq \inf_{\sigma'\in \AA_1} W(\sigma,\sigma'), \quad \sigma\in\Sigma,\, \AA \subset \Sigma.
}
For any subset $\Sigma_*\subset\Sigma$, we can restrict $\HH$ to $\KK(\Sigma_*)$, the space of nonempty, (relatively) closed subsets $\AA\subset\Sigma_*$.
In any case, by \cite[Lem.~3.74]{aliprantis-border06} we have
\eeq{ \label{hausdorff_reform}
\HH(\AA_1,\AA_2) = \sup_{\sigma\in\Sigma_*}|W(\sigma,\AA_1)-W(\sigma,\AA_2)|, \quad \AA_1,\AA_2\in\KK(\Sigma_*)
}
When $\Sigma_*=\Sigma_M$, Lemma \ref{compactness_lemma} and \cite[Thm.~3.85(3)]{aliprantis-border06} together imply that $(\KK(\Sigma_M),\HH)$ is compact. 
In addition, it is guaranteed by \cite[Lem.~3.76(3)]{aliprantis-border06} that for any pair $\AA_1,\AA_2\in\KK(\Sigma_M)$, there exist $\sigma_1\in \AA_1$ and $\sigma_2\in \AA_2$ such that $\HH(\AA_1,\AA_2) = W(\sigma_1,\sigma_2)$.

\section{Definition of constraint set}
Let $(U_e)_{e\in E(\Z^d)}$ be a collection of i.i.d. $[0,1]$-valued random variables \label{U_e_def_2}
supported on the complete probability space $(\Omega,\FF,\P)$. 
Assume the law of $U_e$ is given by $\pp$.
We associate to each nonempty path $\gamma$ the following element of $\Sigma$:
\eq{ 
\sigma_\gamma \coloneqq \frac{1}{\|x-y\|_2}\sum_{e\in\gamma}\delta_{U_e}, \quad \gamma\in\PP(x,y).
}
For $n\in[\sqrt{d},\infty)$ and $\eps>0$, define the random set 
\eeq{ \label{prelimit_to_R} 
\RR^{\bxi,\eps}_n \coloneqq \bigcup_{\substack{\zetab\in\S^{d-1}:\, \|\zetab-\bxi\|_2\leq \eps}}\{\sigma_\gamma :\, \gamma\in\PP(0,n\zetab)\}.
}
Then the constraint set we desire for our variational formula is given by
\eeq{ \label{R_infty_def}
\RR^\bxi_\infty \coloneqq \Big\{\sigma\in\Sigma :\, \limsup_{\eps\searrow0}\liminf_{n\to\infty} W(\sigma,\RR^{\bxi,\eps}_n) = 0\Big\}.
}
We can also define directionless versions of these sets: 
\eeq{  \label{directionless_defs} 
\RR_n &\coloneqq \bigcup_{\bxi\in\S^{d-1}} \{\sigma_\gamma:\, \gamma\in\PP(0,n\bxi)\}, \\ 
\RR_\infty &\coloneqq \Big\{\sigma\in\Sigma:\, \liminf_{n\to\infty} W(\sigma,\RR_n) = 0\Big\}.
}
An alternative description of $\RR_\infty^\bxi$ and $\RR_\infty$ can be given as follows.

\begin{lemma} \label{alternative_R_def}
The following statements always hold.
\begin{enumerate}[label=\textup{(\alph*)}]

\item \label{alternative_R_def_a}
For any $\sigma_*\in\Sigma$ and $\bxi_*\in\S^{d-1}$, we have $\sigma_*\in\RR^{\bxi_*}_\infty$ if and only if there exists a sequence $(x_k)_{k\geq1}$ in $\Z^d$, together with $\gamma_k\in\PP(0,x_k)$, such that
\eeq{ \label{W_2_0}
 \lim_{k\to\infty} \|x_k\|_2 = \infty,\quad \lim_{k\to\infty} \frac{x_k}{\|x_k\|_2} = \bxi_*, \quad \text{and} \quad \lim_{k\to\infty} W(\sigma_*,\sigma_{\gamma_{k}}) = 0.
}

\item \label{alternative_R_def_b}
We have $\RR_\infty = \bigcup_{\bxi\in\S^{d-1}}\RR_\infty^{\bxi}$.

\end{enumerate}
\end{lemma}

\begin{proof}
First we prove \ref{alternative_R_def_a}.
Assume \eqref{W_2_0} holds.
The second limit in \eqref{W_2_0} shows that for every $\eps>0$, the measure $\sigma_{\gamma_k}$ belongs to $\RR_{\|x_k\|_2}^{\bxi_*,\eps}$ for all large $k$.
Therefore, the first and third limits give $\liminf_{n\to\infty} W(\sigma_*,\RR_n^{\bxi_*,\eps})=0$ for every $\eps>0$, and so $\sigma_*\in\RR^{\bxi_*}_\infty$.

Now assume $\sigma_*\in\RR^{\bxi_*}_\infty$.
We wish to construct $(x_k)_{k\geq1}$ for which \eqref{W_2_0} holds.
Given any $k\geq1$, because $\sigma_*\in\RR^{\bxi_*}_\infty$, we can choose $\eps\in(0,k^{-1}]$ and $n_k\geq k+\sqrt{d}$ such that $W(\sigma_*,\RR^{{\bxi_*},\eps}_{n_k}) \leq k^{-1}$.
This permits us to take $x_k =  [n_k\zetab_k]$ and $\gamma_k \in \PP(0,x_k)$ such that $W(\sigma_*,\sigma_{\gamma_k}) \leq k^{-1}$,
where $\zetab_k$ is some unit vector satisfying $\|\zetab_k-{\bxi_*}\|_2\leq\eps\leq k^{-1}$.
It follows that
\eq{
\|x_k\|_2 = \|[n_k\zetab_k]\|_2 \geq \|n_k\zetab_k\|_2 - \sqrt{d} = n_k - \sqrt{d} \geq k,
}
and also that
\eq{
\Big\|\frac{x_k}{\|x_k\|_2}-{\bxi_*}\Big\|_2
&\leq \frac{\|x_k-n_k\zetab_k\|_2 + \|n_k\zetab_k - \|x_k\|_2\zetab_k\|_2}{\|x_k\|_2} + \|\zetab_k-\bxi_*\|_2 \\
&\leq \frac{2\sqrt{d}+1}{k}.
} 
Sending $k\to\infty$, we see that \eqref{W_2_0} holds.

Now we prove \ref{alternative_R_def_b}.
Since $\RR_n^{\bxi,\eps}\subset\RR_n$ for every $\bxi$ and $\eps$, it is clear that $\RR_\infty$ contains $\bigcup_{\bxi\in\S^{d-1}}\RR_\infty^\bxi$.
Conversely, let us consider any $\sigma\in\RR_\infty$.
By definition, there is a sequence $(x_k)_{k\geq1}$ in $\Z^d$, which tends to $\infty$, admitting $\gamma_k\in\PP(0,x_k)$ such that $W(\sigma,\sigma_{\gamma_k})\to0$ as $k\to\infty$.
By compactness of $\S^{d-1}$, we may pass to a subsequence so that $x_k/\|x_k\|_2$ converges to some $\bxi_*\in\S^{d-1}$.
It thus follows from part \ref{alternative_R_def_a} that $\sigma\in\RR_\infty^{\bxi_*}$.
Hence we also have $\RR_\infty \subset \bigcup_{\bxi\in\S^{d-1}}\RR_\infty^\bxi$, which completes the proof.
\end{proof}

The following result proves Theorem \hyperref[var_form_thm_c]{\ref*{var_form_thm}\ref*{var_form_thm_c}}.

\begin{thm} \label{deterministic_limits_thm}
There are deterministic sets $\RR=\RR(d,\pp)$ and $\RR^\bxi=\RR^\bxi(d,\pp)$, $\bxi\in\S^{d-1}$,
such that
\eq{
\P(\RR_\infty=\RR) = 1 \quad \text{and} \quad \P(\RR^\bxi_\infty=\RR^\bxi) = 1 \quad \text{for all $\bxi\in\S^{d-1}$}.
}
Furthermore, there is a constant 
$C = C(d,\pp)$ such that 
\eeq{ \label{abs_cont_quantitative_2}
\hat\sigma(B) \leq C(\log \pp(B)^{-1})^{-1} \quad \text{for all $\sigma\in\RR^\bxi$, measurable $B\subset[0,1]$}.
}
\end{thm}

\begin{remark} \label{direction_dependent_remark_1}
It follows from Lemma \hyperref[{alternative_R_def_b}]{\ref*{alternative_R_def}\ref*{alternative_R_def_b}} and Theorem \ref{deterministic_limits_thm} that $\RR$ contains $\bigcup_{\bxi\in\S^{d-1}}\RR^\bxi$.
It is natural to suspect that the two sets are in fact equal, or equivalently that $\P(\RR_\infty^\bxi=\RR^\bxi \text{ for all $\bxi\in\S^{d-1}$})=1$.
To prove this statement, it would suffice to show that the map $\S^{d-1}\to\KK(\Sigma)$ given by $\bxi\mapsto\RR_\infty^\bxi$ is almost surely continuous, but doing so is made difficult by the fact that all paths under consideration are self-avoiding.
\end{remark}

The proof of Theorem \ref{deterministic_limits_thm} will require the following general fact.

\begin{lemma} \label{deterministic_lemma}
Let $(\YY,\TTT)$ be a second-countable, Hausdorff topological space, whose \mbox{Borel} sigma-algebra is denoted by $\BB(\TTT)$. \label{borel_2}
Let $(X_i)_{i\geq1}$ be independent random variables on a complete probability space $(\Omega,\FF,\P)$, and let $Y\colon \Omega \to \YY$ be $(\FF,\BB(\TTT))$-measurable. 
If $Y$ is measurable with respect to the tail sigma-algebra associated to the $X_i$'s, then there is $y\in\YY$ such that $\P(Y=y)=1$.
\end{lemma}

\begin{proof}
By Kolmogorov's zero-one law, we have
\eeq{ \label{0_1_consequence}
\P(Y\in B) \in \{0,1\} \quad \text{for every $B\in\BB(\TTT)$}.
}
Let $\{\OOO_i\}_{i=1}^\infty\subset\TTT$ be a countable base for the topology $\TTT$.
It is trivial that
$Y \in \bigcap_{i :\, Y\in\OOO_i} \OOO_i$.
But because $\TTT$ is Hausdorff, and the $\OOO_i$'s form a base for $\TTT$, we in fact have
$\bigcap_{i :\, Y\in\OOO_i} \OOO_i = \{Y\}$.
By \eqref{0_1_consequence}, the index set in this expression has an almost sure value, and so $Y$ does as well.
\end{proof}

Before proceeding to the proof of Theorem \ref{deterministic_limits_thm}, we state one other very important lemma.
Recall that $\PP(x)$ is the set of all (finite) self-avoiding paths starting at $x\in\Z^d$.

\begin{lemma} \label{abs_cont_lemma}
For any $\beta>1$, let $\mathbf{C}_j = \beta\log j + \log(2d+1)$, $j\geq1$.
\begin{enumerate}[label=\textup{(\alph*)}]
\item \label{abs_cont_lemma_a}
There is $\eps_0 = \eps_0(d)\in(0,1)$ such that the following holds for all $B\subset[0,1]$ such that $\pp(B)<\eps_0$.
Almost surely we have
\begin{subequations} \label{abs_cont_ineq}
\begin{align}
\label{abs_cont_ineq_1}
\hat\sigma_\gamma(B)\leq 2\mathbf{C}_1(\log\pp(B)^{-1})^{-1}\quad \text{for all $\gamma\in\PP(0)$, $|\gamma|$ large enough},
\end{align}
as well as
\begin{align}
\label{abs_cont_ineq_2}
\hat\sigma(B)\leq 2\mathbf{C}_1(\log\pp(B)^{-1})^{-1} \quad \text{for all $\sigma\in\RR_\infty$}.
\end{align}
\end{subequations}
\item \label{abs_cont_lemma_b}
Let $(\eps_j)_{j\geq1}$ be any sequence of numbers in $(0,1)$ such that
\eeq{ \label{eps_sequence_assumption}
\lim_{j\to\infty} \frac{\log j}{\log \eps_j^{-1}} = 0.
}
If $B_j\subset[0,1]$, $j\geq1$, are such that $\pp(B_j) < \eps_j$, then almost surely 
\eeq{ \label{abs_cont_ineq_2_eq}
\hat\sigma_\gamma(B_j) \leq \mathbf{C}_j(\log \eps_j^{-1})^{-1} \quad \text{$\forall$ $\gamma\in\PP(0)$, $|\gamma|\geq1$, and $j$ large enough}.
}
\end{enumerate}
\end{lemma}

\begin{proof}
Recall the relative entropy of Bernoulli($p$) with respect to Bernoulli($\eps$):
\eq{
h_\eps(p) \coloneqq (1-p)\log\frac{1-p}{1-\eps}+p\log\frac{p}{\eps} , \quad \eps,p\in(0,1).
}
Set $p_0 = 1 - \sqrt{2d/(2d+1)}$ so that for every $p\in(0,p_0]$ and $\eps\in(0,p]$, we have
\begin{subequations}  \label{p0_choice}
\eeq{\label{p0_choice_1}
0\geq(1-p)\log\frac{1-p}{1-\eps} \geq \log(1-p) \geq \log\sqrt{\frac{2d}{2d+1}}.
}
Next set $\eps_0$ sufficiently small that for all $\eps\in(0,\eps_0]$, we have both
\eeq{\label{e0_choice_1}
\eps \leq \mathbf{C}_1(\log\eps^{-1})^{-1} \quad \text{and} \quad
\mathbf{C}_1(\log\eps^{-1})^{-1}\leq p_0, 
}
\end{subequations}
as well as
\eeq{ \label{e0_choice_2}
\frac{\log\sqrt{(2d+1)2d}}{\log(2d+1)}&\leq\Big(1-\frac{\log \log \eps^{-1}}{\log\eps^{-1}}\Big).
}
We have made these choices so that the following holds for every $\eps\in(0,\eps_0]$:
\eeq{ \label{re_bound_0} 
h_{\eps}(\mathbf{C}_1(\log\eps^{-1})^{-1}) 
&\stackref{p0_choice}{\geq} \log\sqrt{\frac{2d}{2d+1}}+\mathbf{C}_1(\log\eps^{-1})^{-1}\log\frac{\mathbf{C}_1(\log\eps^{-1})^{-1}}{\eps} \\
&\stackrefp{p0_choice}{\geq} \log\sqrt{\frac{2d}{2d+1}} + \mathbf{C}_1\Big(1 - \frac{\log\log\eps^{-1}}{\log\eps^{-1}}\Big) \\
&\stackref{e0_choice_2}{\geq} \log(2d).
}

Now suppose that $B\subset[0,1]$ has mass $\pp(B)<\eps\leq\eps_0$.
Because $\pp$ is outer regular, we can find an open subset $\OOO$ of $[0,1]$ such that $B\subset\OOO$ and $\pp(\OOO)<\eps$.
Given any self-avoiding path $\gamma$ of length $\ell$, a standard Chernoff bound (e.g.~see \cite[pg.~24]{boucheron-lugosi-massart13}) yields the following inequality for all $p\geq\eps$:
\eeq{ \label{binomial_chernoff}
\P(\hat\sigma_\gamma(\OOO)>p)
&= \P(\mathrm{Binomial}(\ell,\pp(\OOO)) > p \ell) \\
&\leq \P(\mathrm{Binomial}(\ell,\eps) > p \ell) 
\leq \e^{-\ell h_{\eps}(p)}.
}
On $\Z^d$, there are at most $2d(2d-1)^{\ell-1}$ self-avoiding paths of length $\ell$ starting at $0$.
Taking a union bound over these paths and applying \eqref{binomial_chernoff}, we find
\eeq{ \label{bc_prep}
\P(\hat\sigma_\gamma(\OOO)>p \text{ for some $\gamma\in\PP(0)$, $|\gamma|=\ell$})
&\leq 2\e^{-\ell[h_\eps(p)-\log(2d-1)]}. 
}
It now follows from \eqref{re_bound_0} and \mbox{Borel}--\mbox{Cantelli} that with probability one, there exists $\ell_0$ large enough that
\eq{
\hat\sigma_\gamma(B)\leq\hat\sigma_\gamma(\OOO) \leq \mathbf{C}_1(\log\eps^{-1})^{-1} \quad \text{for all $\gamma\in\PP(0)$ with $|\gamma|\geq\ell_0$}.
}
Furthermore, as any $\sigma\in\RR_\infty$ arises as the weak limit of some sequence $(\sigma_{\gamma_{n_k}})_{k\geq1}$ with $\gamma_{n_k}\in\PP(0,x_k)$ and $\|x_k\|_2\to\infty$, we have
\eq{
\hat\sigma(B)\leq\hat\sigma(\OOO) \leq \liminf_{k\to\infty} \hat\sigma_{\gamma_{n_k}}(\OOO) \leq \mathbf{C}_1(\log\eps^{-1})^{-1} \quad \text{for all $\sigma\in\RR_\infty$}.
}
The proof of part \ref{abs_cont_lemma_a} is completed by our choosing $\eps$ sufficiently close to $\pp(B)$ that $\mathbf{C}_1(\log\eps^{-1})^{-1} \leq 2\mathbf{C}_1(\log\pp(B)^{-1})^{-1}$. 

For part \ref{abs_cont_lemma_b}, observe that \eqref{eps_sequence_assumption} allows us to replace \eqref{e0_choice_1} with
\eq{
\eps_j \leq \mathbf{C}_j(\log\eps_j^{-1})^{-1} \leq p_0
\quad \text{for all large $j$}.
}
We also replace \eqref{e0_choice_2} with the observation that
\eeq{ \label{e0_choice_3}
\frac{\log\sqrt{(2d+1)2d}}{\log(2d+1)}\vee\frac{\beta+1}{2\beta}&\leq\Big(1-\frac{\log \log \eps_j^{-1}}{\log\eps_j^{-1}}\Big) \quad \text{for all large $j$},
}
so that
\eeq{ \label{re_bound_1} 
h_{\eps_j}(\mathbf{C}_j(\log\eps_j^{-1})^{-1}) 
&\stackref{p0_choice}{\geq} \log\sqrt{\frac{2d}{2d+1}}+\mathbf{C}_j(\log\eps_j^{-1})^{-1}\log\frac{\mathbf{C}_j(\log\eps_j^{-1})^{-1}}{\eps_j} \\
&\stackrefp{p0_choice}{\geq} \log\sqrt{\frac{2d}{2d+1}} + \mathbf{C}_j\Big(1 - \frac{\log\log\eps_j^{-1}}{\log\eps_j^{-1}}\Big) \\
&\stackref{e0_choice_3}{\geq} \frac{\beta+1}{2}\log j + \log(2d).
}
We now have the following for all large $j$:
\eq{
&\P\big(\hat\sigma_\gamma(B_j)> \mathbf{C}_j(\log \eps_j^{-1})^{-1} \text{ for some $\gamma\in\PP(0)$, $|\gamma|\geq1$}\big) \\
&\stackref{bc_prep}{\leq} \sum_{\ell\geq1} 2\exp\big\{-\ell[h_{\eps_j}(\mathbf{C}_j(\log\eps_j^{-1})^{-1})-\log(2d-1)]\big\} \\
&\stackref{re_bound_1}{\leq}\sum_{\ell\geq1} 2\exp\Big\{-\ell\Big[\frac{\beta+1}{2}\log j + \log\frac{2d}{2d-1}\Big]\Big\} \\
&\stackrefp{re_bound_1}{=}2\frac{\exp\big\{-\frac{\beta+1}{2}\log j-\log\frac{2d}{2d-1}\big\}}{1-\exp\big\{-\frac{\beta+1}{2}\log j-\log\frac{2d}{2d-1}\Big\}} 
\leq \frac{4}{j^{(\beta+1)/2}}. 
}
Since $\beta>1$, we can apply \mbox{Borel}--\mbox{Cantelli} once more to obtain \eqref{abs_cont_ineq_2_eq}.
\end{proof}

\begin{proof}[Proof of Theorem \ref{deterministic_limits_thm}]
Our strategy is to reduce to the compact case from Lemma \ref{compactness_lemma}, since there all topological issues will be made easier.
Recalling the set $\Sigma_M\subset\Sigma$ defined in \eqref{Sigma_M_def}, we write
\eeq{\label{first_rewrite}
\RR^\bxi_\infty = \bigcup_{M=\lceil \sqrt{d}\:\rceil}^\infty (\RR^\bxi_\infty\cap\Sigma_M).
}

\renewcommand{\qedsymbol}{$\square$ (Claim)}

\begin{claim} \label{nonempty_claim}
For each $\bxi\in\S^{d-1}$ and $M\geq\sqrt{d}$, the set $\RR^\bxi_\infty\cap\Sigma_M$ is nonempty.
\end{claim}

\begin{proof}
Assume $M\geq\sqrt{d}$.
For each integer $k\geq\lceil\sqrt{d}\:\rceil$, let $\gamma_k \in \PP(0,k\bxi)$ be such that $|\gamma_k| = \|[k\bxi]\|_1$.
Since
\eq{
\langle 1,\sigma_{\gamma_k}\rangle 
= \frac{|\gamma_k|}{\|[k\bxi]\|_2} 
\leq \sqrt{d}\frac{|\gamma_k|}{\|[k\bxi]\|_1}  
=\sqrt{d} \leq M,
} 
we have $\sigma_{\gamma_k}\in\Sigma_M$.
By the compactness from Lemma \ref{compactness_lemma}, there is a subsequence $(\gamma_{k_j})_{j\geq1}$ such that $\sigma_{\gamma_{k_j}}$ converges to some $\sigma_*\in\Sigma_M$ as $j\to\infty$.
It is now clear that \eqref{W_2_0} holds, and so $\sigma_*$ belongs to $\RR^\bxi_\infty$ by Lemma \hyperref[alternative_R_def_a]{\ref*{alternative_R_def}\ref*{alternative_R_def_a}}.
Indeed, $\RR^\bxi_\infty\cap\Sigma_M$ is nonemtpy.
\end{proof}

\begin{claim} \label{closed_claim}
The sets $\RR_\infty$ and $\RR_\infty^\bxi$, $\bxi\in\S^{d-1}$, are closed in $\Sigma$.
\end{claim}

\begin{proof}
Consider any sequence $(\sigma_j)_{j\geq1}$ in $\RR_\infty$ converging to some $\sigma\in\Sigma$.
By Lemma \hyperref[alternative_R_def_b]{\ref*{alternative_R_def}\ref*{alternative_R_def_b}}, each $\sigma_j$ belongs to $\RR_\infty^{\bxi_j}$ for some $\bxi_j\in\S^{d-1}$.
By compactness of $\S^{d-1}$, we may pass to a subsequence so that $\bxi_j$ converges to some $\bxi\in\S^{d-1}$.
For each $j$, we can apply Lemma \hyperref[alternative_R_def_a]{\ref*{alternative_R_def}\ref*{alternative_R_def_a}} to identify a sequence $(x_{j,k})_{k\geq1}$ in $\Z^d$ such that \eqref{W_2_0} holds with $\sigma_*=\sigma_j$, $\bxi_*=\bxi_j$, and some choice of $\gamma_{j,k}\in\PP(0,x_{j,k})$.
By passing to a subsequence, we may assume
\eq{
\|x_{j,k}\|_2 \geq k, \quad \Big\|x_{j,k}/\|x_{j,k}\|_2-\bxi_j\Big\|_2 \leq k^{-1}, \quad \text{and} \quad
W(\sigma_j,\sigma_{\gamma_{j,k}}) \leq k^{-1}
}
for all $k\geq1$.
Setting $y_j = x_{j,j}$, we now have
\eq{
 \lim_{j\to\infty} \|y_j\|_2 = \infty,\quad \lim_{j\to\infty} \frac{y_j}{\|y_j\|_2} = \bxi, \quad \text{and} \quad \lim_{j\to\infty} W(\sigma,\sigma_{\gamma_{j,j}}) = 0.
}
By Lemma \ref{alternative_R_def}, we conclude $\sigma\in\RR_\infty^{\bxi}\subset\RR_\infty$.
Therefore, $\RR_\infty$ is closed.
When $\bxi_j=\bxi$ for every $j$, the same argument shows that $\RR_\infty^\bxi$ is closed.
\end{proof}

Let $\TT$ be the tail sigma-algebra associated to the $U_e$'s. \label{tail_1}

\begin{claim} \label{measurability_claim_1}
For every $M\geq\sqrt{d}$ and $\sigma\in\Sigma_M$, the random variables 
\eq{
W(\sigma,\RR^\bxi_\infty\cap\Sigma_M)\quad\text{and}\quad W(\sigma,\RR_\infty\cap\Sigma_M)
}
are $\TT$-measurable.
\end{claim}

\begin{proof}
Let $\sigma\in\Sigma_M$ be given.
We prove the claim just for $W(\sigma,\RR_\infty^\bxi\cap\Sigma_M)$, as the argument for $W(\sigma,\RR_\infty\cap\Sigma_M)$ is completely analogous.
The key observation is that
\eeq{ \label{limit_dist}
W(\sigma,\RR^\bxi_\infty\cap\Sigma_M) = \lim_{j\to\infty}\limsup_{\ell\to\infty}\liminf_{n\to\infty} W(\sigma,\RR^{\bxi,\ell^{-1}}_n\cap\Sigma_{M+j^{-1}}),
}
where $j,\ell\in\Z$ and $n\in\R$.
To see this, we prove inequalities in both directions.
On one hand, since $\RR^\bxi_\infty\cap\Sigma_M$ is an element of $\KK(\Sigma_M)$ by Claims \ref{nonempty_claim} and \ref{closed_claim}, we can choose $\sigma_*\in\RR^\bxi_\infty\cap\Sigma_M$ to achieve
$W(\sigma,\sigma_*) = W(\sigma,\RR^\bxi_\infty\cap\Sigma_M)$ (recall the last sentence of Section \ref{topological_preliminaries}).
By virtue of $\sigma_*$ belonging to $\RR^\bxi_\infty$, there exists a sequence 
$(x_k)_{k\geq1}$ in $\Z^d$ and $\gamma_{k}\in\PP(0,x_k)$ such that \eqref{W_2_0} holds.
Let us write $x_k = [n_k\zetab_k]$ for suitable $n_k\in\R$ and $\zetab_k\in\S^{d-1}$, where $\zetab_k\to\bxi$ as $k\to\infty$.
Since $\lim_{k\to\infty}\langle 1,\sigma_{\gamma_{k}}\rangle = \langle 1,\sigma_*\rangle\leq M$,
we have the following: For any positive integers $j$ and $\ell$, the measure $\sigma_{\gamma_k}$ belongs to $\RR^{\bxi,\ell^{-1}}_{n_k}\cap\Sigma_{M+j^{-1}}$ for all large $k$.
Consequently, one inequality for \eqref{limit_dist} can be established:
\eeq{ \label{limsup}
&\limsup_{j\to\infty}\limsup_{\ell\to\infty}\liminf_{n\to\infty} W(\sigma,\RR^{\bxi,\ell^{-1}}_n\cap\Sigma_{M+j^{-1}}) \\
&\leq \lim_{k\to\infty} W(\sigma,\sigma_{\gamma_{k}}) 
= W(\sigma,\sigma_*) 
= W(\sigma,\RR^\bxi_\infty\cap\Sigma_M).
}
For the other direction, we choose positive integers $j_1<j_2<\cdots$ such that
\eq{
&\lim_{k\to\infty}\limsup_{\ell\to\infty}\liminf_{n\to\infty} W(\sigma,\RR^{\bxi,\ell^{-1}}_n\cap\Sigma_{M+j_k^{-1}}) \\
&= \liminf_{j\to\infty}\limsup_{\ell\to\infty}\liminf_{n\to\infty} W(\sigma,\RR^{\bxi,\ell^{-1}}_n\cap\Sigma_{M+j^{-1}}).
}
Next we identify positive integers $\ell_1<\ell_2<\cdots$ such that
\eq{
&\Big|\liminf_{n\to\infty} W(\sigma,\RR_n^{\bxi,\ell_k^{-1}}\cap\Sigma_{M+j_k^{-1}})\\
&\phantom{\Big|}-\limsup_{\ell\to\infty}\liminf_{n\to\infty} W(\sigma,\RR_n^{\bxi,\ell^{-1}}\cap\Sigma_{M+j_k^{-1}})\Big| \leq k^{-1} \quad \text{for each $k\geq1$}.
}
Finally, we choose real numbers $(n_k)_{k\geq1}$ such that $n_k\to\infty$ as $k\to\infty$, and for each $k$ we have both $n_k\geq\sqrt{d}$ and 
\eq{
\Big|W(\sigma,\RR^{\bxi,{\ell_k}^{-1}}_{n_k}\cap\Sigma_{M+j_k^{-1}})
- \liminf_{n\to\infty} W(\sigma,\RR^{\bxi,{\ell_k}^{-1}}_{n}\cap\Sigma_{M+j_k^{-1}})\Big| \leq k^{-1}.
}
Because $\RR^{\bxi,{\ell_k}^{-1}}_{n_k}\cap\Sigma_{M+j_k^{-1}}$ is a finite set, one of its elements $\sigma_k$ must satisfy
\eq{
W(\sigma,\sigma_k) = W(\sigma,\RR^{\bxi,{\ell_k}^{-1}}_{n_k}\cap\Sigma_{M+j_k^{-1}}).
}
By Lemma \ref{compactness_lemma}, we may pass to a subsequence of $(\sigma_k)_{k\geq1}$ in order to assume that $\sigma_{k}$ converges to some $\sigma_*\in\Sigma_M$ as $k\to\infty$.
Because $\ell_k^{-1}\to0$ and $n_k\to\infty$, Lemma \hyperref[alternative_R_def_a]{\ref*{alternative_R_def}\ref*{alternative_R_def_a}} implies $\sigma_*\in\RR_\infty^{\bxi}$.
Together, the four previous displays now yield
\eeq{ \label{liminf}
&\liminf_{j\to\infty}\limsup_{\ell\to\infty}\liminf_{n\to\infty} W(\sigma,\RR^{\bxi,\ell^{-1}}_n\cap\Sigma_{M+j^{-1}}) \\
&=\lim_{k\to\infty} W(\sigma,\sigma_k) 
= W(\sigma,\sigma_*) \geq W(\sigma,\RR^\bxi_\infty\cap\Sigma_M).
}
Now \eqref{limsup} and \eqref{liminf} combine to give \eqref{limit_dist}, as desired.

Our next step is to observe that by definition we have
\eeq{ \label{sup_rep}
W(\sigma,\RR^{\bxi,\eps}_{n}\cap\Sigma_M) &= \inf_{\sigma_\gamma\in\RR^{\bxi,\eps}_n\cap\Sigma_{M}} W(\sigma,\sigma_\gamma). 
}
Using \eqref{W_explicit}, for any $\gamma\in\PP(0,x)$, we can write
\eq{
W(\sigma,\sigma_\gamma) = \Big|\langle 1,\sigma\rangle - \frac{|\gamma|}{\|x\|_2}\Big|
+ \sum_{i = 1}^{|\gamma|} \int_0^{i/|\gamma|} |G^{-1}(u) - U_{(i)}|\ \dd u,
}
where $G$ is the distribution function associated to $\hat\sigma$, and the order statistics of $\{U_e\}_{e\in\gamma}$ have been denoted by
$U_{(1)} \leq U_{(2)} \leq \cdots \leq U_{(|\gamma|)}$.
This is all to demonstrate that $W(\sigma,\sigma_\gamma)$ is a measurable function of $(U_e)_{e\in\gamma}$.
As the infimum in \eqref{sup_rep} is over a finite set, it follows that the random variable $W(\sigma,\RR^{\bxi,\eps}_n\cap\Sigma_{M})$ is measurable with respect to $(U_e)_{e\in E(\Z^d)}$.
In turn, \eqref{limit_dist} now shows the same to be true for $W(\sigma,\RR^\bxi_\infty\cap\Sigma_M)$, since the set $\RR_n^{\bxi,\ell^{-1}}$ only changes value at countably many (deterministic) values of $n$.

To complete the proof of the claim, we must show that $W(\sigma,\RR^\bxi_\infty\cap\Sigma_M)$ has no dependence on any finite subset of $\{U_e\}_{e\in E(\Z^d)}$.
Indeed, suppose that $(U_e')_{e\in E(\Z^d)}$ is such that $U_e' = U_e$ for all but $N$ many $e\in E(\Z^d)$, where $N<\infty$.
Let us write $\sigma_\gamma' \coloneqq \frac{1}{\|x\|_2}\sum_{e\in\gamma}\delta_{U_e'}$ for $\gamma\in\PP(0,x)$, and
define $(\RR^{\bxi,\eps}_n)'$ analogously to $\RR_n^{\bxi,\eps}$.
In light of \eqref{w_tv_1}, Lemma \ref{tv_bound_lemma} (with $I=J=N$ and $K = |\gamma|-N$) offers the bound
\eq{
W(\sigma_\gamma,\sigma_\gamma')
=W(\hat\sigma_\gamma,\hat\sigma_\gamma')  \leq \frac{N}{|\gamma|} \leq \frac{N}{\|x\|_1} \quad \text{for all $\gamma\in\PP(0,x)$}.
}
In particular, for any $\sigma\in\Sigma$, $\eps>0$, and $n,M>\sqrt{d}$, we have
\eq{
\big|W(\sigma,\RR^{\bxi,\eps}_n\cap\Sigma_M)
- W(\sigma,(\RR^{\bxi,\eps}_n)'\cap\Sigma_M)\big| \leq \frac{N}{\inf_{\zetab\in\S^{d-1}}\|[n\zetab]\|_1}
\leq \frac{N}{n-\sqrt{d}},
}
and so
\eq{
\liminf_{n\to\infty} W(\sigma,\RR^{\bxi,\eps}_n\cap\Sigma_M)
&= \liminf_{n\to\infty} W(\sigma,(\RR^{\bxi,\eps}_n)'\cap\Sigma_M).
}
It is now evident from \eqref{limit_dist} that $W(\sigma,\RR^\bxi_\infty\cap\Sigma_M)$ is measurable with respect to the tail sigma-algebra $\TT$.
The proof for $W(\sigma,\RR_\infty\cap\Sigma_M)$ is identical, except that we replace \eqref{limit_dist} by the (simpler) expression
\eq{
W(\sigma,\RR_\infty\cap\Sigma_M) = \lim_{j\to\infty}\liminf_{n\to\infty} W(\sigma,\RR_n\cap\Sigma_{M+j^{-1}}). 
}
\end{proof}

 Let $\BB$ denote the \mbox{Borel} sigma-algebra for the metric space $(\KK(\Sigma_M),\HH)$. \label{borel_3}

\begin{claim} \label{measurability_claim_2}
For $M\geq\sqrt{d}$, the maps $\Omega\to\KK(\Sigma_M)$ given by $\omega\mapsto\RR^\bxi_\infty\cap\Sigma_M$ and $\omega\mapsto\RR_\infty\cap\Sigma_M$ are $(\TT,\BB)$-measurable.
\end{claim}

\begin{proof}
As in the previous claim, let us provide the argument for only $\omega\mapsto\RR_\infty^\bxi\cap\Sigma_M$, since replacing every $\RR_\infty^\bxi$ with $\RR_\infty$ would prove the claim for the other map.
It suffices to show that for every $\AA\in\KK(\Sigma_M)$ and $\eps>0$, the event
$\{\omega\in\Omega:\, \HH(\RR^\bxi_\infty\cap\Sigma_M,\AA)<\eps\}$ belongs to $\TT$.
In other words, we wish to show that $\HH(\RR^\bxi_\infty\cap\Sigma_M,\AA)$ is a $\TT$-measurable random variable.
Indeed, let $\{\sigma_i\}_{i=1}^\infty$ be a countable dense subset of $\Sigma_M$ (recall from Lemma \ref{compactness_lemma} that $(\Sigma_M,W)$ is a compact metric space and hence separable).
We then have
\eq{
\HH(\RR^\bxi_\infty\cap\Sigma_M,\AA)
&\stackref{hausdorff_reform}{=} \sup_{\sigma\in\Sigma_M} |W(\sigma,\RR^\bxi_\infty\cap\Sigma_M)-W(\sigma,\AA)| \\
&\stackrefp{hausdorff_reform}{=}\sup_{i\geq1} |W(\sigma_i,\RR^\bxi_\infty\cap\Sigma_M)-W(\sigma_i,\AA)|.
}
By Claim \ref{measurability_claim_1}, $W(\sigma_i,\RR^\bxi_\infty\cap\Sigma_M)$ is $\TT$-measurable for each $i$.
Hence the random variable $\HH(\RR^\bxi_\infty\cap\Sigma_M,\AA)$ is $\TT$-measurable as well.
\end{proof}

We can now complete the proof of Theorem \ref{deterministic_limits_thm}.
Once again, we will consider just the direction-specific case; if one deletes $\bxi$ in all superscripts, the following argument also works in the directionless case.
By Lemma \ref{deterministic_lemma} and Claim \ref{measurability_claim_2}, for each $M\geq\sqrt{d}$ the random variable $\RR^\bxi_\infty\cap\Sigma_M$ has an almost sure value, call it $\RR^{\bxi,M}$.
Using this set in \eqref{first_rewrite}, we obtain an almost sure value for $\RR^\bxi_\infty$, namely
$\RR^\bxi \coloneqq \bigcup_{M=\lceil\sqrt{d}\:\rceil}^\infty\RR^{\bxi,M}$.
Our final objective is to show \eqref{abs_cont_quantitative_2}.
So consider any $\sigma\in\RR^\bxi$ and any measurable $B\subset[0,1]$ with $\pp(B)<\eps_0$, where $\eps_0>0$ is the constant from Lemma \hyperref[abs_cont_lemma_a]{\ref*{abs_cont_lemma}\ref*{abs_cont_lemma_a}}.
Since $\sigma$ is almost surely a member of $\RR^\bxi_\infty$, we conclude $\hat\sigma(B)\leq 2\mathbf{C}_1(\log \pp(B)^{-1})^{-1}$ from \eqref{abs_cont_ineq_2}.
To account for $B$ such that $\pp(B)\geq\eps_0$, we simply choose $C\geq 2\mathbf{C}_1$ large enough that 
$C(\log\eps_0^{-1})^{-1} \geq 1$.
\renewcommand{\qedsymbol}{$\square$}
\end{proof}

\section{Properties of sequences converging to constraint set} \label{constraint_properties}
In the previous section, we realized the putative constraint set $\RR^\bxi$ as the almost sure value of $\RR^\bxi_\infty$ from \eqref{R_infty_def}.
Moreover, Lemma \ref{alternative_R_def} gave a description of $\RR^\bxi_\infty$ in terms of limits.
Here we examine how those limits behave under deletions of a small number of edgess, linear functionals, and pushforward operations.
These results will be needed in Chapter \ref{var_form_proof} for the proof of the variational formula.

For a nonempty path $\gamma$ and a proper subset of edges $\gamma'\subsetneq\gamma$, let us write
\eeq{ \label{sigma_minus_def}
\sigma_{\gamma,\gamma'} \coloneqq \frac{1}{\|x-y\|_2}\cdot \frac{|\gamma|}{|\gamma|-|\gamma'|}\sum_{e\in\gamma\setminus\gamma'}\delta_{U_e}, \quad \gamma\in\PP(x,y),\,\gamma'\subsetneq\gamma.
}
Throughout the remainder of the section, $(x_k)_{k\geq1}$ is any sequence in $\Z^d$ such that $\|x_k\|_2\to\infty$, $\gamma_k$ is any element of $\PP(0,x_k)$, and $\gamma_k'$ is any proper subset of $\gamma_k$.

\begin{lemma} \label{modification_lemma}
If $|\gamma_k'|/|\gamma_k|\to0$ as $k\to\infty$,
then
\eq{ 
\lim_{k\to\infty} W(\sigma_{\gamma_{k}},\sigma_{\gamma_{k},\gamma_k'}) = \lim_{k\to\infty} \TV(\hat\sigma_{\gamma_k},\hat\sigma_{\gamma_k,\gamma_k'}) = 0.
}
\end{lemma}

\begin{proof}
Applying Lemma \ref{tv_bound_lemma} (with $I=|\gamma_k'|$, $J=0$, $K=|\gamma_k|-|\gamma_k'|$) results in
\eq{ 
 \TV(\hat\sigma_{\gamma_{k}},\hat\sigma_{\gamma_{k},\gamma_k'})
\leq \frac{|\gamma_k'|}{|\gamma_{k}|}
\to 0 \quad \text{as $k\to\infty$}.
}
Since $\langle 1,\sigma_{\gamma_k,\gamma_k'}\rangle = |\gamma_k|/\|x_k\|_2 = \langle 1,\sigma_{\gamma_k}\rangle$, we also have
\eq{ 
W(\sigma_{\gamma_{k}},\sigma_{\gamma_{k},\gamma_k'})
= W(\hat\sigma_{\gamma_{k}},\hat\sigma_{\gamma_{k},\gamma_k'})
&\stackref{w_tv_1}{\leq} \TV(\hat\sigma_{\gamma_{k}},\hat\sigma_{\gamma_{k},\gamma_k'}) \to 0 \quad \text{as $k\to\infty$}. 
}
\end{proof}

\begin{lemma} \label{test_fnc_lemma}
For each measurable function $f \colon [0,1]\to\R$, there is a probability-one event $\Omega_f^\Rightarrow$ on which the following statements hold whenever $|\gamma_k'|/|\gamma_k|\to0$ and $\sigma_{\gamma_k}$ converges to some $\sigma\in\RR_\infty$ as $k\to\infty$.
\begin{enumerate}[label=\textup{(\alph*)}]
\item \label{test_fnc_lemma_c}
If 
\eeq{ \label{test_fnc_condition}
\langle \e^{\alpha|f(u)|^{\beta}},\pp\rangle<\infty \quad \text{for some $\alpha>0$ and $\beta>1$},
} 
then
\eeq{ \label{test_fnc_lemma_c_eq}
\langle f,\sigma\rangle = \lim_{k\to\infty} \langle f,\sigma_{\gamma_{k},\gamma_k'}\rangle.
}
\item \label{test_fnc_lemma_d}
Let $f^\pm = \pm f\vee0$.
If $\langle \e^{\alpha f^-(u)^\beta},\pp\rangle<\infty$ for some $\alpha>0$ and $\beta>1$, then
\eeq{ \label{test_fnc_lemma_d_eq}
\langle f,\sigma\rangle \leq \liminf_{k\to\infty} \langle f,\sigma_{\gamma_{k},\gamma_k'}\rangle.
}
\end{enumerate}
\end{lemma}

\begin{proof}
For brevity, we will write $\sigma_k = \sigma_{\gamma_{k},\gamma_k'}$.
In light of the hypotheses $|\gamma_k'|/|\gamma_k|\to0$ and $\sigma_{\gamma_{k}}\Rightarrow\sigma$, Lemma \ref{modification_lemma} yields
\eeq{ \label{tv_pair_0}
\lim_{k\to\infty} \TV(\hat\sigma_{\gamma_{k}},\hat\sigma_k)
= 0,
}
as well as
\eeq{ \label{W_pair}
\lim_{k\to\infty} W(\sigma_k,\sigma) = 0.
}

Assume for now that $f$ is bounded.
Let $\mathbf{C}_1$ and $\eps_0$ be as in Lemma \ref{abs_cont_lemma}.
It is then possible to choose, for any $\eta>0$, some $\eps\in(0,\eps_0]$ so small that
\eeq{ \label{from_eta}
4\mathbf{C}_1(\log\eps^{-1})^{-1}\|f\|_\infty \leq \eta.
}
By Lusin's theorem (see, for instance, \cite[Thm.~1.10.8]{tao10}), there exists a continuous $\vphi:[0,1]\to\R$ satisfying
\eq{
\pp(\{u\in[0,1]:\vphi(u)\neq f(u)\}) < \eps \quad \text{and} \quad \|\vphi\|_\infty\leq \|f\|_\infty.
}
Let us set $B = \{u\in[0,1]:\vphi(u)\neq f(u)\}$, so that \eqref{abs_cont_ineq} applies on an almost sure event we call $ \Omega_{f,\eta}$, which depends only on the set $B$.
We also have
\eq{ 
\limsup_{k\to\infty} \hat\sigma_k(B)
\stackref{tv_pair_0}{=} \limsup_{k\to\infty}\hat\sigma_{\gamma_{k}}(B) 
\stackref{abs_cont_ineq_1}{\leq} 2\mathbf{C}_1(\log\eps^{-1})^{-1},
}
while \eqref{abs_cont_ineq_2} gives $\hat\sigma(B) \leq 2\mathbf{C}_1(\log \eps^{-1})^{-1}$.
Now, we wish to prove that the following quantity tends to $0$ as $k\to\infty$:
\eq{
|\langle f,\sigma_k-\sigma\rangle|
&= \langle f,\hat\sigma_k\rangle\langle1,\sigma_k\rangle-\langle f,\hat\sigma\rangle\langle1,\sigma\rangle| \\
&= |\langle f,\hat\sigma_k\rangle\langle1,\sigma_k-\sigma\rangle + \langle f,\hat\sigma_k-\hat\sigma\rangle\langle1,\sigma\rangle| \\
&\leq \|f\|_\infty|\langle1,\sigma_k-\sigma\rangle| + \langle f,\hat\sigma_k-\hat\sigma\rangle\langle1,\sigma\rangle.
}
The first term in the final line is easy to control; by the weak convergence $\sigma_k\Rightarrow\sigma$ shown in \eqref{W_pair}, we have $\langle1,\sigma_k-\sigma\rangle\to0$ as $k\to\infty$.
Concerning the second term, we observe that
\eq{
\limsup_{k\to\infty}|\langle f,\hat\sigma_k-\hat\sigma\rangle| 
&\stackrefp{W_pair}{=}\limsup_{k\to\infty} |\langle \vphi,\hat\sigma_k-\hat\sigma\rangle + \langle (f-\vphi)\one_B,\hat\sigma_k-\hat\sigma\rangle| \\
&\stackrefp{W_pair}{\leq}\lim_{k\to\infty} |\langle \vphi,\hat\sigma_k-\hat\sigma\rangle| + 2\|f\|_\infty\limsup_{k\to\infty}|\langle\one_B,\hat\sigma_k-\hat\sigma\rangle| \\
&\stackref{W_pair}{\leq} 0 + 4\mathbf{C}_1(\log\eps^{-1})^{-1}\|f\|_\infty
\stackref{from_eta}{\leq} \eta.
}
Choosing any decreasing sequence $(\eta_j)_{j\geq1}$ converging to zero, we can now conclude that $\langle\sigma_k,f\rangle\to\langle\sigma,f\rangle$ on the almost sure event
$\Omega_f^\Rightarrow \coloneqq \bigcap_{j=1}^\infty \Omega_{f,\eta_{j}}$.

Next consider an unbounded but nonnegative $f \colon [0,1]\to[0,\infty)$.
For each $L>0$, define the bounded function $f_L \coloneqq f \wedge L$.
Then define the event
\eeq{ \label{unbounded_almost_sure}
\Omega_f^\Rightarrow \coloneqq \bigcap_{L=1}^\infty\Omega_{f_L}^\Rightarrow,
}
so that $\P(\Omega_f^\Rightarrow) = 1$, because $\Omega_{f_L}^\Rightarrow$ was just defined in the bounded case and has probability $1$ for every $L$. 
By monotone convergence, we have
\eq{
\lim_{L\to\infty} \langle f_L,\sigma_k\rangle = \langle f,\sigma_k\rangle \quad \text{for each $k$}.
}
Furthermore, 
we have
\eeq{ \label{conclusion_b}
\langle f,\sigma\rangle = \lim_{L\to\infty}\langle f_L,\sigma\rangle =\lim_{L\to\infty} \lim_{k\to\infty} \langle f_L,\sigma_k\rangle \leq \liminf_{k\to\infty} \langle f,\sigma_k\rangle \quad \mathrm{a.s.}
}

Now suppose $f \colon [0,1]\to[0,\infty)$ is unbounded but satisfies $\langle\e^{\alpha f(u)^{\beta}},\pp\rangle<\infty$ for some $\alpha>0$ and $\beta>1$.
In particular, we have
\eq{
\infty>\langle\e^{\alpha f(u)^{\beta}},\pp\rangle 
&= \int_1^\infty \pp(\{u\in[0,1]:\, \e^{\alpha f(u)^{\beta}}\geq t\})\ \dd t \\
&= \int_0^\infty \pp(\{u\in[0,1]:\, f(u) \geq s\})\cdot\alpha\beta s^{\beta-1}\e^{\alpha s^{\beta}}\ \dd s\\
&\geq \sum_{j=0}^\infty \pp(\{u\in[0,1]:\,  f(u) \geq j+1\})\cdot\alpha \beta j^{\beta-1}\e^{\alpha j^{\beta}}.
}
If we write $B_j = \{u\in[0,1] :\, f(u) \geq j+1\}$, then the final line implies that
$\lim_{j\to\infty}\pp(B_j)\e^{\alpha j^{\beta}} = 0$.
It thus follows from Lemma \hyperref[abs_cont_lemma_b]{\ref*{abs_cont_lemma}\ref*{abs_cont_lemma_b}} that on an almost almost sure event we call $\Omega_{f,\mathrm{tail}}$ (depending only on $f$), we have
\eq{
\hat\sigma_{\gamma_{k}}(B_j)
\leq \frac{\mathbf{C}_j}{\alpha j^\beta}\quad \text{for all $k\geq1$ and $j$ large enough}.
}
Consequently,
\eeq{ \label{hat_sigma_Bj}
\hat\sigma_k(B_j) \leq \frac{|\gamma_k|}{|\gamma_k|-|\gamma_k'|}\cdot\frac{\mathbf{C}_j}{\alpha j^\beta}\quad \text{for all $k\geq1$ and $j$ large enough}.
}
In this case, we replace \eqref{unbounded_almost_sure} by
$\Omega_f^{\Rightarrow} \coloneqq \Omega_{f,\mathrm{tail}} \cap \bigcap_{L=1}^\infty \Omega_{f_L}^{\Rightarrow}$.
On this event, we have the following for all large enough $L$:
\eq{
\limsup_{k\to\infty} \langle f,\sigma_k\rangle
&\stackrefp{W_pair,hat_sigma_Bj}{\leq} \limsup_{k\to\infty}[\langle f_L,\sigma_k\rangle + \langle f - f_L,\sigma_k\rangle] \\
&\stackrefp{W_pair,hat_sigma_Bj}{\leq} \langle f_L,\sigma\rangle + \limsup_{k\to\infty}\int_L^\infty \sigma_k(\{u\in[0,1]:\, f(u)\geq t\})\ \dd t\\
&\stackrefp{W_pair,hat_sigma_Bj}{\leq} \langle f_L,\sigma\rangle +\limsup_{k\to\infty} \sum_{j=L-1}^\infty \sigma_k(B_j)  \\
&\stackref{W_pair,hat_sigma_Bj}{\leq} \langle f_L,\sigma\rangle + \langle 1,\sigma\rangle\sum_{j=L-1}^\infty\frac{\mathbf{C}_j}{\alpha j^\beta}.
}
Since $\beta>1$, we can take $L\to\infty$ and conclude
\eq{
\limsup_{k\to\infty} \langle f,\sigma_k\rangle \leq \lim_{L\to\infty} \langle f_L,\sigma\rangle = \langle f,\sigma\rangle,
}
which together with \eqref{conclusion_b} shows $\langle f,\sigma_k\rangle\to\langle f,\sigma\rangle$. 
To complete the proof of part \ref{test_fnc_lemma_c} for general $f \colon [0,1]\to\R$ satisfying $\langle\e^{\alpha|f(u)|^{\beta}},\pp\rangle<\infty$, we apply this conclusion separately to $f^-$ and $f^+$.
Meanwhile, part \ref{test_fnc_lemma_d} is obtained by applying \ref{test_fnc_lemma_c} to $f^-$ and \eqref{conclusion_b} to $f^+$.
\end{proof}

Before proceeding to our last lemma, we record the following consequence of Lemma \ref{test_fnc_lemma}.
It says that within the set $\RR$ from Theorem \ref{deterministic_limits_thm}, weak convergence is equivalent to strong convergence.

\begin{cor} \label{weak_to_strong_cor}
If $(\sigma_{j})_{j\geq1}$ is a sequence in $\RR$ converging (weakly) to $\sigma$, then
\eq{
\lim_{j\to\infty} \langle f,\sigma_j\rangle = \langle f,\sigma\rangle \quad \text{for all measurable $f:[0,1]\to\R$ satisfying \eqref{test_fnc_condition}}.
}
In particular, $\sigma_j(B)\to\sigma(B)$ for every measurable $B\subset[0,1]$.
\end{cor}

\begin{proof}
Let us assume the almost sure event that $\RR_\infty$ is equal to $\RR$, so that $\sigma_{j}\in\RR_\infty$ for each $j$, and thus $\sigma$ also belongs to $\RR_\infty$ by Claim \ref{closed_claim}.
For each $j$, there is some sequence $(n_{j,k})_{k\geq1}$, tending to infinity as $k\to\infty$, which admits $\sigma_{j,k}\in\RR_{n_{j,k}}$ such that $W(\sigma_{j,k},\sigma_{j})\to0$ as $k\to\infty$.
Now let $f:[0,1]\to\R$ be any measurable function that satisfies \eqref{test_fnc_condition}.
On the almost sure event $\Omega_f^\Rightarrow$ from Lemma \ref{test_fnc_lemma}, for every $j$ we have
$\langle f,\sigma_{j,k}\rangle\to\langle f,\sigma_{j}\rangle$  as $k\to\infty$.

We now construct a new sequence $(\rho_\ell)_{\ell\geq1}$ in $\Sigma$ as follows.
First each $\ell$, choose $j_{\ell}$ such that $j_\ell\geq\ell$ and $W(\sigma_{j_\ell},\sigma) \leq \ell^{-1}$.
Following this, select $k_{\ell}$ sufficiently large that all three of the following inequalities hold:
\eq{
W(\sigma_{j_\ell,k_\ell},\sigma_{j_\ell})\leq \ell^{-1},\quad
|\langle f,\sigma_{j_\ell,k_\ell}-\sigma_{j_\ell}\rangle| \leq \ell^{-1}, \quad \text{and} \quad
n_{j_{\ell},k_\ell} \geq \ell.
}
Setting $\rho_\ell = \sigma_{j_\ell,k_\ell}$, we have $W(\rho_\ell,\sigma)\leq 2\ell^{-1}$, and thus $W(\rho_\ell,\sigma)\to0$ as $\ell\to\infty$.
Moreover, since $\rho_\ell$ belongs to $\RR_{n_{j_\ell,k_\ell}}$ and $n_{j_\ell,k_\ell}\to\infty$ as $\ell\to\infty$,
we have $\langle f,\rho_\ell\rangle\to\langle f,\sigma\rangle$ on the event $\Omega_f^\Rightarrow$.
On the other hand, we also know $\langle f,\rho_\ell-\sigma_{j_\ell}\rangle\to0$.
We are thus left to conclude that $\langle f,\sigma_{j_\ell}\rangle\to\langle f,\sigma\rangle$, which is a deterministic statement.
Finally, notice that if $(\sigma_{j})_{j\geq1}$ were replaced by any of its subsequences, we could have fashioned the same argument to find a further subsequence $(\sigma_{j_\ell})_{\ell\geq1}$ such that  $\langle f,\sigma_{j_\ell}\rangle\to\langle f,\sigma\rangle$ as $\ell\to\infty$.
Therefore, we have proved $\langle f,\sigma_{j}\rangle\to\langle f,\sigma\rangle$ as $j\to\infty$.
\end{proof}

Our last lemma concerns the pushforwards of measures belonging to a convergent sequence.
It ensures that the resulting sequence of pushfoward measures also converges.
The statements are trivial if $\tau$ is bounded and continuous, but would be false in general without restricting to an almost sure event.
Recall the definitions of $\tau_*(\sigma)$ and $\tau_*^+(\sigma)$ from \eqref{push_def} and \eqref{pos_push_def}.

\begin{lemma} \label{pushforward_lemma}
For each measurable $\tau \colon [0,1]\to\R$,
there exists a probability-one event $\Omega_\tau^*$ on which the following holds.
Whenever $\sigma_{\gamma_k}$ converges to some $\sigma\in\RR_\infty$ as $k\to\infty$, we have
\eq{
\tau_*(\sigma_{\gamma_k}) \Rightarrow \tau_*(\sigma), \quad \hat\tau_*(\sigma_{\gamma_k})\Rightarrow\hat\tau_*(\sigma), \quad \tau_*^+(\sigma_{\gamma_k})\Rightarrow\tau_*^+(\sigma) \quad \text{as $k\to\infty$}.
}
If $\tau_*(\sigma)$ is not an atom at zero, then also
\eq{
\hat\tau_*^+(\sigma_{\gamma_k}) \Rightarrow \hat\tau_*^+(\sigma) \quad \text{as $k\to\infty$}.
}
\end{lemma}

\begin{proof}
Recall that for every $L>0$, the space $C^0([-L,L])$ of continuous functions $f \colon  [-L,L]\to\R$, is separable with respect to the norm
\eq{
\|f\|_{\infty,L} \coloneqq \sup_{t\in[-L,L]} |f(t)|.
}
Furthermore, every element of $C^0([-L,L])$ can be easily extended to a bounded, continuous function on $\R$ with the same sup-norm.
Therefore, it is possible to identify a countable collection $\{f_i\}_{i=1}^\infty$ of bounded, continuous functions on $\R$ such that $\{f_i\}_{i=1}^\infty$ is dense with respect to $\|\cdot\|_{\infty,L}$ for every $L>0$.
Let us assume $f_1\equiv1$.

Now let $\tau \colon  [0,1]\to\R$ be given, and define the sets
\eq{
A_L \coloneqq \{u\in[0,1]:|\tau(u)|\leq L\}, \quad L\geq0.
}
For brevity, let us write $\sigma_k = \sigma_{\gamma_k}$.
On the almost sure event
\eeq{ \label{tau_star_almost_sure}
\Omega_\tau^* \coloneqq \bigcap_{i=1}^\infty\bigcap_{L=0}^\infty\Big(\Omega^\Rightarrow_{(f_i\circ\tau)\one_{A_L}}\cap\Omega^\Rightarrow_{(f_i\circ\tau)\one_{A_L}\one_{\Omega\setminus A_0}}\Big),
}
where $\Omega_{(\cdot)}^\Rightarrow$ is as in Lemma \ref{test_fnc_lemma}, we have the following limit whenever $\sigma_k$ converges as $k\to\infty$ to some $\sigma\in\RR_\infty$:
\begin{subequations}
\eeq{
\lim_{k\to\infty} \langle(f_i\circ \tau)\one_{A_L},\sigma_k\rangle = \langle (f_i\circ \tau)\one_{A_L},\sigma\rangle 
\quad \text{for every $i$ and $L$}.
}
In particular, the case $i=1$ yields
 \label{choices}
\eeq{
 \label{i_1_case}
\lim_{k\to\infty} \langle\one_{A_L},\sigma_k\rangle = \langle\one_{A_L},\sigma\rangle \quad \text{for every $L$}.
}
Now consider any bounded, continuous $f \colon  \R\to\R$.
Given any $\eps>0$, we can take an integer $L$ so large that
\eeq{ \label{L_choice}
\langle\one_{\Omega\setminus A_L},\sigma\rangle \leq \frac{\eps}{\|f\|_\infty}.
}
Next we choose $i$ such that 
\eeq{\label{i_choice}
\|f - f_i\|_{\infty,L} \leq \frac{\eps}{\langle 1,\sigma\rangle}.
}
\end{subequations}
Several applications of the triangle inequality yield
\eeq{ \label{several_triangles} 
&|\langle f\circ\tau,\sigma_k\rangle - \langle f\circ\tau,\sigma\rangle| \\
&\leq \langle |f\circ\tau|\one_{\Omega\setminus A_L},\sigma_k+\sigma\rangle
+ \langle |f\circ\tau-f_i\circ\tau|\one_{A_L},\sigma_k+\sigma\rangle \\
&\hphantom{\leq}
+ |\langle (f_i\circ\tau)\one_{A_L},\sigma_k\rangle-\langle (f_i\circ\tau)\one_{A_L},\sigma\rangle| \\
&\leq \|f\|_\infty\langle\one_{\Omega\setminus A_L},\sigma_k+\sigma\rangle + \|f-f_i\|_{\infty,L}\langle1,\sigma_k+\sigma\rangle \\
&\hphantom{\leq}
+ |\langle (f_i\circ\tau)\one_{A_L},\sigma_k\rangle-\langle (f_i\circ\tau)\one_{A_L},\sigma\rangle|,
}
and so \eqref{choices} shows
\eq{
\limsup_{k\to\infty} |\langle f\circ\tau,\sigma_k\rangle-\langle f\circ\tau,\sigma\rangle| \leq 2\eps+2\eps+0 =4\eps.
}
As $\eps$ is arbitrary, we in fact have
\eq{
\lim_{k\to\infty}\langle f,\tau_*(\sigma_k)\rangle
= \lim_{k\to\infty}\langle f\circ\tau,\sigma_k\rangle
= \langle f\circ\tau,\sigma\rangle
= \langle f,\tau_*(\sigma)\rangle.
}
As these equalities hold for every bounded, continuous $f \colon \R\to\R$, we have shown that $\tau_*(\sigma_k)\Rightarrow\tau_*(\sigma)$ as $k\to\infty$ on the event $\Omega_\tau^*$ from \eqref{tau_star_almost_sure}.
Since $\langle1,\sigma_k\rangle\to\langle 1,\sigma\rangle\geq1$, it immediately follows that $\hat\tau_*(\sigma_k)\Rightarrow\hat\tau_*(\sigma)$.

Next notice that \eqref{choices} and \eqref{several_triangles} remain true if every instance of $\one_{A_L}$ and $\one_{\Omega\setminus A_L}$ is replaced by $\one_{A_L}\one_{\Omega\setminus A_0}$ and $\one_{\Omega\setminus A_L}\one_{\Omega\setminus A_0}$, respectively.
Therefore, 
\eq{
\lim_{k\to\infty}\langle f,\tau_*^+(\sigma_k)\rangle
= \lim_{k\to\infty}{\langle (f\circ\tau)\one_{\Omega\setminus A_0},\sigma_k\rangle} 
=  {\langle(f\circ\tau)\one_{\Omega\setminus A_0},\sigma\rangle} 
= \langle f,\tau_*^+(\sigma)\rangle.
}
Hence $\tau_*^+(\sigma_k)\Rightarrow\tau_*^+(\sigma)$ on the event $\Omega_\tau^*$.
Finally, if $\langle \one_{\Omega\setminus A_0},\sigma\rangle > 0$, then this string of equalities can be modified to show
\eq{
\lim_{k\to\infty}\langle f,\hat\tau_*^+(\sigma_k)\rangle
= \lim_{k\to\infty}\frac{\langle (f\circ\tau)\one_{\Omega\setminus A_0},\sigma_k\rangle}{\langle\one_{\Omega\setminus A_0},\sigma_k\rangle}
=  \frac{\langle(f\circ\tau)\one_{\Omega\setminus A_0},\sigma\rangle}{\langle\one_{\Omega\setminus A_0},\sigma\rangle}
= \langle f,\hat\tau_*^+(\sigma)\rangle.
}
That is, $\hat\tau_*^+(\sigma_k)\Rightarrow\hat\tau^+_*(\sigma)$.
\end{proof}

\chapter{Proof of Variational Formula} \label{var_form_proof}

In order to return to the original setup from Chapters \ref{var_form_sec}--\ref{modification_sec}, we fix the Radon measure $\pp$ from Chapter \ref{constraint_sec} to be Lebesgue measure $\Lambda$ on $[0,1]$.
That is, $(U_e)_{e\in E(\Z^d)}$ is a collection of i.i.d.~uniform $[0,1]$-valued random variables.
Recall that $\LL=\tau_*(\Lambda)$ is the law of $\tau(U_e)\geq0$, with distribution function $F$.

\section{Statements of most general results}
Here we state the variational formula and our most general results regarding convergence to its minimizers.
In the next section, we will use the two theorems given below to prove the results from Chapters \ref{intro} and \ref{var_form_sec}.
First we address the critical and supercritical cases.

\begin{thm} \label{var_formula_thm_super}
Assume $F(0)\geq p_\cc(\Z^d)$.

\begin{enumerate}[label=\textup{(\alph*)}]

\item \label{var_formula_thm_super_a}
For every $\bxi\in\S^{d-1}$, we have
\eeq{ \label{linear_functional_0}
\inf_{\sigma\in\RR^\bxi}\langle\tau,\sigma\rangle=\inf_{\sigma\in\RR}\langle\tau,\sigma\rangle = 0 \quad \text{for every $\bxi\in\S^{d-1}$}.
}

\item \label{var_formula_thm_super_b}
If $F(0)>p_\cc(\Z^d)$, then $\langle \tau,\sigma\rangle=0$ for some $\sigma\in\RR^\bxi$.

\item \label{var_formula_thm_super_c} 
Under the coupling \eqref{function_coupling}, there is a probability-one event on which the following statements hold.

\begin{enumerate}[label=\textup{(c\arabic*)}]

\item \label{var_formula_thm_super_c1}
For any sequence of paths $\gamma_k\in\PP(0,x_k)$ and $\gamma_k'\subset\gamma_k$ such that
\eeq{ \label{geodesic_like_1}
\lim_{k\to\infty} \|x_k\|_2 = \infty, \quad
\lim_{k\to\infty} \frac{|\gamma_k'|}{|\gamma_k|} = 0, \quad
\lim_{k\to\infty} \frac{T(\gamma_k\setminus\gamma_k')}{\|x_k\|_2} = 0,
}
we have
\eq{
\hat\nu_{\gamma_k}\Rightarrow\delta_0 \quad \text{as $k\to\infty$}.
}

\item \label{var_formula_thm_super_c2}
For any infinite geodesic $\Gamma$, we have
\eq{
\hat\nu_{\Gamma^{(\ell)}}\Rightarrow\delta_0 \quad \text{as $\ell\to\infty$}.
}

\end{enumerate}
\end{enumerate}
\end{thm}

\begin{remark} \label{length_remark}
It may be that the conclusion of Theorem \hyperref[var_formula_thm_super_b]{\ref*{var_formula_thm_super}\ref*{var_formula_thm_super_b}} holds even if $F(0)=p_\cc(\Z^d)$, 
but we have been unable to prove this in general because critical FPP does not necessarily admit geodesics with bounded linear length; recall Example \ref{critical_length}.
For \eqref{minimizers_def} to hold, it would suffice to produce paths $\gamma_k\in\PP(0,x_k)$, not necessarily geodesics, such that \eqref{geodesic_like_1} holds, $x_k/\|x_k\|_2\to\bxi$ as $k\to\infty$, and  $|\gamma_k|/\|x_k\|_2$ is uniformly bounded so that Lemma \ref{compactness_lemma} applies.
\end{remark}

In the subcritical setting, we allow our passage times to be perturbed. 
Given measurable functions $\tau:[0,1]\to[0,\infty)$ and $\tau^\pert:[0,1]\to\R$, \label{pert_fnc_2}
we assume the coupling
\eeq{ \label{function_coupling_2}
\tilde\tau_e = \tilde\tau(U_e), \quad e\in E(\Z^d), \quad \text{where} \quad \tilde\tau \coloneqq \tau+\tau^\pert.
}
Under this coupling, passage times can be expressed in the notation of \eqref{sigma_minus_def}:
\eeq{ \label{coupling_statistics}
\frac{\wt T(\gamma\setminus\gamma')}{\|x-y\|_2} = \frac{|\gamma|-|\gamma'|}{|\gamma|}\langle \tilde\tau,\sigma_{\gamma,\gamma'}\rangle \quad \text{for all $\gamma\in\PP(x,y)$, $\gamma'\subsetneq\gamma$}.
}
To avoid writing the dreadful notation $\hat{\tilde\tau}_*$ (that is, the normalization of a pushforward measure under $\tilde\tau$) more than once, we replace it with $\check\tau_*$. \label{check_def}

\begin{thm} \label{var_formula_thm_sub}
Assume $F(0)<p_\cc(\Z^d)$ and $\|\tau^\pert\|_\infty\leq \af $, where $\af >0$ is the constant from Proposition \ref{negative_thm}.

\begin{enumerate}[label=\textup{(\alph*)}]

\item \label{var_formula_thm_sub_a}
For every $\bxi\in\S^{d-1}$, the time constant from \eqref{shell_to_mu} is given by
\eeq{ \label{linear_functional}
\mu_\bxi(\tilde\tau) = \inf_{\sigma\in\RR^\bxi} \langle \tilde\tau,\sigma\rangle.
}

\item \label{var_formula_thm_sub_b} 
For every $\bxi\in\S^{d-1}$, the set of minimizers is nonempty:
\eq{
\RR^\bxi_{\tilde\tau} \coloneqq \{\sigma\in\RR^\bxi :\, \langle\tilde\tau,\sigma\rangle=\mu_\bxi(\tilde\tau)\} \neq \varnothing.
}

\item \label{var_formula_thm_sub_c}
Under the coupling \eqref{function_coupling_2}, there is a probability-one event $\Omega_{\tilde\tau}$ on which the following two statements hold for every $\bxi\in\S^{d-1}$. 

\begin{enumerate}[label=\textup{(c\arabic*)}]

\item \label{var_formula_thm_sub_c1}
For any sequence of paths $\gamma_k\in\PP(0,x_k)$ and $\gamma_k'\subset\gamma_k$ satisfying
\eeq{ \label{geodesic_like_2}
\lim_{k\to\infty} \|x_k\|_2 = \infty,\ 
\lim_{k\to\infty} \frac{x_k}{\|x_k\|_2} = \bxi,\ 
\lim_{k\to\infty} \frac{|\gamma_k'|}{|\gamma_k|} = 0,\
\lim_{k\to\infty} \frac{\wt T(\gamma_k\setminus\gamma_k')}{\|x_k\|_2} = \mu_\bxi,
}
there is a subsequence $(\gamma_{{k_j}})_{j\geq1}$ such that
\eq{ 
\sigma_{\gamma_{{k_j}}}\Rightarrow\sigma \quad \text{as $j\to\infty$}, \quad \text{for some $\sigma\in\RR_\infty^\bxi$ with $\langle\tilde\tau,\sigma\rangle = \mu_\bxi(\tilde\tau)$},
}
in which case
\eeq{ \label{to_minimizers_push}
\nu_{\gamma_{{k_j}}}\Rightarrow {\tilde\tau}_*(\sigma), \quad
\hat\nu_{\gamma_{{k_j}}}\Rightarrow \check{\tau}_*(\sigma), \quad
\nu_{\gamma_{{k_j}}}^+\Rightarrow {\tilde\tau}_*^+(\sigma), \quad
\hat\nu_{\gamma_{{k_j}}}^+\Rightarrow \check{\tau}_*^+(\sigma).
}

\item \label{var_formula_thm_sub_c2}
For any increasing sequence of nonnegative integers $(\ell_k)_{k\geq1}$ and any $\Gamma\in\Geo_\infty(\bxi)$, there is a subsequence $(\ell_{k_j})_{j\geq1}$ such that
\eq{ 
\sigma_{\Gamma^{(\ell_{k_j})}}\Rightarrow\sigma \quad \text{as $j\to\infty$}, \quad \text{for some $\sigma\in\RR_\infty^\bxi$ with $\langle\tilde\tau,\sigma\rangle = \mu_\bxi(\tilde\tau)$},
}
in which case
\eq{ 
\nu_{\Gamma^{(\ell_{k_j})}}\Rightarrow \tau_*(\sigma), \quad
\hat\nu_{\Gamma^{(\ell_{k_j})}}\Rightarrow \check{\tau}_*(\sigma), \quad
\nu_{\Gamma^{(\ell_{k_j})}}^+\Rightarrow \tau_*^+(\sigma), \quad
\hat\nu_{\Gamma^{(\ell_{k_j})}}^+\Rightarrow \check{\tau}_*^+(\sigma).
}
\end{enumerate}
\end{enumerate}
\end{thm}

\begin{remark} \label{direction_dependent_remark_2}
The almost sure event $\Omega_{\tilde\tau}$ in part \ref{var_formula_thm_sub_c} does not depend on the direction $\bxi$.
In contrast, the almost sure event $\Omega^\bxi_\tau$ from Theorem \ref{minimizers_thm}  \textit{does} depend on $\bxi$, but this is only because of Theorem \ref{deterministic_limits_thm}.
Therefore, this dependence could be removed if one shows $\P(\RR_\infty^\bxi=\RR^\bxi \text{ for all $\bxi\in\S^{d-1}$})=1$; see Remark \ref{direction_dependent_remark_1}.
\end{remark}

\section{Proofs of main theorems} \label{main_final_proof_sec}
We first assume Theorems \ref{var_formula_thm_super} and \ref{var_formula_thm_sub} and use them to establish Theorems \ref{zero_thm}, \ref{zero_thm_infty}, \ref{var_form_thm}, \ref{minimizers_thm} and Proposition \ref{slightly_negative_lemma}.

\begin{proof}[Proof of Theorem \ref{var_form_thm} and Proposition \ref{slightly_negative_lemma}]
First note that Proposition \hyperref[slightly_negative_lemma_a]{\ref*{slightly_negative_lemma}\ref*{slightly_negative_lemma_a}} was proved in Proposition \ref{negative_thm}.
Meanwhile, Theorem \hyperref[var_form_thm_a]{\ref*{var_form_thm}\ref*{var_form_thm_a}} and Proposition \hyperref[slightly_negative_lemma_c]{\ref*{slightly_negative_lemma}\ref*{slightly_negative_lemma_c}} are established by Theorems \hyperref[var_formula_thm_super_a]{\ref*{var_formula_thm_super}\ref*{var_formula_thm_super_a}} and \hyperref[var_formula_thm_sub_a]{\ref*{var_formula_thm_sub}\ref*{var_formula_thm_sub_a}}.
(In the case $F(0)\geq p_\cc(\Z^d)$, recall from Theorem \ref{time_constant_thm} or Proposition \ref{negative_thm} that $\mu_\bxi=0$ for all $\bxi\in\S^{d-1}$.)
Similarly, Theorem \hyperref[var_form_thm_b]{\ref*{var_form_thm}\ref*{var_form_thm_b}} is stated in \hyperref[var_formula_thm_super_b]{\ref*{var_formula_thm_super}\ref*{var_formula_thm_super_b}} and \hyperref[var_formula_thm_sub_b]{\ref*{var_formula_thm_sub}\ref*{var_formula_thm_sub_b}}.
Finally, Theorem \hyperref[var_form_thm_c]{\ref*{var_form_thm}\ref*{var_form_thm_c}} is given by Theorem \ref{deterministic_limits_thm}.
\end{proof}

\begin{proof}[Proof of Theorem \ref{zero_thm}]
Apply Theorem \hyperref[var_formula_thm_super_c1]{\ref*{var_formula_thm_super}\ref*{var_formula_thm_super_c1}}
to the subpath $\gamma_k\setminus\gamma_k' = \gamma_k^{(a_0(\gamma_k),a_1(\gamma_k))}$, where $a_0(\gamma_k)$ and $a_1(\gamma_k)$ are obtained from Proposition \ref{replacement_thm}.
\end{proof}

\begin{proof}[Proof of Theorem \ref{zero_thm_infty}]
This is exactly Theorem \hyperref[var_formula_thm_super_c2]{\ref*{var_formula_thm_super}\ref*{var_formula_thm_super_c2}}.
\end{proof}

\begin{proof}[Proof of Theorem \ref{minimizers_thm}]
Let $\Omega^\bxi$ be the event that $\RR_\infty^\bxi = \RR^\bxi$, which occurs with probability one by Theorem \ref{deterministic_limits_thm}.
On the event $\Omega^\bxi_\tau \coloneqq \Omega^\bxi\cap\Omega_\tau$, the conclusions of Theorems \hyperref[minimizers_thm_a]{\ref*{minimizers_thm}\ref*{minimizers_thm_a}} and \hyperref[minimizers_thm_b]{\ref*{minimizers_thm}\ref*{minimizers_thm_b}} follow from Theorems \hyperref[var_formula_thm_sub_c1]{\ref*{var_formula_thm_sub}\ref*{var_formula_thm_sub_c1}} and \hyperref[var_formula_thm_sub_c2]{\ref*{var_formula_thm_sub}\ref*{var_formula_thm_sub_c2}}, respectively, where in \ref{var_formula_thm_sub_c1} we take $\gamma_k\setminus\gamma_k' = \gamma_k^{(a_0(\gamma_k),a_1(\gamma_k))}$.
Once again, $a_0(\gamma_k)$ and $a_1(\gamma_k)$ are obtained from Proposition \ref{replacement_thm}.
\end{proof}

Now we turn to proving Theorems \ref{var_formula_thm_super} and \ref{var_formula_thm_sub}. 
We will consider Theorem \ref{var_formula_thm_sub} first, since the proof of Theorem \hyperref[var_formula_thm_super_a]{\ref*{var_formula_thm_super}\ref*{var_formula_thm_super_a}} will require \hyperref[var_formula_thm_sub_c]{\ref*{var_formula_thm_sub}\ref*{var_formula_thm_sub_b}}.

\begin{proof}[Proof of Theorem \ref{var_formula_thm_sub}]
Our first step is to prove the upper bound for \eqref{linear_functional}.
Suppose that $\tau$ is bounded, so that Proposition \hyperref[replacement_thm_d]{\ref*{replacement_thm}\ref*{replacement_thm_d}} applies.
Consider any $\sigma\in\RR^\bxi_\infty$. 
By Lemma \hyperref[alternative_R_def_a]{\ref*{alternative_R_def}\ref*{alternative_R_def_a}}, there is some sequence of paths $\gamma_{k}\in\PP(0,x_k)$ such that 
\eqref{W_2_0} holds with $\bxi_*=\bxi$ and $\sigma_*=\sigma$.
It thus follows from Lemma \hyperref[test_fnc_lemma_c]{\ref*{test_fnc_lemma}\ref*{test_fnc_lemma_c}} that
\eq{
\langle{\tilde\tau},\sigma\rangle 
\stackref{test_fnc_lemma_c_eq}{=} \lim_{k\to\infty} \langle{\tilde\tau},\sigma_{\gamma_{k}}\rangle
\stackref{coupling_statistics}{=} \lim_{k\to\infty} \frac{\wt T(\gamma_{k})}{\|x_k\|_2}
\geq \lim_{k\to\infty} \frac{\wt T(0,x_k)}{\|x_k\|_2}
\stackref{adjusted_convergence}{=} \mu_\bxi({\tilde\tau}) \quad \mathrm{a.s.}
}
We have thus shown
\eeq{ \label{inf_direction_2}
\inf_{\sigma\in\RR^\bxi_\infty}\langle{\tilde\tau},\sigma\rangle \geq \mu_\bxi({\tilde\tau}) \quad \mathrm{a.s.}
}
To recover this inequality for general $\tau$, we reduce to the bounded case.
Suppose toward a contradiction that there were some $\sigma\in\RR^\bxi_\infty$ such that $\mu_\bxi({\tilde\tau}) >  \langle{\tilde\tau},\sigma\rangle$.
Let us write $\delta$ for the difference $\mu_\bxi({\tilde\tau})-\langle{\tilde\tau},\sigma\rangle$.
By \cite[Thm.~6.9]{kesten86}, there exists $L>0$ large enough that $\mu_\bxi({\tilde\tau}) - \mu_\bxi({\tilde\tau}_L)<\delta$, where
${\tilde\tau}_L \coloneqq {\tilde\tau} \wedge L$.\footnote{The theorem referenced deals only with $\bxi=\mathbf{e}_1$ and nonnegative $\tau$, but its proof works just as well for general directions and for functions $\tilde\tau$ that are uniformly bounded from below, provided one has Lemma \ref{eta_choice_lemma}.\label{also_for_negative_fn}}
We then have
\eq{
\mu_\bxi({\tilde\tau}_L) - \langle{\tilde\tau}_L,\sigma\rangle 
> \mu_\bxi({\tilde\tau})-\delta-\langle{\tilde\tau}_L,\sigma\rangle
= \langle\tilde\tau,\sigma\rangle-\langle\tilde\tau_L,0\rangle
\geq0,
}
which contradicts \eqref{inf_direction_2} in the bounded case.
Hence \eqref{inf_direction_2} must hold in general.

To obtain the reverse inequality, we consider any sequence of paths $\gamma_k\in\PP(0,x_k)$ that satisfies \eqref{geodesic_like_2}.
(Such a sequence almost surely exists by Proposition \ref{replacement_thm}.)
By Lemma \ref{eta_choice_lemma}, the quantity $\Qvc(0;\af ,\qf )$ is almost surely finite, in which case
\eq{
|\{e\in\gamma_k:\, \tilde\tau_e \leq \af \}| \leq \qf<1  \quad \text{for all large enough $k$}.
} 
In particular, we have
\eq{
\mu_\bxi(\tilde\tau)
\stackref{geodesic_like_2}{=}\lim_{k\to\infty}\frac{\wt T(\gamma_k\setminus\gamma_k')}{\|x_k\|_2} 
&\stackrefp{geodesic_like_2}{\geq} \limsup_{k\to\infty}\frac{(1-\qf )|\gamma_k|-|\gamma_k'|}{\|x_k\|_2}\af  \\
&\stackref{geodesic_like_2}{=} \af (1-\qf )\limsup_{k\to\infty}\frac{|\gamma_k|}{\|x_k\|_2} \quad \mathrm{a.s.},
}
and so
\eq{
\limsup_{k\to\infty} \langle1,\sigma_{\gamma_k}\rangle 
= \limsup_{k\to\infty}\frac{|\gamma_k|}{\|x_k\|_2}
\leq \frac{\mu_\bxi(\tilde\tau)}{\af (1-\qf )} \quad \mathrm{a.s.}
}
Consequently, Lemma \ref{compactness_lemma} guarantees the existence 
of a subsequence $(\sigma_{\gamma_{{k_j}}})_{j\geq1}$ that converges to some $\sigma\in\Sigma$ as $j\to\infty$.
In light of \eqref{geodesic_like_2}, Lemma \hyperref[alternative_R_def_a]{\ref*{alternative_R_def}\ref*{alternative_R_def_a}} tells us that $\sigma$ belongs to $\RR^\bxi_\infty$. 
Therefore, Lemma \hyperref[test_fnc_lemma_d]{\ref*{test_fnc_lemma}\ref*{test_fnc_lemma_d}} gives
\eeq{ \label{unbounded_case}
\langle {\tilde\tau},\sigma\rangle
&\stackref{test_fnc_lemma_d_eq}{\leq} \lim_{j\to\infty}\langle{\tilde\tau},\sigma_{\gamma_{k_j},\gamma'_{k_j}}\rangle \\
&\stackrefpp{coupling_statistics}{test_fnc_lemma_d_eq}{=} \lim_{j\to\infty} \frac{|\gamma_{k_j}|}{|\gamma_{k_j}|-|\gamma_{k_j}'|}\frac{\wt T(\gamma_{k_j}\setminus\gamma_{k_j}')}{\|x_{k_j}\|_2} \\
&\stackrefpp{geodesic_like_2}{test_fnc_lemma_d_eq}{=}\lim_{j\to\infty} \frac{\wt T(\gamma_{k_j}\setminus\gamma'_{k_j})}{\|x_{k_j}\|_2}
\stackref{geodesic_like_2}{=}\mu_\bxi({\tilde\tau}) \quad \mathrm{a.s.}
}
Since $\sigma$ belongs to $\RR_\infty^\bxi$, we conclude
\eeq{ \label{inf_direction_1}
\inf_{\sigma\in\RR^\bxi_\infty} \langle\tilde\tau,\sigma\rangle\leq \mu_\bxi(\tilde\tau) \quad \mathrm{a.s.}
}
On the other hand, we have already verified \eqref{inf_direction_2}, and so \eqref{unbounded_case} forces equality: $\langle{\tilde\tau},\sigma\rangle=\mu_\bxi({\tilde\tau})$.
With this observation, we have proved claim \ref{var_formula_thm_sub_c1}, and then \eqref{to_minimizers_push} follows from Lemma \ref{pushforward_lemma} (recall from Proposition \hyperref[negative_thm_c]{\ref*{negative_thm}\ref*{negative_thm_c}} that $\mu_\bxi>0$, and so $\tilde\tau_*(\sigma)$ is not an atom at zero).
Since $\RR^\bxi_\infty$ is almost surely equal to $\RR^\bxi$  by Theorem \ref{deterministic_limits_thm}, the two inequalities \eqref{inf_direction_2} and \eqref{inf_direction_1} together prove part \ref{var_formula_thm_sub_a}.
Our extraction of the limit $\sigma$ establishes \ref{var_formula_thm_sub_b}.

All that remains to prove is claim \ref{var_formula_thm_sub_c2}.
Consider any $\Gamma\in\Geo_\infty(\bxi)$.
Label the vertices traversed by $\Gamma$ as $x_0,x_1,\dots$ so that $\Gamma^{(\ell)}\in\Geo(x_0,x_\ell)$.
Recall from \eqref{direction_implication} that $(x_{\ell}\cdot\bxi)/\|x_{\ell}-x_0\|_2\to1$ as $\ell\to\infty$, and so
Proposition \ref{addition_thm} gives
\eeq{ \label{time_constant_infinite_0}
\lim_{\ell\to\infty} \langle {\tilde\tau},\sigma_{\Gamma^{(\ell)}}\rangle
\stackref{coupling_statistics}{=} \lim_{\ell\to\infty} \frac{\wt T(\Gamma^{(\ell)})}{\|x_{\ell}-x_0\|_2}
= \lim_{\ell\to\infty} \frac{\wt T(\Gamma^{(\ell)})}{x_{\ell}\cdot\bxi}
\stackref{single_infinite}{=} \mu_\bxi \quad \mathrm{a.s.}
}
Next, by applying \mbox{Borel}--\mbox{Cantelli} with Proposition \ref{geo_length_cor}, we obtain
\eq{
\limsup_{\ell\to\infty}\frac{|\Gamma^{(\ell)}|}{\ell} \leq \mf  \quad \mathrm{a.s.}
}
We can then appeal to Lemma \ref{compactness_lemma} once more in order to deduce the following: Any increasing sequence of positive integers $(\ell_k)_{k\geq1}$ contains a subsequence $(\ell_{k_j})_{j\geq1}$ such that, for some $\sigma\in\Sigma_{\mf }$, we have
\eeq{ \label{weak_conv_infinite_0}
\sigma_{\Gamma^{(\ell_{k_j})}}\Rightarrow\sigma \quad \text{as $j\to\infty$}.
}
While it is not immediate from the definition \eqref{R_infty_def} that $\sigma$ belongs to $\RR_\infty^\bxi$, we can easily verify this membership as follows.

Consider any self-avoiding path $\pi$ that starts at $0$ and terminates upon reaching a vertex of $\Gamma$.
Say the first intersection of $\pi$ and $\Gamma$ is at the vertex $x_{\ell_*}$, and label the vertices of $\pi$ as $0=y_{\ell_*-|\pi|},y_{\ell_*-|\pi|-1},\dots,y_{\ell_*}$.
Replacing the initial segment $\Gamma^{(\ell_*)}$ by $\pi$, we obtain an infinite self-avoiding path $\bar\Gamma$ starting at $0$ and traversing the vertices $(y_\ell)_{\ell\geq\ell_*-|\pi|}$, where $y_\ell=x_\ell$ for all $\ell\geq\ell_*$.
While $\bar\Gamma$ may not be geodesic, the limits \eqref{time_constant_infinite_0} and \eqref{weak_conv_infinite_0} still hold if $\Gamma$ is replaced by $\bar\Gamma$, and $x$'s were replaced by $y$'s, since we have only changed a fixed number of edges.
Therefore, Lemma \hyperref[alternative_R_def_a]{\ref*{alternative_R_def}\ref*{alternative_R_def_a}} shows that $\sigma\in\RR^\bxi_\infty$, while Lemma \hyperref[test_fnc_lemma_d]{\ref*{test_fnc_lemma}\ref*{test_fnc_lemma_d}} yields
\eq{
\mu_\bxi \stackref{inf_direction_2}{\leq}\langle{\tilde\tau},\sigma\rangle
\stackref{weak_conv_infinite_0,test_fnc_lemma_d_eq}{\leq}\lim_{j\to\infty} \langle\tilde\tau,\sigma_{\bar\Gamma^{(\ell_{k_{j}})}}\rangle
\stackref{coupling_statistics}{=}\lim_{j\to\infty} \frac{\wt T(\bar\Gamma^{(\ell_{k_j})})}{y_{\ell_{k_j}}\cdot\bxi}
\stackref{time_constant_infinite_0}{=} \mu_\bxi \quad \mathrm{a.s.}
}
Hence $\langle{\tilde\tau},\sigma\rangle = \mu_\bxi$ (in particular, $\check{\tau}_*(\sigma)\neq\delta_0$).
Consequently, Lemma \ref{pushforward_lemma} yields the following as $j\to\infty$:
\eq{
\nu_{\Gamma^{(\ell_{{k_{j}}})}}\Rightarrow {\tilde\tau}_*(\sigma), \quad
\hat\nu_{\Gamma^{(\ell_{{k_{j}}})}}\Rightarrow \check{\tau}_*(\sigma), \quad
\nu_{\Gamma^{(\ell_{{k_{j}}})}}^+\Rightarrow {\tilde\tau}_*^+(\sigma), \quad
\hat\nu_{\Gamma^{(\ell_{{{k_j}}})}}^+\Rightarrow \check{\tau}_*^+(\sigma).
}
This last display and \eqref{weak_conv_infinite_0} together prove part \ref{var_formula_thm_sub_c2}.
\end{proof}

\begin{proof}[Proof of Theorem \ref{var_formula_thm_super}]
Given that $\tau$ is nonnegative, the lower bound for \eqref{linear_functional_0} is trivial:
$\inf_{\sigma\in\RR^\bxi}\langle\tau,\sigma\rangle \geq 0$.
In establishing the upper bound, let us first consider the case $F(0) > p_\cc(\Z^d)$.

As we did in Chapter \ref{modification_sec}, call an edge $e$ \textit{open} if $\tau_e=0$.
We write $x\leftrightarrow y$ if there exists $\gamma\in\PP(x,y)$ containing only open edges, and define $D_\oo(x,y)$ to be the minimum length of such a path $\gamma$.
Because $F(0)>p_\cc(\Z^d)$,
the subgraph of $\Z^d$ induced by the open edges has a unique infinite component, which we call $\OO$.
Let $E(\OO)$ denote the set of edges in $\OO$.
By \cite[Thm.~1.1]{antal-pisztora96}, there are positive constants $C_1$, $C_2$, $c_3$, whose values depend only on $F(0)$, such that
\eeq{ \label{ap96}
\P(x\leftrightarrow y,\, D_\oo(x,y)\geq C_1\|x-y\|_2) \leq C_2\e^{-c_3\|x-y\|_2} \quad \text{$\forall$ $x,y\in\Z^d$}.
}
Now consider any $\bxi\in\S^{d-1}$, and let $\zetab_k\in\S^{d-1}\cap\Q^d$ be such that $\zetab_k\to\bxi$ as $k\to\infty$.
For each $k$, let $b_k$ be a positive integer such that $b_k\zetab_k\in\Z^d$, and let us assume 
\eq{
b_{k+1} \geq b_{k}\quad \text{and} \quad
\e^{-c_3b_k} \leq (k+1)^{-2} \quad \text{for all $k\geq1$}.
}
For $\ell\geq1$, define $y_{k,\ell} = \ell b_k\zetab_k$ and note that $\|y_{k,\ell}\|_2 = \ell b_k$.
For any $x_0\in\Z^d$, the estimate \eqref{ap96} gives
\eq{
&\sum_{k,\ell=1}^\infty \P(y_{k,\ell}\leftrightarrow x_0, D_\oo(x_0,y_{k,\ell})\geq C_1\|x_0-y_{k,\ell}\|_2)
\leq C_2\sum_{k,\ell=1}^\infty\e^{-c_3(\|y_{k,\ell}\|_2-\|x_0\|_2)} \\
&\qquad= C_2\e^{c_3\|x_0\|_2}\sum_{k=1}^\infty \frac{\e^{-c_3b_k}}{1-\e^{-c_3b_k}} 
\leq \frac{4C_2\e^{c_3\|x_0\|_2}}{3}\sum_{k=1}^\infty(k+1)^{-2}<\infty.
}
By Borel--Cantelli, we conclude that for any $x_0$ and all $k$ large enough (depending on $x_0$), we have
\eeq{ \label{good_Do}
D_\oo(x_0,y_{k,\ell})< C_1\|x_0-y_{k,\ell}\|_2 \quad\text{whenever $x_0,y_{k,\ell}\in\OO$}.
}
At the same time, by ergodicity with respect to translations of $\Z^d$, the following is true almost surely:
For each $k$, there are infinitely many $\ell\geq1$ such that $y_{k,\ell}\in\OO$.
Therefore, we can inductively define a sequence $(x_k)_{k\geq1}$ in $\Z^d$ as follows.
Let $\ell_1 \geq 1$ be such that $y_{1,\ell_1}\in\OO$, and set $x_1 = y_{1,\ell_1}$.
For $k\geq2$, choose $\ell_k$ such that $\ell_k>\ell_{k-1}$ and $y_{k,\ell_k}\in\OO$; then set $x_k = y_{k,\ell_k}$.
In particular, we have $\|x_k\|_2=\ell_kb_k\geq k$ and $x_k/\|x_k\|_2 = \zetab_k\to\bxi$ as $k\to\infty$.

We next define a corresponding sequence of paths $\gamma_k\in\PP(0,x_k)$ as follows.
Choose any $x_0\in\OO$ and any $\gamma_0\in\PP(0,x_0)$.
Continuing from $x_0$, append to $\gamma_0$ any path in $\OO$ that reaches $x_k$ in minimal length; after removing any loops, call the resulting path $\gamma_k\in\PP(0,x_k)$.
By construction, we have
\eq{
\limsup_{k\to\infty} \frac{|\gamma_k|}{\|x_k\|_2} \leq 
\limsup_{k\to\infty} \frac{|\gamma_0| + D_\oo(x_0,x_k)}{\|x_k\|_2}
\stackref{good_Do}{\leq} C_1.
}
Therefore, by Lemma \ref{compactness_lemma}, there is some subsequence $(k_j)_{j\geq1}$ such that $\sigma_{\gamma_{k_j}}$ converges to some $\sigma\in\Sigma_{C_1}$ as $j\to\infty$.
From Lemma \hyperref[alternative_R_def_a]{\ref*{alternative_R_def}\ref*{alternative_R_def_a}}, it is clear that $\sigma\in\RR_\infty^\bxi$.
On the other hand, we have
\eq{
\limsup_{k\to\infty}\frac{|\{e\in\gamma_k : \tau_e > 0\}|}{\|x_k\|_2}
\leq \limsup_{k\to\infty} \frac{|\gamma_0|}{\|x_k\|_2} = 0,
}
and so $\hat\tau_*(\sigma_{\gamma_k})\Rightarrow\delta_0$ as $k\to\infty$.
By Lemma \ref{pushforward_lemma}, it follows that $\tau_*(\sigma)=\delta_0$.
In particular, we have $\langle\tau,\sigma\rangle=0$.
Since $\RR_\infty^\bxi$ is almost surely equal to $\RR^\bxi$ by Theorem \ref{deterministic_limits_thm}, we have completed the proof of parts \ref{var_formula_thm_super_a} and \ref{var_formula_thm_super_b} in the supercritical case.

Now suppose $F(0)=p_\cc(\Z^d)$, and again consider any $\bxi\in\S^{d-1}$.
Set $B= \{u \in [0,1] :\, \tau(u) = 0\}$ so that $\Lambda(B) = p_\cc(\Z^d) > 0$.
Given any $\eps\in(0,p_\cc(\Z^d))$, choose a subset $B_\eps\subset B$ such that $\Lambda(B_\eps) = \eps$ (e.g.~see \cite{hourieh14}).
Now write $\tau^{(\eps)} = \tau + \eps\one_{B_\eps}$ so that $\P(\tau^{(\eps)}=0) = p_\cc(\Z^d)-\eps$.
By Theorem \hyperref[var_formula_thm_sub_b]{\ref*{var_formula_thm_sub}\ref*{var_formula_thm_sub_b}}, there is $\sigma^{(\eps)}\in\RR^\bxi$ such that $\langle \sigma^{(\eps)},\tau^{(\eps)}\rangle= \mu_\bxi(\tau^{(\eps)})$.
But of course $\langle\sigma,\tau\rangle\leq\langle\sigma,\tau^{(\eps)}\rangle$ for every $\sigma\in\RR^\bxi$, and so
\eq{
\inf_{\sigma\in\RR^\bxi}\langle \sigma,\tau\rangle \leq \liminf_{\eps\searrow0}\langle\sigma^{(\eps)},\tau^{(\eps)}\rangle =\liminf_{\eps\searrow0} \mu_\bxi(\tau^{(\eps)}).
}
The proof of part \ref{var_formula_thm_super_a} is completed by invoking \cite[Thm.~6.9]{kesten86}, which shows that $\mu_\bxi(\tau^{(\eps)})\to\mu_\bxi(\tau)=0$ as $\eps\to0$.

Finally, we turn our attention to part \ref{var_formula_thm_super_c}.
Consider any sequence of paths $\gamma_k\in\Geo(0,x_k)$ that satisfies \eqref{geodesic_like_1}.
Mimicking the notation $\sigma_{\gamma,\gamma'}$ from \eqref{sigma_minus_def}, we define 
\eeq{ \label{nu_minus_def} 
\hat\nu_{\gamma,\gamma'} \coloneqq \frac{1}{|\gamma|-|\gamma'|}\sum_{e\in\gamma\setminus\gamma'} \delta_{\tau_e}, \quad \gamma'\subsetneq\gamma.
}
Now the last equation in \eqref{geodesic_like_1} reads as
\eeq{\label{first_pass}
\lim_{k\to\infty} \frac{|\gamma_k|-|\gamma_k'|}{\|x_k\|_2} \langle t, \hat\nu_{\gamma_k,\gamma_k'}\rangle = 0.
}
Since $|\gamma_k'|/|\gamma_k|\to0$ as $k\to\infty$, we trivially have
\eq{
\liminf_{k\to\infty}\frac{|\gamma_k|-|\gamma_k'|}{\|x_k\|_2} = 
\liminf_{k\to\infty}\frac{|\gamma_k|}{\|x_k\|_2} 
\geq \liminf_{k\to\infty}\frac{|\gamma_k|}{\|x_k\|_1}\geq 1,
}
which forces the following from \eqref{first_pass}:
\eq{ 
\lim_{k\to\infty} \langle t, \hat\nu_{\gamma_k,\gamma_k'}\rangle = 0.
}
It can now be easily deduced from Markov's inequality that $\hat\nu_{\gamma_k,\gamma_k'}\Rightarrow\delta_0$.
At the same time, 
Lemma \ref{tv_bound_lemma} (with $I=|\gamma_k'|$, $J=0$, $K=|\gamma_k|-|\gamma_k'|$) gives
\eq{
\TV(\hat\nu_{\gamma_k},\hat\nu_{\gamma_k,\gamma_k'})
\leq |\gamma_k'|/|\gamma_k|.
}
Since $|\gamma_k'|/|\gamma_k|\to0$ and $\hat\nu_{\gamma_k,\gamma_k'}\Rightarrow\delta_0$, we conclude
$\hat\nu_{\gamma_k}\Rightarrow\delta_0$. 

A similar argument goes through for any infinite geodesic $\Gamma$.
Label the vertices traversed by $\Gamma$ (in the order traversed) as $x_0,x_1,\dots$ so that $\Gamma^{(\ell)}\in\Geo(x_0,x_\ell)$, and take $n_\ell$ such that $x_\ell = x_0+[n_\ell\bxi_\ell]$ for some $\bxi_\ell\in\S^{d-1}$.
By Proposition \ref{addition_thm}, we have
\eeq{ \label{time_constant_infinite}
0=\lim_{\ell\to\infty}\frac{T(\Gamma^{(\ell)})}{n_\ell} = \lim_{\ell\to\infty} \frac{\ell}{n_\ell}\langle t,\hat\nu_{\Gamma^{(\ell)}}\rangle \quad \mathrm{a.s.}
}
Since $\|x_\ell-x_{\ell-1}\|_2=1$ for all $\ell\geq0$, we clearly have $\liminf_{\ell\to\infty} \ell/n_\ell \geq 1$, and so \eqref{time_constant_infinite} forces
$\lim_{\ell\to\infty} \langle t,\hat\nu_{\Gamma^{(\ell)}}\rangle = 0$.
Now Markov's inequality yields $\hat\nu_{\Gamma^{(\ell)}}\Rightarrow\delta_0$.
\end{proof}

\section{Remaining proofs} \label{length_lemma_proof}
Here we tie up a few loose ends: verification of Theorem \ref{zero_length_thm}, Lemma \ref{good_to_length}, and Theorem \ref{derivative_ineq_thm}.
For $x\in\Z^d$ and $S\subset\Z^d$, let us write $\Geo(x,S)$ for the set of self-avoiding paths $\gamma\in\PP(x,y)$ such that $y\in S$ and $T(\gamma) = \inf_{z\in S}T(x,z)$.
Recall from \eqref{box_defs} that $\Bbf_r(x)$ is the box of radius $r$ centered at $x\in\Z^d$, and $\partial \Bbf_r(x)$ denotes its boundary.

\begin{proof}[Proof of Theorem \ref{zero_length_thm}]
Let us first assume $F(0) > p_\cc(\Z^d)$.
As before, call an edge $e$ \textit{open} if $\tau_e=0$.
The subgraph of $\Z^d$ induced by the open edges has a unique infinite component, which we call $\OO$.
Let $E(\OO)$ denote the set of edges in $\OO$.
By \cite[Thm.~1]{zhang95}, there exists a constant $c = c(\LL,d)>0$ such that for any positive integer $r$, we have
\eeq{ \label{zhang_estimate}
\P(\exists\, \gamma\in\Geo(x,\partial \Bbf_r(x)) :\, \gamma\cap E(\OO)=\varnothing) \leq \e^{-cr^{1/d}}.
}
Observe that if every $\gamma_1\in \Geo(x,\partial \Bbf_{r_1}(x))$ and every $\gamma_2\in\Geo(y,\partial \Bbf_{r_2}(y))$ intersects $\OO$, and $\|x-y\|_\infty\geq r_1+r_2$, then every $\gamma\in\Geo(x,y)$ must remain entirely in $\OO$ between $\partial \Bbf_{r_1}(x)$ and $\partial \Bbf_{r_2}(x)$.
In this case, we have $|\gamma\setminus E(\OO)| \leq (2r_1+1)^d+(2r_2+1)^d$.
Of course, $\gamma\setminus E(\OO)$ contains as a subset all the edges $e\in\gamma$ such that $\tau_e>0$, and so
$|\gamma_+|\leq (2r_1+1)^d+(2r_2+1)^d$.

To utilize these observations, we take any $\eps\in(0,1/d)$.
Invoking \eqref{zhang_estimate} with a union bound over $\partial \Bbf_r(x)\cup\{0\}$, we have
\eq{
\P\Big(\bigcup_{x\in\partial \Bbf_{r}(0)\cup\{0\}} \{\exists\, \gamma\in\Geo(x,\partial \Bbf_{\lceil r^\eps\rceil}(x)) :\, \gamma\cap E(\OO)=\varnothing\}\Big)\leq Cr^{d-1}\e^{-cr^{\eps/d}},
}
where $C$ is a constant depending only $d$.
Applying Borel--Cantelli, we conclude that on an almost sure event, the event displayed above occurs for only finitely many $r$.
By the discussion in the previous paragraph, it follows that whenever $\|x\|_\infty$ is sufficiently large, we have
\eq{
|\gamma_+| \leq 2(2\lceil \|x\|_\infty^\eps\rceil+1)^d \quad \text{for every $\gamma\in\Geo(0,x)$}.
}
By our choice of $\eps$, the upper bound seen here is $o(\|x\|_\infty)$.
Therefore, we do have the claimed \eqref{zero_length_thm_eq}.

Now we turn to the case $F(0)=F(h)=p_\cc(\Z^d)$ for some $h>0$.
Consider any $x_k\in\mathbf{S}_{n_k}$ (recall the definition of $\mathbf{S}_n$ from \eqref{Sn_def}) and any $\gamma_k\in\Geo(0,x_k)$.
Let $A_0$, $(A_n)_{n\in[1,\infty)}$, $a_0(\cdot)$, and $a_1(\cdot)$ be as in Proposition \ref{replacement_thm}.
That is, $[a_0(\gamma_k)+a_1(\gamma_k)]/n_k\to0$ as $k\to\infty$ by \hyperref[replacement_thm_a]{\ref*{replacement_thm}\ref*{replacement_thm_a}}, while \hyperref[replacement_thm_b]{\ref*{replacement_thm}\ref*{replacement_thm_b}} gives
\eq{
\lim_{k\to\infty} \frac{T(\gamma_k^{(a_0(\gamma_k),a_1(\gamma_k))})}{n_k}=0 \quad \mathrm{a.s.}
}
Since $\tau_e\geq h$ whenever $\tau_e>0$, the following inequality is trivial:
\eq{
T(\gamma_k^{(a_0(\gamma_k),a_1(\gamma_k))}) \geq h(|\gamma_k|_+-a_0(\gamma_k)-a_1(\gamma_k)).
}
We conclude that
\eq{
0 \geq h\limsup_{k\to\infty} \frac{|\gamma_k|_+-a_0(\gamma_k)-a_1(\gamma_k)}{n_k}
= h\limsup_{k\to\infty} \frac{|\gamma_k|_+}{\|x_k\|_2}.
}
Since $h>0$, we again conclude that \eqref{zero_length_thm_eq} holds.
\end{proof}

\begin{proof}[Proof of Lemma \ref{good_to_length}]
Assume a coupling of the form \eqref{function_coupling_2}, as well as the occurrence of the almost sure events from Proposition \ref{replacement_thm} and Theorem \hyperref[var_formula_thm_sub_c]{\ref*{var_formula_thm_sub}\ref*{var_formula_thm_sub_c}}.
Consider any sequence of $\gamma_{n_k}\in\Geo(0,n_k\bxi)$ such that $\hat\nu_{\gamma_{n_k}}$ converges weakly to some $\hat\nu$ as $k\to\infty$.
By Proposition \ref{replacement_thm} and Theorem \hyperref[var_formula_thm_sub_c1]{\ref*{var_formula_thm_sub}\ref*{var_formula_thm_sub_c1}}, every subsequence of $(\gamma_{n_k})_{k\geq1}$ admits a further subsequence $(\gamma_{n_{k_j}})_{j\geq1}$ such that $\sigma_{\gamma_{n_{k_j}}}\Rightarrow\sigma$ as $j\to\infty$, for some $\sigma\in\RR^\bxi_\infty$ satisfying $\langle\tau,\sigma\rangle=\mu_\bxi(\tau)$.
We deduce that
\eq{
\lim_{j\to\infty} \frac{|\gamma_{n_{k_j}}|}{n_{k_j}} 
=\lim_{j\to\infty} \frac{|\gamma_{n_{k_j}}|}{\|[n_{k_j}\bxi]\|_2} 
&\stackrefp{to_minimizers_push}{=} \lim_{j\to\infty} \langle 1,\nu_{\gamma_{n_{k_j}}}\rangle\\
&\stackref{to_minimizers_push}{=} \langle 1,\tau_*(\sigma)\rangle
= \frac{\langle t,\tau_*(\sigma)\rangle}{\langle t,\hat\tau_*(\sigma)\rangle}
= \frac{\langle \tau,\sigma\rangle}{\langle t,\hat\nu\rangle}
= \frac{\mu_\bxi(\tau)}{\langle t,\hat\nu\rangle}.
}
As this holds for an arbitrary subsequence of $(\gamma_{n_k})_{k\geq1}$, we conclude
\eq{
\lim_{k\to\infty} \frac{|\gamma_{n_k}|}{n_k} = \frac{\mu_\bxi(\tau)}{\langle t,\hat\nu\rangle}.
}
The second part of the lemma is proved in exactly the same way, replacing all instances of $\hat\nu_\gamma$, $\hat\nu$, $\tau_*(\sigma)$, $|\gamma|$ with $\hat\nu_\gamma^+$, $\hat\nu^+$, $\tau_*^+(\sigma)$, $|\gamma|_+$, respectively.
\end{proof}

\begin{remark} \label{good_to_length_generality_remark}
The conclusions of Lemma \ref{good_to_length} also hold for any sequence satisfying \eqref{geodesic_like_2}.
We restricted our attention to geodesics because the primary purpose of Lemma \ref{good_to_length} is to realize \eqref{good_to_length_implication}.
\end{remark}

\begin{proof}[Proof of Theorem \ref{derivative_ineq_thm}]
First we address the subcritical regime in which $F(0)< p_\cc(\Z^d)$.
Consider any sequence of $\gamma_n\in\Geo(0,n\bxi)$.
Recall that
$|\gamma_n|/\|[n\bxi]\|_2$ can be expressed as $\langle 1,\sigma_{\gamma_n}\rangle$.
Since $\bxi$ is a unit vector, we also have $\|[n\bxi]\|_2/n\to1$ as $n\to\infty$.
Finally, recall the notation $\RR_\tau^\bxi = \{\sigma\in\RR^\bxi : \langle\tau,\sigma\rangle=\mu_\bxi(\tau)\}$.
On the almost sure event from Theorem \ref{minimizers_thm}, specifically \eqref{optimizer_weak_limit}, we can use an argument of subsequences to obtain
\eq{
\inf_{\sigma\in\RR^\bxi_\tau}\langle1,\sigma\rangle \leq \liminf_{n\to\infty} \frac{|\gamma_n|}{n} \leq \limsup_{n\to\infty} \frac{|\gamma_n|}{n} \leq \sup_{\sigma\in\RR^\bxi_\tau}\langle1,\sigma\rangle.
}
Now \eqref{derivative_ineq_thm} follows from Lemma \hyperref[one_d_lemma_c]{\ref*{one_d_lemma}\ref*{one_d_lemma_c}} with $\psi\equiv 1$.

Next consider the case when $F(0)\geq p_\cc(\Z^d)$. 
By Proposition \ref{slightly_negative_lemma_2}, we formally have $\mu_\bxi(\tau-h)=-\infty$ for all $h>0$, which leads to $D^-\mu_\bxi(\tau+h)\big|_{h=0}=\infty$.
This means the left-derivative inequality \eqref{derivative_ineq_b} is trivial.
Regarding the right derivative, if
$\liminf_{n\to\infty} n^{-1}\barbelow{N}_n^\bxi = \infty$ with probability one
(which includes the possibility that $\Geo(0,n\bxi)$ is empty for all large $n$),
then \eqref{derivative_ineq_a} is also trivial.
Otherwise, we can (with positive probability)  identify a sequence of $\gamma_{n_k}\in\PP(0,n_k\bxi)$, where $n_k\to\infty$ as $k\to\infty$, such that
\eq{
\lim_{k\to\infty} \frac{|\gamma_{n_k}|}{n_k} = \liminf_{n\to\infty}\frac{1}{n}\barbelow{N}_n^\bxi< \infty.
}
Therefore, Lemma \ref{compactness_lemma} allows us to pass to a further subsequence and assume  $\sigma_{\gamma_{n_k}}$ converges to some $\sigma_*\in\RR^\bxi$, which necessarily satisfies $\langle\tau,\sigma_*\rangle=0$ by Theorem \ref{zero_thm}.
In this case, we have
\eq{
\inf_{\sigma\in\RR^\bxi_\tau}\langle 1,\sigma\rangle 
\leq\langle1,\sigma_*\rangle
=\lim_{k\to\infty} \frac{|\gamma_{n_k}|}{n_k}
= \liminf_{n\to\infty} \frac{1}{n}\barbelow{N}_n^\bxi.
}
Once more, Lemma \ref{one_d_lemma} completes the proof, in this case by \ref{derivative_ineq_minimizers_a} with $h=0$.
\end{proof}

\chapter{Proof of Empirical Measure Convergence in Tree Case} \label{tree_proof}
In this section, we prove Theorem \ref{tree_thm}.
As in Chapter \ref{var_form_sec}, let $(U_x)_{x\in\T_d}$ be a family of i.i.d.~$[0,1]$-valued random variables \label{U_x_def_2}
distributed according to a Radon probability measure $\pp$.
Assume these variables are supported on a complete probability space $(\Omega,\FF,\P)$.
Given a measurable function $\tau\colon [0,1]\to\R$ whose negative part $\tau^- = -\tau\vee 0$ satisfies
\eeq{ \label{negative_side_moment_assumption}
\langle \e^{\alpha\tau^-(u)^\beta},\pp\rangle < \infty \quad \text{for some $\alpha>0$ and $\beta>1$},
}
we couple the FPP model \eqref{tree_FPP_model} to $(U_x)_{x\in\T_d}$ by $\tau_x = \tau(U_x)$ (recall from Section \ref{coupling_env_sub} that when $\pp$ is equal to Lebesgue measure $\Lambda$, any law for $\tau_x$ can be realized in this way).
Therefore, the empirical measures
\eq{
\hat\nu_x = \frac{1}{|x|}\sum_{0<y\leq x} \delta_{\tau_y} \quad \text{and} \quad
 \hat\sigma_x = \frac{1}{|x|}\sum_{0<y\leq x}\delta_{U_y}
}
are related via the pushforward operation: $\hat\nu_x = \tau_*(\hat\sigma_x)$.
Once we define
\eeq{ \label{prelimit_tree} 
\RR_n^{\T_d} \coloneqq \{\hat\sigma_x :\, x\in\T_d, |x|=n\}\subset\hat\Sigma, \quad n\in\{1,2,\dots\}.
}
all the arguments of Chapters \ref{constraint_sec} and \ref{var_form_proof} remain valid in the tree case.
The only notable difference is that the length of a path is no longer relevant and so the proofs actually become simpler.
Moreover, Propositions \ref{replacement_thm} and \ref{addition_thm} are not needed, as Theorem \ref{tree_time_constant_thm} gives almost sure convergence to the time constant (rather than just in probability).
In summary, we may use as starting points the following results:
\begin{itemize}
\item The set
$\RR_\infty^{\T_d} \coloneqq \big\{\hat\sigma\in\hat\Sigma :\, \liminf_{n\to\infty}  W(\hat\sigma,\RR_n^{\T_d})=0\big\}$
\label{R_infty_tree_def}
has an almost sure value $\RR^{\T_d} = \RR^{\T_d}(\pp)$, every $\hat\sigma\in\RR^{\T_d}$ is absolutely continuous with respect to $\pp$, and Corollary \ref{weak_to_strong_cor} applies.
Furthermore, $\RR^{\T_d}$ is compact because it is a closed subset of $\hat\Sigma$.
(In the lattice case, $\RR$ was closed in $\Sigma$, which is isometric to $\hat \Sigma\times[1,\infty)$ and thus non-compact.)
\item When $\pp=\Lambda$, the time constant $\mu_{\T_d}$ from \eqref{tree_time_constant_def} is given by
\eq{
\mu_{\T_d}(\tau) = \inf_{\hat\sigma\in\RR^{\T_d}} \langle \tau,\sigma\rangle,
}
and the set of minimizers $\RR_\tau^{\T_d}$ is \textit{always} nonempty (cf.~Remark \ref{length_remark}).
\item Almost surely, every sequence of $x_{n_k} \in \Geo_{n_k}$ (or more generally, any sequence with $T(x_{n_k})/n_k\to\mu_{\T_d}(\tau)$) admits a subsequence $(x_{n_{k_\ell}})_{\ell\geq1}$ such that
$\hat\sigma_{x_{n_{k_\ell}}}$ converges weakly to some $\hat\sigma\in\RR_\tau^{\T_d}$ as $\ell\to\infty$,
in which case
$\hat\nu_{x_{n_{k_\ell}}} \Rightarrow \tau_*(\hat\sigma)$.
The inequality in \eqref{unbounded_case} relies on Lemma \ref{test_fnc_lemma}, which explains the moment assumption \eqref{negative_side_moment_assumption} made in Theorem \ref{tree_thm}.
\end{itemize}
Given these facts, we need only show the following two propositions.
Recall that $\mathfrak{b}$ denotes the essential infimum of the random variable $\tau_x$.

\begin{prop} \label{tree_prop_1}
The set $\RR^{\T_d}(\pp)$ is equal to $\{\hat\sigma\in\hat\Sigma :\, \KL{\hat\sigma}{\pp} \leq \log d\}$.
\end{prop}

\begin{prop} \label{tree_prop_2}
When $\pp = \Lambda$, the pushforward measure $\tau_*(\hat\sigma)$ is the same for every $\hat\sigma\in\RR_\tau^{\T_d}$.
More precisely, $\tau_*(\hat\sigma) = \delta_{\mathfrak{b}}$ if $\P(\tau_x=\mathfrak{b})\geq 1/d$; otherwise $\RR_\tau^{\T_d}=\{\sigma_\star\}$, where $\sigma_\star$ is the unique solution to \eqref{minimum_RE}.
\end{prop}

Concerning the first of these two results, the following argument was suggested by L. Addario-Berry.

\begin{proof}[Proof of Proposition \ref{tree_prop_1}]
Since $\RR_\infty^{\T_d}$ is almost surely equal to $\RR^{\T_d}$, it suffices to show the following equivalence for every $\hat\sigma\in\hat\Sigma$:
\eq{
\P(\hat\sigma\in\RR^{\T_d}_\infty) = 1 \quad \iff \quad \KL{\hat\sigma}{\pp}\leq\log d.
}
By a classical result of Donsker and Varadhan \cite[Thm.~4.5]{donsker-varadhan76}, we have the following large deviations principle commonly known as Sanov's theroem.
For every closed set $A\subset\hat\Sigma$,
\eeq{ \label{closed_sanov}
\limsup_{|x|\to\infty} \frac{1}{|x|}\log\P(\hat\sigma_x\in A) \leq -\inf_{\hat\sigma\in A} \KL{\hat\sigma}{\pp},
}
and for every open set $B\subset\hat\Sigma$,
\eeq{ \label{open_sanov}
\liminf_{|x|\to\infty} \frac{1}{|x|}\log\P(\hat\sigma_x\in B) \geq -\inf_{\hat\sigma\in B} \KL{\hat\sigma}{\pp}.
}
Now suppose $\KL{\hat\sigma}{\pp}>\log d$, and take $\eps>0$ such that $\KL{\hat\sigma}{\pp}\geq\log d+3\eps$.
Since $\KL{\cdot}{\pp}$ is lower semi-continuous with respect to weak convergence, there is $\delta>0$ such that
\eq{
\KL{\hat\rho}{\pp} \geq \log d + 2\eps \quad \text{whenever} \quad \text{$ W(\hat \sigma,\hat\rho) \leq \delta$}.
}
Applying \eqref{closed_sanov}, we find
\eq{
\limsup_{|x|\to\infty}\frac{1}{|x|}\log \P( W(\hat\sigma,\hat\sigma_x)\leq\delta) \leq -(\log d + 2\eps),
}
which is enough to imply that
\eq{
\limsup_{|x|\to\infty} \frac{\P(W(\hat\sigma,\hat\sigma_x)\leq\delta)}{\e^{-(\log d+\eps)|x|}}= 0.
}
This last observation, together with a union bound over $x\in\T_d$ belonging to the $n^\text{th}$ generation, shows
\eq{
\limsup_{n\to\infty} \frac{\P( W(\hat\sigma,\RR^{\T_d}_n)\leq\delta)}{\e^{-n\eps}} \leq
\limsup_{|x|\to\infty} \frac{d^{|x|}\cdot\P( W(\hat\sigma,\hat\sigma_x)\leq\delta)}{\e^{-|x|\eps}} = 0.
}
By \mbox{Borel}--\mbox{Cantelli}, we conclude that $\liminf_{n\to\infty}  W(\hat\sigma,\RR_n^{\T^d})$ is almost surely at least $\delta$,
which implies $\hat\sigma\notin\RR^{\T_d}$.

On the other hand, suppose $\KL{\hat\sigma}{\pp} \leq \log d$.
Let $\eps>0$ be given.
We claim there is some $\hat\rho\in\hat\Sigma$ such that $ W(\hat\sigma,\hat\rho)\leq \eps$ and $\KL{\hat\rho}{\pp} < \log d$.
Indeed, let $\alpha\in(0,1)$ be sufficiently small that $ W(\hat\sigma,(1-\alpha)\hat\sigma + \alpha\pp) \leq \eps$.
By the convexity of $\KL{\cdot}{\pp}$, we have
\eq{
\KL{(1-\alpha)\hat\sigma+\alpha\pp}{\pp} &\leq (1-\alpha)\KL{\hat\sigma}{\pp} + \alpha\KL{\pp}{\pp}  \\
&= (1-\alpha)\KL{\hat\sigma}{\pp} < \log d,
}
thus making $\hat\rho = (1-\alpha)\hat\sigma+\alpha\pp$ the desired measure.
Now choose $\kappa>0$ such that $\KL{\hat\rho}{\pp}\leq \log d - 2\kappa$, so that \eqref{open_sanov} gives
\eq{
\liminf_{|x|\to\infty}\frac{1}{|x|}\log\P( W(\hat\rho,\hat\sigma_x)<\eps) \geq -\log d +2\kappa.
}
So let $n_0$ be sufficiently large that $e^{n_0\kappa}>1$ and
\eq{
\P( W(\hat\rho,\hat\sigma_x)<\eps) \geq \e^{-n_0(\log d - \kappa)} \quad \text{for $|x|=n_0$}.
}
We then have
\eq{
\E|\{x\in\T_d :\, |x|=n_0,  W(\hat\rho,\hat\sigma_x) < \eps\}| \geq d^{n_0}\e^{-n_0(\log d-\kappa)} > 1.
}
Therefore, each $x_0\in\T_d$ can be viewed as the root of some supercritical Galton--Watson process $\mathrm{GW}^{\hat\rho,\eps}_{x_0}$, defined as follows.
For any $x\in\T_d$, the $\mathrm{GW}^{\hat\rho,\eps}$-children of $x$ (not to be confused with the $\T_d$-children of $x$) are those $\T_d$-descendants $z> x$ such that
\eq{
|z| = |x| + n_0\quad\text{and} \quad W\bigg(\hat \rho,\frac{1}{n_0}\sum_{x<y\leq z} \delta_{U_y}\bigg) < \eps.
}
Furthermore, if we write $\T_d(x)$ for the subset of $\T_d$ consisting of $x$ and all its $\T_d$-descendants, then $\mathrm{GW}^{\hat\rho,\eps}_{x_1},\dots,\mathrm{GW}^{\hat\rho,\eps}_{x_n}$ are independent whenever $\T_d(x_1),\dots,\T_d(x_n)$ are disjoint.
Consequently, there is almost surely some $x_0\in\T_d$ for which $\mathrm{GW}^{\hat\rho,\eps}_{x_0}$ avoids extinction. 
That is, there is a sequence $x_0=z_0< z_1< z_2<\cdots$
such that $|z_k| = |z_0| + kn_0$, and
\eq{
 W\bigg(\hat \rho,\frac{1}{n_0}\sum_{z_{k-1}<y\leq z_k} \delta_{Y_y}\bigg) < \eps \quad \text{for every $k\geq1$}.
}
Since $ W(\hat\rho,\cdot)$ is convex, it follows that
\eq{
\limsup_{k\to\infty}  W(\hat\rho,\hat\sigma_{z_k}) \leq \limsup_{k\to\infty}\Big( \frac{|z_0|}{|z_k|} W(\hat\rho,\hat\sigma_{z_0}) + \frac{|z_k|-|z_0|}{|z_k|}\eps\Big) = \eps.
}
By compactness, $(\hat\sigma_{z_k})_{k\geq0}$ admits some subsequence converging to some $\hat\sigma_\infty\in\RR_\infty$, thereby giving
\eq{
 W(\hat\sigma,\RR_\infty^{\T_d}) \leq  W(\hat\sigma,\hat\rho) +  W(\hat\rho,\hat\sigma_\infty) \leq 2\eps.
}
As $\eps$ is arbitrary and $\RR_\infty^{\T_d}$ is closed, we conclude that $\hat \sigma$ almost surely belongs to $\RR_\infty^{\T_d}$.
\end{proof}

\begin{proof}[Proof of Proposition \ref{tree_prop_2}]
Now we assume $\pp=\Lambda$.
It is a routine exercise that $\mu_{\T_d}=\mathfrak{b}$ if and only if $\P(\tau_x=\mathfrak{b})\geq1/d$.
Furthermore, by definition we have $\tau(u)\geq \mathfrak{b}$ for Lebesgue-almost every $u\in[0,1]$.
Therefore, if $\hat\sigma\in\hat\Sigma$ is such that $\langle\hat\sigma,\tau\rangle=\mathfrak{b}$, we must have $\tau_*(\hat\sigma) = \delta_{\mathfrak{b}}$.
In particular, when $\mu_{\T_d}=\mathfrak{b}$, it follows that $\tau_*(\hat\sigma) = \delta_{\mathfrak{b}}$ for every $\hat\sigma\in\RR_\tau^{\T_d}$.
So the proposition holds in this case.

Otherwise, we have $\mu_{\T_d}>\mathfrak{b}$, and so there is a set $B\subset[0,1]$ of positive Lebesgue measure such that $\tau(u)<\mu_{\T_d}(\tau)$ for all $u\in B$.
Let $\Lambda_B$ be the uniform probability measure on $B$, and note that
\eq{
\KL{\Lambda_B}{\Lambda} = -\log\Lambda(B) < \infty.
}
Now suppose that $\hat\sigma$ is any element of $\hat\Sigma$ such that $\KL{\hat\sigma}{\Lambda}<\log d$.
Since $\KL{\cdot}{\Lambda}$ is convex wherever it is finite (strictly so, in fact), we can choose $\eps\in(0,1)$ sufficiently small that $\KL{(1-\eps)\hat\sigma+\eps\Lambda_B}{\Lambda}\leq\log d$.
Therefore, both $\hat\sigma$ and $(1-\eps)\hat\sigma+\eps\Lambda_B$ are candidate measures in the variational formula \eqref{linear_functional_tree}, which means
\eq{
\mu_{\T_d}(\tau)
\leq \langle \tau,(1-\eps)\hat\sigma+\eps\Lambda_B\rangle
= (1-\eps)\langle\tau,\hat\sigma\rangle + \eps\langle\tau,\Lambda_B\rangle
< \langle\tau,\hat\sigma\rangle.
}
We have thus shown that any minimizer $\hat\sigma\in\RR_\tau^{\T_d}$ must satisfy $\KL{\hat\sigma}{\Lambda}=\log d$.
Since $\RR_\tau^{\T_d}$ is a convex set, the strict convexity of $\KL{\cdot}{\Lambda}$ now implies $\RR_\tau^{\T_d}$ can contain at most one element.
On the other hand, we know $\RR_\tau^{\T_d}$ is nonempty, and so $\RR_\tau^{\T_d} = \{\sigma_\star\}$, where $\langle\tau,\sigma_*\rangle = \mu_{\T_d}(\tau)$ by definition, and the relative entropy $\KL{\sigma_\star}{\Lambda}=\log d$ is minimal by the discussion from before.
\end{proof}

Our final result is unrelated to FPP but follows from what we have done.

\begin{thm} \label{weak_to_strong_thm}
Let $\pp$ be a Radon probability measure on $[0,1]$.
Suppose $(\hat\sigma_j)_{j\geq1}$ is a sequence of Borel probability measures such that $\KL{\hat\sigma_j}{\pp}$ is uniformly bounded in $j$ by a finite constant.
If $\hat\sigma_j$ converges weakly to $\hat\sigma$ as $j\to\infty$, then 
\eq{
\lim_{j\to\infty}\langle f,\hat\sigma_j\rangle = \langle f,\sigma\rangle \quad \text{for every measurable $f\colon [0,1]\to\R$ satisfying \eqref{test_fnc_condition}}.
}
In particular, $\hat\sigma_j(B)\to\hat\sigma(B)$ for every measurable $B\subset[0,1]$.
\end{thm}

\begin{proof}
Suppose $\KL{\hat\sigma_j}{\pp}\leq C<\infty$ for all $j$.
Choose a positive integer $d$ such that $\log d\geq C$, so that $\hat\sigma_j\in\RR^{\T_d}(\pp)$ by Proposition \ref{tree_prop_1}.
Moreover, if $\hat\sigma_j\Rightarrow\hat\sigma$ as $j\to\infty$, then $\hat\sigma$ also belongs to $\RR^{\T_d}(\pp)$.
Therefore, the proof is completed by recalling Corollary \ref{weak_to_strong_cor}.
\end{proof}

\appendix
\chapter*{List of Symbols} \label{list_of_symbols}

\begin{longtable}{lp{0.85in}p{0.6\textwidth}}
{\sf Notation} & {\sf pg.~or (eq.)} & {\sf Description ($*$ means w.r.t.~$\tilde\tau_e$ when applicable)} \\ \hlineB{3}
\rowcolor{gray!25}\multicolumn{3}{l}{\sf Universal objects}\\
$(\Omega,\FF,\P)$ & \pageref{cps_def} & complete probability space on which all random variables are defined \\
$E(\Z^d)$ & \pageref{edges_def} & undirected edges of integer lattice $\Z^d$ \\
$\S^{d-1}$ & \pageref{unit_sphere_def} &unit sphere in $\R^d$ \\
$\mathbf{e}_i$    & \pageref{sbv_def} &  $i^\text{th}$ standard basis vector in $\R^d$ \\
$p_\cc(\Z^d)$ &\pageref{pczd_def} &critical probability for bond percolation on $\Z^d$ \\
$\T_d$ & \pageref{tree_d_def} & infinite complete $d$-ary tree \\
$\Lambda$ & \pageref{lebesgue_meas_def} &Lebesgue measure on $[0,1]$ \\
$\pp$ & \pageref{radon_meas_def} &Radon prob.~measure on $[0,1]$, often equal to $\Lambda$ \\
$\BBB_\bv$    & \pageref{bv_banach_space} &   sequences with bounded variation \\
$\ell^1(\N)$ &\pageref{ell1_banach_space} & summable sequences \vspace{0.5\baselineskip}\\

\rowcolor{gray!25}\multicolumn{3}{l}{\sf Random weights and random environment}\\
$\tau_e$ & \pageref{edge_weights_def} &nonnegative edge-weight, $e\in E(\Z^d)$ \\
$\tau_e^\pert$ &\pageref{pert_edge_weights_def} & real-valued edge-weight, $|\tau_e^\pert| \leq \af$ \\
$\tilde\tau_e$ &\pageref{tilde_edge_weights_def} & real-valued edge-weight, equal to $\tau_e+\tau_e^\pert$ \\
$\tau_x$ & \pageref{vertex_weight_def} & real-valued vertex-weight, $x\in\T_d$ \\
$U_e$, $U_x$ & \pageref{U_e_def_1}, \pageref{U_e_def_2} &$[0,1]$-valued, $\pp$-distributed random variable \\
$\TT$ &\pageref{tail_1} &tail sigma-algebra associated to $U_e$'s \\
$\mathfrak{T}_z$ &\eqref{shift_operator} &shift transformation of random environment \\
$\tau$ & \pageref{tau_measurable_def_1}, \pageref{tau_measurable_def_2} &coupling map, nonnegative in lattice setting \\
$\tau^\pert$ & \pageref{pert_fnc_1}, \pageref{pert_fnc_2} &coupling map, real-valued \\
$\tilde\tau$ & \pageref{pert_fnc_1}, \eqref{function_coupling_2} &coupling map, equal to $\tau+\tau^\pert$ \\
$F$    & \eqref{distribution_fnc_def} &  cumulative distribution function for $\tau_e$ \\
$\LL$ & \pageref{weight_law_def_1}, \pageref{weight_law_def_2} & law of $\tau_e$ or $\tau_x$ \\
$\LL\oplus h$ & \pageref{variable_law_1} & law of $\tau_e+h$ \\
$\LL\oplus h\one_{\{t>0\}}$ & \pageref{variable_law_2} & law of $\tau_e+h\one_{\{\tau_e>0\}}$ \\
$\PPP$    & \pageref{PPP_def} &space of Borel probability measures on $[0,\infty)$   \\
$\PPP_\mathrm{emp}(\bxi)$    & \pageref{good_def} &edge-weight laws satisfying Definition \ref{good_def}   \\
$\PPP_\mathrm{emp}^\infty(\bxi)$    & \pageref{good_def_infty}   &edge-weight laws satisfying Definition \ref{good_def_infty}   \\
$\PPP_\mathrm{length}(\bxi)$   & \pageref{good_length_def} &edge-weight laws satisfying Definition \ref{good_length_def} \\
$\mathfrak{b}$ & \pageref{nonzero_atom_eg}, \pageref{bfrak_def}  & essential infimum of $\LL$ \\
$\tf$ & \pageref{tfrak_def} & large constant depending on $\LL$ and $d$ \\ 
$\af$, $\qf$, $\mathfrak{s}$, $\mf$ &\pageref{eta_choice_lemma}, \pageref{T_hat_upper}, \pageref{geo_length_cor} &special constants depending on $\LL$ and $d$ \\
$C$, $c$ &\pageref{C_c_def} & large/small positive constants depending on $\LL$ and $d$, values can change from line to line \vspace{0.5\baselineskip} \\

\rowcolor{gray!25}\multicolumn{3}{l}{\sf Paths and their properties; geometry of $\Z^d$ }\\
$[x]$    & \pageref{integer_approx_def} & $\Z^d$-approximation of $x\in\R^d$   \\
$\mathbf{S}_n$    & \eqref{Sn_def} &   $\Z^d$-approximation of circle in $\R^d$ of radius $n$ \\
$\Bbf_r(x)$ &\eqref{box_def_a} &box of radius $r$ centered at $x$   \\
$\partial\Bbf_r(x)$ &\eqref{box_def_b} &boundary of $\Bbf_r(x)$   \\
$\PP(x,y)$    & \pageref{path_xy_def} & set of self-avoiding paths between $[x]$ and $[y]$  \\
$\PP(x)$    & \pageref{path_x_def} & set of self-avoiding paths starting at $x$  \\
$\Gamma^{(\ell)}$ & \eqref{sub_of_infinite} &first $\ell$ edges in the infinite path $\Gamma$ \\
$\gamma^{(a_0,a_1)}$ &\eqref{delete_edges_def} &$\gamma$ without its first $a_0$ edges and last $a_1$ edges \\
$|\gamma|$ &\pageref{gamma_length_def} &number of edges in $\gamma$ \\
$|\gamma|_0$ &\eqref{gamma_length_0_def} &number of zero-weight edges in $\gamma$ \\
$|\gamma|_+$ & \eqref{gamma_length_0_def} &number of positive-weight edges in $\gamma$ \\
$|x|$ &\pageref{gen_number_def} &generation number of $x\in\T_d$ \\
$\SS(x)$ & \pageref{shell_def} &shell around $x$, a random subset of $\Z^d$ \\
$\SS_\mathrm{int}(x)$ & \pageref{interior_shell_def} &``interior" of $\SS(x)$ \\
$\Rvc(x)$ &\eqref{radius_x_def} &max.~distance from $x$ to an element of $\SS(x)$ \\ 
$\Avc(x)$    & \eqref{LQ_def} &max.~length of s.a.~path before hitting $\SS(x)$ \\
$\OO$    & \pageref{infinite_open_cluster}, \pageref{infinite_open_cluster_2} &unique infinite cluster formed by open edges   \\
$\WW$    & \pageref{infinite_white_cluster} &unique infinite cluster formed by white vertices   \\
$D_\oo(x,y)$  & \eqref{chemical_def_1} &min.~length of path $x\leftrightarrow y$ using open edges  \\
$D_\ww(x,y)$  & \eqref{chemical_def_2} &min.~length of path $x\leftrightarrow y$ using white vertices \vspace{0.1\baselineskip} \\

\rowcolor{gray!25}\multicolumn{3}{l}{\sf Passage times and geodesics} \\
$T(\gamma)$    & \eqref{fpp_def} &sum of edge-weights $\tau_e$, $e\in\gamma$   \\
$\wt T(\gamma)$    & \eqref{fpp_def_perturbed} &sum of edge-weights $\tilde\tau_e$, $e\in\gamma$   \\
$T(x,y)$    &  \eqref{fpp_def} &first-passage time between $[x]$ and $[y]$, w.r.t.~$\tau_e$  \\
$\wt T(x,y)$    &  \eqref{fpp_def_perturbed} &first-passage time between $[x]$ and $[y]$, w.r.t.~$\tilde\tau_e$  \\
$\wh T(x,y)$ &\eqref{wh_T_def}$^*$ &first-passage time between $\SS([x])$ and $\SS([y])$ \\
$T_n$ &\eqref{tree_FPP_model} &first-passage time to $n^\text{th}$ generation of $\T_d$ \\
$\mu_\bxi$    & \eqref{time_constant_def}, \eqref{shell_to_mu}$^*$ &time constant in $\bxi$-direction \\
$\mu_{\T_d}$    & \eqref{tree_time_constant_def} &time constant on $\T_d$   \\
$B_0$ &\pageref{limit_shape_def} &FPP limit shape \\
$\Lvc(x,h)$    & \eqref{LQ_def} &max.~length of s.a.~path with $\wt T(\gamma)<h|\gamma|$ \\
$\Qvc(x,h;q)$    & \eqref{LQ_def} &max.~length of s.a.~path having $\geq q$ fraction of edges with weight $\tilde\tau_e\leq h$ \\
$\Delta_\mathrm{int}(x)$ &\eqref{deltas_def} &quantity to control difference between $\wt T$ and $\wh T$ \\
$\Delta_\mathrm{ext}(x)$ &\eqref{deltas_def} &correction term in subadditivity inequalities \\
$\Geo(x,y)$ &\pageref{geo_def}$^*$ &set of all geodesics between $[x]$ and $[y]$ \\
$\Geo_\infty$ &\pageref{geo_infty_def}$^*$ &set of all infinite geodesics \\
$\Geo_\infty(\bxi)$ &\pageref{geo_infty_bxi_def}$^*$ &set of all $\bxi$-directed infinite geodesics \\
$\Geo_n$ &\eqref{geo_n_def} &level-$n$ nodes in $\T_d$ with min.~passage time \\
$\barbelow{N}_n^\bxi$, $\barabove{N}_n^\bxi$ &\eqref{Ln_def} &min./max. length of geodesic $0\leftrightarrow[n\bxi]$ \vspace{0.5\baselineskip} \\

\rowcolor{gray!25}\multicolumn{3}{l}{\sf Empirical measures and their limits} \\
$\delta_t$ &\pageref{dirac_def} &Dirac delta measure at $t\in\R$ \\
$\langle f,\nu\rangle$ &\eqref{f_nu_def} &integral of $f$ with respect to finite measure $\nu$ \\
$\hat\nu$    & \pageref{normalization_def} &normalization of $\nu$ to a probability measure  \\
$\nu_\gamma$, $\sigma_\gamma$    & \eqref{nu_def}$^*$, \eqref{sigma_gamma} &partially normalized empirical measure along $\gamma$, w.r.t.~$\tau_e$ and $U_e$, respectively   \\
$\nu_{\gamma,\gamma'}$, $\sigma_{\gamma,\gamma'}$    &  \eqref{nu_minus_def}$^*$, \eqref{sigma_minus_def} &partially normalized empirical measure along subpath $\gamma\setminus\gamma'$, w.r.t.~$\tau_e$ and $U_e$, respectively   \\
$\nu_\gamma^+$    & \eqref{nu_def}$^*$ &partially normalized empirical measure along $\gamma$, w.r.t.~$\tau_e\neq 0$   \\
$\hat\nu_x$ &\eqref{empirical_measures_tree} &empirical measure along path to $x\in\T_d$, w.r.t.~$\tau_y$ \\
$\hat\sigma_x$ &\pageref{empirical_measures_tree} &empirical measure along path to $x\in\T_d$, w.r.t.~$U_y$ \\
$\tau_*(\sigma)$    & \eqref{push_def} &pushforward of $\sigma$ by map $\tau$  \\
$\hat\tau_*(\sigma)$    & \pageref{push_def} &normalization of $\tau_*(\sigma)$  \\
$\check\tau_*(\sigma)$    & \pageref{check_def} &normalization of $\tilde\tau_*(\sigma)$  \\
$\tau_*^+(\sigma)$    & \eqref{pos_push_def} &pushforward of $\sigma$ by $\tau$, restricted to $\{\tau\neq0\}$  \\
$\hat\tau_*^+(\sigma)$    & \pageref{push_def} &normalization of $\tau_*^+(\sigma)$  \\
$\check\tau_*^+(\sigma)$    & \pageref{check_def} &normalization of $\tilde\tau_*^+(\sigma)$  \\
$\Sigma$ &\pageref{Sigma_def} &space of finite, positive Borel measures on $[0,1]$ with total mass at least $1$ \\
$\hat\Sigma$ &\pageref{hat_Sigma_def} &space of Borel probability measures on $[0,1]$ \\
$\Sigma_M$ &\eqref{Sigma_M_def} &set of $\sigma\in\Sigma$ with total mass at most $M$  \\
$W(\cdot,\cdot)$ &\eqref{wasserstein_def} &Wasserstein distance, suitably extended to $\Sigma$ \\
$\TV(\cdot,\cdot)$ &\eqref{w_tv_1} &total variation distance, a metric on $\hat\Sigma$ \\
$\KL{\cdot}{\cdot}$    & \eqref{KL_div_def} &Kullback--Leibler divergence (relative entropy)   \\
$\KK(\Sigma_*)$ &\pageref{KK_def} &set of nonempty, closed subsets of $\Sigma_*$ \\
$\HH$ &\eqref{hausdorff_def} &Hausdorff distance on $\KK(\Sigma_*)$ \\
$\RR^\bxi_\infty$ &\eqref{R_infty_def} &all limits of empirical measures in $\bxi$-direction \\
$\RR_\infty$ &\eqref{directionless_defs} &union of $\RR^\bxi_\infty$ over $\bxi\in\S^{d-1}$ \\
$\RR_n^{\bxi,\eps}$ &\eqref{prelimit_to_R} &prelimit of $\RR^\bxi_\infty$ \\
$\RR_n$ &\eqref{directionless_defs} &prelimit of $\RR_\infty$ \\
$\RR^\bxi$ &\pageref{var_form_thm}, \pageref{deterministic_limits_thm} &constraint set in variational formula for $\mu_\bxi$, almost sure value of $\RR_\infty^\bxi$ \\
$\RR$ &\pageref{deterministic_limits_thm} &almost sure value of $\RR_\infty$ \\
$\RR^{\T_d}_\infty$ &\pageref{R_infty_tree_def} &all possible limits of empirical measures for $\T_d$ \\
$\RR_n^{\T_d}$ &\eqref{prelimit_tree} &prelimit of $\RR_\infty^{\T_d}$ \\
$\RR^{\T_d}$ &\pageref{R_infty_tree_def} &almost sure value of $\RR^{\T_d}_\infty$ \\ 
$\RR^\bxi_{\tau}$ &\eqref{minimizers_def} &minimizers in variational formula for $\mu_\bxi(\tau)$ \\
$\RR^{\T_d}_\tau$ &\eqref{minimizers_tree_def} &minimizers in variational formula for $\mu_{\T_d}(\tau)$
\end{longtable}

\backmatter
\bibliographystyle{amsplain}
\bibliography{empirical.bib}
\printindex

\end{document}